\documentclass[11pt, oneside]{book}
\usepackage[margin=1in]{geometry}

\usepackage[T1]{fontenc}
\usepackage[utf8]{inputenc}

\usepackage[parfill]{parskip}
\usepackage{enumerate}
\usepackage{xcolor}
\definecolor{aoenglish}{rgb}{0.0, 0.5, 0.0}

\usepackage{hyperref}
\hypersetup{
    colorlinks,
    citecolor=blue,
    filecolor=blue,
    linkcolor=blue,
    urlcolor=blue,
    bookmarksnumbered
}
\usepackage[backend=biber, style=alphabetic]{biblatex}
\usepackage{fancyhdr}
\usepackage[nottoc,notlot,notlof]{tocbibind}

\usepackage{makeidx}
\makeindex

\usepackage{bbm} 
\newcommand{\bbone}{\mathbbm{1}}

\usepackage{adjustbox} 

\usepackage{mymath}

\newif\ifverbose
\verbosefalse

\counterwithout{footnote}{chapter}

\author{Antonio Mejías Gil}
\title{A Main Conjecture in non-commutative Iwasawa theory}

\begin{document}

    \frontmatter

    {\pagestyle{plain}

\thispagestyle{empty}

\centerline{\Huge A Main Conjecture}

\vspace{1em}

\centerline{\Huge in non-commutative Iwasawa theory}

\vspace{5em}

This is the doctoral dissertation of
\vspace{1.5em}

\centerline{\huge Antonio Mejías Gil}

\centerline{\large (\texttt{antonio.mejias-gil@stud.uni-due.de})}

\vspace{0.5em}
submitted to the Faculty of Mathematics of the 
\vspace{1.5em}

\centerline{\Large University of Duisburg-Essen}

\vspace{.5em}
on November 2, 2022.

\vspace{3em}

\centerline{\Large\textbf{Abstract}}

    We formulate a new equivariant Main Conjecture in Iwasawa theory of number fields and study its properties. This is done for arbitrary one-dimensional $p$-adic Lie extensions $L_\infty/K$ containing the cyclotomic $\ZZ_p$-extension $K_\infty$ of the base field. As opposed to existing conjectures in the area, no requirement that $L_\infty/K$ be abelian or that $L_\infty$ be totally real is imposed. We prove the independence of the Main Conjecture of essentially all of its parameters and explore its functorial behaviour. It is furthermore shown that, to a large extent, this new conjecture generalises existing ones of Burns, Kurihara and Sano and Ritter and Weiss, which enables us to deduce its validity in several cases.


\newpage

\thispagestyle{empty}

\centerline{\Large\textbf{Acknowledgements}}

\vspace{1.5em}

    I am indebted to the following people for their role in this project:

    First and foremost, my supervisor Andreas Nickel, whose mathematical insight has been invaluable. Thank you for your immense generosity, dedication, patience and encouragement.

    My parents, my fiancée Marysia and my brother Gonza, who have supported me and made this journey so much more meaningful.
    
    The people from the ESAGA, especially Bence, Chirantan, Gürkan and Nils.

    On a more academic note, I would like to thank Alexandre Maksoud, Dominik Bullach, Otmar Venjakob and Sören Kleine for insightful conversations.

    I cannot forget my Master's advisor, Daniel Macías Castillo, without whom I might not have started a doctorate in Essen in the first place.

    Last but not least, many thanks goes to Werner Bley for agreeing to review this dissertation.


        \setcounter{page}{1}
        \tableofcontents

    \cleardoublepage}

\newpage
\chapter*{Introduction}
\addcontentsline{toc}{chapter}{Introduction}

    The study of class numbers occupies a central place in modern number theory. In this regard, much credit is due to Lamé's 1847 incorrect proof of Fermat's last theorem for prime exponent $p$, which failed to account for the failure of unique factorisation in what we recognise today as rings of integers of cyclotomic fields. Kummer realised that the argument could be salvaged by replacing the complex numbers with what he referred to as \textit{ideal numbers} under the assumption that $p$ does not divide the \textit{class number} of $\QQ(\zeta_p)$. This and many other examples showcase the usefulness of this invariant for solving Diophantine equations and, more broadly, understanding the arithmetic of more general rings than the integers.

    In many instances, the problem of determining the class number $h_K$ of an arbitrary number field $K$ is intractable in isolation. One of the most successful approaches is to instead consider its behaviour along a tower of such fields. The inception of this area of research can be traced back to the work of Iwasawa in the early 1960s, which resulted in an asymptotic law for the $p$-part of the class number of a number field $K$ along its \textit{cyclotomic $\ZZ_p$-extension} $K = K_0 \subseteq K_1 \subseteq \cdots \subseteq K_\infty$. This behaviour turned out to be governed by a certain module $X_{nr}$ over the so-called Iwasawa algebra $\Lambda(\Gamma)$, a completed group ring isomorphic to the power series ring $\ZZ_p[[T]]$. This sparked interest in the study of this and other related modules, which were eminently algebraic in nature, and in the question of whether they could be constructed using analytic methods as well. Iwasawa postulated, in what came to be known as his Main Conjecture, that the characteristic polynomial of $X_{nr}$ for the base field $K = \QQ$ coincides, up to a unit of $\Lambda(\Gamma)$, with a particular power series arising from the \textit{$p$-adic $L$-function} defined by Kubota and Leopoldt - an analytic object interpolating values of the Riemann zeta function.

    Given the depth and richness of both disciplines, the idea of relating algebraic and analytic invariants seems, with the benefit of hindsight, extremely natural. The quintessential example of such a connection is the analytic class number formula
    \[
        \lim_{s \to 1} (s - 1) \zeta_K(s) = \frac{2^{r_1} (2\pi)^{r_2} h_K R_K}{w_K \sqrt{\abs{d_K}}}.
    \]
    The left-hand side is the residue at 1 of the Dedekind zeta function of the number field $K$, which coincides with Riemann's celebrated function when $K = \QQ$. On the right-hand side are a number of algebraic invariants describing the arithmetic of $K$ - most conspicuously, its class number $h_K$, which is often the most challenging magnitude to determine directly. The class number formula therefore offers an alternative, often effective method to compute it.

    A vast generalisation of said formula comes in the form of the equivariant Tamagawa Number Conjecture (eTNC) proposed by Burns and Flach in \cite{bf_etnc}. Much research has been devoted to this conjecture, originally formulated for arbitrary motives. However, it is only Tate motives that concern this discussion, as they amount to certain extensions of number fields. In their case, the eTNC can intuitively be thought of as a global, finite-level version of a Main Conjecture - and yet making this connection precise is the subject of active research. On one hand, the limit of the $p$-part of the eTNC along an infinite tower should imply the validity of a Main Conjecture for that tower. On the other, the Main Conjecture for a given tower should yield a particular instance of the eTNC when passing to finite level. In practice, the latter is a more difficult problem due to \textit{descent issues}. The article \cite{bks} of Burns, Kurihara and Sano, one of the main inspirations for this work, explores precisely this question. More specifically, it is shown in it that eTNC can indeed be recovered from a specific Main Conjecture under the additional assumption of the Mazur-Rubin-Sano conjecture and a classical conjecture of Gross. Another example of the significance of the eTNC is the fact that the Birch and Swinnerton-Dyer conjecture also arises as a particular case of it, namely when the motive in question arises from an elliptic curve.

    Although it would be futile to try and list all existing Main Conjectures in Iwasawa theory of number fields, we point out a few milestones which may help put this text into perspective. For a much more authoritative survey on the topic, we refer to \cite{arithmetic_of_l}. Iwasawa's original conjecture described above was quickly extended to the case of the cyclotomic $\ZZ_p$-extension $L_\infty$ of a totally real finite Galois extension  $L$ of $K$. The Kubota-Leopoldt $p$-adic $L$-function still serves as the analytic object when $L$ is an abelian extension of $K = \QQ$, but in the general case it must be replaced by a new type of $p$-adic $L$-function constructed, independently, by Pierrette Cassou-Nogu\`es (\cite{cassou-nogues}), Deligne and Ribet (\cite{deligne_ribet}) and Barsky (\cite{barsky}). The earliest conjectures in this context were \textit{character-wise}. In other words, they claimed a relation between a quotient of power series coming from the $L$-function associated to a single irreducible character $\chi$ of $\Gal(L/K)$, and one arising from the \textit{$\chi$-part} of some Iwasawa module. The first major breakthrough in this direction came in 1984, when Mazur and Wiles settled the abelian case with $K = \QQ$ in \cite{mazur_wiles}. Wiles then extended these techniques and gave a proof of the general totally real case six years later (cf. \cite{wiles}). A completely different approach would later be put forward by Rubin, who gave a proof of the Main Conjecture with base field $\QQ$ making use of the notion of an Euler system introduced by Kolyvagin and ideas of Thaine. This tool also allowed him to prove a Main Conjecture for abelian extensions of \textit{imaginary quadratic} fields in \cite{rubin}.

    In the early 2000s, the first examples of \textit{equivariant} Main Conjectures made an appearance - due in large part to the interest in the eTNC. In this context, the term ``equivariant'' alludes to the fact that information about \textit{all characters} of some Galois group is considered simultaneously, by contrast to character-wise conjectures. One of the first examples of this is Burns and Greither's equivariant conjecture from \cite{burns_greither}, which they proved for abelian extensions of $\QQ$ and used to deduce the corresponding case of the eTNC. Shortly after, Ritter and Weiss formulated Main Conjecture for arbitrary totally real number fields in \cite{rwii} - with no abelianity requirement. A series of follow-up papers would culminate with a proof of their conjecture in the so-called $\mu = 0$ case (cf. \cite{rw_on_the}). Independently, Kakde formulated and proved in \cite{kakde_mc} a Main Conjecture in the totally real, $\mu = 0$ case, which additionally allowed for $\ZZ_p$-extensions of higher rank. Nickel (cf. \cite{nickel_plms}) and Venjakob (cf. \cite{venjakob_on_rw}) independently proved Kakde's conjecture to be equivalent to that of Ritter and Weiss. The case of imaginary quadratic base field initiated by Rubin did not stay isolated from this equivariant trend, with some cases of such conjectures being settled by Bley in \cite{bley} and, more recently, Bullach and Hofer in \cite{bullach_hofer}, as an intermediate step towards proving new instances of the eTNC.

    Research in this field remains very active. In 2017, Burns, Kurihara and Sano proposed the aforementioned Main Conjecture for abelian extensions of arbitrary number fields which, as a distinctive feature, is formulated in the language of determinant functors. We single out recent work of Dasgupta and Kakde proving important cases of the Brumer-Stark conjecture and hence of Hilbert's twelfth problem (see \cite{dasgupta_kakde_twelfth} and \cite{dasgupta_kakde}); and the subsequent proof of an abelian Main Conjecture in remarkable generality by Johnston and Nickel in \cite{johnston_nickel}.

    The aim of this work is to formulate a Main Conjecture which generalises, and in doing so also unifies, some of the above. Namely, our two conjectures of reference are those of Ritter and Weiss, and of Burns, Kurihara and Sano. There are essentially two aspects in favour this new conjecture's generality. Firstly, unlike in the conjecture of Ritter and Weiss, the field $L_\infty$ is not required to be totally real. As an immediate consequence, the existence of $p$-adic $L$-functions is not guaranteed, and must therefore be stated as part of the conjecture and its implications studied. A more subtle effect of this departure from the totally real case is that some of the relevant arithmetic modules are no longer torsion over the relevant Iwasawa algebra, which complicates the definition of an algebraic object in the conjecture. This is resolved through the use of refined Euler characteristics. Secondly, unlike in the conjecture of Burns, Kurihara, and Sano, we do not assume $\cG = \Gal(L_\infty/K)$ to be abelian. This has a profound impact on the algebraic machinery and representation theory involved. Furthermore, it prevents the conjecture from being stated in the language of classical determinant functors.

    As far as generality is concerned, a special mention is due to Fukaya and Kato's article \cite{fk}, which outlines a very general - and hence necessarily non-explicit - framework for the formulation of Main Conjectures \textit{for motives}. However, the discussion below does not seem to be fully covered by the proposed broad  motivic umbrella: the article in question concerns \textit{critical motives} only, which in the number field case corresponds to the non-vanishing of the relevant $L$-value - but the analytic objects in our conjecture interpolate \textit{leading coefficients}.

    This work is structured as follows: In the preparatory chapter 0, we establish some notation and lay out the basic tools needed throughout the text in four different areas: Iwasawa algebras and their modules, representation theory of finite groups, Artin $L$-series and algebraic $K$-theory of rings. The last of these is especially useful for illustrating the general outline of the conjecture, since it is in the so-called \textit{localisation sequence} of $K$-theory that the statement of the conjecture takes place. More specifically, two fundamental rings in our endeavours will be a certain Iwasawa algebra $\Lambda(\cG)$ and its total ring of fractions $\cQ(\cG)$, which is Artinian semisimple. Then, letting $Z(\cQ(\cG))$ denote the centre of the latter, $K$-theory provides homomorphisms of abelian groups
    \[
        \units{Z(\cQ(\cG))} \xleftarrow{\nr} K_1(\cQ(\cG)) \xrightarrow{\partial} K_0(\Lambda(\cG), \cQ(\cG))
    \]
    known as the \textit{reduced norm} and the \textit{boundary} (or \textit{connecting}) \textit{homomorphism}. The strategy is to define an analytic and an algebraic element in the first and last terms, respectively, and conjecture that they both arise as images of the same \textit{zeta element} in $K_1(\cQ(\cG))$.

    The aim of chapter \ref{chap:construction_of_the_complex} is to construct the \textit{main complex} of $\Lambda(\cG)$-modules $\cC_{S, T}\q$, which constitutes the central algebraic object of the Main Conjecture. It arises as the cone of a morphism from local to global complexes, both of which are instances of the \textit{translation functor} developed by Ritter and Weiss in \cite{rw_tate_sequence}. As proved in section \ref{sec:description_in_terms_of_rgamma_complexes}, the resulting complex is isomorphic in the derived category of $\Lambda(\cG)$-modules to that employed by Burns, Kurihara and Sano in \cite{bks} - so this chapter gives a new, arguably more explicit description of the previously known object. In the process of defining $\cC_{S, T}\q$, we will also become acquainted with a multitude of Iwasawa modules relevant to the sequel.

    Chapter \ref{chap:formulation_of_the_main_conjecture} is devoted to the formulation of the Main Conjecture and therefore occupies a distinguished place in the discussion. The first section recalls the definition of refined Euler characteristics and explains how one can be defined for our main complex by means of a certain map $\alpha$. It is through the use of this machinery that $\cC_{S, T}\q$, together with a \textit{trivialisation} constructed from that map, can be regarded as an element of the relative $K_0$ group $K_0(\Lambda(\cG), \cQ(\cG))$. Sections \ref{sec:morphisms_on_finite_level} and \ref{sec:regulators} delve into the definition of \textit{regulators}, complex $p$-adic numbers that the special $L$-values need to be divided by before they can stand a chance of being interpolated. These regulators, based on the work of Tate on Stark's conjectures \cite{tate}, measure the difference between a finite-level map induced by $\alpha$ and the classical Dirichlet regulator map. The next section continues this analytic trend by studying how $Z(\cQ(\cG))$ decomposes as a direct product of quotient fields of power series rings by virtue of structural results of Ritter and Weiss. Only in the last section is the Main Conjecture finally stated, which takes place in two steps: first an \textit{Interpolation Conjecture} asserting the existence of power series quotients in $Z(\cQ(\cG))$ which interpolate regulated special $L$-values; and then an \textit{equivariant Main Conjecture}, which claims the existence of a zeta element in $K_1(\cQ(\cG))$ which is mapped to these analytic objects as well as to the refined Euler characteristic of $\cC_{S, T}\q$. A version of this conjecture \textit{with uniqueness} is also formulated, in line with existing conjectures

    The properties of the Main Conjecture are the subject of chapter \ref{chap:properties_of_the_main_conjecture}, the longest by a substantial margin. The conjecture is shown to be unconditionally independent of all of its parameters - where we do not include the extension $L_\infty/K$ itself - with the possible exception of an arbitrarily chosen isomorphism ${\beta \colon \CC_p \isoa \CC}$. Independence of the latter is only proved assuming Stark's conjecture as formulated by Tate. Nonetheless, we also prove a weaker, yet unconditional, form of independence of $\beta$. The remainder of the chapter explores functoriality properties. We first show that the conjecture for $L_\infty/K$ implies those for $L_\infty'/K$ and $L_\infty/K'$ with $L_\infty' \subseteq L_\infty$ and $K' \supseteq K$. Then the converse problem is studied, resulting in a characterisation of the conjecture for $L_\infty/K$ in terms of those for two different families of subextensions. Many particular features of the formulation of the Main Conjecture find their justification in the proofs in this chapter.

    The final chapter draws a rigorous connection to the aforementioned conjectures of Ritter and Weiss and of Burns, Kurihara and Sano. In the first section, we show that our conjecture is equivalent to a modified version of that in \cite{bks} in the cases when both can be formulated - that is, $L_\infty/K$ abelian and $K_\infty/K$ cyclotomic. This modification amounts to the replacement of a claim ``for all characters'' by the same one ``for almost all characters'' (meaning all but finitely many), a feature of our conjecture which is necessary for some of the results in chapter \ref{chap:properties_of_the_main_conjecture}. Section \ref{sec:the_conjecture_of_ritter_and_weiss} focuses instead on the conjecture of Ritter and Weiss. We show that, when $L$ is a CM field containing a primitive $p$-th root of unity, the \textit{minus part} of our Main Conjecture for $L_\infty/K$ holds if and only if that of Ritter and Weiss does for $L_\infty\pl/K$. It follows from these equivalence results that our conjecture holds whenever $L_\infty/K$ is abelian if the base field is either $\QQ$ or, under some additional restrictions, imaginary quadratic; and its minus part also does for the aforementioned type of CM extensions under some assumptions.

    Two small appendices are included at the end in order not to interrupt the flow of the exposition. The first one contains a proof that refined Euler characteristics behave well under extension and restriction of scalars, a fact which is crucial for the functoriality proofs of section \ref{sec:functoriality}. The second one recalls some basic properties of determinant functors and shows two simple results which are used in the comparison to the conjecture of Burns, Kurihara and Sano.

    Several directions for further research present themselves as natural continuations of this work. Ideally, the cyclotomic assumption should be done away with completely. For the vast majority of our discussion, it is enough for $K_\infty/K$ to satisfy the \textit{weak Leopoldt conjecture} and for no finite places of $K$ to split completely in $L_\infty$ - two properties which notably hold in the cyclotomic case. Removing these two assumptions would, however, require substantial changes. If, in addition, the rank-one $p$-adic Lie group $\cG = \Gal(L_\infty/K)$ could be replaced by one of arbitrary higher rank, the resulting conjecture would in fact rank among the most general in the Iwasawa theory of number fields. Another, perhaps more relevant question, is that of the specific relation to the eTNC - as already mentioned, this is the motivation behind several of the most recent Main Conjectures. However, this problem would also necessitate some new ideas beyond what is contained below. Lastly, we would be remiss not to point out the most evident of challenges: that of a general proof of the conjecture. Although hardly within reach at the moment, that is indeed the ultimate goal. For every Main Conjecture dreams of being proved, after all.

    \mainmatter


\newpage
\chapter{Notation, conventions and preliminaries}
\label{chap:notation_conventions_and_preliminaries}

    This preparatory chapter summarises some properties of certain classical objects which will be featured repeatedly in the sequel: Iwasawa algebras, representations of finite groups, $K$-theory and Artin $L$-functions. The notation for many of these is also established here. We will be less detailed than in subsequent chapters and defer some of the explanations to other sources. None of the results in this chapter are original.

    Some conventions and notation in place throughout the entire text are the following:
    \begin{itemize}
        \item{
            $0 \in \NN$.
        }
        \item{
            Rings have a unit element different from 0 but are not necessarily commutative. Given a ring $R$, we denote its centre by $Z(R)$. Analogously, we denote the centre of a group $G$ by $Z(G)$.
        }
        \item{
            The left ideal generated by an element $r$ of a ring $R$ is denoted by $\ideal{r}_R = \set{sr : s \in R}$, or simply $\ideal{r}$ if the ring is clear from context.
        }
        \item{
            Given a ring $R$ and a positive $n \in \NN$, we denote the ring of $n$-by-$n$ matrices over $R$ by $M_n(R)$.
        }
        \item{
            Modules are left modules unless otherwise stated. A finite module is one of finite order, rather than a finitely generated one.
        }
        \item{
            All topological groups are assumed to be Hausdorff.
        }
        \item{
            Unadorned $\otimes$ of modules always denotes $\otimes_\ZZ$. However, on basic tensors $m \otimes n \in M \otimes_R N$ we always omit the ring.
        }
        \item{
            If $\varphi \colon R \to S$ is a ring homomorphism and $f \colon M \to N$ is a homomorphism of (left) $R$-modules, then $S \otimes_\varphi f$ denotes the extension of $f$ to $S \otimes_\varphi M \to S \otimes_\varphi N$ by the identity on $S$. We may write $S \otimes_R -$ instead of $S \otimes_\varphi -$ if $\varphi$ is clear from context.
        }
        \item{
            If $M$ is a module over a ring $R$, then $M[0]$ denotes the cochain complex consisting of $M$ in degree 0 and trivial modules elsewhere; and for $i \in \ZZ$, $M[i] = (M[0])[i]$ denotes the shift of $M[0]$ by $i$, which consists of $M$ \textit{placed in degree $-i$} (not $i$) and trivial modules elsewhere.
        }
        \item{
            Given a profinite group $G$, we denote by $G^{ab}$ its abelianisation - that is, the quotient of $G$ by the closure of the subgroup generated by all commutators. We also denote by $G^{ab}(p)$ the maximal abelian pro-$p$ quotient of $G$.
        }
        \item{
            If $G$ is a topological group, then $H \leq_o G$ (resp. $H \leq_c G$) denotes that $H$ is an open (resp. closed) subgroup of $G$. The corresponding notation for normal subgroups is $\trianglelefteq_o, \trianglelefteq_c$.
        }
        \item{
            Given a field $E$ of characteristic $0$, we denote by $E^c$ an algebraic closure of $E$ and set ${G_E = \Gal(E^c/E)}$.
        }
        \item{
            An (algebraic) \textbf{number field}\index{number field} is a finite extension of $\QQ$, that is, a global field of characteristic 0. Given a rational prime $p$, a \textbf{$p$-adic field}\index{padic field@$p$-adic field} is a finite extension of $\QQ_p$, that is, a non-archimedean local field of characteristic 0. If $K$ is a number field or a $p$-adic field, then $\cO_K$ denotes its ring of integers.
        }
        \item{
            Given a set of places $\Sigma$ of a number field $K$ and an algebraic extension $L/K$, the set of places of $L$ above places in $\Sigma$ is denoted by $\Sigma(L)$. If the place $w$ is above $v$, we may refer to $w$ as a \textbf{prolongation} of $v$. In particular, $\set{v}(L)$ denotes the set of prolongations of $v$ to $L$.
        }
        \item{
            Given a set of places $\Sigma$ of a number field $K$ and a tower of algebraic extensions $L/K'/K$, we say that $L/K'$ is unramified at $\Sigma$ if it is so at every place of $K'$ above $\Sigma$  (i.e. at $\Sigma(K')$), and analogously for the terminology \textit{completely split at $\Sigma$}. The phrase \textbf{$\Sigma$-ramified} means \textit{unramified outside $\Sigma$}.
        }
        \item{
            Given a place $v$ of a number field $K$, we denote the corresponding normalised absolute value by $\abs{-}_v$. The completion of $K$ with respect to $\abs{-}_v$ is denoted by $K_v$.
        }
        \item{
            We denote the residue field of a number field at a non-archimedean place $v$ by $\kappa(v)$ (the number field being implicit in $v$). The \textbf{(absolute) norm} $\fN(v)$ of $v$ is defined as $\abs{\kappa(v)}$.
        }
        \item{
            If $K$ is a number field, $L/K$ is a (not necessarily finite) Galois extension and $w$ is a place of $L$, then $\Gal(L/K)_w$ denotes the decomposition group of $L/K$ at $w$. If $w$ is non-archimedean, we denote the corresponding inertia group by $I_w$ (the extension being implicit). Thus, in our notation, $\Gal(L/K)_w/I_w \iso \Gal(\kappa(w)/\kappa(v))$ for archimedean places (where $v$ is the place of $K$ below $w$).
        }
        \item{
            Given an archimedean place $v$ of a number field $K$ and an algebraic extension $L/K$, we say that a prolongation $w$ of $v$ to $L$ is \textbf{ramified} if it is complex and $v$ is real. In other words, if $v$ does not split completely in $L/K$. If $L/K$ is Galois, this is in turn equivalent to $\Gal(L/K)_w$ having order $2$ (as opposed to being trivial).
        }
    \end{itemize}

    We write $\blacksquare$ at the end of proofs and $\square$ at the end of other roman-font environments unless they consist of a bulleted list.

    \newpage
    \section{Iwasawa algebras and modules}
    \label{sec:iwasawa_algebras_and_modules}

        We shall first recall some basic properties of Iwasawa algebras and their modules. These objects are at the heart of the algebraic machinery necessary for the construction of our complex of interest and the formulation of the Main Conjecture. All of the results presented are well known and can be found in \cite{nsw}, although a more specific reference will be given on several occasions.

        Let $p$ be a rational prime, $E$ a $p$-adic field and $G$ a profinite group. The ring of integers $\cO_E$ is Noetherian and local with maximal ideal generated by a uniformiser $\pi$. Furthermore, it is compact and complete with respect to its $\pi$-adic topology and $\cO_E/\pi^n\cO_E$ is finite of order a power of $p$ for all $n \in \NN$. The \textbf{Iwasawa algebra}\index{Iwasawa algebra} of $G$ with coefficients in $\cO_E$ is the completed group ring
        \[
            \Lambda^{\cO_E}(G) = \cO_E[[G]] = \varprojlim_{U \trianglelefteq_o G} \cO_E \Big[\faktor{G}{U}\Big]
        \]
        with transition maps induced by the projections $G/U \sa G/V$ for $U \leq V$. Note that $\cO_E$ embeds canonically as a subring of $\Lambda^{\cO_E}(G)$ and $G$ embeds canonically as a subgroup of the units $\units{\Lambda^{\cO_E}(G)}$. If $G$ is finite, this Iwasawa algebra coincides with the usual group ring $\cO_E[G]$. The superscript $\ZZ_p$ is often omitted in $\Lambda(G) = \Lambda^{\ZZ_p}(G)$ and related objects such as $\Delta(G) = \Delta^{\ZZ_p}(G)$ (see below).

        We endow $\Lambda^{\cO_E}(G)$ with the topology determined by the basis of open subgroups
        \[
            \bigg\{\Delta_{n, U} = \ker\bigg(\Lambda^{\cO_E}(G) \sa \faktor{\cO_E}{\pi^n \cO_E} \Big[\faktor{G}{U}\Big]\bigg) \colon n \in \NN, U \trianglelefteq_o G \bigg\}.
        \]
        By a $\Lambda^{\cO_E}(G)$-module we always mean a (Hausdorff left) \textit{topological} module, which amounts to a topological $\cO_E$-module with a continuous $G$-action. Similarly, by a homomorphism of $\Lambda^{\cO_E}(G)$-modules we always mean a continuous such homomorphism. It is easy to show that topologically finitely generated modules and finitely generated modules coincide and they are compact; and that any abstract finitely generated $\Lambda^{\cO_E}(G)$-module can be endowed with a \textit{unique} topology which makes it into a compact topological module (cf. \cite{nsw} section V\S II). Compact modules over Iwasawa algebras are known as \textbf{Iwasawa modules}\index{Iwasawa module}.

        Given $H \trianglelefteq_c G$, the \textbf{augmentation map}\index{augmentation map} of $\Lambda^{\cO_E}(G)$ with respect to $H$ is the canonical continuous surjection $\aug_H \colon \Lambda^{\cO_E}(G) \sa \Lambda^{\cO_E}(G/H)$. Its kernel $\Delta^{\cO_E}(G, H)$ fits into the short exact sequence of $\Lambda(G)$-modules
        \begin{equation}
        \label{eq:general_augmentation_ideal}
            0 \to \Delta^{\cO_E}(G, H) \to \Lambda^{\cO_E}(G) \xrightarrow{\aug_H} \Lambda^{\cO_E}\Big(\faktor{G}{H}\Big) \to 0.
        \end{equation}
        and is therefore a two-sided closed ideal known as the \textbf{augmentation ideal}\index{augmentation ideal}. When $H = G$, we denote it simply by $\Delta^{\cO_E}(G)$ and we write $\aug$ for $\aug_G$. One has $\Delta^{\cO_E}(G, H) = \Delta^{\cO_E}(H) \Lambda^{\cO_E}(G)$.

        With $H$ as above, the \textbf{module of $H$-invariants} $M^H$\index{invariants}\index{Iwasawa module!of invariants} of a $\Lambda^{\cO_E}(G)$-module $M$ is defined as the largest submodule with trivial $H$-action. It consists of all elements of $M$ fixed by $H$. The \textbf{module of $H$-coinvariants}\index{coinvariants}\index{Iwasawa module!of coinvariants} of $M$ is the largest quotient module with trivial $H$-action, and it is given by
        \begin{equation}
        \label{eq:definition_coinvariants}
            M_H = \faktor{M}{\Delta^{\cO_E}(G, H)M} = \faktor{M}{\Delta^{\cO_E}(H)M}
        \end{equation}
        as the denominator is generated by elements of the form $\sigma m - m$ with $\sigma \in H, m \in M$. Both $M^H$ and $M_H$ inherit a natural $\Lambda^{\cO_E}(G/H)$-module structure (cf. equation \eqref{eq:general_augmentation_ideal}). The notation for the module of coinvariants only features $H$, since it is irrelevant whether $M$ is regarded as a $\Lambda^{\cO_E}(G)$-module or one over $\Lambda^{\cO_E}(H)$ in the first place.

        The only modules over Iwasawa algebras which will be of interest to us are compact and discrete ones. Every compact $\Lambda^{\cO_E}(G)$-module $M$ coincides with the inverse limit of its modules of coinvariants
        \[
            M = \varprojlim_{U \trianglelefteq_o G} M_U
        \]
        and is also a pro-$p$ group (cf. \cite{nsw} proposition 5.2.4). Conversely, every discrete $\Lambda^{\cO_E}(G)$-module $N$ coincides with the direct limit of its modules of invariants
        \[
            N = \varinjlim_{U \trianglelefteq_o G} N^U
        \]
        and is also a $p$-torsion group (by which we mean every element has order a power of $p$). This points to some manner of duality between the two, which is formalised in the notion of \textbf{Pontryagin duality}\index{Pontryagin duality}. Given a locally compact (Hausdorff) abelian group $A$, we define its Pontryagin dual as
        \[
            A^\vee = \Hom_{cts}(A, \RR/\ZZ)
        \]
        where the unit circle $\RR/\ZZ$ is endowed with the quotient topology of the usual topology on $\RR$. The compact-open topology (cf. \cite{nsw} section I\S 1) makes $A^\vee$ into a locally compact abelian group, and thus $-^\vee$ induces a contravariant autoduality on the category of locally compact abelian groups. It restricts to a contravariant duality between the categories of compact abelian groups and discrete abelian groups.

        If $M$ is a locally compact $\Lambda^{\cO_E}(G)$-module, then $M^\vee$ inherits a natural $\Lambda^{\cO_E}(G)$-module structure with $o \in \cO_{E}$ acting as $(of)(m) = f(om)$ and $\sigma \in G$ acting as $(\sigma f)(m) = f(\sigma^{-1} m)$. This action is compatible with Pontryagin duality, which therefore induces a contravariant autoduality on the category of locally compact $\Lambda^{\cO_E}(G)$-modules and a contravariant duality between compact and discrete $\Lambda^{\cO_E}(G)$-modules. In particular, it is exact on short exact sequences of $\Lambda^{\cO_E}(G)$-modules.

        It is easy to verify that, for $U \trianglelefteq_o G$, one has natural isomorphisms $(M_U)^\vee \iso (M^\vee)^U$ for $M$ compact; and $(N^U)^\vee \iso (N^\vee)_U$ for $N$ discrete. Note that, since compact and discrete $\Lambda^{\cO_E}(G)$-modules are pro-$p$ and $p$-torsion respectively, the functor $-^\vee$ coincides with $\Hom_{cts}(-, \QQ_p/\ZZ_p)$ on these two classes of modules, where $\QQ_p/\ZZ_p$ is the Prüfer group.

        The usual notion of the tensor product is often not sufficient in the theory of Iwasawa modules. Suppose, for instance, that $M$ is a compact right $\Lambda^{\cO_E}(G)$-module and $N$ is a compact left one. If both modules have infinitely many open neighbourhoods of $0$ and we choose two countable families of nested ones $M \supseteq U_0 \supseteq U_1 \supseteq \cdots$ and $N \supseteq V_0 \supseteq V_1 \supseteq \cdots$ and two families of elements $\set{m_i \in U_i}_{i \in \NN}$ and $\set{y_i \in V_i}_{i \in \NN}$, then $\sum_{i = 0}^\infty x_i \otimes y_i$ will in general not converge to an element in $M \otimes_{\Lambda^{\cO_E}(G)} N$ even though the terms become arbitrarily small. This is addressed by introducing the \textbf{complete tensor product}\index{complete tensor product}
        \[
            M \ctp_{\Lambda^{\cO_E}(G)} N = \varprojlim_{U, V} \faktor{M}{U} \otimes_{\Lambda^{\cO_E}(G)} \faktor{N}{V},
        \]
        where $M$ and $N$ are compact as above and $U$ (resp. $V$) runs over the compact open submodules of $M$ (resp. $N$). This is itself a compact $\cO_E$-module with respect to the inverse limit topology, a basis of which is given by the images of the subgroups $M \otimes V + U \otimes N$ with $U$ and $V$ as above (it is in fact the completion of $M \otimes_{\Lambda^{\cO_E}(G)} N$ with respect to that topology). There is also a natural monomorphism $M \otimes_{\Lambda^{\cO_E}(G)} N \ia M \ctp_{\Lambda^{\cO_E}(G)} N$. These and other properties of the construction can be found in \cite{nsw} section V\S2. Of particular relevance to us is the fact that, if $M$ or $N$ is finitely presented, then $M \otimes_{\Lambda^{\cO_E}(G)} N$ is already complete and therefore the natural monomorphism from before is in fact an equality
        \begin{equation}
        \label{eq:completed_tensor_product_equal}
            M \otimes_{\Lambda^{\cO_E}(G)} N = M \ctp_{\Lambda^{\cO_E}(G)} N.
        \end{equation}

        The complete tensor product allows us to define a \textbf{compact induction}\index{compact induction} for Iwasawa modules. Namely, if $H \leq_c G$ is a closed subgroup and $M$ is a compact left $\Lambda^{\cO_E}(H)$-module, then we let
        \[
            \Ind_H^G M = \Lambda^{\cO_E}(G) \ctp_{\Lambda^{\cO_E}(H)} M.
        \]
        This is itself a compact left $\Lambda^{\cO_E}(G)$-module (the action being on the first term) with the above topology, and thus compact induction defines an additive functor from the category of compact left $\Lambda^{\cO_E}(H)$-modules to that of compact left $\Lambda^{\cO_E}(G)$-modules (with continuous homomorphisms in both cases).

        Suppose $N$ is already a compact left $\Lambda^{\cO_E}(G)$-module and $f \colon M \to N$ is a $\Lambda^{\cO_E}(H)$-homomorphism with $M$ as above (where $N$ has restricted scalars). Then $f$ induces a natural homomorphism of $\Lambda^{\cO_E}(G)$-modules
        \begin{equation}
        \label{eq:partial_induction}
            \Ind_H^G f \colon \Ind_H^G M \to N.
        \end{equation}
        This should not be confused with the homomorphism $\Ind_H^G M \to \Ind_H^G N$ given by functoriality, which is usually denoted by $\Ind_H^G f$ as well. We will clarify what $\Ind_H^G f$ refers to wherever ambiguity exists.

        The following properties of induction will be used on occasion:
        \begin{lem}
        \label{lem:properties_of_induction}
            Let $G$ be a profinite group, $H \leq_c G$ a closed subgroup and $M$ a compact left $\Lambda^{\cO_E}(H)$-module. Then
            \begin{enumerate}[i)]
                \item{
                    $\Ind_H^G M \iso \varprojlim_{U \trianglelefteq_o G} \Ind^{G/U}_{H/H \cap U} M_{H \cap U}$ as $\Lambda^{\cO_E}(G)$-modules.
                }
                \item{
                    $(\Ind_H^G M)^\vee \iso \Map_{cts, H}(G, M^\vee)$ as $\Lambda^{\cO_E}(G)$-modules, where a continuous map $f \colon G \to M^\vee$ belongs to $\Map_{cts, H}(G, M^\vee)$ if $f(\tau\sigma) = \tau f(\sigma)$ for all $\tau \in H, \sigma \in G$. The $G$-action on $\Map_{cts, H}(G, M^\vee)$ is given by $(\sigma f)(\sigma') = f(\sigma' \sigma)$.
                }
            \end{enumerate}
        \end{lem}
        Both facts can be found in \cite{nsw} section XI\S3. Since a compact $\Lambda^{\cO_E}(H)$-module $M$ coincides with $\varprojlim_{U \trianglelefteq_o H} M_U$, point i) in the lemma can be worded as: \textit{induction commutes with inverse limits}. This will become meaningful in the introduction to chapter \ref{chap:construction_of_the_complex}.

        In the rest of this section we consider only the case $E = \QQ_p$, although most of the following facts still hold for any $E$ as above. Turning our attention briefly to cohomological considerations, let $M$ be a discrete $\Lambda(G)$-module. One can construct \textbf{cohomology groups} $H^i(G, M)$ for $i \geq 0$ in several equivalent ways, for instance as the right derived functors of the $G$-invariants functor ${-^G \iso \Hom_{\Lambda(G)}(\ZZ_p, -)}$ in the category of discrete $\Lambda(G)$-modules (which has sufficiently many injectives). Various definitions can be found in \cite{nsw} sections I\S2, II\S5 and V\S2. The resulting cohomology groups are canonically isomorphic to the direct limit $H^i(G, M) \iso \varinjlim_{U \trianglelefteq_o G} H^i(G/U, M^U)$ of classical cohomology groups and have a natural discrete $\Lambda(G)$-module structure. A short exact sequence $M' \ia M \sa M''$ of discrete $\Lambda(G)$-modules induces the usual long exact cohomology sequence
        \begin{equation}
        \label{eq:long_exact_sequence_cohomology_iwasawa}
            0 \to (M')^G \to M^G \to (M'')^G \to H^1(G, M') \to H^1(G, M) \to H^1(G, M'') \to H^2(G, M') \to \cdots
        \end{equation}
        of discrete $\Lambda(G)$-modules.

        Cohomology groups find their dual notion in the homology groups $H_i(G, N)$ given by the left derived functors of the $G$-coinvariants functor $-_G \iso \ZZ_p \ctp_{\Lambda(G)} -$ on a compact module $N$. These are again compact modules and satisfy the following relation:
        \begin{lem}
        \label{lem:cohomology_homology_duality}
            Let $G$ be a profinite group and $N$ a compact left $\Lambda(G)$-module. Then there exist functorial isomorphisms of $\Lambda(G)$-modules
            \[
                H^i(G, N^\vee) \iso H_i(G, N)^\vee
            \]
            for all $i \geq 0$.
        \end{lem}
        For a proof, see \cite{nsw} theorem 2.6.9.

        When studying Iwasawa modules which are arithmetic in nature, the notion of Tate twists is sometimes relevant. Let $L/K$ be a Galois extension of fields of characteristic 0 such that $L$ contains the group $\mu_{p^\infty}$ of all $p$-power roots of unity in some algebraic closure of $K$. Then the \textbf{$p$-adic cyclotomic character}\index{cyclotomic character}\index{character!cyclotomic} of $G$ is defined as the unique group homomorphism
        \[
            \chi_{\tcyc} \colon \Gal(L/K) \to \units{\ZZ_p}
        \]
        such that, for all $\sigma \in \Gal(L/K)$, one has $\sigma(\zeta) = \zeta^{\chi_{\tcyc}(\sigma)}$ for all $\zeta \in \mu_{p^\infty}$ - which makes it automatically continuous. In other words, $\chi_{\tcyc}$ captures the action of $\Gal(L/K)$ on the $p$-power roots of unity.  It is injective if and only if $L = K(\mu_{p^\infty})$. Surjectivity is equivalent to $K \cap \mu_{p^\infty} = \set{1}$ if $p$ is odd, and to $K \cap \mu_{p^\infty} = \set{\pm 1}$ if $p = 2$.

        Suppose now that $G = \Gal(L/K)$ is as above and $M$ is a $\Lambda(G)$-module. For $i \in \ZZ$, the \textbf{$i$-th Tate twist}\index{Tate twist} of $M$, denoted by $M(i)$, is defined as the $\ZZ_p$-module $M$ endowed with the $G$-action
        \[
            \sigma \cdot^i m = \chi_{\tcyc}^i (\sigma) \sigma m,
        \]
        with $\chi_{\tcyc}(\sigma)^i \sigma \in \Lambda(G)$ acting on $m \in M$ via the original action. It follows immediately that $M(0) = M$ and $M(i)(j) = M(i + j) = M(j)(i)$ for all $i, j \in \ZZ$. If $\QQ_p/\ZZ_p$ and $\ZZ_p$ are given the trivial $G$-action, then
        \begin{equation}
        \label{eq:tate_twist_rou}
            \faktor{\QQ_p}{\ZZ_p}(1) \iso \mu_{p^\infty} = \varinjlim_n \mu_{p^n} \quad \text{and} \quad \ZZ_p(1) \iso \varprojlim_n \mu_{p^n},
        \end{equation}
        where $\mu_{p^n}$ is the $p^n$-torsion of $\mu_{p^\infty}$. Pontryagin duality changes the sign of Tate twists: if $M$ is a locally compact $\Lambda(G)$-module, then $M(i)^\vee$ is canonically isomorphic to $M^\vee(-i)$. This comes from inversion of the $G$-action on duals: $(\sigma f)(m) = f(\sigma^{-1} m)$.

        We now specialise to the \textit{classical case} $G = \Gamma \iso \ZZ_p, E = \QQ_p$. There is a fundamental non-canonical isomorphism of topological $\ZZ_p$-algebras
        \begin{equation}
        \label{eq:iso_iwasawa_algebra}
            \Lambda(\Gamma) \iso \ZZ_p[[T]],
        \end{equation}
        where $\ZZ_p[[T]]$ (the ring of formal power series) is endowed with the $\ideal{p, T}$-adic topology (see \cite{nsw} proposition 5.3.5 for a proof). This isomorphism arises from the choice of a topological generator $\gamma$ of $\Gamma$ (which we shall always write multiplicatively) by mapping $\gamma \to T + 1$. The diagram
        \begin{center}
            \begin{tikzcd}
            \ZZ_p[\Gamma/\Gamma^{p^{n + 1}}] \arrow[r, "\sim"] \arrow[d, two heads] & \faktor{\ZZ_p[T]}{\ideal{w_{n + 1}}} \arrow[d, two heads] \arrow[r, "\sim"] & \faktor{\ZZ_p[[T]]}{\ideal{w_{n + 1}}} \arrow[d, two heads] \\
            \ZZ_p[\Gamma/\Gamma^{p^n}] \arrow[r] \arrow[r, "\sim"]                  & \faktor{\ZZ_p[T]}{\ideal{w_{n + 1}}} \arrow[r, "\sim"]                      & {\faktor{\ZZ_p[[T]]}{\ideal{w_n}}}
            \end{tikzcd}
        \end{center}
        with $w_i = (T + 1)^{p^i} - 1$ then commutes and yields an isomorphism on inverse limits. We denote the successive quotients of the $w_i$ for $i \geq 1$ by $\xi_i = w_i / w_{i - 1}$, all of which are irreducible in $\ZZ_p[T]$ and hence, although not obvious, in $\ZZ_p[[T]]$. We also set $\xi_0 = w_0 = T$. The $\xi_i$ are known as \textbf{cyclotomic polynomials}. Both $w_i$ and $\xi_i$ are examples of \textbf{Weierstrass polynomials}\index{Weierstrass polynomial}, which are defined as monic polynomials whose non-leading coefficients lie in $p\ZZ_p$ - that is, they become a power of $T$ after projecting to $\FF_p[[T]]$. A recurring idea in chapter \ref{chap:formulation_of_the_main_conjecture} will be the fact that, under the isomorphism $\Lambda(\Gamma) \iso \ZZ_p[[T]]$, the augmentation map $\aug = \aug_\Gamma$ corresponds to evaluation at $T = 0$. In particular, the latter is independent of the non-canonical choice of $\gamma$, since so is $\aug$.

        This isomorphism with the ring of formal powers series allows one to deduce a plethora of properties of the classical Iwasawa algebra $\Lambda(\Gamma)$. It is an integrally closed, Noetherian unique factorisation domain and a two-dimensional regular local ring whose maximal ideal is $\ideal{p, \gamma - 1}$ (the kernel $\Delta_{1, \Gamma}$ of the projection $\Lambda(\Gamma) \sa \FF_p$) and whose height-one prime ideals are all principal, generated by either $p$ or an irreducible Weierstrass polynomial. We denote the localisation of $\Lambda(\Gamma)$ at a prime ideal $\fp$ by $\Lambda_\fp(\Gamma)$.

        The \textbf{support} of a $\Lambda(\Gamma)$-module $M$ is
        \[
            \supp(M) = \set{\fp \et{a prime ideal of} \Lambda(\Gamma) : M_\fp \neq 0},
        \]
        which is finite if $M$ is finitely generated and torsion. $M$ is said to be \textbf{pseudo-null}\index{Iwasawa module!pseudo-null} if $\supp(M)$ contains no height-one prime ideals, which is equivalent to $M$ being finite (cf. \cite{nsw} remark after definition 5.1.4). A homomorphism $f \colon M \to N$ of $\Lambda(\Gamma)$-modules is said to be a \textbf{pseudo-isomorphism}\index{pseudo-isomorphism} if the localisation $f_\fp$ at every height-one prime ideal $\fp$ is an isomorphism of $\Lambda_\fp(\Gamma)$-modules, i.e. if $\ker(f)$ and $\coker(f)$ are finite. We often denote this by $f \colon M \xrightarrow{\approx} N$, which defines a reflexive and transitive relation on $\Lambda(\Gamma)$-modules. It is furthermore symmetric (and thus an equivalence relation) on finitely generated \textit{torsion} Iwasawa modules. One of the central results of classical Iwasawa theory is the following classification:
        \begin{thm}[Structure theorem for Iwasawa modules]
        \label{thm:structure_theorem_iwasawa}
            Let $M$ be a finitely generated $\Lambda(\Gamma)$-module, where $\Gamma \iso \ZZ_p$. Then there exists a pseudo-isomorphism of (topological) $\Lambda(\Gamma)$-modules
            \begin{equation}
            \label{eq:structure_theorem}
                M \xrightarrow{\approx} \Lambda(\Gamma)^r \oplus \bigoplus_{i = 1}^s \faktor{\Lambda(\Gamma)}{\ideal{p^{m_i}}} \oplus \bigoplus_{j = 1}^t \faktor{\Lambda(\Gamma)}{\ideal{F_j^{l_j}}}
            \end{equation}
            for some $r, s, t \in \NN$, $m_i, l_j > 0$ and irreducible Weierstrass polynomials $F_j$, and these are all unique up to order.
        \end{thm}
        This is \cite{nsw} theorem 5.3.8. Here $r = \rank_{\Lambda(\Gamma)} M$ and one defines the \textbf{Iwasawa invariants} $\mu(M) = \sum_{i = 1}^s m_i$ and $\lambda(M) = \sum_{j = 1}^t l_j \deg(F_j)$, as well as the characteristic polynomial\footnote{In some sources, the characteristic polynomial is scaled by $p^{\mu(M)}$.} ${\Char(M) = \prod_{j = 1}^t F_j^{l_j}}$. If $M$ is a torsion module, then the height-one prime ideals in $\supp(M)$ are precisely the $\ideal{F_j}$, and potentially $\ideal p$ (if $\mu(M) \neq 0$). If the pseudo-isomorphism in \eqref{eq:structure_theorem} is an isomorphism, $M$ is said to be \textbf{elementary}\index{Iwasawa module!elementary}.

        Given a short exact sequence $M' \ia M \sa M''$ of $\Lambda(\Gamma)$-modules, an easy application of the snake lemma to the endomorphism of each module given by multiplication by $\gamma - 1$ yields the \textbf{invariants-coinvariants exact sequence}\index{invariants-coinvariants exact sequence}
        \begin{equation}
        \label{eq:invariants_coinvariants}
            0 \to (M')^\Gamma \to M^\Gamma \to (M'')^\Gamma \to M_\Gamma' \to M_\Gamma \to M_\Gamma'' \to 0.
        \end{equation}
        Here one may replace $\Gamma$ by $\Gp{n}$ either by considering the endomorphism $\gp{n} - 1$ instead, or by noting that $\Gp{n} \iso \ZZ_p$ and therefore $\Lambda(\Gp{n})$ is itself a classical Iwasawa algebra.

    \newpage
    \section{Representations of finite groups}
    \label{sec:representations_of_finite_groups}

        Group representations and group algebras will play a central role in the subsequent discussion. Both finite and infinite groups will be involved in the formulation of the Main Conjecture, yet much of the necessary machinery will be concerned with finite groups alone - even for infinite groups, we shall only consider characters which factor through a finite quotient. The following definitions and properties can be found in any standard reference on the topic, for instance \cite{cr1} (primarily subsections \S9A and \S9B).

        Let $G$ be a finite group and $E$ a field. We assume $\Char(E) = 0$, although this requirement can be weakened to $\Char(E) \nmid \abs{G}$ in most instances. A \textbf{representation $(V, \rho)$ of $G$ over $E$}\index{representation} is a non-trivial $E$-vector space $V$ together with a group homomorphism $\rho \colon G \to \GLg(V)$. This is equivalent to endowing $V$ with a left $E[G]$-module structure with $G$-action via $\sigma v = \rho(\sigma)(v)$ for $\sigma \in G, v \in V$, and both descriptions will be used interchangeably. We often denote the representation simply by $V$ or $\rho$, the other element being implicit. The dimension of $V$ is $d = \dim_E V$, which will always be finite in this text. After choosing a basis, this representation corresponds uniquely to a homomorphism $\rho \colon G \to \GLg_d(E)$. Two representations with chosen bases $\rho, \rho' \colon G \to \GLg_d(E)$ are said to be \textbf{equivalent} if they differ by a change of basis, i.e. if $\rho'(\sigma) = M \rho(\sigma) M^{-1}$ for all $\sigma \in G$ for a fixed $M \in \GLg_d(E)$.

        The \textbf{character}\index{character} $\chi$ of $(V, \rho)$ is the map
        \begin{align*}
            \chi \colon & G         \to     E \\
                        & \sigma    \mapsto \Tr(\rho(\sigma)),
        \end{align*}
        which is well defined regardless of the choice of a basis for $V$. In particular, $\chi(1)$ is the dimension of the representation. We say $V$ \textbf{affords} $\chi$. Since the trace is invariant under conjugation, characters are \textbf{class functions}, i.e. they are invariant under conjugation by elements of the group. If $\rho$ is linear (i.e. $\dim_E(V) = 1$), then $\chi$ and $\rho$ coincide and the character becomes a group homomorphism $\chi \colon G \to \units{E}$. Two representations have the same character if and only if they are equivalent, which in turn happens if and only if the associated vector spaces are isomorphic as $E[G]$-modules.

        Let $\rho \colon G \to \GLg_d(E)$ be a representation with character $\chi$ and choose $\sigma \in G$. Since $G$ is finite, $\rho(\sigma)$ has finite (multiplicative) order and thus all of its eigenvalues in an algebraic closure $E^c$ of $E$ are roots of unity. In particular, $\chi(\sigma) \in E$ is a sum of roots of unity in $E^c$. The \textbf{kernel} of $\chi$ is defined as the kernel of $\rho$ as a group homomorphism, and it can be shown that
        \[
            \ker(\chi) = \set{\sigma \in G : \chi(\sigma) = \chi(1)}.
        \]
        The \textbf{trivial representation} is the linear representation which takes the constant value 1, i.e. it makes every element $\sigma \in G$ act as the identity on $V = E$. We denote its character by $\bbone_G$, or $\bbone$ if the group is clear from context. In particular, $\ker(\bbone_G) = G$.

        A representation $(V, \rho)$ of $G$ over $E$ (or, equivalently, its character $\chi$) naturally gives rise to several other representations:
        \begin{itemize}
            \item{
                If $G$ is a subgroup of a finite group $A$, then the \textbf{induction} of $V$ to $A$ is the $E[A]$-module
                \[
                    \indu_G^A V = E[A] \otimes_{E[G]} V,
                \]
                where $E[A]$ is regarded as a right $E[G]$-module. Its character $\indu_G^A \chi$ sends $\sigma \in A$ to
                \begin{equation}
                \label{eq:induction_character}
                    (\indu_G^A \chi)(\sigma) = \frac{1}{\abs{G}} \sum_{\substack{\tau \in A \\ \tau \sigma \tau^{-1} \in G}} \chi(\tau \sigma \tau^{-1}).
                \end{equation}
            }
            \item{
                If $G$ is a quotient of a finite group $B$ and we denote the corresponding projection map by $\pi \colon B \sa G$, then the \textbf{inflation} of  $\rho$ to $B$ is the representation
                \[
                    \infl_G^B \rho = \rho \circ \pi \colon B \sa G \to \GLg(V),
                \]
                which has character $\infl_G^B \chi = \chi \circ \pi$.
            }
            \item{
                If $C$ is a subgroup of $G$, then the \textbf{restriction} of $\rho$ to $C$ is the representation
                \[
                    \rest_C^G \rho = \restr{\rho}{C} \colon C \hookrightarrow G \xrightarrow{\rho} \GLg(V),
                \]
                which has character $\rest_C^G \chi = \restr{\chi}{C}$.
            }
            \item{
                The \textbf{dual} or \textbf{contragredient} representation $(\rho^\ast, V^\ast)$ is the $E$-vector space $V^\ast = \Hom_E(V, E)$ together with the $G$-action $(\sigma f)(v) = f(\sigma^{-1} v)$ for $\sigma \in G, f \in V^\ast, v \in V$. If $\cB$ is an $E$-basis of $V$ and $\cB^\ast$ its dual, then the matrix of $\sigma$ acting on $V^\ast$ with respect to $\cB^\ast$ is the transpose inverse of that of $\sigma$ acting on $V$ with respect to $\cB$. In other words, $\rho^\ast = (\rho^{-1})^t \colon G \to \GLg_{\dim_E(V)}(E)$. We denote the character of $\rho^\ast$ by $\check{\chi}$, which satisfies $\check{\chi}(\sigma) = \chi(\sigma^{-1})$ for all $\sigma \in G$.

                In the case $E = \CC$, the eigenvalues of $\rho^\ast(\sigma) = (\rho(\sigma)^{-1})^t$ are the inverses of those of $\rho(\sigma)$, i.e. their complex conjugates (because, as argued above, these are roots of unity). In particular, $\check{\chi}(\sigma) = \overline{\chi(\sigma)} \in \CC$ for all $\sigma \in G$.
            }
            \item{
                If $H$ is a subgroup of $G$, we denote the subspace of \textbf{$H$-invariants} by
                \[
                    V^H = \set{v \in V : hv = v \et{for all} h \in H}.
                \]
                This is an $E$-vector subspace of $V$, but it may fail to be closed under the action of $G$. However, it is so if $H$ is normal, in which case $V^H$ is again a representation of $G$ over $E$.
            }
            \item{
                If $(V', \rho')$ is another representation of $G$ over $E$ and we denote its character by $\chi'$, then the \textbf{sum} of $\rho$ and $\rho'$ is the $E$-vector space $V \oplus V'$ with $G$-action given by $(\rho \oplus \rho')(\sigma) = \rho(\sigma) \oplus \rho'(\sigma)$  for $\sigma \in G$. After choosing bases, this amounts to the matrix block sum of the two representations. If we denote the character of $\rho \oplus \rho'$ by $\chi + \chi'$, then one has $(\chi + \chi')(\sigma) = \chi(\sigma) + \chi'(\sigma)$ for all $\sigma \in G$. In particular, $\chi + \chi' = \chi' + \chi$.
            }
            \item{
                If $V', \rho'$ and $\chi'$ are as above, then the \textbf{tensor product} of $\rho$ and $\rho'$ is the $E$-vector space $V \otimes_E V'$ with $G$-action given by $(\rho \otimes \rho')(\sigma) = \rho(\sigma) \otimes \rho'(\sigma)$ (the Kronecker product of matrices after choosing bases) for $\sigma \in G$. If we denote the character of $\rho \otimes \rho'$ by $\chi \otimes \chi'$, then one has $(\chi \otimes \chi')(\sigma) = \chi(\sigma) \cdot \chi'(\sigma)$ for all $\sigma \in G$. In particular, $\chi \otimes \chi' = \chi' \otimes \chi$.
            }
        \end{itemize}

        We say a representation $(V, \rho)$ (or its associated character $\chi$) is \textbf{irreducible} if $V$ is a simple $E[G]$-module, and reducible otherwise. In the latter case, $V$ can be decomposed into a sum of several non-empty $E$-vector subspaces which are closed under the action of $G$. We denote the set of all (resp. all irreducible) $E$-valued characters of $G$ by $\CChars{E}{G}$ (resp. $\Irrs{E}{G}$). Character addition turns $\CChars{E}{G}$ into a commutative semigroup generated by $\Irrs{E}{G}$, since every character decomposes uniquely as a sum of irreducible ones. This semigroup can be enlarged to an abelian group by allowing formal integer combinations of characters and identifying them according to character addition - that is, the relation $(\chi) + (\chi') \sim (\chi + \chi')$. This is the usual process of Grothendieck completion after a formal identity (which maps every $\sigma \in G$ to 0) has been added. The result coincides in fact with the free abelian group on $\Irrs{E}{G}$, many of whose elements - the so-called \textbf{virtual characters}\index{character!vritual} - no longer arise as traces of representations. If we endow this additive group with the tensor product of characters, we obtain the \textbf{ring $\cR_E(G)$ of $E$-valued virtual characters of $G$}. A central result in the complex case is the following:
        \begin{thm}[Brauer's induction theorem]
        \label{thm:brauer_induction}
            The ring $\cR_\CC(G)$ is additively generated by
            \[
                \set{\indu_H^G \lambda \colon H \et{an elementary subgroup of} G, \lambda \et{a linear} \CC\text{-valued character of} \ H}.
            \]
        \end{thm}
        Recall that a finite group $H$ is said to be \textbf{$p$-elementary}\index{p-elementary@$p$-elementary!finite group} (with $p$ a prime) if it decomposes as $H = C_n \times H_p$ with $C_n$ a cyclic group of order $n$ coprime to $p$ and $H_p$ a $p$-group; and \textbf{elementary}\index{elementary finite group} if it is $p$-elementary for at least one $p$.

        The set of $E$-valued class functions on $G$ has a natural $E$-vector space structure, and it can be shown that $\Irrs{E}{G}$ constitutes a basis. This space admits a \textbf{scalar product} $\sprod{-, -}_G$ (where the subindex $G$ might be omitted if clear from the context) given by
        \[
            \sprod{\chi, \psi}_G = \frac{1}{\abs{G}} \sum_{\sigma \in G} \chi(\sigma) \psi(\sigma^{-1}) \in E.
        \]
        This is symmetric, bilinear over $E$ and satisfies \textbf{Frobenius reciprocity}\index{Frobenius reciprocity}: given a subgroup of $H \leq G$ and two characters (not arbitrary class functions) $\chi$ and $\psi$ of $G$ and $H$ respectively, one has
        \begin{equation}
        \label{eq:frobenius_reciprocity}
            \sprod{\chi, \indu_H^G \psi}_G = \sprod{\rest_H^G \chi, \psi}_H.
        \end{equation}
        The following relation is a direct consequence of the definition of the scalar product: for characters $\chi, \psi$ and $\varphi$ of $G$, one has
        \begin{equation}
        \label{eq:scalar_product_dual}
            \sprod{\chi, \psi \otimes \varphi}_G = \sprod{\chi \otimes \check{\psi}, \varphi}_G.
        \end{equation}

        Characters provide a natural framework to understand the \textbf{Wedderburn-Artin structure theorem}. Assume for the rest of this section that \textit{$E$ is algebraically closed} (and still of characteristic 0). Then, a version of the aforementioned theorem states that the semisimple Artinian ring $E[G]$ is in fact a finite sum of matrix rings over $E$. With some character theory, one can show
        \begin{equation}
        \label{eq:wedderburn}
            E[G] = \bigoplus_{\chi \in \Irrs{E}{G}} E[G] e(\chi) \iso \bigoplus_{\chi \in \Irrs{E}{G}} M_{\chi(1)}(E),
        \end{equation}
        where the isomorphism is character-wise. Here the $e(\chi)$ are (all) the \textbf{primitive central idempotents}\footnote{A common notation for these in the literature is $e_\chi$. In this text, however, the notation $e_\chi$ will be reserved for a different - although related - object (cf. section \ref{sec:evaluation_maps}) and the two should not be confused.} of $E[G]$\index{primitive central idempotent!e1@$e(\chi)$}, meaning that $e(\chi) \in Z(E[G])$, $e(\chi)^2 = e(\chi)$, and $e(\chi)$ cannot be expressed as a sum of two non-zero central idempotents. They are given by the formula
        \begin{equation}
        \label{eq:definition_pci}
            e(\chi) = \frac{\chi(1)}{\abs{G}} \sum_{\sigma \in G} \chi(\sigma^{-1}) \sigma.
        \end{equation}

        In the isomorphism \eqref{eq:wedderburn}, $E[G]e(\chi) \iso M_{\chi(1)}(E)$ should be understood as the ring of $E$-linear endomorphisms $\End_E(V_\chi)$ (with an implicit choice of basis), where $V_\chi$ is an irreducible $E[G]$-module which affords $\chi$. Given $\sigma \in G$, the element $\sigma e(\chi) \in E[G]e(\chi) \subseteq E[G]$ is sent to the matrix describing the action of $\sigma$ on $V_\chi$. In particular, $e(\chi)$ itself is mapped to the identity matrix. By the Wedderburn-Artin theorem, $E[G]e(\chi)$ is a ring with unity $e(\chi)$, but not a subring of $E[G]$ in general since the unity does not coincide.

        A remarkable property of these primitive central idempotents is their behaviour with respect to projection: if $N$ is a normal subgroup of $G$ and we let $Q = G/N$, the canonical ring surjection ${\varepsilon \colon E[G] \sa E[Q]}$ satisfies the following: for an irreducible character $\chi \in \Irr_E(G)$,
        \begin{equation}
        \label{eq:pci_projection}
            \varepsilon(e(\chi)) =
            \begin{cases}
                e(\overline{\chi}), & N \subseteq \ker(\chi) \\
                0, & \text{otherwise},
            \end{cases}
        \end{equation}
        where $\overline{\chi}$ is the projection of $\chi$ to $Q$ (which is necessarily irreducible), i.e. it is defined by ${\chi = \infl_N^G \overline{\chi}}$.

        It follows from the arithmetic of the primitive central idempotents that if $\chi$ and $\chi'$ are irreducible, then
        \[
            \sprod{\chi, \chi'}_G =
            \begin{cases}
                1, & \chi = \chi' \\
                0, & \text{otherwise.}
            \end{cases}
        \]
        This is known as \textbf{Schur's orthogonality relation}. In particular, the irreducible characters of $G$ form an \textit{orthonormal} basis of the $E$-vector space of class functions on $G$. As mentioned above, any $\psi \in \CChars{E}{G}$ admits a unique decomposition
        \[
            \psi = \sum_{\chi \in \Irrs{E}{G}} n_\chi \chi
        \]
        with $n_\chi \in \NN$. By the orthogonality relation, we can compute the multiplicities as $n_\chi = \sprod{\chi, \psi}_G$. We say $\chi$ \textbf{divides} $\psi$ if $n_\chi > 0$; and $\psi$ is \textbf{isotypic} if only one irreducible character divides it. The above is simply a restatement of the fact that every module over a semisimple Artinian ring decomposes uniquely as a direct sum of simple modules: the representations $V_\chi$ for $\chi \in \Irrs{E}{G}$ constitute a system of representatives of simple $E[G]$-modules up to isomorphism, and therefore any finitely generated $E[G]$-module $M$ has a unique expression as
        \begin{equation}
        \label{eq:character_associated_module}
            M \iso \bigoplus_{\chi \in \Irrs{E}{G}} V_\chi^{n_{\chi, M}}
        \end{equation}
        with $n_{\chi, M} \in \NN$. $M$ is then the unique (up to isomorphism) module with character $\psi_M = \sum_\chi \chi^{n_{\chi, M}}$.

        Not all hope is lost for fields which are not algebraically closed: given an irreducible $E$-valued representation $(V, \rho)$ of $G$ with character $\chi$ (still assuming $E = E^c$), the entries of all matrices in the image of $\rho$ after the choice of a basis $\cB$ only amount to a finite set of elements of $E$. If $F$ is a subfield of $E$ which contains all of those entries, we say $\chi$ has a \textbf{realisation}\index{character!realisation}\index{realisation of a character} (or realises) over $F$. This implies the image of $\chi$ is also contained in $F$, and therefore, $e(\chi) \in F[G] \subseteq E[G]$ is a primitive central idempotent in $F[G]$. In this case, $F[G] e(\chi)$ is only a matrix ring over a \textit{skew field} whose centre contains $F$.

    \newpage
    \section{Artin $L$-series}
    \label{sec:artin_l-series}

        Artin $L$-series generalise classical objects such as the Dedekind zeta function and Dirichlet $L$-series. Their special values play a central role in modern number theory, and they will do so in our Main Conjecture as well. All properties below can be found in \cite{ant} section 7\S10 and \cite{tate} section I\S3. The definitions from section \ref{sec:representations_of_finite_groups} will be particularly relevant.

        Let $L/K$ be a Galois extension of number fields with Galois group $G = \Gal(L/K)$. Given a prolongation $w$ to $L$ of a non-archimedean place $v$ of $K$ and a $\CC[G]$-module $V_\chi$ with character ${\chi \in \CChars{\CC}{G}}$, there is a well-defined $\CC$-linear action of the Frobenius element $\varphi_w \in \Gal(\kappa(w)/\kappa(v)) \iso G_w / I_w$ on the complex vector space $V_\chi^{I_w}$. We define the local \textbf{Euler factors}\footnote{$L_v$ does not denote the completion of $L$ at $v$ here. There will not be any ambiguity in practice.}\index{Euler factor}
        \begin{equation}
        \label{eq:local_L-factors}
            L_v(\chi, s) = \det(1 - \fN(v)^{-s} \varphi_w \mid V_\chi^{I_w})^{-1}
        \end{equation}
        and
        \begin{equation}
        \label{eq:local_delta-factors}
            \delta_v(\chi, s) = \det(1 - \fN(v)^{1 - s} \varphi_w \mid V_\chi^{I_w})
        \end{equation}
        for $s \in \CC$, where the first 1 is common notation for the identity matrix. The subscript $-_v$ is justified: different prolongations $w$ of $v$ lead to conjugate decomposition and inertia groups, and $\varphi_w$ acts on $V_\chi^{I_w}$ via the same matrix (under a suitable basis) as $\varphi_{\sigma(w)}$ does on $V_\chi^{I_{\sigma(w)}} = V_\chi^{\sigma I_w \sigma^{-1}} = \sigma(V_\chi^{I_w})$. A similar argument shows that the Euler factors depend on the character $\chi$ rather than on the representation $V_\chi$: two representations with the same character $\chi$ are equivalent and hence differ only by matrix conjugation, which does not affect the determinant.

        If the action of $\varphi_w$ on $V_\chi^{I_w}$ has (not necessarily distinct) eigenvalues $\lambda_1, \ldots, \lambda_n \in \CC$, then
        \begin{equation}
        \label{eq:euler_factors_analytic}
            \det(1 - \fN(v)^{-s} \varphi_w \mid V_\chi^{I_w}) = \prod_{i = 1}^n (1 - \fN(v)^{-s} \lambda_i)
        \end{equation}
        is an entire complex function on $s$. In particular, its inverse $L_v(\chi, s)$ is meromorphic. By the same token, $\delta_v(\chi, s)$ is entire as well.

        Let $S$ and $T$ be two finite sets of places of $K$ such that $S \supseteq S_\infty$ (the archimedean places) and $S \cap T = \varnothing$. These are sometimes referred to as the \textbf{depletion} and \textbf{smoothing} set, respectively. The \textbf{$(S, T)$-modified\footnote{Traditionally, Artin $L$-series have been defined for a single chosen set $S$. This is the case in the two references cited at the beginning of the present section. The $(S, T)$-modified version is of more recent introduction, and can be found, for instance, in \cite{burns_derivatives}. However, the proofs of the properties in lemma \ref{lem:properties_of_L-functions} for the $S$-version found in the given references also apply to the $(S, T)$-version, since the $\delta$-factors have the same functorial properties as the $L$-ones and do not vanish at $s = 0$} Artin $L$-series attached to $\chi$}\index{L-series@$L$-series}\index{Artin L-series@Artin $L$-series} is defined as
        \begin{equation}
        \label{eq:artin_l_series}
            L_{K, S, T}(\chi, s) = \prod_{v \notin S} L_v(\chi, s) \cdot \prod_{v \in T} \delta_v(\chi, s)
        \end{equation}
        for $s \in \CC, \re(s) > 1$. Note that there are infinitely many $L$-factors but only finitely many $\delta$-ones (see the proposition below for convergence) and that primes in $T$ appear in both types. We denote $L_{S, \varnothing}(\chi, s)$ simply by $L_S(\chi, s)$. For instance, in the case $L = K$ and $S = S_\infty$, the Artin $L$-series $L_{K, S_\infty}(\bbone, s)$ coincides with the series defining the Dedekind zeta function $\zeta_K(s)$.

        Before stating some basic properties of these series, we need to introduce a few classical $\ZZ[G]$-modules. The first one is
        \[
            \cY_{L, S}^\ZZ = \bigoplus_{w \in L(S)} \ZZ \cdot w \iso \bigoplus_{v \in S} \Ind_{G_w}^G \ZZ,
        \]
        where, in the last term, $w$ is an arbitrarily chosen prolongation of $v$ to $L$ and $\Ind_{G_w}^G -$ stands for $\ZZ[G] \otimes_{\ZZ[G_w]} -$. Note that, if $G_w$ is normal in $G$, then $\Ind_{G_w}^G \ZZ$ is simply $\ZZ[G/G_w]$. Different choices of prolongations yield canonically isomorphic modules, since they have conjugate decomposition groups. The natural $G$-action is given by permuting the places or, under the isomorphism, by simple multiplication in $G$. If we extend scalars from $\ZZ$ to a field of characteristic 0, then $\cY_{L, S}^\ZZ$ becomes (canonically isomorphic to) the sum of inductions of the trivial representations $\bbone_{G_w}$ of each $G_w$ in the language of section \ref{sec:representations_of_finite_groups}.

        The second relevant Galois module is the kernel $\cX_{L, S}^\ZZ$ of the \textbf{augmentation map}\index{augmentation!map}
        \begin{align}
        \label{eq:definition_y_z}
            \cY_{L, S}^\ZZ & \twoheadrightarrow \ZZ \\
            \sum_{w \in L(S)} z_w \cdot w & \mapsto \sum_{w \in L(S)} z_w, \nonumber
        \end{align}
        which is a $\ZZ[G]$-homomorphism when $\ZZ$ is endowed with the trivial $G$-action. This kernel is closely related to the group of \textbf{$S$-units}
        \begin{equation}
        \label{eq:definition_s-units}
            \units{\cO_{L, S}} = \set{u \in \units{L} \colon \abs{u}_w = 1 \et{for every place} w \et{of} L \et{outside} S(L)},
        \end{equation}
        which contains the finite-index subgroup of \textbf{$(S, T)$-units}
        \begin{equation}
        \label{eq:definition_st-units}
            \units{\cO_{L, S, T}} = \set{u \in \units{\cO_{L, S}} \colon \abs{u - 1}_w < 1 \et{for all} w \in S(T)}.
        \end{equation}
        In other words, the $S$-units are the non-zero elements of $L$ which are integral locally at all places outside $S(L)$, and the $(S, T)$-units are the $S$-units which are principal locally at all places in $S(T)$. Both are equipped with the natural action of $G$ on $\units{L}$. The connection to $\cX_{L, S}^\ZZ$ comes in the form of \textbf{Dirichlet's unit theorem}\index{Dirichlet's unit theorem}, which asserts that the \textbf{Dirichlet regulator map}\index{Dirichlet regulator map}
        \begin{align}
        \label{eq:dirichlet_regulator_map_real}
            \RR \otimes \units{\cO_{L, S}} & \to \RR \otimes \cX_{L, S}^\ZZ \\
            r \otimes u & \mapsto - \sum_{w \in S(L)} r \log(\abs{u}_w) \otimes w \nonumber
        \end{align}
        is an isomorphism of $\RR[G]$-modules.

        We point out that the definition of $\cY_{L, S}^\ZZ$ and $\cX_{L, S}^\ZZ$ does not truly require $S \supseteq S_\infty$ (but Dirichlet's unit theorem does). Furthermore, one can consider all four modules above over a single number field rather a Galois extension - this simply corresponds to the case $L = K, G = 1$.

        The relevance of these objects in the discussion of Artin $L$-series is clear in point iv) of the following result:
        \begin{lem}
        \label{lem:properties_of_L-functions}
            In the above notation, one has the following:
            \begin{enumerate}[i)]
                \item{
                    $L_{K, S, T}(\chi, s)$ converges on $\re(s) > 1$ and has a meromorphic continuation to all of $\CC$, which we denote the same and refer to as the \textbf{$(S, T)$-modified Artin $L$-function attached to $\chi$}.
                }
                \item{
                    If $M$ is a finite Galois extension of $K$ containing $L$, then $L_{K, S, T}(\chi, s) = L_{K, S, T}(\infl_G^{\Gal(M/K)}\chi, s)$.
                }
                \item{
                    If $H$ is a subgroup of $G$ with fixed field $L' = L^H$ and $\chi = \indu_H^G \psi$ for some $\psi \in \CChars{\CC}{H}$, then $L_{K, S, T}(\chi, s) = L_{L', S(L'), T(L')}(\psi, s)$.
                }
                \item{
                    If $\chi'$ is another complex character of $G$, then $L_{K, S, T}(\chi + \chi', s) = L_{K, S, T}(\chi, s) L_{K, S, T}(\chi', s)$.
                }
                \item{
                    The order of vanishing $r_S(\chi)$ of $L_{K, S, T}(\chi, s)$ at $s = 0$ coincides with
                    \[
                        \bigg(\sum_{v \in S} \dim_\CC V_\chi^{G_w} \bigg) - \dim_\CC V_\chi^G = \sprod{\chi, \psi_\cX} = \dim_\CC(\Hom_{\CC[G]}(V_\chi^\ast, \CC \otimes \cX_{L, S}^\ZZ)),
                    \]
                    where, in the second term, $w$ denotes an arbitrarily chosen prolongation of $v$; and, in the third term, $\psi_\cX$ is the character of $\CC \otimes \cX_{L, S}^\ZZ$ as a $\CC[G]$-module. In particular, it is independent of $T$ and non-negative, i.e. $L_{K, S, T}(\chi, s)$ does not have a pole at $s = 0$.
                }
            \end{enumerate}
        \end{lem}
        In properties ii) - iv), the equality means equality as series and hence equality of Artin $L$-functions. These functorial properties constitute a powerful tool when combined with representation-theoretic results such as Brauer's induction theorem \ref{thm:brauer_induction}. For instance, it follows immediately that Artin $L$-functions decompose as products of those attached to linear characters - which are easier to manipulate. Since virtual characters $R_\CC(G)$ are freely generated by $\Irrs{\CC}{G}$, property iv) above provides a natural way to define $L$-functions attached to virtual characters. Another observation is that property ii) justifies why the field $L$ is not featured in the notation $L_{K, S, T}$ (although it is implicit in $\chi$ at any rate): enlarging or - when possible - shrinking the top field does not affect the $L$-function. In that sense, there is a minimal field over which the function can be defined: $L^{\ker(\chi)}$, which is sometimes referred to as the field \textbf{cut out} by $\chi$.

        The \textbf{leading coefficient}\index{leading coefficient} of $L_{K, S, T}(\chi, s)$ at $s = 0$ is defined as
        \begin{equation}
        \label{eq:definition_leading_coefficient}
            L_{K, S, T}^\ast(\chi, 0) = \lim_{s \to 0} s^{-r_S(\chi)} L_{K, S, T}(\chi, s) \in \units{\CC}.
        \end{equation}
        In other words, it is the first non-zero coefficient of the series expansion of $L_{K, S, T}(\chi, s)$ around 0. This is the special value featured in our Main Conjecture.

    \newpage
    \section{Algebraic $K$-theory}
    \label{sec:algebraic-k-theory}

        Algebraic $K$-theory will provide the right setting for the formulation of our Main Conjecture. The algebraic and analytic sides of the conjecture come together in a classical so-called localisation sequence. We will only be concerned with $K$-groups in low degrees, namely 0 and 1. Our main reference is \cite{swan}.

        Let $R$ be a ring and denote by $\PPf_0(R)$ the category of finitely generated projective left $R$-modules with $R$-module homomorphisms. We define the following $K$-groups\index{K-group@$K$-group|(}:
        \begin{itemize}
            \item{
                $K_0(R)$ is the quotient of the free abelian group on the set of isomorphism classes of modules in $\PPf_0(R)$ by the subgroup generated by
                \[
                    (\overline{M \oplus N}) - (\overline{M}) - (\overline{N}),
                \]
                where $M, N \in \PPf_0(R)$ and $\overline{-}$ means isomorphism class. Here and below, $(-)$ denotes the image of $-$ in the corresponding free abelian group. Given $M \in \PPf_0(R)$, we denote the class of $(\overline{M})$ in $K_0(R)$ by $[M]$.

                For $M, N \in \PPf_0(R)$, one has $[M] = [N]$ in $K_0(R)$ if and only if there exists a $P \in \PPf_0(R)$ such that $M \oplus P \iso N \oplus P$ (cf. \cite{swan}, theorem 1.10). In that case, $M$ and $N$ are said to be \textbf{stably isomorphic}, which is in general a weaker condition than isomorphism.

                This construction can be formulated naturally in terms of a Grothendieck completion, a notion which already appeared in section \ref{sec:representations_of_finite_groups}.
            }
            \item{
                Given a ring homomorphism $R \xrightarrow{\varphi} S$, consider the category $\PPf_0^\varphi(R, S)$ whose objects are triples $(M, f, N)$ with $M, N \in \PPf_0(R)$ and $f \colon S \otimes_R M \xrightarrow{\sim} S \otimes_R N$ an isomorphism of $S$-modules. The map $f$ \textit{need not} arise as the scalar extension of a morphism of $R$-modules $M \to N$. A morphism of triples $\nu \colon (M, f, N) \to (M', f', N')$ consists of a pair of homomorphisms of $R$-modules $\nu_M \colon M \to M'$, $\nu_N \colon N \to N'$ making the diagram
                \begin{center}
                    \begin{tikzcd}[column sep=large]
                        S \otimes_R M \arrow[r, "S \otimes_R \nu_M"] \arrow[d, "f"] & S \otimes_R M' \arrow[d, "f'"] \\
                        S \otimes_R N \arrow[r, "S \otimes_R \nu_N"]                & S \otimes_R N'
                    \end{tikzcd}
                \end{center}
                commute. Note that, since the two vertical arrows are isomorphisms, either horizontal arrow determines the other uniquely. Composition and identity morphisms in this category are defined in the obvious way.

                The relative $K_0$ group\footnote{Some authors define the relative $K_0$ group for $R$ and a two-sided ideal $\fa \trianglelefteq R$, in which case it is often denoted by $K_0(R, \fa)$. This coincides with the above definition of $K_0^\varphi(R, R/\fa)$ (cf. \cite{swan} p. 214) if one chooses as $\varphi$ the canonical projection $R \sa R/\fa$. In particular, Swan's definition is more general (and necessary for our purposes).} of $R$ and $S$, denoted by $K_0^\varphi(R, S)$, is the quotient of the free abelian group on the set of isomorphism classes of objects in $\PPf_0^\varphi(R, S)$ (in order to avoid notational clutter, we do not distinguish between $(M, f, N)$ and its isomorphism class) by the subgroup generated by all elements of the form
                \begin{enumerate}[i)]
                    \item{
                        $((M, gf, P)) - ((M, f, N)) - ((N, g, P))$.
                    }
                    \item{
                        $((M, f, N)) - ((M', f', N')) - ((M'', f'', N''))$ if there exists a short exact sequence
                        \[
                            0 \to (M', f', N') \to (M, f, N) \to (M'', f'', N'') \to 0
                        \]
                        in $\PPf_0^\varphi(R, S)$ in the obvious sense.
                    }
                \end{enumerate}
                We denote the class of $((M, f, N))$ by $[M, f, N]$. An immediate consequence of i) is that $[M, \Id, M] = 0$ and $[M, f, M] = - [M, f^{-1}, M]$. As a direct application of ii), we have the following: if $f_R \colon M \xrightarrow{\sim} N$ is an isomorphism of (finitely generated projective) $R$-modules, then
                \begin{equation}
                \label{eq:iso_k0_over_r_trivial}
                    [M, S \otimes_R f_R, N] = 0.
                \end{equation}

                The notation $K_0(R, S)$ may be used instead of $K_0^\varphi(R, S)$ if the morphism $\varphi$ is clear from context.
            }
            \item{
                Let $\GLg(R)$ denote the infinite general linear group over $R$, that is, the direct limit $\varinjlim_n \GLg_n(R)$ with transition maps $\GLg_n(R) \ia \GLg_{n + 1}(R)$ given by $M \mapsto M \oplus 1$. Then the \textbf{Whitehead group} of $R$ is the abelianisation
                \[
                    K_1(R) = \GLg(R)^{ab}.
                \]
                By a classical result, the commutator of $\GLg(R)$ is precisely the subgroup generated by elementary matrices. We denote the class of a matrix $A \in \GLg_n(R) \subset \GLg(R)$ in $K_1(R)$ by $[A]$.

                We will on occasion need an alternative description of $K_1$ resembling that of the relative $K_0$ group. Let $\PPf_0^{\Aut}(R)$ denote the category whose objects are pairs $(M, f)$ with $M \in \PPf_0(R)$ and $f \in \Aut_R(M)$; and where a morphism $\nu \colon (M, f) \to (M', f')$ is a morphism of $R$-modules $\nu \colon M \to M'$ making the diagram
                \begin{center}
                    \begin{tikzcd}
                        M \arrow[r, "\nu"] \arrow[d, "f"] & M' \arrow[d, "f'"] \\
                        M \arrow[r, "\nu"]                & M'
                    \end{tikzcd}
                \end{center}
                commute. Composition and identity morphisms are defined in the obvious way. The group $K_1^{\det}(R)$ is then defined as the quotient of the free abelian group on the set of isomorphism classes of objects in $\PPf_0^{\Aut}(R)$ (which, as before, we do not distinguish from the objects themselves) by the subgroup generated by all elements of the form
                \begin{enumerate}[i)]
                    \item{
                        $((M, gf)) -  ((M, f)) - ((M, g))$.
                    }
                    \item{
                        $((M, f)) - ((M', f')) - ((M'', f''))$ if there exists a short exact sequence
                        \[
                            0 \to (M', f') \to (M, f) \to (M'', f'') \to 0
                        \]
                        in $\PPf_0^{\Aut}(R)$ in the obvious sense.
                    }
                \end{enumerate}
                Analogously to the relative $K_0$, we denote the class of $((M, f))$ by $[M, f]$. It follows from i) that $[M, \Id] = 0$ and $[M, f] = - [M, f^{-1}]$.

                These two descriptions of $K_1$ give rise to the same group: the map
                \begin{align*}
                    K_1(R) & \to K_1^{\det}(R) \\
                       [A] & \mapsto [R^n, f_A]
                \end{align*}
                is a well-defined group isomorphism (cf. \cite{cr2} theorem 40.6 or \cite{swan} theorem 13.4). Here $n$ is any positive integer such that $A \in \GLg_n(R) \subseteq \GLg(R)$ and $f_A$ is the $R$-automorphism of $R^n$ induced by multiplication by $A$. Note that, since $R$ may not be commutative, multiplying elements of $R^n$ by $A$ is only guaranteed to yield a homomorphism of left $R$-modules if done \textit{on the right}. In other words, $x \in R^n$ is regarded as a row vector and $f_A(x) = xA$.

                The inverse of the above isomorphism can be described as follows: let $[M, f] \in K_1^{\det}(R)$. There exists an $N \in \PPf_0(R)$ and an isomorphism $\iota \colon M \oplus N \xrightarrow{\sim} R^n$ for some $n \in \NN$. Then $\iota \circ (f \oplus \Id_N) \circ \iota^{-1}$ is an $R$-automorphism of $R^n$, which in particular has a matrix expression $A_f \in \GLg_n(R)$ - once again regarding elements of $R^n$ as row vectors. Mapping $[M, f]$ to $[A_f]$ is well defined and yields the inverse $K_1^{\det}(R) \xrightarrow{\sim} K_1(R)$ (cf. loc. cit.).

                In the sequel, we regard this canonical isomorphism as an identification $K_1(R) = K_1^{\det}(R)$.
            }
        \end{itemize}\index{K-group@$K$-group|)}

        We shall write the groups $K_0(R)$, $K_0(R, S)$ and $K_1^{\det}(R)$ in additive notation, since their operations are related to the direct sum of modules; and the group $K_1(R)$ in multiplicative notation, since its operation comes from matrix multiplication. Note that all three groups are abelian.

        Let now $\varphi \colon R \to S$ be any ring homomorphism. Consider the sequence abelian groups
        \begin{equation}
        \label{eq:exact_sequence_k-theory}
            K_1(R) \xrightarrow{K_1(\varphi)} K_1(S) \xrightarrow{\partial} K_0(R, S) \xrightarrow{\psi} K_0(R) \xrightarrow{K_0(\varphi)} K_0(S)
        \end{equation}
        where each arrow is defined as follows:
        \begin{itemize}
            \item{
                $K_1(\varphi)$ sends $[A]$ to $[\varphi(A)]$, where $\varphi$ is applied entry-wise.
            }
            \item{
                $\partial$ sends $[A]$ to $[R^n, f_A, R^n]$ for $A \in \GLg_n(S)$, where $f_A$ is the automorphism of $S^n$ given by multiplication by $A$ on the right. This map is often called the \textbf{connecting}\index{connecting homomorphism} or \textbf{boundary homomorphism}\index{boundary homomorphism}\footnote{Another common choice for $\partial$ the inverse of the above, i.e. $\partial([A]) = [R^n, f_A^{-1}, R^n]$. This does not affect the exactness of the sequence, but it has implications on other conventions. Our choice is consistent with the sources we will refer to, most notably \cite{swan} and \cite{bb}.}.
            }
            \item{
                $\psi$ sends $[M, f, N]$ to $[M] - [N]$.
            }
            \item{
                $K_0(\varphi)$ is induced by extension of scalars $S \otimes_R - \colon \PPf_0(R) \to \PPf_0(S)$ via $\varphi$. In other words, it sends $[M]$ to $[S \otimes_R M]$.
            }
        \end{itemize}
        These four maps are well-defined group homomorphisms which make \eqref{eq:exact_sequence_k-theory} into an exact sequence (cf. \cite{swan} theorem 15.5), known as the \textbf{exact sequence of $K$-theory}\index{exact sequence of K-theory@exact sequence of $K$-theory} - although technically this term refers to an infinite exact sequence continuing to the left with higher $K$-groups. If $R \xrightarrow{\varphi} S$ is the embedding of $R$ into some localisation, the sequence is also referred to as the \textbf{localisation sequence}\index{localisation sequence}.

        As an addendum to the definition of the boundary homomorphism, we also let ${\partial \colon K_1^{\det}(S) \to K_0(R, S)}$ denote the map resulting from the identification $K_1(S) = K_1^{\det}(S)$. Relation ii) in the definition of $K_0(R, S)$ then implies the following: given $M \in \PPf_0(R)$ and an $S$-automorphism $f$ of $S \otimes_R M$, one has
        \begin{equation}
        \label{eq:partial_on_k1det}
            \partial([S \otimes_R M, f]) = [M, f, M].
        \end{equation}

        $K_0$ and $K_1$ as above define covariant functors from $\fR$ (rings with unity) to $\fA\fb$ (abelian groups). There is a natural way to define \textit{contravariant} functors on a certain subcategory instead. To do so, let $\varphi \colon R \to S$ be a ring homomorphism such that $S$ is finitely generated and projective as an $R$-module via $\varphi$. Then restriction of scalars $\restr{-}{R}$ defines a covariant functor $\PPf_0(S) \to \PPf_0(R)$. This in turn induces a \textbf{transfer} functor on $K$-theory (cf. \cite{weibel} definition IV.6.3.2), which we can easily make explicit in low degree. On $K_0$, the map is clear:
        \begin{align*}
            K_0^\rest(\varphi) \colon K_0(S) & \to K_0(R) \\
               [M] & \mapsto [\restr{M}{R}].
        \end{align*}
        In the case of $K_1$, the map has an immediate description in terms of $K_1^{\det}$:
        \begin{align*}
            K_1^{{\det}, \rest}(\varphi) \colon K_1^{\det}(S) & \to K_1^{\det}(R) \\
               [M, f] & \mapsto [\restr{M}{R}, \restr{f}{R}].
        \end{align*}
        This then induces a map on the usual $K_1$ by imposing commutativity of the diagram
        \begin{center}
            \begin{tikzcd}
                K_1(S) \arrow[d, "K_1^{\rest}(\varphi)", dashed] \arrow[r, "\sim"] & K_1^{\det}(S) \arrow[d, "{K_1^{{\det}, \rest}(\varphi)}"] \\
                K_1(R) \arrow[r, "\sim"]                                           & {K_1^{\det}(R)\nospacecomma}
            \end{tikzcd}
        \end{center}
        The resulting homomorphism $K_1^{\rest}(\varphi)$ does not have a seamless definition in terms of matrices, but it can be described as follows. For simplicity, assume that $S$ is free as an $R$-module and hence there exists an $R$-isomorphism $\iota \colon R^r \to S$. Let $A = (a_{i , j})_{i, j} \in \GLg_n(S)$. For each $i, j$, let $m_{i, j}$ be the homomorphism of $S$-(and $R$-)modules $S \to S$ given by multiplication by $a_{i, j}$ on the right. The composition $\iota^{-1} \circ m_{i, j} \circ \iota$ is an endomorphism of the left $R$-module $R^r$ and hence can be represented by a matrix $M_{i, j} \in M_r(R)$. Then $K_1^\rest(\varphi)([A])$ is the class in $K_1(R)$ of the $rn$-by-$rn$ matrix resulting from replacing each $a_{i, j}$ by $M_{i, j}$ in $A$.

        Having defined change-of-ring maps for $K_0$ and $K_1$, it remains to extend these to the relative $K$-group $K_0(R, S)$ in a compatible manner. This will be essential to prove the good functorial behaviour of refined Euler characteristics (appendix \ref{app:functoriality_of_refined_euler_characteristics}) and hence of the Main Conjecture. The construction naturally involves four rings $R, S, R'$ and $S'$ and homomorphisms making the diagram
        \begin{equation}
        \label{eq:functoriality_relative_k0_setup}
            \begin{tikzcd}
                R \arrow[d, "\rho"] \arrow[r, "\varphi"] & S \arrow[d, "\sigma"] \\
                R' \arrow[r, "\varphi'"]                 & S'
            \end{tikzcd}
        \end{equation}
        commute. Extension of scalars is then given by
        \begin{align*}
            K_0(\rho, \sigma) \colon K_0(R, S) & \to K_0(R', S') \\
            [M, f, N] & \mapsto [R' \otimes_R M, S' \otimes_S f, R' \otimes_R N].
        \end{align*}
        The fact that this is well defined relies crucially on $S' \otimes_S (S \otimes_R -)$ being the same as $S' \otimes_{R'} (R' \otimes_R -)$, which is ensured by \eqref{eq:functoriality_relative_k0_setup}. The exact $K$-theory sequences \eqref{eq:exact_sequence_k-theory} induced by $\varphi$ and $\varphi'$ are then connected into a commutative diagram
        \begin{equation}
        \label{eq:functoriality_relative_k0_payoff}
            \begin{tikzcd}
                K_1(R) \arrow[r, "K_1(\varphi)"] \arrow[d, "K_1(\rho)"] & K_1(S) \arrow[r] \arrow[d, "K_1(\sigma)"] & {K_0(R, S)} \arrow[r] \arrow[d, "{K_0(\rho, \sigma)}"] & K_0(R) \arrow[r, "K_0(\varphi)"] \arrow[d, "K_0(\rho)"] & K_0(S) \arrow[d, "K_0(\sigma)"] \\
                K_1(R') \arrow[r, "K_1(\varphi')"]                      & K_1(S') \arrow[r]                         & {K_0(R', S')} \arrow[r]                                & K_0(R') \arrow[r, "K_0(\varphi')"]                      & K_0(S')\nospaceperiod
            \end{tikzcd}
        \end{equation}

        In the case of restriction of scalars, we place the same additional requirement on diagram \eqref{eq:functoriality_relative_k0_setup} as before, namely that $R'$ be a finitely generated projective $R$-module via $\rho$ and $S'$ a finitely generated projective $S$-module via $\sigma$. Then the map
        \begin{align*}
            K_0^\rest(\rho, \sigma) \colon K_0(R', S') & \to K_0(R, S) \\
            [M, f, N] & \mapsto [\restr{M}{R}, \restr{f}{S}, \restr{N}{R}]
        \end{align*}
        is well defined and makes the diagram
        \begin{equation}
        \label{eq:localisation_sequence_restriction}
            \begin{tikzcd}
                K_1(R') \arrow[r, "K_1(\varphi')"] \arrow[d, "K_1^\rest(\rho)"] & K_1(S') \arrow[r] \arrow[d, "K_1^\rest(\sigma)"] & {K_0(R', S')} \arrow[r] \arrow[d, "{K_0^\rest(\rho, \sigma)}"] & K_0(R') \arrow[r, "K_0(\varphi')"] \arrow[d, "K_0^\rest(\rho)"] & K_0(S') \arrow[d, "K_0^\rest(\sigma)"] \\
                K_1(R) \arrow[r, "K_1(\varphi)"]                      & K_1(S) \arrow[r]                         & {K_0(R, S)} \arrow[r]                                & K_0(R) \arrow[r, "K_0(\varphi)"]                      & K_0(S)
            \end{tikzcd}
        \end{equation}
        commute.

        We end this section by briefly introducing the \textbf{reduced norm}\index{reduced norm} $\nr \colon K_1(R) \to \units{Z(R)}$ for a semisimple Artinian ring $R$. This map will be instrumental in formulating the analytic side of the Main Conjecture, as it connects the localisation sequence above to certain (fraction fields of) rings of power series where the analogues of $p$-adic $L$-series are expected to live. The specifics of the construction will not play an important role. More details can be found in \cite{weibel} III.1.2.4.

        Consider first a division algebra $D$ of finite dimension $d$ over its centre $F = Z(D)$, which is necessarily a field. Choose a splitting field $E$, that is, an extension of $F$ such that $E \otimes_F D \iso M_n(E)$ - in particular, $d = n^2$ ($n$ is called the \textbf{Schur index} of $D$). Splitting fields always exist - any maximal field inside $D$ is one. The reduced norm map
        \[
            \nr \colon D \ia E \otimes_F D \iso M_n(E) \xrightarrow{\det} E
        \]
        is known to have image contained in $F$ and to be independent of the choice of $E$ and the isomorphism in the middle. It restricts to a group homomorphism $\units{D} \to \units{F}$. The same construction can be used to define maps $M_r(D) \to F$ which restrict to homomorphisms $\GLg_r(D) \to \units{F}$ for all $r \geq 1$. These are compatible with the embeddings $\GLg_r(D) \ia \GLg_{r + 1}(D)$ and hence extend to a homomorphism
        \[
            \nr \colon K_1(D) \to \units{F}.
        \]
        This is easily generalised to the case where $R$ is a semisimple Artinian ring in the form of
        \[
            \nr \colon K_1(R) \iso \prod_{i = 1}^s K_1(M_{n_i}(D_i)) = \prod_{i = 1}^s K_1(D_i) \xrightarrow{\prod \nr} \prod_{i = 1}^s \units{F_i} = \units{Z(R)},
        \]
        where each $D_i$ is a division ring of finite dimension $d_i$ over its centre $F_i$. Here the first isomorphism is given by the Wedderburn-Artin theorem $R \iso \prod M_{n_i}(D_i)$ together with the fact that $K_1$ sends products to products. The second equality is known as \textbf{Morita invariance} (cf. \cite{weibel} example III.1.1.4 and proposition III.1.6.4) and comes from the natural identifications $\GLg_{a}(M_b(D)) = \GLg_{ab}(D)$. As before, the resulting homomorphism $\nr$ is independent of the choice of an isomorphism in the Wedderburn-Artin theorem.

        When $R$ is a \textit{commutative} semisimple Artinian ring, the same theorem implies it decomposes as a finite product of fields. In particular, the reduced norm becomes an isomorphism
        \begin{equation}
        \label{eq:reduced_norm_iso}
            \nr \colon K_1(R) \isoa \units{R}.
        \end{equation}

\newpage
\chapter{Construction of the complex}
\label{chap:construction_of_the_complex}

    This chapter is devoted to the construction of a certain complex $\cC_{S, T}\q$ in explicit Galois-theoretic and cohomological terms. This complex is the central algebraic object in our Main Conjecture, and it coincides in the relevant derived category with that employed by Burns, Kurihara and Sano in their recent conjecture (cf. \cite{bks}). Although not strictly necessary for the formulation of the Main Conjecture, having this explicit description of $\cC_{S, T}\q$ is valuable in itself, and it will simplify subsequent tasks such as the study of how changes in the parameters (most notably the sets of places $S$ and $T$) affect the complex.

    The first two sections deal with the definition of some global and local complexes, respectively, which are based on the \textbf{translation functor} of Ritter and Weiss (cf. \cite{rw_tate_sequence}). In section \ref{sec:local-to-global_maps}, we define a natural morphism from the local complexes to the global one. The definition of $\cC_{S, T}\q$ as the cone of that morphism, along with some of its key properties, are the object of section \ref{sec:the_main_complex}. The remaining section proves the relation to the complex of Burns, Kurihara and Sano.

    The reader who is familiar with the language of \textit{Weil-étale cohomology complexes} or is otherwise not interested in the construction may skip this chapter and refer only to proposition \ref{prop:complexes_are_perfect} and theorem \ref{thm:cohomology_of_complex} for the essential properties of the complex. However, the next few pages (until the beginning of section \ref{sec:the_global_complex}) lay out some important notation for subsequent chapters - alongside that introduced at the beginning of chapter \ref{chap:notation_conventions_and_preliminaries}, which we assume now. Specifically, the following setting will apply to most of our considerations\footnote{For the sake of clarity, we will reference this and subsequent settings in the definitions and results where it is in place.}:

    \begin{sett}
    \label{sett:construction}
    \addcontentsline{toc}{section}{Setting A}

        Let $K$ be a number field. We consider:
        \begin{itemize}
            \item{
                A rational prime $p$, different from 2 if $K$ is not totally imaginary.
            }
            \item{
                A fixed algebraic closure $\QQ^c$ of $\QQ$. For each place $v$ of $K$, we choose and fix a distinguished prolongation $v^c$ of $v$ to $\QQ^c$. We denote the place of a number field $K'$ below $v^c$ by $v(K')$, which is simply $K' \cap v^c$ if $v$ is non-archimedean. This should not be confused with $\set{v}(K')$, the set of prolongations of $v$ to $K' \supseteq K$ in the notation introduced in chapter \ref{chap:notation_conventions_and_preliminaries}.
            }
            \item{
                $L_\infty$, a \textit{one-dimensional $p$-adic Lie extension} of $K$ containing its cyclotomic $\ZZ_p$-extension $K_\infty$ - which simply amounts to $H = \Gal(L_\infty/K_\infty)$ being finite. We denote $\Gal(K_\infty/K)$ by $\Gamma_K$, $\Gal(L_\infty/K)$ by $\cG$ and choose and fix an open central subgroup $\Gamma \trianglelefteq_o \cG$ such that $\Gamma \iso \ZZ_p$ as topological groups. Then $\Gamma \cap H = \set{1}$ and therefore, letting $L = L_\infty^\Gamma$, one has $L_\infty = LK_\infty$. In particular, $L_\infty$ is itself the cyclotomic $\ZZ_p$-extension of $L$. Lemma \ref{lem:splitting_and_central} below shows that such a $\Gamma$ exists.

                $K_n$ denotes the $n$-th layer of $K_\infty/K$, and thus $\Gal(K_\infty/K_n) = \Gamma_K^{p^n}$ and $\Gal(K_n/K)$ is cyclic of order $p^n$. The same notation applies to $L_n$ in $L_\infty/L$. We define $N \in \NN$ by $L \cap K_\infty = K_N$, so the image of the projection $\Gamma \ia \cG \sa \Gamma_K$ is precisely $\Gamma_K^{p^N}$. Since $\Gamma \subseteq Z(\cG)$, the extension $L_n/K$ is Galois (and finite) for all $n \in \NN$ and we denote the corresponding Galois group by $\cG_n = \cG/\Gp{n}$.
            }
            \item{
                $S \supseteq S_\infty \cup S_{\ram}(L_\infty/K)$, a finite set of places of $K$ containing the archimedean places and the ramified places in $L_\infty/K$. We denote the set of non-archimedean places in $S$ by $S_f$. Since the places which ramify in the cyclotomic $\ZZ_p$-extension of any number field are precisely the $p$-adic ones, we have $S_f \supseteq S_p$ (the set of $p$-adic places of $K$).

                Non-archimedean places are finitely decomposed in the cyclotomic $\ZZ_p$-extension, and therefore the subgroup
                \[
                    \Gamma \cap \bigcap_{w_\infty \in S_f(L_\infty)} \cG_{w_\infty}
                \]
                is open in $\Gamma$, i.e. of the form $\Gamma^{p^{n(S)}}$ for some $n(S) \in \NN$. In particular, $L_{n(S)}$ is the smallest layer $L_n$ of $L_\infty/L$ such that every place in $S_f(L_n)$ is non-split in $L_\infty/L_n$.
            }
            \item{
                $T$, a finite set of places of $K$ disjoint from $S$.  We let
                \[
                    T^p = \set{v \in T \colon L_w \et{contains a primitive} p\text{-th root of unity for some (any) place} \ w \in \set{v}(L)},
                \]
                where the equivalence between \textit{some} and \textit{any} comes from the fact that the elements of $\Gal(L/K)$ induce isomorphisms among the residue fields $\kappa(w)$ of the various prolongations $w \mid v$. By Hensel's lemma, $L_w$ contains a primitive $p$-th root of unity if and only if $p \mid \fN(w) - 1$.
            }
        \end{itemize}
    \end{sett}

    We \textbf{do not assume}
    \begin{itemize}
        \item{
            $\cG$ to be abelian.
        }
        \item{
            $L$ or $K$ to be totally real or imaginary, with the aforementioned caveat if $p = 2$. Furthermore, we do not prescribe the behaviour of archimedean places in $L_\infty/K$.
        }
        \item{
            $\abs{H}$ (i.e. $[L_\infty : K_\infty]$) to be prime to $p$.
        }
    \end{itemize}

    \begin{rem}
        An alternative approach to the above setting is to start with an extension $M/K$ of number fields and assume that the cyclotomic $\ZZ_p$-extension $M_\infty$ of $M$ is Galois over $K$ - which is always the case, for instance, if $M/K$ is Galois. It follows that $[M_\infty : K_\infty]$ is finite (with $K_\infty$ the cyclotomic $\ZZ_p$-extension of $K$) and one is thus in the same situation as above with $L_\infty = M_\infty$. In fact, $L$ can then be chosen to be one of the layers $M_n$ by simply forming the compositum of $M$ and an arbitrarily chosen $L$ as in the setting. Since our Main Conjecture will not depend on the choice of $L$ (as shown in subsection \ref{subsec:the_choice_of_l}), both approaches are equivalent. \qedef
    \end{rem}

    A simple argument (see for instance \cite{rwii} section 1) shows that a $\Gamma$ as in setting \ref{sett:construction} exists:
    \begin{lem}
    \label{lem:splitting_and_central}
        Let $K$ be a number field and $L_\infty/K$ a Galois extension such that $\cG = \Gal(L_\infty/K)$ contains a finite (and hence closed) subgroup $H$ with quotient $\Gamma_K = \cG/H \iso \ZZ_p$. Then, the short exact sequence
        \[
            1 \to H \to \cG \to \Gamma_K \to 1
        \]
        of profinite groups splits and $\cG$ contains an open central subgroup $\Gamma \iso \ZZ_p$. For any such $\Gamma$, the field $L = L_\infty^\Gamma$ is a finite Galois extension of $K$ and one has $\Gamma \cap H = \set{1}$ and $L_\infty^H L = L_\infty$.
    \end{lem}

    \begin{proof}
        Let $\pi$ denote the projection $\cG \sa \Gamma_K$ and choose a topological generator $\gamma_K$ of $\Gamma_K$. Choose a preimage $g$ of $\gamma_K$ under $\pi$. If $n$ is the index of a $p$-Sylow subgroup $\cG_p$ of $\cG$ (i.e. the prime-to-$p$ part of $\abs{H}$), then $\pi(g^n) = \gamma_K^n$ is also a topological generator of $\Gamma_K$ and $g^n$ therefore generates a procyclic, and hence pro-$p$, subgroup of $\cG_p$. The assignment $\gamma_K^n \mapsto g^n$ uniquely determines a continuous embedding $\sigma \colon \Gamma_K \ia \cG$ which splits the short exact sequence from the proposition.

        Let $\Gamma_\cG = \sigma(\Gamma_K) \iso \ZZ_p$. The splitting $\sigma$ yields a semidirect product decomposition $\cG = H \rtimes \Gamma_\cG$ via the conjugation homomorphism $\Gamma_\cG \to \Aut(H)$. But $\Aut(H)$ is finite and hence there exists an open subgroup $\Gamma = \Gamma_\cG^{p^m} \trianglelefteq_o \Gamma_\cG$ which is mapped to the identity automorphism, i.e. whose elements commute with those of $H$ and therefore with all of $\cG$. Since $H$ is finite and $\Gamma$ has no elements of finite order, one has $\Gamma \cap H = \set{1}$, from which the last statement follows.
    \end{proof}

    It should be noted that the choice of $\Gamma$ is far from unique. In fact, if $\Gamma$ is as in the setting, then $\Gp{n}$ is another valid choice for any $n \in \NN$. In other words, we can move freely along the cyclotomic tower $L_\infty/L$. This will be addressed in section \ref{sec:independence_of_the_choice_of_parameters}.

    We now introduce notation for some distinguished fields:
    \begin{itemize}
        \item{
            $M_S$ is defined as the maximal $S$-ramified (i.e. unramified outside $S$) extension of $L_\infty$, and analogously for $M_{S \cup T}$.
        }
        \item{
            $M_{T, S}^{cs}$ is defined as the maximal $T$-ramified extension of $L_\infty$ which is completely split at $S$.
        }
        \item{
            $M_S^{cs} = M_{\varnothing, S}^{cs}$ is defined as the maximal unramified extension of $L_\infty$ which is completely split at $S$.
        }
    \end{itemize}

    By maximality, the above fields are all Galois over $L_\infty$, and so is $M_{S \cup T}$ over $K$ (because $L_\infty/K$ is unramified outside $S \cup T$). We set $H_{S \cup T} = \Gal(M_{S \cup T}/L_\infty) \trianglelefteq_c G_{S \cup T} = \Gal(M_{S \cup T}/K)$. The following diagram illustrates some of the relevant global fields and Galois groups:
    
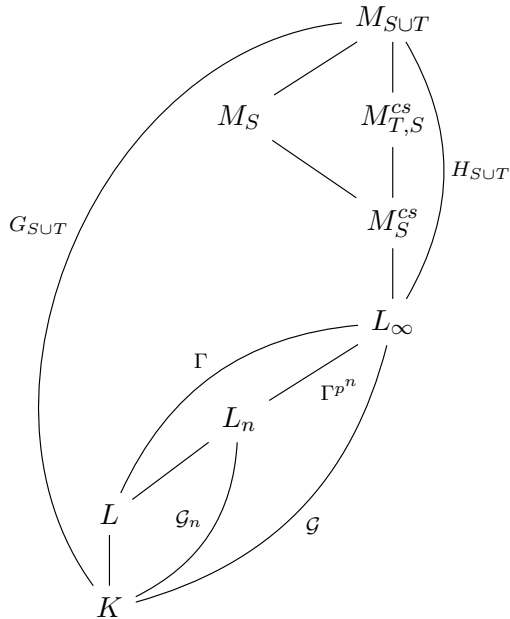
\vspace{-0.2em}
    \begin{center}
        \begin{tikzcd}
                                                                                                                                    &                                                                   & M_{S \cup T} \arrow[ddd, "H_{S \cup T}", no head, bend left]    \\
                                                                                                                                    & M_S \arrow[ru, no head]                                           & {M_{T, S}^{cs}} \arrow[u, no head]                              \\
                                                                                                                                    &                                                                   & M_S^{cs} \arrow[lu, no head] \arrow[u, no head]                 \\
                                                                                                                                    &                                                                   & L_\infty \arrow[u, no head] \arrow[ld, "\Gamma^{p^n}", no head] \\
                                                                                                                                    & L_n \arrow[ld, no head] \arrow[ldd, "\cG_n"', no head, bend left] &                                                                 \\
            L \arrow[rruu, "\Gamma", no head, bend left]                                                                            &                                                                   &                                                                 \\
            K \arrow[u, no head] \arrow[rruuu, "\cG"', no head, bend right] \arrow[rruuuuuu, "G_{S \cup T}", no head, bend left=60] &                                                                   &
        \end{tikzcd}
    \end{center}

    One of the reasons we restrict ourselves to the cyclotomic $\ZZ_p$-extension is the fact that non-archimedean places of the bottom field are \textit{finitely split} in it, that is, they split only into finitely many places in the top field. The other main reason is the \textbf{weak Leopoldt conjecture}\index{weak Leopoldt conjecture}, which is known in the cyclotomic case. It asserts the boundedness of the so-called \textit{Leopoldt defect} along a $\ZZ_p$-tower, although we shall make use of a cohomological formulation:
    \begin{thm}
    \label{thm:weak_leopoldt}
        Setting \ref{sett:construction}. One has $H_2(H_{S \cup T}, \ZZ_p) = 0$.
    \end{thm}
    The homology group is defined as in lemma \ref{lem:cohomology_homology_duality} (or immediately before it) by regarding $\ZZ_p$ as a $\Lambda(H_{S \cup T})$-module, where the $H_{S \cup T}$-action is trivial. A proof can be found in \cite{nsw} theorems 10.3.22 and 10.3.25.

    The Iwasawa algebra at the centre of our construction is $\Lambda(\cG)$. Some general properties of such rings are discussed in section \ref{sec:iwasawa_algebras_and_modules}. $\Lambda(\cG)$ contains the integral domain $\Lambda(\Gamma) \iso \ZZ_p[[T]]$, over which it is a free module of finite rank equal to $[\cG : \Gamma] = [L : K]$. In particular, a $\Lambda(\cG)$-module is finitely generated if and only if it is so over $\Lambda(\Gamma)$. As a consequence, since $\Lambda(\Gamma)$ is Noetherian, so is $\Lambda(\cG)$. Unlike $\Lambda(\cG)$, the classical Iwasawa algebra $\Lambda(\Gamma)$ has many desirable properties, and restriction of scalars to it will be an essential algebraic tool when studying the $\Lambda(\cG)$-modules of interest to us.

    In setting \ref{sett:construction}, we have fixed a prolongation $v^c$ to $\QQ^c$ of each place $v$ of $K$. If $K'/K$ is Galois with group $G$, we may denote the decomposition group $G_{v(K')}$ simply by $G_v$ - and we will often do so for $K' = L_\infty$. This is a slight abuse of notation as the choice of $v(K')$ (that is, of $v^c$) is not canonical. However, the resulting objects will always be canonically isomorphic. The clearest example is the compact induction $\Ind_{\cG_{w_\infty}}^\cG M = \Lambda(\cG) \ctp_{\Lambda(\cG_{w_\infty})} M$ (section \ref{sec:iwasawa_algebras_and_modules}) of a compact $\Lambda(\cG_{w_\infty})$-module $M$, where $w_\infty$ is an arbitrary prolongation of $v$ to $L_\infty$. If $w_\infty'$ is a different prolongation and $\sigma \in \cG$ is any Galois automorphism sending $w_\infty$ to $w_\infty'$, one has $\cG_{w_\infty} = \sigma^{-1} \cG_{w_\infty'} \sigma$. Consider the $\Lambda(\cG_{w_\infty'})$-module ${}^{\sigma^{-1}}M$, which is identical to $M$ as a $\ZZ_p$-module but has $\cG_{w_\infty'}$-action given by $\tau \cdot^{\sigma^{-1}} m = (\sigma^{-1} \tau \sigma) m$. Then the continuous $\ZZ_p$-linear map given by
    \begin{align*}
        \Ind_{\cG_{w_\infty}}^\cG M & \to \Ind_{\cG_{w_\infty'}}^\cG {}^{\sigma^{-1}}M \\
        \lambda \otimes m & \mapsto \sigma \lambda \sigma^{-1} \otimes m
    \end{align*}
    for $\lambda \in \Lambda(\cG), m \in M$ is a well-defined isomorphism of $\Lambda(\cG)$-modules.

    We are thus justified in using the notation $\cG_v$ for $\cG_{v(L_\infty)}$, and in particular $\Ind_{\cG_v}^\cG -$ for $\Ind_{\cG_{v(L_\infty)}}^\cG -$. What does play an important role in some formal arguments is the fact that the chosen set of prolongations is compatible, that is, $v^c \mid \cdots \mid v(L_\infty) \mid v(L_n) \mid v(L) \mid v$. A natural way to think of compact induction is as the sum of one copy of the local (i.e. $\Lambda(\cG_v)$-)module $M$ for each place of $L_\infty$ above $v$, endowed with a natural $\Lambda(\cG)$-action by having $\cG$ permute these local components. This is especially clear if $\cG_v$ is open in $\cG$, which is the case for all non-archimedean $v$.

    A situation we will face repeatedly in this chapter is the following: suppose given an inverse system $(\set{M_n}, \set{\varphi_n \colon M_{n + 1} \to M_n}, n \in \NN)$, where each $M_n$ is a compact $\Lambda((\cG_n)_v)$-module and $\varphi_n$ is a $\Lambda((\cG_{n + 1})_v)$-homomorphism. Formally, one can regard this as an inverse system in the category of left $\Lambda(\cG)$-modules via the canonical augmentation maps $\Lambda(\cG) \sa \Lambda(\cG_n)$. The inductions $\Ind_{(\cG_n)_v}^{\cG_n} M_n$ then have natural transition maps from level $n + 1$ to level $n$, induced by the $\varphi_n$ and the projection of the local component at a prime of $L_{n + 1}$ to the prime of $L_n$ below it. The crucial idea now is the fact that the inverse limit of this new system coincides with the outcome of first taking the inverse limit $\varprojlim_n M_n$, which yields a $\Lambda(\cG_v)$-module, and only then inducing this local module to $\cG$. More rigorously, there is a canonical $\Lambda(\cG)$-isomorphism
    \begin{equation}
    \label{eq:induction_inverse_limit}
        \varprojlim_n \Ind_{(\cG_n)_v}^{\cG_n} M_n \iisoo \Ind_{\cG_v}^\cG \varprojlim_n M_n,
    \end{equation}
    or in other words, \textit{induction commutes with inverse limits}. Note that this is essentially part i) of lemma \ref{lem:properties_of_induction} together with the fact that a compact module coincides with the inverse limit of its modules of coinvariants (section \ref{sec:iwasawa_algebras_and_modules}).

    This concludes the necessary preparations for the construction of the complex $\cC_{S, T}\q$.

    \section{\texorpdfstring{The global complex $\cT_{S \cup T}^\bullet$}{The global complex}}
    \label{sec:the_global_complex}

        Our global complex is given by a four-term exact sequence of $\Lambda(\cG)$-modules, which shall be regarded as a complex concentrated in degrees 0 and 1. It is essentially an application of the \textit{translation functor} defined by Ritter and Weiss (cf. \cite{rw_tate_sequence}), so the aim of this brief section is to show what explicit form it takes in our setting.

        For the sake of generality, we consider a finite set $\Sigma$ of places of $K$ containing $S_\infty \cup S_{\ram}(L_\infty/K)$ - the choice of interest to us is $\Sigma = S \cup T$. As discussed above, the maximal $\Sigma$-ramified extension $M_\Sigma$ of $L_\infty$ is Galois over both $L_\infty$ and $K$. We use the same notation $G_\Sigma = \Gal(M_\Sigma/K)$ and $H_\Sigma = \Gal(M_\Sigma/L_\infty)$. In particular, $H_\Sigma \trianglelefteq_c G_\Sigma$ and $\cG \iso  G_\Sigma / H_\Sigma$.

        The starting point of the construction is the canonical short exact sequence of $\Lambda(G_\Sigma)$-modules
        \begin{equation}\label{eq:ses_augmentation_gs}
            0 \to \Delta(G_\Sigma) \to \Lambda(G_\Sigma) \xrightarrow{\aug} \ZZ_p \to 0
        \end{equation}
        in the notation of section \ref{sec:iwasawa_algebras_and_modules}. Recall that taking $H_\Sigma$-coinvariants of a $\Lambda(G_\Sigma)$-module amounts to taking its quotient by $\Delta(G_\Sigma, H_\Sigma)$, which is defined by the short exact sequence
        \begin{equation}\label{eq:ses_generalised_augmentation_gs}
            0 \to \Delta(G_\Sigma, H_\Sigma) \to \Lambda(G_\Sigma) \to \Lambda(\cG) \to 0
        \end{equation}
        Applying the $H_\Sigma$-coinvariants functor to \eqref{eq:ses_augmentation_gs} yields a long exact sequence of $\Lambda(\cG)$-modules
        \begin{equation}
        \label{eq:coinvariants_sequence_global}
            \cdots \to H_1(H_\Sigma, \Lambda(G_\Sigma)) \to H_1(H_\Sigma, \ZZ_p) \to \Delta(G_\Sigma)_{H_\Sigma} \to \Lambda(G_\Sigma)_{H_\Sigma} \to (\ZZ_p)_{H_\Sigma} \to 0,
        \end{equation}
        the dual of \eqref{eq:long_exact_sequence_cohomology_iwasawa}. Looking at each term more closely, we have:
        \begin{itemize}
            \item{
                $H_1(H_\Sigma, \Lambda(G_\Sigma)) = 0$. Note that $H_1(H_\Sigma, \Lambda(G_\Sigma)) \iso H^1(H_\Sigma, \Lambda(G_\Sigma)^\vee)^\vee$ by lemma \ref{lem:cohomology_homology_duality} and
                \begin{align*}
                    \Lambda(G_\Sigma)^\vee & = \Hom_{cts}\left(\Lambda(G_\Sigma), \faktor{\QQ_p}{\ZZ_p}\right) \\
                    & \iso \varinjlim_{\substack{n \in \NN \\ U \trianglelefteq_o G_\Sigma}} \Hom_{cts}\left(\fzmod{p^n}\left[\faktor{G_\Sigma}{U}\right], \faktor{\QQ_p}{\ZZ_p}\right) \\
                    & \iso \varinjlim_{U \trianglelefteq_o G_\Sigma} \Map\left(\faktor{G_\Sigma}{U}, \faktor{\QQ_p}{\ZZ_p}\right) \\
                    & = \Map_{cts}\left(G_\Sigma, \faktor{\QQ_p}{\ZZ_p}\right)
                \end{align*}
                 is $G_\Sigma$-induced. Therefore, it is also $H_\Sigma$-induced and hence cohomologically trivial by \cite{nsw} propositions 1.3.6 (ii) and 1.3.7.
            }
            \item{
                $H_1(H_\Sigma, \ZZ_p) \iso H_\Sigma^{ab}(p)$, the maximal abelian pro-$p$ quotient of $H_\Sigma$, or the Galois group over $L_\infty$ of its maximal $\Sigma$-ramified abelian pro-$p$ extension. This is the classical $\Sigma$-ramified Iwasawa module, which we denote by $X_\Sigma$.
            }
            \item{
                A middle term $\Delta(G_\Sigma)_{H_\Sigma}$, which we denote by $Y_\Sigma$.
            }
            \item{
                $\Lambda(G_\Sigma)_{H_\Sigma} \iso \Lambda(\cG)$ by \eqref{eq:ses_generalised_augmentation_gs}.
            }
            \item{
                $(\ZZ_p)_{H_\Sigma} = \ZZ_p$, since the $H_\Sigma$-action on $\ZZ_p$ is trivial.
            }
        \end{itemize}

        Thus, \eqref{eq:coinvariants_sequence_global} becomes a four-term exact sequence
        \[
            0 \to X_\Sigma \to Y_\Sigma \to \Lambda(\cG) \to \ZZ_p \to 0,
        \]
        which proves the following:
        \begin{prop}
        \label{prop:global_complex}
            Setting \ref{sett:construction}. Let $\Sigma$ be a finite set of places of $K$ containing $S_\infty \cup S_{\ram}(L_\infty/K)$. Then there exists a cochain complex of $\Lambda(\cG)$-modules concentrated in degrees 0 and 1
            \[
                \cT_\Sigma\q = [\stackrel{0}{Y_\Sigma} \to \stackrel{1}{\Lambda(\cG)}]
            \]
            such that $H^0(\cT_\Sigma\q) = X_\Sigma$ and $H^1(\cT_\Sigma\q) = \ZZ_p$. It is referred to as the \textbf{global complex}\index{global complex}\index{complex!global}.
        \end{prop}

    \section{\texorpdfstring{The local complexes $\cL_v\q$}{The local complexes}}
    \label{sec:the_local_complexes}

        We now use a similar process to construct a local complex $\cL_v$ at any place $v$ of $K$ (whether archimedean or non-archimedean). Several such complexes will be added together and induced to the global Galois group $\cG$ in order to define the main local-to-global morphism in the next section.

        Consider first the case of a \textbf{non-archimedean place} $v$ of $K$. Let us explore the connection between the global and local settings. By \textit{local} we mean we consider the tower of fields
        \[
            (\QQ^c)_{v^c} \ / \ (L_\infty)_{v(L_\infty)} \ / \ L_{v(L)} \ / \ K_v,
        \]
        where $(\QQ^c)_{v^c}$ denotes the union $\varinjlim_{K'/\QQ} K_{v(K')}'$ of the completions of all number fields $K'$ at the compatible set of places determined by $v^c$. This field identifies naturally with the algebraic closure $K_v^c$ of $K_v$ (cf. \cite{nsw} proposition 8.1.5), which determines an isomorphism $(G_K)_{v^c} \iso G_{K_v}$. Analogous notation will be used for other infinite algebraic extensions of $K$. For instance, $(L_\infty)_{v(L_\infty)}$ is the union of the completions of all number fields contained in $L_\infty$, and one has ${\Gal((L_\infty)_{v(L_\infty)}/K_v) \iso \cG_v}$. We regard these isomorphisms, and similar ones below, as identifications. Their dependence on the choice of $v^c$ is rendered irrelevant after taking induction, as explained in the introduction to this chapter.

        The extension $(L_\infty)_{v(L_\infty)} / L_{v(L)}$ has Galois group $\Gamma_{v(L_\infty)}$, which is open in $\Gamma$ and hence has the form $\Gamma^{p^n}$ for some $n \in \NN$. This means that $L_{v(L)} = (L_n)_{v(L_n)}$, that is, localising may cause the first few (finitely many) layers of the cyclotomic $\ZZ_p$-extension to \textit{collapse}. Note that the field $(L_\infty)_{v(L_\infty)}$ is still the cyclotomic $\ZZ_p$-extension of $L_{v(L)}$. In the non-$p$-adic case $v \notin S_p$, it must necessarily be the so-called \textit{unramified $\ZZ_p$-extension} of $L_{v(L)}$, since that is the only $\ZZ_p$-extension of an $l$-adic field for $l \neq p$. To avoid further notational clutter, we denote the absolute Galois group of $(L_\infty)_{v(L_\infty)}$ by $G_{L_\infty, v} \trianglelefteq_c G_{K_v}$, instead of $G_{(L_\infty)_{v(L_\infty)}}$. As before, this coincides with the decomposition group of $v^c$ in $G_{L_\infty}$.

        As for the extension $L/K$, it is locally trivial ($L_{v(L)} = K_v$) if $v$ splits completely in it, and the converse is true if the extension is Galois.

        \begin{rem}
            The field $M_{S \cup T}$, which played a prominent role in the construction of the global complex, does not make an appearance in the local case. As the maximal $(S \cup T)$-ramified extension of $K$, one could expect it to locally coincide with the entire algebraic closure of $K_v$ for $v \in S \cup T$. However, this is a subtle question related to the global realisation of local extensions, and we shall only briefly touch on it in the next section. The correct top field to consider in the local case is  $K_v^c$. \qedef
        \end{rem}

        We now proceed along the same lines as in section \ref{sec:the_global_complex}, starting at the short exact sequence of $\Lambda(G_{K_v})$-modules
        \[
            0 \to \Delta(G_{K_v}) \to \Lambda(G_{K_v}) \to \ZZ_p \to 0
        \]
        and taking $G_{L_{\infty, v}}$-coinvariants. This yields a long exact sequence of compact $\Lambda(\cG_v)$-modules
        \begin{equation}\label{eq:ses_augmentation_l}
            \cdots \to H_1(G_{L_{\infty, v}}, \Lambda(G_{K_v})) \to H_1(G_{L_{\infty, v}}, \ZZ_p) \to \Delta(G_{K_v})_{G_{L_{\infty, v}}} \to \Lambda(G_{K_v})_{G_{L_{\infty, v}}} \to (\ZZ_p)_{G_{L_{\infty, v}}} \to 0.
        \end{equation}
        By the same argument as for the global sequence \eqref{eq:coinvariants_sequence_global}, this reduces to
        \[
            0 \to G_{L_{\infty, v}}^{ab}(p) \to Y_v \to \Lambda(\cG_v) \to \ZZ_p \to 0,
        \]
        where $Y_v = \Delta(G_{K_v})_{G_{L_{\infty, v}}}$. Note that $G_{L_{\infty, v}}^{ab}(p)$ is the Galois group over $(L_\infty)_{v(L_\infty)}$ of its maximal abelian pro-$p$ extension.

        We now move on to the case of an \textbf{archimedean place} $v$ of $K$. The completions at places above $v$ then have the usual interpretation as $\RR$ or $\CC$ according on whether those places are real or complex, but this will not play a part in the construction. Decomposition groups of archimedean places are \textit{very small}, either trivial or of order 2 generated by complex conjugation. This simplifies the argument considerably, with the caveat that inducing from $\cG_v$ to $\cG$, as we will do later on, is no longer a finite process - now decomposition groups are not open.

        This time we directly consider the short exact sequence
        \begin{equation}
        \label{eq:ses_augmentation_la}
            0 \to \Delta(\cG_v) \to \Lambda(\cG_v) \to \ZZ_p \to 0
        \end{equation}
        and distinguish two cases:
        \begin{enumerate}[i)]
            \item{
                If $v(L_\infty)$ is \textit{unramified} in $L_\infty/K$, then $\Delta(\cG_v)$ is trivial and the above sequence is an equality $\ZZ_p = \ZZ_p$. This is always the case if $p = 2$ by assumption.
            }
            \item{
                If $v(L_\infty)$ is \textit{ramified} in $L_\infty/K$, let $\tau_{v(L_\infty)}$ denote the complex conjugation induced by $v(L_\infty)$ on $L_\infty$. We then have a canonical decomposition
                \[
                    \Lambda(\cG_v) \iisoo \left(\frac{1 + \tau_{v(L_\infty)}}{2}\right) \Lambda(\cG_v) \oplus \left(\frac{1 - \tau_{v(L_\infty)}}{2}\right) \Lambda(\cG_v)
                \]
                (because $2 \neq p$ is invertible in $\ZZ_p$), the two coefficients being precisely the primitive (central) idempotents of the semisimple algebra $\QQ_p^c[\cG_v]$ (cf. section \ref{sec:representations_of_finite_groups}). The two summands on the right-hand side are $\ZZ_p$-free of rank 1 (generated by these idempotents), sometimes denoted by $\ZZ_p^+$ and $\ZZ_p^-$. They are the maximal submodules of $\Lambda(\cG_v)$ on which $\tau_{v(L_\infty)}$ acts as the identity and as -1, respectively. The augmentation ideal $\Delta(\cG_v)$ coincides with $\ZZ_p^-$, and mapping $z \in \ZZ_p \subseteq \Lambda(\cG_v)$ to $z (1 + \tau_{v(L_\infty)})/2$ makes the diagram
                \begin{center}
                    \begin{tikzcd}
                        0 \arrow[r] & \Delta(\cG_v) \arrow[r] \arrow[d, equals] & \Lambda(\cG_v) \arrow[r] \arrow[d, equals] & \ZZ_p \arrow[d, "\rsim"] \arrow[r] & 0 \\
                        0 \arrow[r] & \ZZ_p^- \arrow[r]                 & \Lambda(\cG_v) \arrow[r]           & \ZZ_p^+ \arrow[r]         & 0
                    \end{tikzcd}
                \end{center}
                commute.
            }
        \end{enumerate}

        \begin{rem}
        \phantomsection
        \label{rem:local_complexes_general_construction}
            \begin{enumerate}[i)]
                \item{
                    Instead of starting directly at \eqref{eq:ses_augmentation_la}, we could have considered the analogous sequence for $G_{K_v}$ and taken $G_{L_\infty, v}$-coinvariants as we did in the non-archimedean case. Note that the term $H_1(G_{L_\infty, v}, \ZZ_p)$ of the analogous sequence to \eqref{eq:ses_augmentation_l} would vanish, since $G_{L_\infty, v}$ (of order $2$) has no non-trivial pro-$p$ quotient. This approach leads to the very same complexes.
                }
                \item{
                    The ramification (or splitting) behaviour of an archimedean place in $L_\infty/K$ is already determined at $L/K$, since these places are never ramified in $\ZZ_p$-extensions: their decomposition groups are finite, but $\ZZ_p$ has no non-trivial finite subgroups.
                }
            \end{enumerate}
        \end{rem}

        In order to uniformise the notation, we set $Y_v = \Delta(G_{K_v})_{G_{L_{\infty, v}}}$ as before. By the previous remark, this module is trivial if $v$ is unramified in $L_\infty/K$, and $\ZZ_p^-$  otherwise. We can now collect the constructions from this section:

        \begin{prop}
        \label{prop:local_complexes}
            Setting \ref{sett:construction}. For every place $v$ of $K$ there exists a cochain complex of $\Lambda(\cG_v)$-modules
            \[
                \cL_v\q = [\stackrel{0}{Y_v} \to \stackrel{1}{\Lambda(\cG_v)}]
            \]
            such that $H^0(\cL_v\q) = G_{L_\infty, v}^{ab}(p)$ and $H^1(\cL_v\q) = \ZZ_p$. If $v$ is archimedean, then $H^0(\cL_v\q) = 0$ and
            \[
                \cL_v\q =
                \begin{cases}
                    [0 \to \ZZ_p], & v \et{is unramified in} L_\infty/K \\
                    [\Delta(\cG_v) \to \Lambda(\cG_v)], & otherwise.
                \end{cases}
            \]
            It is referred to as the \textbf{local complex}\index{local complex}\index{complex!local} at $v$.
        \end{prop}

        Although this construction works for arbitrary $v$, the definition of our main complex only relies on the local complexes $\cL_v\q$ for places $v \in S$.

    \section{Local-to-global maps}
    \label{sec:local-to-global_maps}

        At this point, we have defined a global complex $\cT_{S \cup T}\q$ and local complexes $\cL_v\q$. It is time to define morphisms from the latter to the former, which we shall then induce from $\cG_v$ to $\cG$. As mentioned at the beginning of this chapter, conceptually, induction amounts to adding a copy of $\cL_v\q$ for each prolongation of $v$ to $L_\infty$ and justifies the notation $\cL_v\q$ for an object which is actually defined at $v(L_\infty)$. Before applying induction, the goal is therefore to construct morphisms of complexes of $\Lambda(\cG_v)$-modules $\cL_v\q \to \cT_{S \cup T}\q$.

        We first address degree 0 - that is, the modules $Y_v$ and $Y_{S \cup T}$. For a place $v \in S$, consider the continuous group homomorphism $f_v$ given by either composition in the commutative diagram
        \begin{equation}
        \label{eq:local_to_global_maps_galois}
            \begin{tikzcd}
                G_{K_v} \arrow[r, hook] \arrow[d, two heads] \arrow[rd, "f_v", dashed] & G_K \arrow[d, two heads] \\
                (G_{S \cup T})_{v(M_{S \cup T})} \arrow[r, hook]                       & G_{S \cup T}
            \end{tikzcd}
        \end{equation}
        in the notation of the previous sections. Recall here that the injection $G_{K_v} \ia G_K$ relies on the choice of a prolongation $v^c$ of $v$ to $\QQ^c$ made in setting \ref{sett:construction}. The map $f_v$ induces a continuous ring homomorphism $\Lambda(G_{K_v}) \to \Lambda(G_{S \cup T})$ which commutes with the augmentation and hence restricts to $\Delta(G_{K_v}) \to \Delta(G_{S \cup T})$. Since $f_v(G_{L_{\infty, v}}) \subseteq H_{S \cup T}$, the image of $\Delta(G_{L_{\infty, v}})$ under this homomorphism is contained in $\Delta(H_{S \cup T})$, which allows us to define a homomorphism of $\Lambda(\cG_v)$-modules
        \[
            \alpha_v^0 \colon \faktor{\Delta(G_{K_v})}{\Delta(G_{L_{\infty, v}}) \Delta(G_{K_v})} \to \faktor{\Delta(G_{S \cup T})}{\Delta(H_{S \cup T})\Delta(G_{S \cup T})},
        \]
        that is, $Y_v \to Y_{S \cup T}$ (cf. \eqref{eq:definition_coinvariants}). Note that this argument does not depend on whether $v$ is archimedean or not: as shown in remark \ref{rem:local_complexes_general_construction}, the description of the degree-zero term $Y_v$ of $\cL_v\q$ as $\Delta(G_{K_v})_{G_{L_{\infty, v}}}$ applies to both archimedean and non-archimedean $v$.

        Before continuing with the construction, we briefly touch on the natural question: for $v$ non-archimedean, should we expect $f_v \colon G_{K_v} \to G_{S \cup T}$ to be injective or surjective? The answer is, in general, \textit{no} to both. The injectivity of $f_v$ is equivalent to that of the left vertical arrow in \eqref{eq:local_to_global_maps_galois}, i.e. to whether $(M_{S \cup T})_{v(M_{S \cup T})}$ coincides with $K_v^c$. This is known to be the case for \textit{very large} (of density 1, in particular infinite) sets $S \cup T$ (see for instance \cite{nsw} theorem 9.4.3), but that is far from our situation. As to the failure of surjectivity in general, \cite{nsw} theorem 10.8.1 contains the following example: if $K$ is not totally real and $v \notin S_p$, then the image of $f_v$ is far from all of $G_{S \cup T}$. Note that the cited theorem refers to the maximal pro-$p$ quotients of these groups, but this implies $f_v$ itself cannot be surjective.

        The maps in degree 1 are substantially simpler: we let
        \[
            \alpha_v^1 \colon \Lambda(\cG_v) \ia \Lambda(\cG)
        \]
        be the canonical embedding. Together with $\alpha_v^0$, this induces a morphism of complexes of $\Lambda(\cG_v)$-modules $\alpha_v \colon \cL_v\q \to \cT_{S \cup T}\q$, which is to say that the diagram
        \begin{center}
            \begin{tikzcd}
                Y_v \arrow[d, "\alpha_v^0"] \arrow[r] & \Lambda(\cG_v) \arrow[d, "\alpha_v^1"] \\
                Y_{S \cup T} \arrow[r]                & \Lambda(\cG)
            \end{tikzcd}
        \end{center}
        commutes. Indeed, all the arrows are homomorphisms of topological $\ZZ_p$-algebras, so it suffices to verify this for the classes in $Y_v = \Delta(G_{K_v})_{G_{L_{\infty, v}}}$ of the generators $\set{\sigma - 1 : \sigma \in G_{K_v}}$ of $\Delta(G_{K_v})$. But on these elements, commutativity is clear.

        Now that we have the desired maps, we induce the local complexes from $\cG_v$ to all of $\cG$. The first important realisation is that, on the exact sequence
        \[
            0 \to H^0(\cL_v\q) \to [Y_v \to \Lambda(\cG_v)] \to H^1(\cL_v\q) \to 0
        \]
        defining $\cL_v\q$, the compact induction functor $\Ind_{\cG_v}^\cG - = \Lambda(\cG) \ctp_{\Lambda(\cG_v)} -$ introduced in section \ref{sec:iwasawa_algebras_and_modules} coincides with $\Lambda(\cG) \otimes_{\Lambda(\cG_v)} -$ and is exact:
        \begin{itemize}
            \item{
                If $v$ is non-archimedean, then $\cG_v$ is open in $\cG$ and therefore $\Lambda(\cG)$ is a free right $\Lambda(\cG_v)$-module of finite rank - the generators being any set of left coset representatives of $\cG/\cG_v$.
            }
            \item{
                If $v$ is archimedean, all $\Lambda(\cG_v)$-modules in $\cL_v\q$ and its cohomology ($\Delta(\cG_v), \Lambda(\cG_v)$ and $\ZZ_p$) are finitely generated, and hence finitely presented. Note here that $\Lambda(\cG_v)$ is Noetherian as it contains the subring $\ZZ_p$, which is itself Noetherian and of finite index in $\Lambda(\cG_v)$.
            }
        \end{itemize}
        In both cases, equation \eqref{eq:completed_tensor_product_equal} shows the desired equality of functors. Since $\Lambda(\cG)$ is a flat right $\Lambda(\cG_v)$-module (this is even true after replacing $\cG_v$ by any closed subgroup of $\cG$ by  \cite{venjakob} lemma B.1), exactness follows. In other words, the complex
        \[
            \Ind_{\cG_v}^\cG \cL_v\q = [\stackrel{0}{\Ind_{\cG_v}^\cG Y_v} \to \stackrel{1}{\Ind_{\cG_v}^\cG \Lambda(\cG_v)}]
        \]
        has cohomology groups
        \begin{equation}
        \label{eq:local_cohomology_compatible_induction}
            H^i(\Ind_{\cG_v}^\cG \cL_v\q) = \Ind_{\cG_v}^\cG H^i(\cL_v\q).
        \end{equation}

        $\cT_{S \cup T}\q$ has the structure of a complex of $\Lambda(\cG)$-(and not only $\Lambda(\cG_v)$-)modules, so induction yields a morphism of complexes of $\Lambda(\cG)$-modules
        \[
            \Ind_{\cG_v}^\cG \alpha_v \colon \Ind_{\cG_v}^\cG \cL_v\q \to \cT_{S \cup T}\q
        \]
        as mentioned in \eqref{eq:partial_induction}. Some of the resulting maps will be studied more closely when determining the cohomology of the main complex in the next section.

    \section{\texorpdfstring{The main complex $\cC_{S, T}\q$}{The main complex}}
    \label{sec:the_main_complex}

        We may now bring together the construction of the local and global complexes, as well as the morphisms between them, into the definition of our complex of interest. On the local side, the complexes at all places $v \in S$ are combined into
        \[
            \cL_S\q = \bigoplus_{v \in S} \Ind_{\cG_v}^\cG \cL_v\q.
        \]
        The maps from the last section then produce a morphism $\alpha_{S, T} = \sum_{v \in S} \Ind_{\cG_v}^\cG \alpha_v \colon \cL_S\q \to \cT_{S \cup T}\q$.
        \begin{defn}
        \label{defn:main_complex}
            Setting \ref{sett:construction}. The \textbf{main complex}\index{main complex}\index{complex!main} $\cC_{S, T}\q$ of $\Lambda(\cG)$-modules is defined as
            \[
                \cC_{S, T}\q = \Cone(\cL_S\q \xrightarrow{\alpha_{S, T}} \cT_{S \cup T}\q)[-1].
            \]
            \qedef
        \end{defn}

        The fact that this complex is perfect will be essential for the formulation of the Main Conjecture. We recall that a \textbf{strictly perfect} complex\index{strcitly perfect complex}\index{complex!strictly perfect} is a bounded complex consisting of finitely generated projective modules, and a \textbf{perfect} complex\index{perfect complex}\index{complex!perfect} is one which is isomorphic to a strictly perfect complex in the relevant derived category - in this case, $\cD(\Lambda(\cG))$. The perfection of $\cC_{S, T}\q$ boils down to the following technical result from the theory of presentations of Iwasawa modules (cf. \cite{nsw} section V\S6):

        \begin{lem}
        \label{lem:y_modules_projective_dimension}
            Let $N \ia G \sa Q$ be a short exact sequence of profinite groups such that $G$ is topologically finitely generated and has cohomological $p$-dimension $\cd_p G \leq 2$, and $H_2(N, \ZZ_p) = 0$. Then the $\Lambda(Q)$-module $Y_{G, N} = \Delta(G)_N$ has projective dimension $\pd_{\Lambda(Q)} Y_{G, N} \leq 1$.

            In particular, the complex of $\Lambda(Q)$-modules
            \[
                C\q = [\stackrel{0}{Y_{G, N}} \xrightarrow{v} \stackrel{1}{\Lambda(Q)}]
            \]
            is perfect.
        \end{lem}

        \begin{proof}
            By proposition \cite{nsw} 5.6.7, there exists an exact sequence of $\Lambda(Q)$-modules
            \begin{equation}
            \label{eq:differential_presentation_y}
                0 \to H_2(N, \ZZ_p) \to (M^{ab}(p))_N \to \Lambda(Q)^d \to Y_{G, N} \to 0,
            \end{equation}
            where $M^{ab}(p)$ is the so called \textit{$p$-relation module} of $G$ with respect to a presentation
            \[
                1 \to M \to F \to G \to 1
            \]
            of $G$ by some free profinite group $F$ of rank $d$ - which is not unique: larger $d$ simply results in a larger kernel $M$. Then $M^{ab}(p)$ has a natural $\Lambda(G)$-module structure (with $G$ acting by lifting and conjugating), which in turn makes $(M^{ab}(p))_N$ into a $\Lambda(Q)$-module. The latter is is projective whenever $\cd_p G \leq 2$ by the same proposition, which together with the triviality of $H_2(N, \ZZ_p)$ shows that \eqref{eq:differential_presentation_y} is a projective resolution of $Y_{G, N}$ of length at most 1.

            The last claim in the lemma follows from the classical fact that a bounded complex consisting of modules of finite projective dimension is perfect - which can be shown by induction on the length of the complex using the fact that the cone of a morphism of perfect complexes is perfect.
        \end{proof}

        The consequence of relevance to us is the following:

        \begin{prop}
        \label{prop:complexes_are_perfect}
            Setting \ref{sett:construction}. The global complex $\cT_{S \cup T}\q$ and the main complex $\cC_{S, T}\q$ are perfect complexes of $\Lambda(\cG)$-modules. For any place $v$ of $K$, the local complex $\cL_v\q$ is a perfect complex of $\Lambda(\cG_v)$-modules.
        \end{prop}

        \begin{proof}
            Both the global and local complexes are of the form considered in lemma \ref{lem:y_modules_projective_dimension}. In the global case $\cT_{S \cup T}\q$, we apply it to the short exact sequence of profinite groups $H_{S \cup T} \ia G_{S \cup T} \sa \cG$ in the notation of section \ref{sec:the_global_complex}. The group $G_{S \cup T}$ is topologically finitely generated by \cite{nsw} corollary 10.11.15 and has $\cd_p G_{S \cup T} = 2$ by \cite{nsw} proposition 10.11.3. The hypothesis $H_2(H_{S \cup T}, \ZZ_p) = 0$ holds by the validity of the weak Leopoldt conjecture (theorem  \ref{thm:weak_leopoldt}).

            For the local complex $\cL_v\q$, we distinguish two cases. If $v$ is archimedean, then $\cL_v\q$ is in fact strictly perfect as seen directly from proposition \ref{prop:local_complexes}. Recall at this point that, in the archimedean ramified case (where $p \neq 2$ by assumption), $\Delta(\cG_v) \iso \ZZ_p^-$ is a direct summand of $\Lambda(\cG_v)$ and hence projective.

            For non-archimedean $v$, we again apply lemma \ref{lem:y_modules_projective_dimension} to $G_{L_\infty, v} \ia G_{K_v} \sa \cG_v$ in the notation of section \ref{sec:the_local_complexes}. This time around, theorems 7.4.1 and 7.1.8 from \cite{nsw} give finite topological generatedness of $G_{K_v}$ and $\cd_p G_{K_v} = 2$, respectively. The vanishing of $H_2(G_{L_\infty, v}, \ZZ_p) \iso H^2(G_{L_\infty, v}, \QQ_p/\ZZ_p)^\vee$ (cf. lemma \ref{lem:cohomology_homology_duality}) follows from
            \[
                H^2(G_{L_\infty, v}, \QQ_p/\ZZ_p) = \varinjlim_n H^2(G_{L_n, v}, \QQ_p/\ZZ_p) = 0,
            \]
            where $G_{L_n, v}$ denotes the absolute Galois group of $(L_n)_{v(L_n)}$ and the equalities are \cite{nsw} propositions 1.5.1 and 7.3.10.

            In order to conclude that the main complex is perfect, we first note that $\Ind_{\cG_v}^\cG \cL_v\q$ is perfect over $\Lambda(\cG)$ for any $v$: both modules in $[Y_v \to \Lambda(\cG_v)]$ have projective dimension $\leq$ 1 by the previous proposition, and compact induction sends projective resolutions to projective resolutions. Hence, as (a shift of) the cone of a morphism between perfect complexes, $\cC_{S, T}\q = \Cone(\cL_S\q \xrightarrow{\alpha_{S, T}} \cT_{S \cup T}\q)[-1]$ is itself perfect.
        \end{proof}

        We now turn our attention to the cohomology\footnote{The structure of these cohomology groups will also follow from the isomorphism (in the derived category $\cD(\Lambda(\cG))$) between $\cC_{S, T}\q$ and certain $R\Gamma$-complex proved in section \ref{sec:description_in_terms_of_rgamma_complexes}. We compute them explicitly here mainly for two reasons: first, this involves the study of some classical objects and sequences which will be used in later chapters. Second, doing so provides a self-contained argument which does not rely on known results for the $R\Gamma$-complex. As a minor additional justification, section \ref{sec:description_in_terms_of_rgamma_complexes} requires $\cG$ to be abelian for compatibility with \cite{bks}, although this is only formal - the same isomorphism holds in the general case.} of the main complex. Its determination involves the following $\Lambda(\cG)$-modules, all of which arise by taking classical Galois modules, extending scalars to $\ZZ_p$ and passing to the inverse limit over the cyclotomic tower $L_\infty/L$:
        \begin{itemize}
            \item{
                $E_{S, T} = \varprojlim_n (\ZZ_p \otimes \units{\cO_{L_n, S, T}})$, where $\units{\cO_{L_n, S, T}}$ is the group of $(S, T)$-units of $L_n$ (cf. \eqref{eq:definition_st-units}) and the inverse limit is taken with respect to the norm maps.
            }
            \item{
                $E_S = E_{S, \varnothing} = \varprojlim_n (\ZZ_p \otimes \units{\cO_{L_n, S}})$, where $\units{\cO_{L_n, S}}$ is the group of $S$-units of $L_n$ (cf. \eqref{eq:definition_s-units}).
            }
            \item{
                $\cY_S = \bigoplus_{v \in S} \Ind_{\cG_v}^\cG \ZZ_p$, where $\ZZ_p$ is endowed with the trivial $\cG_v$-action. This is an instance of the situation described by equation \eqref{eq:induction_inverse_limit} with $M_n = \ZZ_p$ for all $n$ and transition maps $M_{n + 1} \to M_n$ equal to the identity. Therefore, $\cY_S$ can be identified with the inverse limit of
                \begin{equation}
                \label{eq:y_finite_level}
                    \cY_{L_n, S} = \bigoplus_{v \in S} \Ind_{(\cG_n)_{v(L_n)}}^{\cG_n} \ZZ_p \iso \bigoplus_{w_n \in S(L_n)} \ZZ_p \cdot w_n,
                \end{equation}
                where the last term is a free $\ZZ_p$-module with $\cG_n$-action given by Galois conjugation of places.

                By \eqref{eq:local_cohomology_compatible_induction}, $\cY_S$ coincides with $H^0(\cL_S\q)$
            }
            \item{
                $\cX_S = \ker(\cY_S \sa \ZZ_p)$ is the kernel the of canonical projection (which is essentially an augmentation map). For the same reason as above, it can be identified with the inverse limit of the kernels $\cX_{L_n, S} = \ker(\cY_{L_n, S} \sa \ZZ_p)$. Note that $\cX_{L_n, S}$ is simply $\ZZ_p \otimes \cX_{L_n, S}^\ZZ$ in the notation of section \ref{sec:artin_l-series}, and analogously for $\cY_{L_n, S}$.

                The module $\cX_S$ should not be confused with $X_S$, the Galois group over $L_\infty$ of its maximal $S$-ramified abelian pro-$p$ extension; and neither should $\cY_S$ with the $Y_S = \Delta(G_S)_{H_S}$ defined in section \ref{sec:the_global_complex}.
            }
        \end{itemize}

        In subsequent chapters, we will use analogous notation $E_{S', T'}, E_{S'}, \cY_{S'}, \cX_{S'}$ for sets of places $S'$ and $T'$ other than $S$ and $T$. These modules will be instrumental in the formulation of the Main Conjecture (see the beginning of chapter \ref{chap:formulation_of_the_main_conjecture}, in particular the homomorphism $\alpha \colon \cY_{S_\infty} \ia E_{S, T}$). All of them are finitely generated over $\Lambda(\cG)$ (for $E_S$, this follows from its structure as a $\Lambda(\Gamma)$-module given in \cite{nsw} theorem 11.3.11; and for $E_{S, T} \subseteq E_S$, this is a consequence of the Noetherianity of $\Lambda(\cG)$).

        In preparation for the computation of the cohomology of $\cC_{S, T}\q$, we determine the kernel and cokernel of
        \[
            H^0(\alpha_{S, T}) \colon H^0(\cL_S\q) \to H^0(\cT_{S \cup T}\q).
        \]
        By section \ref{sec:the_local_complexes}, $H^0(\cL_S\q) = \bigoplus_{v \in S_f} \Ind_{\cG_v}^\cG G_{L_\infty, v}^{ab}(p)$. Local class field theory yields an isomorphism of $\Lambda(\cG_v)$-modules
        \[
            G_{L_\infty, v}^{ab}(p) \iisoo \varprojlim_n \big(\units{(L_n)_{v(L_n)}}\big)^\wedge,
        \]
        where
        \[
            \big(\units{(L_n)_{v(L_n)}}\big)^\wedge = \varprojlim_m \faktor{\units{(L_n)_{v(L_n)}} \; }{\big(\units{(L_n)_{v(L_n)}}\big)^{p^m}}
        \]
        denotes the $p$-completion of the group of units of the local field $(L_n)_{v(L_n)}$. Thus, $H^0(\cL_S\q)$ identifies with the Iwasawa module $A_S$ in the notation of \cite{nsw} section XI\S3 (note that, for archimedean places $v \in S_\infty$, there is nothing to add as cohomology vanishes in degree 0). It should be pointed out that the definitions of $E_{S, T}$ and $E_S$ (and also $\cY_S$ and $\cX_S$) boil down to a limit of $p$-completions as well, but in those instances $p$-completion amounts to $\ZZ_p \otimes -$ because the abelian groups involved are finitely generated.

        \begin{prop}
        \label{prop:kernel_cokernel_h0_alpha}
            Setting \ref{sett:construction}. There exists a short exact sequence of $\Lambda(\cG)$-modules
            \[
                0 \to E_{S, T} \to H^0(\cL_S\q) \xrightarrow{H^0(\alpha_{S, T})} H^0(\cT_{S \cup T}\q) \to X_{T, S}^{cs} \to 0,
            \]
            where $X_{T, S}^{cs} = \Gal(M_{T, S}^{cs}/L_\infty)^{ab}(p)$ is the Galois group over $L_\infty$ of its maximal $T$-ramified abelian pro-$p$ extension which is completely split at $S$.
        \end{prop}

        \begin{proof}
            The case $T = \varnothing$, i.e. $H^0(\alpha_{S, \varnothing}) \colon H^0(\cL_S\q) \to H^0(\cT_S\q) = X_S$, is settled in \cite{nsw} theorem\footnote{Because of the hypothesis $p \nmid [L : K]$ in p. 739 of the reference, the cited theorem does not directly refer to the $\Lambda(\cG)$-module structure of the objects. However, it can still be used to determine the kernel and cokernel of $H^0(\alpha_{S, \varnothing})$ as a $\Lambda(\Gamma)$-homomorphism (i.e. choosing $L$ rather than $K$ as the base field), after which we simply point out that the maps $E_S \ia A_S \iso H^0(\alpha_{S, \varnothing})$ and $X_S \sa X_S^{cs}$ are $\cG$-equivariant.} 11.3.10: the kernel and cokernel are isomorphic to $E_S$ and $X_{\varnothing, S}^{cs}$, respectively, where the latter is usually simply denoted by $X_S^{cs}$. In applying the cited theorem, we note that $H_2(H_S, \ZZ_p) = 0$ by the weak Leopoldt conjecture (theorem \ref{thm:weak_leopoldt}). The fact that $H^0(\alpha_{S, \varnothing})$ coincides with the $A_S \to X_S$ therein under $H^0(\cL_S\q) \iso A_S$ is a consequence of our definition of the local-to-global-maps, which are induced by natural inclusions and projections of Galois groups.

            We therefore need to measure the change caused by the introduction of $T$. Consider the commutative diagram with exact rows
            \begin{center}
                \begin{tikzcd}
                    0 \arrow[r] & 0 \arrow[r] \arrow[d] & H^0(\cL_S\q) \arrow[d, "H^0(\alpha_{S, T})"] \arrow[r, equal] & H^0(\cL_S\q) \arrow[r] \arrow[d, "H^0(\alpha_{S, \varnothing})"] & 0 \\
                    0 \arrow[r] & \ker(\pi) \arrow[r]    & X_{S \cup T} \arrow[r, "\pi"]  & X_S \arrow[r]           & 0
                \end{tikzcd}
            \end{center}
            where $\pi$ is the natural projection - that is, restriction of Galois automorphisms. The snake lemma then yields an exact sequence
            \begin{equation}
            \label{eq:five_term_sequence_alpha}
                0 \to \ker(H^0(\alpha_{S, T})) \to E_S \to \ker(\pi) \to \coker(H^0(\alpha_{S, T})) \to X_S^{cs} \to 0.
            \end{equation}

            In order to determine the unknown terms, we construct the same sequence in a different manner, starting at the exact sequence of $\ZZ[\cG_n]$-modules
            \begin{equation}
            \label{eq:classical_five_term_sequence_units}
                1 \to \units{\cO_{L_n, S, T}} \to \units{\cO_{L_n, S}} \to \bigoplus_{w_n \in T(L_n)} \units{\kappa(w_n)} \to Cl_{L_n, S, T} \to Cl_{L_n, S} \to 1,
            \end{equation}
            where $Cl_{L_n, S}$ and $Cl_{L_n, S, T}$ denote the $S$-class group and $(S, T)$-ray class group at the layer $L_n$, respectively. These are the quotients of the class group and ray class group mod ${\fm_T = \prod_{w_n \in T(L_n)} w_n}$ by the subgroup generated by classes of ideals in $S_f$ (see \cite{nickel_plms} subsection 1.4 for the precise definition and the above sequence). Applying $\ZZ_p \otimes -$ is exact and yields a new sequence whose last three terms (not counting $1$) are the $p$-parts of those above, since they are finite abelian groups. In particular, $\ZZ_p \otimes \units{\kappa(w_n)} = \units{\kappa(w_n)}(p)$ is isomorphic to the group of $p$-power roots of unity of $(L_n)_{w_n}$ (by definition, places above $T$ cannot be $p$-adic).

            All terms resulting from tensoring \eqref{eq:classical_five_term_sequence_units} with $\ZZ_p$ are compact $\Lambda(\cG_n)$-modules, as they are finitely generated over $\ZZ_p$ (for the first two, this follows from Dirichlet's unit theorem: see \eqref{eq:dirichlet_regulator_map_real}). Therefore, taking inverse limits along the cyclotomic tower with respect to the norm maps produces a sequence of $\Lambda(\cG)$-modules
            \begin{equation}
            \label{eq:modified_five_term_sequence_alpha}
                0 \to E_{S, T} \to E_S \to \bigoplus_{v \in T^p} \Ind_{\cG_v}^\cG \ZZ_p(1) \to X_{T, S}^{cs} \to X_S^{cs} \to 0
            \end{equation}
            which is exact by \cite{nsw} proposition 5.2.4. In the middle term, we have used the fact that induction commutes with inverse limits (recall \eqref{eq:induction_inverse_limit}). Note that the group of $p$-power roots of unity in the cyclotomic $\ZZ_p$-extension of $L_{v(L)}$ is trivial if $L_{v(L)}$ itself contains no primitive $p$-th root of unity, and $\mu_{p^\infty}$ (i.e. all possible) otherwise. Hence the only terms that survive in the limit are those corresponding to places $v \in T^p$ as defined in setting \ref{sett:construction}, which yield precisely $\ZZ_p(1)$ (see equation \eqref{eq:tate_twist_rou}). Finally, the fact that the limits of $Cl_{L_n, S, T}$ and $Cl_{L_n, S}$ are precisely $X_{T, S}^{cs}$ and $X_S^{cs}$ is a well-known consequence of global class field theory.

            The last step is to connect sequences \eqref{eq:five_term_sequence_alpha} and \eqref{eq:modified_five_term_sequence_alpha}. The middle terms $\bigoplus_{v \in T^p} \Ind_{\cG_v}^\cG \ZZ_p(1)$ and $\ker(\pi)$ are isomorphic by \cite{nsw} theorem 11.3.5 (the isomorphism of the summands with $\Ind_{\cG_v}^\cG \ZZ_p(1)$ is addressed only in the proof) - this follows from applying it to $S \cup T$ and $S$ separately, then using the snake lemma. For the fourth terms, we use the fact that the canonical projection ${X_{S \cup T} \sa X_{T, S}^{cs}}$ is trivial on the image of $H^0(\cL_S\q)$, as the latter is generated precisely by the projections to $X_{S \cup T}$ of the decomposition groups $G_{L_\infty, v}^{ab}(p) = (G_{L_\infty})_{v^c}^{ab}(p)$ for $v \in S$. It therefore factors as ${\varepsilon \colon \coker(H^0(\alpha_{S, T})) \sa X_{T, S}^{cs}}$. This yields a diagram with exact rows
            \begin{center}
                \begin{tikzcd}[column sep=small]
                    0 \arrow[r] & {\ker(H^0(\alpha_{S, T}))} \arrow[r] \arrow[d, dashed] & E_S \arrow[r] \arrow[d, equals] & \ker(\pi) \arrow[r] \arrow[d, "\rsim"]                            & {\coker(H^0(\alpha_{S, T}))} \arrow[r] \arrow[d, two heads, "\varepsilon"] & X_S^{cs} \arrow[d, equals] \arrow[r] & 0 \\
                    0 \arrow[r] & {E_{S, T}} \arrow[r]                           & E_S \arrow[r]                       & \bigoplus_{v \in T^p} \Ind_{\cG_v}^\cG \ZZ_p(1) \arrow[r] & {X_{T, S}^{cs}} \arrow[r]                                          & X_S^{cs} \arrow[r]                       & 0
                \end{tikzcd}
            \end{center}
            where the first vertical arrow is defined by commutativity of the corresponding square, and all other squares commute by the aforementioned naturality of the maps. Now a simple diagram chase shows the first and fourth vertical arrows are isomorphisms as well.
        \end{proof}

        \begin{rem}
            The argument used to define $\varepsilon \colon \coker(H^0(\alpha_{S, T})) \sa X_{T, S}^{cs}$ above is not far from showing directly that it is an isomorphism. The bulk of the proof is aimed at determining the cohomology in degree 0, which is subtler as it relies on class field theory. The proof also showcases some objects and techniques which will be relevant later on - for instance, in subsection \ref{subsec:the_choice_of_s_and_t}. \qedef
        \end{rem}

        We can now determine the cohomology of the main complex:

        \begin{thm}
        \label{thm:cohomology_of_complex}
            Setting \ref{sett:construction}. The perfect complex $\cC_{S, T}\q$ is acyclic outside degrees 0 and 1 and satisfies:
            \begin{itemize}
                \item{
                    $H^0(\cC_{S, T}\q) \iso E_{S, T}$.
                }
                \item{
                    $H^1(\cC_{S, T}\q)$ fits into the short exact sequence of $\Lambda(\cG)$-modules
                    \[
                        0 \to X_{T, S}^{cs} \to H^1(\cC_{S, T}\q) \to \cX_S \to 0,
                    \]
                    with $X_{T, S}^{cs}$ and $\cX_S$ as defined in and before proposition \ref{prop:kernel_cokernel_h0_alpha}, respectively.
                }
            \end{itemize}
        \end{thm}

        \begin{proof}
            In the derived category $\cD(\Lambda(\cG))$, the complex fits in an exact triangle
            \[
                \cC_{S, T}\q \to \cL_S\q \xrightarrow{\alpha_{S, T}} \cT_{S \cup T}\q \to
            \]
            by definition. The cohomology of $\cL_S\q$ and $\cT_{S \cup T}\q$ is given by propositions \ref{prop:global_complex} and \ref{prop:local_complexes}, and it is trivial outside degrees 0 and 1 in both cases. Therefore, the long exact cohomology sequence becomes
            \begin{center}
                \begin{tikzcd}
                    0 \arrow[r] & {H^0(\cC_{S, T}\q)} \arrow[r] & H^0(\cL_S\q) \arrow[rr, "{H^0(\alpha_{S, T})}"] &  & H^0(\cT_{S \cup T}\q) \arrow[llld, out=0, in=180, looseness=1.5, overlay] &                               &   \\
                                & {H^1(\cC_{S, T}\q)} \arrow[r] & H^1(\cL_S\q) \arrow[rr, "{H^1(\alpha_{S, T})}"] &  & H^1(\cT_{S \cup T}\q) \arrow[r]    & {H^2(\cC_{S, T}\q)} \arrow[r] & 0
                \end{tikzcd}
            \end{center}
            and $\cC_{S, T}\q$ is acyclic outside degrees 0, 1 and 2. It follows that $H^0(\cC_{S, T}\q)$ is isomorphic to the kernel of $H^0(\alpha_{S, T})$ and $H^1(\cC_{S, T}\q)$ fits into the short exact sequence
            \[
                0 \to \coker(H^0(\alpha_{S, T})) \to H^1(\cC_{S, T}\q) \to \ker(H^1(\alpha_{S, T})) \to 0.
            \]
            As already discussed, $H^1(\cL_S\q) \iso \cY_S$ and the map $H^1(\alpha_{S, T}) \colon H^1(\cL_S\q) \to H^1(\cT_{S \cup T}\q) = \ZZ_p$ identifies with the canonical projection $\cY_S \sa \ZZ_p$, whose kernel is $\cX_S$ by definition. Since it is surjective, $H^2(\cC_{S, T}\q)$ is trivial. The kernel and cokernel of $H^0(\alpha_{S, T})$ were computed in proposition \ref{prop:kernel_cokernel_h0_alpha}, which concludes the proof.
        \end{proof}

    \section{\texorpdfstring{Description in terms of $R\Gamma$-complexes}{Description in terms of RGamma-complexes}}
    \label{sec:description_in_terms_of_rgamma_complexes}

        This section is devoted to showing that the complex constructed in this chapter is isomorphic to that employed by Burns, Kurihara and Sano in \cite{bks} (pp. 1534 and 1539-1540) in the cases where the settings coincide. This is a necessary step for section \ref{sec:the_conjecture_of_burns_kurihara_and_sano}, where we prove that the Main Conjectures formulated here (section \ref{sec:the_main_conjecture}) and in said article (conjecture 3.1) are essentially the same whenever they can be compared.

        The differences between the settings are as follows:
        \begin{itemize}
            \item{
                Here $L_\infty/K$ is only assumed to be Galois, whereas it is required to be abelian in \cite{bks}. The definition of the complex therein does not truly require commutativity (see for instance \cite{burns_flach}, first paragraph of section 3.2), but the formulation of the Main Conjecture (conjecture 3.1) does.
            }
            \item{
                Here $K_\infty$ is the cyclotomic $\ZZ_p$-extension of $K$, whereas it is an arbitrary $\ZZ_p$-extension in the article. Only at times do they specialise to the cyclotomic case - namely in subsections 3D and 4B and, to some extent, section 5. We note that the cyclotomic $\ZZ_p$-extension is the only one if $K$ is totally real and Leopoldt's conjecture holds for $K$ and $p$.
            }
            \item{
                Our condition \textit{$p \neq 2$ if $K$ is not totally imaginary} does not appear in the earlier part of the cited article - neither in the definition of the complex nor in the formulation of the Main Conjecture. The condition $p \neq 2$ is indeed required in their reformulation of the same conjecture (\cite{bks} conjecture 3.14).
            }
        \end{itemize}

        The present section, slightly more technical in its use of derived categories, $R\Gamma$-complexes and duality properties, may be skipped without significant effect on the comprehension of the sequel - the fundamental takeaway is theorem \ref{thm:iso_bks_complex}, i.e. the isomorphism between the complexes outlined above.

        We start by introducing the complex used in \cite{bks}, which is essentially given by an inverse limit of finite-level complexes from \cite{burns_flach}:

        \begin{defn}
        \phantomsection
        \label{defn:complexes_bks}
            \begin{enumerate}[i)]
                \item{
                    Let $p$ be a prime number, $F/K$ a Galois extension of number fields and $S$ and $T$ two disjoint finite sets of places of $K$ such that $S \supseteq S_\infty \cup S_{\ram}(F/K)$. We set
                    \[
                        \cB_{F, S, \varnothing}\q = R \Hom_{\ZZ_p}(R \Gamma_c(H_{F, S}, \ZZ_p), \ZZ_p)[-2],
                    \]
                    where $H_{F, S}$ denotes the Galois group over $F$ of its maximal $S$-ramified extension and $R \Gamma_c$ denotes the compactly supported cohomology complex (cf. \cite{burns_flach} equation (3)). We then define $\cB_{F, S, T}\q$ by the exact triangle
                    \begin{equation}
                    \label{eq:exact_triangle_bks}
                        \cB_{F, S, T}\q \to \cB_{F, S, \varnothing}\q \to \bigoplus_{w \in T(F)} \ZZ_p \otimes \units{\kappa(w)}[0] \to,
                    \end{equation}
                    where the second arrow is as in \cite{bks}.
                }
                \item{
                    Setting \ref{sett:construction}. We define
                    \[
                        \cB_{S, T}\q = \varprojlim_n \cB_{L_n, S, T}\q,
                    \]
                    where the limit is taken with respect to the natural transition maps induced by restriction (note that $H_{L_{n + 1}, S} \leq H_{L_n, S}$). This is the complex featured in \cite{bks} conjecture 3.1.
                }
            \end{enumerate}
        \end{defn}

        Let us examine in which category one may naturally consider these complexes. In part i), $\cB_{F, S, \varnothing}\q$ has a canonical structure as a complex of $\Lambda(\Gal(F/K))$-modules in the usual way (cf. \cite{nsw} proposition 1.6.3) because the maximal $S$-ramified extensions of $F$ and $K$ coincide, and hence $H_{F, S}$ is normal in the Galois group of this extension over $K$. The term $\bigoplus_{w \in T(F)} \ZZ_p \otimes \units{\kappa(w)}$ also has a natural $G = \Gal(F/K)$-action, which in particular allows us to regard it as $\bigoplus_{v \in T} \Ind_{G_v}^G \ZZ_p \otimes \units{\kappa(w)}$ (as usual, here $G_v$ denotes $G_{w_0}$ for any chosen prolongation $w_0$ of $v$). This endows $\cB_{F, S, T}\q$ itself with the structure of a complex of $\Lambda(\Gal(F/K))$-modules. As for part ii), the transition maps in $\varprojlim_n \cB_{L_n, S, T}$ are Galois-equivariant, and hence $\cB_{S, T}\q$ is a complex of $\Lambda(\cG)$-modules.

        We recall at this point that $R\Gamma(G, M)$ is notation for the complex which computes the cohomology of $G$ with coefficients in $M$ (in some suitable category, which in our case was described in section \ref{sec:iwasawa_algebras_and_modules}) - either the one resulting from the standard resolution, or any isomorphic one in the relevant derived category. We must first relate some fashion of $R\Gamma$-complexes to complexes of the form considered in previous sections. This is achieved by the following proposition, the idea of whose proof is largely taken from \cite{nickel_plms} (see theorem 2.4 therein):

        \begin{prop}
        \label{prop:iso_rgamma_complexes}
            Let $G$ be a profinite group and $N \trianglelefteq_c G$ a closed normal subgroup such that $H_n(N, \ZZ_p) = 0$ for all $n \geq 2$, and set $Q = G/N$. Then there exists an isomorphism of complexes
            \[
                R\Hom(R\Gamma(N, \QQ_p/\ZZ_p), \QQ_p/\ZZ_p) \iisoo [\stackrel{-1}{Y_{G, N}} \xrightarrow{v} \stackrel{0}{\Lambda(Q)}]
            \]
            in the derived category $\cD(\Lambda(Q))$, where $Y_{G, N} = \Delta(G)_N = \Delta(G)/\Delta(N)\Delta(G)$ and $\upsilon$ is the natural composition of the lift to $\Delta(G)$ with the projection to $\Lambda(Q)$.
        \end{prop}

        \begin{proof}
            Cohomology $H^i(N, -)$ is by definition the right derived functor of $-^N = \Hom_{\Lambda(N)}(\ZZ_p, -)$. Therefore, $R\Gamma(N, \QQ_p/\ZZ_p)$ can be identified with $R\Hom_{\Lambda(N)}(\ZZ_p, \QQ_p/\ZZ_p)$ in $\cD(\Lambda(Q))$ (here and below, the $\Lambda(Q)$-module structure is given by \cite{nsw} proposition 1.6.3 as before).

            The contravariant functor $R\Hom(-, \QQ_p/\ZZ_p)$ coincides with $R\Hom_{\ZZ_p}(-, \QQ_p/\ZZ_p)$ on complexes of discrete modules, and it amounts to applying $\Hom_{\ZZ_p}(-, \QQ_p/\ZZ_p) = (-)^\vee$ at each degree as $\QQ_p/\ZZ_p$ is $\ZZ_p$-injective. Noting that $M \ctp_{\Lambda(N)} \ZZ_p \iso \Hom_{\Lambda(N)}(M, \QQ_p/\ZZ_p)^\vee$ for any compact right $\Lambda(N)$-module $M$ (see the proof of \cite{nsw} corollary 5.2.9), we conclude that
            \begin{equation}
            \label{eq:iso_rg_complex_tensor}
                    R\Hom(R\Gamma(N, \QQ_p/\ZZ_p), \QQ_p/\ZZ_p) \iisoo \ZZ_p \ctp_{\Lambda(N)}^\LL \ZZ_p
            \end{equation}
            in $\cD(\Lambda(Q))$ (where $- \ctp_{\Lambda(N)}^\LL \ZZ_p$ denotes the left derived functor of $- \ctp_{\Lambda(N)} \ZZ_p$) by dualising a projective resolution of $\ZZ_p$ into an injective one of $\QQ_p/\ZZ_p$.

            Now we consider the canonical short exact sequence $\Delta(G) \ia \Lambda(G) \sa \ZZ_p$ of compact $\Lambda(N)$-modules, which induces an exact triangle
            \[
                \Delta(G) \ctp_{\Lambda(N)}^\LL \ZZ_p \to \Lambda(G) \ctp_{\Lambda(N)}^\LL \ZZ_p \to \ZZ_p \ctp_{\Lambda(N)}^\LL \ZZ_p \to
            \]
            in $\cD(\Lambda(Q))$ and hence an isomorphism
            \begin{equation}
            \label{eq:zp_derived_cone}
                \ZZ_p \ctp_{\Lambda(N)}^\LL \ZZ_p \iisoo \Cone(\Delta(G) \ctp_{\Lambda(N)}^\LL \ZZ_p \to \Lambda(G) \ctp_{\Lambda(N)}^\LL \ZZ_p).
            \end{equation}

            The groups $H_i(N, \Lambda(G))$ vanish for all $i \geq 1$, since $\Lambda(G)^\vee \iso \Map_{cts}(G, \QQ_p/\ZZ_p)$ is $G$-induced and therefore cohomologically trivial (see also the first point after \eqref{eq:coinvariants_sequence_global}). In particular, there are isomorphisms $H_{i + 1}(N, \ZZ_p) \iso H_i(N, \Delta(G))$ for all $i \geq 1$. This implies $H_i(N, \Delta(N)) = 0$ for all $i \geq 1$ by hypothesis, and hence the complexes $\Delta(G) \ctp_{\Lambda(N)}^\LL \ZZ_p$ and $\Lambda(G) \ctp_{\Lambda(N)}^\LL \ZZ_p$ are both acyclic outside degree 0. But a complex with a single non-trivial cohomology group is isomorphic in the derived category to the complex consisting of that group concentrated in the corresponding degree, i.e.
            \[
                \Delta(G) \ctp_{\Lambda(N)}^\LL \ZZ_p \iso (\Delta(G) \ctp_{\Lambda(N)} \ZZ_p)[0] = Y_{G, N}[0]
            \]
            and
            \[
                \quad \Lambda(G) \ctp_{\Lambda(N)}^\LL \ZZ_p \iso (\Lambda(G) \ctp_{\Lambda(N)} \ZZ_p)[0] = \Lambda(Q)[0].
            \]
            This concludes the proof by \eqref{eq:iso_rg_complex_tensor} and \eqref{eq:zp_derived_cone}, since by definition of the cone we have
            \[
                \Cone(Y_{G, N}[0] \to \Lambda(Q)[0]) = [\stackrel{-1}{Y_{G, N}} \xrightarrow{v} \stackrel{0}{\Lambda(Q)}]
            \]
            and the morphisms defining both cones are compatible.
        \end{proof}

        \begin{rem}
        \label{rem:complex_independent_of_g_q}
            It follows that, as a complex of abelian groups, $[Y_{G, N} \to \Lambda(Q)]$ is actually independent of $G$ and $Q$ up to isomorphism in the derived category - as long as they fit into a short exact sequence $N \ia G \sa Q$. This should not be too surprising, as the cohomology groups of complexes of that form were shown to be $N^{ab}(p)$, $\ZZ_p$ and $\set{1}$ in sections \ref{sec:the_global_complex} and \ref{sec:the_local_complexes}. Naturally, the $\Lambda(Q)$-structure does depend on the choice of groups, and it will be fundamental in our discussion. \qedef
        \end{rem}

        The above proposition can now be applied to the global and local complexes to obtain isomorphisms
        \begin{equation}
        \label{eq:iso_global_rhom}
            \cT_{S \cup T}\q \iisoo R\Hom(R\Gamma(H_S, \QQ_p/\ZZ_p), \QQ_p/\ZZ_p)[-1]
        \end{equation}
        and
        \[
            \cL_v\q \iisoo R\Hom(R\Gamma(G_{L_\infty, v}, \QQ_p/\ZZ_p), \QQ_p/\ZZ_p)[-1],
        \]
        where $G_{L_\infty, v}$ denotes $G_{(L_\infty)_{v(L_\infty)}} = (G_{L_\infty})_{v^c}$ as in section \ref{sec:the_local_complexes}. The hypothesis on the vanishing of $H_n(N, \ZZ_p)$ is satisfied in these cases by the same reasons described at the beginning of the proof of proposition \ref{prop:complexes_are_perfect}. The only case not covered there is that of an archimedean place $v$, which is even simpler by classical homology of finite groups: if $G_{L_\infty, v}$ is trivial, then its homology in degree $\geq 1$ vanishes for any module; and otherwise, $\abs{G_{L_\infty, v}} = 2$ (so $p$ is odd by setting \ref{sett:construction}) and hence the groups $H_n(G_{L_\infty, v}, \ZZ_p)$ are 2-torsion $\ZZ_p$-modules, i.e. trivial, for all $n \geq 1$.

        In combination with some duality properties, this new description of the global and local complexes allows us to prove the main result of this section. As mentioned before, nothing in the construction of $\cB_{S, T}\q$ requires $\cG$ to be abelian and we can therefore state the following theorem in the generality of setting \ref{sett:construction}.

        \begin{thm}
        \label{thm:iso_bks_complex}
            Setting \ref{sett:construction}. There is an isomorphism of complexes $\cC_{S, T}\q \iso \cB_{S, T}\q$ in the derived category $\cD(\Lambda(\cG))$.
        \end{thm}

        \begin{proof}
            As in some of the previous proofs, it is convenient to initially disregard the set $T$ altogether and consider
            \[
                \cC_{S, \varnothing}\q = \Cone(\cL_S\q \xrightarrow{\alpha_{S, \varnothing}} \cT_S\q)[-1].
            \]

            Let us consider the local side first, letting $G_{L_n, v} = G_{(L_n)_{v(L_n)}} = (G_{L_n})_{v^c}$ by analogy with $G_{L_\infty, v}$. We also set $\cG_{n, v} = (\cG_n)_{v(L_n)}$. Then there are isomorphisms
            \begin{align*}
                \cL_S\q  = \bigoplus_{v \in S} \Ind_{\cG_v}^\cG \cL_v\q & \iso \bigoplus_{v \in S} \Ind_{\cG_v}^\cG R\Hom(R\Gamma(G_{L_\infty, v}, \QQ_p/\ZZ_p), \QQ_p/\ZZ_p)[-1] \\
                & \iso \bigoplus_{v \in S} \Ind_{\cG_v}^\cG R\Hom(\varinjlim_n R\Gamma(G_{L_n, v}, \QQ_p/\ZZ_p), \QQ_p/\ZZ_p)[-1] \\
                & \iso \varprojlim_n \bigoplus_{v \in S} \Ind_{\cG_{n, v}}^{\cG_n} R\Hom(R\Gamma(G_{L_n, v}, \QQ_p/\ZZ_p), \QQ_p/\ZZ_p)[-1] \\
                & \iso \varprojlim_n  R\Hom\big(\bigoplus_{v \in S} \Ind_{\cG_{n, v}}^{\cG_n} R\Gamma(G_{L_n, v}, \QQ_p/\ZZ_p), \QQ_p/\ZZ_p\big)[-1],
            \end{align*}
            in $\cD(\Lambda(\cG))$, where
            \begin{itemize}
                \item{
                    the first isomorphism is essentially proposition \ref{prop:iso_rgamma_complexes}, which we can apply as explained after remark \ref{rem:complex_independent_of_g_q}. On the left-hand side, $\Ind_{\cG_v}^\cG -$ denotes the functor $\Lambda(\cG) \ctp_{\Lambda(\cG_v)} -$ applied to the local complex degree-wise. The isomorphism $ \cL_v\q \iso R\Hom(R\Gamma(G_{L_\infty, v}, \QQ_p/\ZZ_p), \QQ_p/\ZZ_p)[-1]$, however, takes place in the derived category $\cD(\Lambda(\cG_v))$, where induction is really \textit{derived induction}, i.e. the derived completed tensor product $\Lambda(\cG) \ctp_{\Lambda(\cG_v)}^\LL -$ applied to complexes. The reason these two notions of induction coincide in this case is that, as discussed at the end of section \ref{sec:local-to-global_maps}, $\Lambda(\cG) \ctp_{\Lambda(\cG_v)} -$ is exact on the sequence defining $\cL_v\q$ for all $v$.
                }
                \item{
                    the second isomorphism is \cite{nsw} proposition 1.5.1. Note that $G_{L_\infty, v} = \bigcap_n G_{L_n, v}$ and this intersection is simply an inverse limit with respect to the canonical embeddings. Even though the cited result only refers to cohomology groups, its proof shows the isomorphism exists on the level of $R\Gamma$-complexes.
                }
                \item{
                    the last two isomorphisms are formal: $R\Hom$ takes colimits (in particular direct limits and sums) in the first component to limits (in particular inverse limits and products) and, in the cases above, $\Ind_{\cG_{n, v}}^{\cG_n} -$ commutes with $\varprojlim_n$ in the derived category. To see this, note that if $v$ is a non-archimedean place, then $\Ind_{\cG_v}^\cG -$ applied to a complex amounts to a finite sum of complexes isomorphic to it; and if $v$ is archimedean, then the inverse limit coincides with the complex at the layer $L_0 = L$ (such places do not split in the cyclotomic tower).
                }
            \end{itemize}

            By the same token, the global complex is isomorphic to
            \[
                \cT_S\q \iisoo R\Hom(R\Gamma(H_S, \QQ_p/\ZZ_p), \QQ_p/\ZZ_p)[-1] \iisoo \varprojlim_n R\Hom(R\Gamma(H_{L_n, S}, \QQ_p/\ZZ_p), \QQ_p/\ZZ_p)[-1],
            \]
            where $H_S = \Gal(M_S/L_\infty)$ and $H_{L_n, S} = \Gal(M_S/L_n)$. Recall that $M_S$ was defined as the the maximal $S$-ramified extension of $L_\infty$ - and hence of $L_n$ for all $n$.

            It follows that $\cC_{S, \varnothing}\q[2]$ is isomorphic to the cone of some morphism
            \[
                \bigg(\varprojlim_n R\Hom\big(\bigoplus_{v \in S} \Ind_{\cG_{n, v}}^{\cG_n} R\Gamma(G_{L_n, v}, \QQ_p/\ZZ_p), \QQ_p/\ZZ_p\big)\bigg) \xrightarrow{\alpha} \bigg(\varprojlim_n R\Hom(R\Gamma(H_{L_n, S}, \QQ_p/\ZZ_p), \QQ_p/\ZZ_p)\bigg)
            \]
            and therefore to the inverse limit of the cones of the finite-level morphisms
            \[
                R\Hom\big(\bigoplus_{v \in S} \Ind_{\cG_{n, v}}^{\cG_n} R\Gamma(G_{L_n, v}, \QQ_p/\ZZ_p), \QQ_p/\ZZ_p\big) \xrightarrow{\alpha_n} R\Hom(R\Gamma(H_{L_n, S}, \QQ_p/\ZZ_p), \QQ_p/\ZZ_p)
            \]
            because cones are defined in terms of direct sums and shifts, both of which commute with inverse limits. The map $\alpha$ is induced by the map $\alpha_{S, T}$ from section \ref{sec:the_main_complex} (with $T = \varnothing$) by definition and hence comes from the inflation and restriction associated to $G_{L_\infty} \sa H_S$ and $G_{L_\infty, v} \ia G_{L_\infty}$ respectively, as these are (after changing base field from $L_\infty$ to $K$) the homomorphisms used in the definition of the local-to-global maps (equation \eqref{eq:local_to_global_maps_galois}). Note that $R\Hom(-, \QQ_p/\ZZ_p)$ causes a reversal in the direction of the resulting arrows between $R\Gamma$-complexes. It follows that $\alpha_n$ is induced from $G_{L_n} \sa H_{L_n, S}$ and $G_{L_n, v} \ia G_{L_n}$.

            The definition of compact-support cohomology as the shifted cone of global-to-local restriction will yield the desired isomorphism by virtue of local and global Tate duality in their derived-categorical form. Namely, \cite{lim} theorem 4.2.6 (right column of the diagram therein) shows that the cone of $\alpha_n$ above is isomorphic in $\cD(\Lambda(\cG_n))$ to $R\Hom(R\Gamma_c(H_{L_n, S}, \QQ_p/\ZZ_p), \QQ_p/\ZZ_p)$. We therefore have isomorphisms
            \[
                \cC_{S, \varnothing}\q \iso \varprojlim_n R\Hom(R\Gamma_c(H_{L_n, S}, \QQ_p/\ZZ_p), \QQ_p/\ZZ_p)[-2] \iso \varprojlim_n R\Gamma(H_{L_n, S}, \ZZ_p(1))[1],
            \]
            the last one being global Tate duality (bottom row of the diagram in the cited theorem). Lastly, we apply Nekovář duality. Although introduced in \cite{nekovar}, a closer formulation to our setting is given by Lim and Sharifi in \cite{lim_sharifi} (theorem in p. 623). We choose $R = \ZZ_p$, dualising complex $\omega_R = \ZZ_p[0]$ and complex $T = \ZZ_p(1)[0]$. Here the Tate twist refers to the $H_{L_n, S}$-action, which is not relevant in the category $\cD(\ZZ_p)$ where we shall apply the cited theorem, but it is inherited by the cohomology groups of complexes in that category. It is then easy to see that $R\Hom_{\ZZ_p}(T, \omega_R)$ is represented by $T^\ast = \ZZ_p(-1)[0]$, and thus the aforementioned theorem states that
            \begin{equation}
            \label{eq:rgamma_finite_level}
                R\Gamma(H_{L_n, S}, \ZZ_p(1)) \iisoo R\Hom_{\ZZ_p}(R\Gamma_c(H_{L_n, S}, \ZZ_p), \ZZ_p)[-3]
            \end{equation}
            and hence
            \begin{equation}
            \label{eq:iso_inverse_zp(1)}
                \cC_{S, \varnothing}\q \iso \varprojlim_n R\Gamma(H_{L_n, S}, \ZZ_p(1))[1] \iso \varprojlim_n R\Hom_{\ZZ_p}(R\Gamma_c(H_{L_n, S}, \ZZ_p), \ZZ_p)[-2] = \cB_{S, \varnothing}\q.
            \end{equation}

            It remains to bring $T$ into the discussion. The difference between $\cC_{S, T}\q$ and $\cC_{S, \varnothing}\q$ is given by the following exact triangle in $\cD(\Lambda(\cG))$:
            \[
                \cC_{S, T}\q \to \cC_{S, \varnothing}\q \to \bigoplus_{v \in T^p} \Ind_{\cG_v}^\cG \ZZ_p(1)[0] \to.
            \]
            This is not difficult to show using sequence \eqref{eq:modified_five_term_sequence_alpha}, and we shall give a formal proof much later on, namely in proposition \ref{prop:exact_triangle_difference_s_and_t}. We stress that said proof does not use any result from this section.

            As to the complex of Burns, Kurihara and Sano, the limit of \eqref{eq:exact_triangle_bks} along the cyclotomic tower yields an exact triangle
            \[
                \cB_{S, T}\q \to \cB_{S, \varnothing}\q \to \varprojlim_n \bigoplus_{w \in T(L_n)} \ZZ_p \otimes \units{\kappa(w)}[0] \to.
            \]
            In the last term, the limit is taken with respect to the norm maps (cf. \cite{bks} p. 1540) and it was already determined in the proof of proposition \ref{prop:kernel_cokernel_h0_alpha}: it is precisely $\bigoplus_{v \in T^p} \Ind_{\cG_v}^\cG \ZZ_p(1)[0]$. The result now follows by comparing the last two exact triangles and noting that the isomorphisms between their second and third terms are compatible because they come from the natural dualities and transition maps.
        \end{proof}

\newpage
\chapter{Formulation of the Main Conjecture}
\label{chap:formulation_of_the_main_conjecture}

    At the core of our discussion lies the Main Conjecture, where a multitude of objects and techniques come together: number theory, homological algebra, representation theory, some $p$-adic analysis and, gluing it all, the localisation sequence of $K$-theory. This chapter develops the necessary tools in those areas and culminates in the formulation of the conjecture in section \ref{sec:the_main_conjecture}. Whereas the preceding chapter delved into an explicit construction of the arithmetic complex $\cC_{S, T}\q$, and subsequent ones will study some fundamental properties of the Main Conjecture, it is this chapter that plays the central role in our pursuits.

    On its most essential level, the conjecture is an existence claim in a certain group $K_1(\cQ(\cG))$. As is often the case in Iwasawa theory, the predicted \textit{zeta element $\zeta_{S, T}^{\alpha, \beta}$} should have some \textit{analytic} properties, in terms of interpolation of special $L$-values; and some \textit{arithmetic-algebraic} properties, in the form of a specific relation to $\cC_{S, T}\q$. These two conditions are captured by the vertical and horizontal arrows in the key diagram
    \begin{center}
        \begin{tikzcd}
            K_1(\cQ(\cG)) \arrow[d, "\psi_\chi"] \arrow[r, "\partial"] & {K_0(\Lambda(\cG), \cQ(\cG))} \\
            \units{\cQ^c(\Gamma_\chi)} \arrow[d, "ev_{\gamma_{\chi \otimes \rho}}"]     &                               \\
            \QQ_p^c \cup \set{\infty}                                   &
        \end{tikzcd}
    \end{center}
    respectively. Some of these objects have been introduced in a general setting in the preliminary chapter: Artin $L$-series (section \ref{sec:artin_l-series}), the connecting homomorphism $\partial \colon K_1(\cQ(\cG)) \to K_0(\Lambda(\cG), \cQ(\cG))$ (section \ref{sec:algebraic-k-theory}) and, although concealed behind the map $\psi_\chi$, reduced norms (ibid).

    In the first section, we address the concept of trivialised complexes and refined Euler characteristics. This homological tool will allow us to regard $\cC_{S, T}\q$ as an element in the relative $K$-group $K_0(\Lambda(\cG), \cQ(\cG))$ after a suitable homomorphism $\alpha \colon \cY_{S_\infty} \to E_{S, T}$ has been chosen, paving the way for the arithmetic side of the conjecture. It is here that we will first become acquainted with some algebraic properties of $\cQ(\cG)$, the total ring of fractions of $\Lambda(\cG)$. Much of the necessary machinery was developed in \cite{rwii}.

    Section \ref{sec:morphisms_on_finite_level} is devoted to showing that certain maps $\varphi_n^\alpha$ induced by $\alpha$ on the layer $L_n$ have desirable properties - namely, we prove they are isomorphisms on $\chi$-parts for almost all \textit{Artin characters} $\chi$ of $\cG$. This allows us to define \textit{regulators} (for the special $L$-values featured in the Main Conjecture) in section \ref{sec:regulators}. Regulators are a common occurrence at the intersection between algebraic and analytic number theory, and some of the rationale behind them in our context is explained at the beginning of the corresponding section. Our definition of the regulator is based on the work of Stark and Tate, especially \cite{tate}.

    We then move on to study certain evaluation maps in section \ref{sec:evaluation_maps}. The term $\units{\cQ^c(\Gamma_\chi)}$ in the above diagram is the group of units of the fraction field of a power series ring with $p$-adic coefficients. The analytic side of the conjecture claims that, when evaluated at 0, the series quotient $\psi_\chi(\zeta_{S, T}^{\alpha, \beta})$ takes the (regulated) special $L$-value for the character $\chi$. However, series associated to so-called \textit{$W$-equivalent} characters turn out to have closely related evaluation-at-0 maps, which leads us to define \textit{twisted evaluation maps $ev_{\gamma_{\chi \otimes \rho}}$}. These correspond to evaluation at (infinitely many) points in the open unit ball of $\QQ_p^c$ centred at 0. Requiring interpolation of $L$-values at almost all of these points determines a unique conjectural element $F_{S, T, \chi}^{\alpha, \beta} \in \units{\cQ^c(\Gamma_\chi)}$, which plays the role of $p$-adic $L$-functions. This Interpolation Conjecture, formulated in the final section of the chapter, constitutes the first half of the Main Conjecture. The second half is the equivariant Main Conjecture, which combines the analytic and algebraic sides by requiring that the zeta element be mapped to the interpolating series quotient $F_{S, T, \chi}^{\alpha, \beta}$ via $\psi_\chi$ for all $\chi$, and to the inverse of the refined Euler characteristic of $\cC_{S, T}\q$ via $\partial$.

    The division of the Main Conjecture into these two halves is only conceptual: one may instead postulate the existence of a $\zeta_{S, T}^{\alpha, \beta} \in K_1(\cQ(\Gamma))$ which is sent to the regulated special $L$-value directly by $ev_{\gamma_\chi} \psi_\chi$ for almost all $\chi$ (and with the same arithmetic property regarding $\cC_{S, T}\q$) - which yields a completely equivalent conjecture. This latter approach is used, for instance, in the Main Conjecture in \cite{bks} (although in the rather different language of \textit{determinant functors}). In our setting, however, the definition of twisted evaluation maps and a separate Interpolation Conjecture will facilitate the study of its properties, as will become apparent in the next chapter. As already mentioned, the consideration of $F_{S, T, \chi}^{\alpha, \beta} \in \units{\cQ^c(\Gamma_\chi)}$ also brings about a direct connection to the $p$-adic $L$-functions featured in Main Conjectures for totally real fields.

    Now that the path has been outlined, we proceed to incorporate two new elements into the setting of the previous chapter which will be necessary for the formulation of the conjecture:

    \begin{sett}
    \label{sett:formulation}
    \addcontentsline{toc}{section}{Setting B}
        All objects and assumptions from setting \ref{sett:construction} carry over to this one. Additionally, we fix the following:
        \begin{itemize}
            \item{
                A homomorphism of $\Lambda(\cG)$-modules $\alpha \colon \cY_{S_\infty} \to E_{S, T}$ such that $\ker(\alpha)$ and $\coker(\alpha)$ are  $\Lambda(\Gamma)$-torsion, where $\cY_{S_\infty}$ and $E_{S, T}$ are as in section \ref{sec:the_main_complex}. Proposition \ref{prop:existence_injectivity_alpha} below shows that such an $\alpha$ exists and is necessarily injective.
            }
            \item{
                An isomorphism of abstract fields $\beta \colon \CC_p \xrightarrow{\sim} \CC$, where $\CC_p$ denotes the field of $p$-adic complex numbers. See for instance the remark after \cite{nsw} proposition 10.3.2 for existence.
            }
        \end{itemize}
    \end{sett}

    Neither $\alpha$ nor $\beta$ is uniquely determined by the properties above. In the next chapter, we will study how the choice of these parameters affects the validity of the Main Conjecture. Note that setting \ref{sett:construction} (sometimes with the addition of $\alpha$ as above) still suffices for a few of the results in this chapter - which shall be made clear each time.

    The local structure of $E_{S, T}$ determined by Nickel allows us to prove the existence of $\alpha$:
    \begin{prop}
    \label{prop:existence_injectivity_alpha}
        Setting \ref{sett:construction}. There exists a homomorphism of $\Lambda(\cG)$-modules $\alpha \colon \cY_{S_\infty} \to E_{S, T}$ such that $\ker(\alpha)$ and $\coker(\alpha)$ are $\Lambda(\Gamma)$-torsion. Furthermore, any such $\alpha$ is injective.
    \end{prop}

    \begin{proof}
        Recall that $\Lambda_\fp(\Gamma)$ denotes the localisation of the integral domain $\Lambda(\Gamma)$ at a prime ideal $\fp$. Since $\Lambda(\Gamma)$ is central in $\Lambda(\cG)$, the localisation $\Lambda_\fp(\cG) = \Lambda_\fp(\Gamma) \otimes_{\Lambda(\Gamma)} \Lambda(\cG)$ has a natural ring structure. In particular, any $\Lambda(\cG)$-module $M$ can be localised to a $\Lambda_\fp(\cG)$-module
        \[
            M_\fp = \Lambda_\fp(\cG) \otimes_{\Lambda(\cG)} M = \Lambda_\fp(\Gamma) \otimes_{\Lambda(\Gamma)} M.
        \]

        By \cite{nickel_swan} theorem 5.5, there exists an isomorphism of $\Lambda_{\fp}(\cG)$-modules
        \[
            \varphi \colon T_\fp \oplus \Lambda_{\fp}(\cG)^{r_1 - r_1' + r_2} \oplus \bigoplus_{v \in S_\infty'} (\Ind_{\cG_v}^{\cG} \ZZ_p)_\fp \xrightarrow{\sim} (E_S)_\fp
        \]
        for almost any (and conjecturally any) height-one prime ideal $\fp$ of $\Lambda(\Gamma)$, where $T$ is either the trivial module or $\ZZ_p(1)$ (the inverse limit of $p$-power roots of unity), $S_\infty'$ is the set of archimedean places of $K$ which ramify in $L_\infty$, and $r_1, r_1'$ and $r_2$ are the number of real, real ramified (in $L_\infty$) and complex places of $K$, respectively.

        Fix a prime ideal $\fp$ for which such $\varphi$ exists and set
        \[
            \cY' = T \oplus \Lambda(\cG)^{r_1 - r_1' + r_2} \oplus \bigoplus_{v \in S_\infty'} \Ind_{\cG_v}^{\cG} \ZZ_p,
        \]
        so that $\varphi \colon (\cY')_\fp \to (E_S)_\fp$. There is a canonical embedding of
        \[
            \cY_{S_\infty} = \bigoplus_{v \in S_\infty} \Ind_{\cG_v}^{\cG} \ZZ_p = \bigoplus_{v \in S_\infty \setminus S_\infty'} \Lambda(\cG) \oplus \bigoplus_{v \in S_\infty'} \Ind_{\cG_v}^{\cG} \ZZ_p
        \]
        into $\cY'$ with cokernel $T$. Let $\cQ(\Gamma)$ denote the field of fractions of $\Lambda(\Gamma)$. Consider the composition $\varphi'$ given by
        \begin{center}
        \vspace{-1.5em}
            \begin{tikzcd}[column sep=small]
                \cQ(\Gamma) \otimes_{\Lambda(\Gamma)} \cY_{S_\infty} \arrow[r] &
                \cQ(\Gamma) \otimes_{\Lambda(\Gamma)} \cY' \arrow[r, equal] &
                \cQ(\Gamma) \otimes_{\Lambda_\fp(\Gamma)} (\cY')_\fp \arrow[d, "\cQ(\Gamma) \otimes \varphi"]                      &                  &   \\
                &                                       &
                \cQ(\Gamma) \otimes_{\Lambda_\fp(\Gamma)} (E_S)_\fp \arrow[r, equal] &
                \cQ(\Gamma) \otimes_{\Lambda(\Gamma)} E_S \arrow[r] &
                \cQ(\Gamma) \otimes_{\Lambda(\Gamma)} E_{S, T},
            \end{tikzcd}
        \end{center}
        where the last map is the inverse of the isomorphism induced by the embedding $E_{S, T} \ia E_S$ (which has $\Lambda(\Gamma)$-torsion cokernel, as can be deduced from sequence \eqref{eq:modified_five_term_sequence_alpha} and the structure theorem \ref{thm:structure_theorem_iwasawa}). Since $T$ is torsion as well, $\varphi'$ is a $\cQ(\Gamma) \otimes_{\Lambda(\Gamma)} \Lambda(\cG)$-isomorphism. We can now clear denominators by observing that the natural map
        \begin{equation}
        \label{eq:tensor_hom}
            \cQ(\Gamma) \otimes_{\Lambda(\Gamma)} \Hom_{\Lambda(\cG)}(M, N) \to \Hom_{\cQ(\Gamma) \otimes_{\Lambda(\Gamma)} \Lambda(\cG)}(\cQ(\Gamma) \otimes_{\Lambda(\Gamma)} M, \cQ(\Gamma) \otimes_{\Lambda(\Gamma)} N)
        \end{equation}
        is an isomorphism whenever $M$ is finitely generated over $\Lambda(\cG)$ (cf. \cite{cr1} corollary 8.18). Since $\cY_{S_\infty}$ is so (by $\abs{S_\infty}$ elements), there exist a non-zero $d \in \Lambda(\Gamma)$ and a homomorphism of $\Lambda(\cG)$-modules $\alpha \colon \cY_{S_\infty} \to E_{S, T}$ such that $\varphi' = d^{-1}\alpha$. In particular,
        \[
            \cQ(\Gamma) \otimes_{\Lambda(\Gamma)} \alpha \colon \cQ(\Gamma) \otimes_{\Lambda(\Gamma)} \cY_{S_\infty} \to \cQ(\Gamma) \otimes_{\Lambda(\Gamma)} E_{S, T}
        \]
        is an isomorphism (because multiplication by $d$ is itself an isomorphism after localising), which implies that $\ker(\alpha)$ and $\coker(\alpha)$ are $\Lambda(\Gamma)$-torsion.

        We now show that $\cY_{S_\infty}$ is $\Lambda(\Gamma)$-torsion free and consequently $\ker(\alpha)$ is trivial. As shown after \eqref{eq:ses_augmentation_la}, $\ZZ_p$ is a projective $\Lambda(\cG_v)$-module for all $v \in S_\infty$, and hence $\Ind_{\cG_v}^\cG \ZZ_p$ is projective over $\Lambda(\cG)$. But $\Lambda(\cG)$ is itself free as a $\Lambda(\Gamma)$-module, so $\Ind_{\cG_v}^\cG \ZZ_p$ is also projective over $\Lambda(\Gamma)$. Since projective modules over integral domains are torsion-free, the claim follows.
    \end{proof}

    \begin{rem}
    \phantomsection
    \label{rem:existence_alpha}
        \begin{enumerate}[i)]
            \item{
                It follows from the last part of the proof that $\cY_{S_\infty}$ is in fact free as $\Lambda(\Gamma)$-module, as the classical Iwasawa algebra is a local ring. This can be seen from the definition too: as a $\Lambda(\Gamma)$-module, $\cY_{S_\infty}$ is simply $\bigoplus_{w \in S_\infty(L)} \Ind_{\Gamma_{w_\infty}}^\Gamma \ZZ_p$ (use lemma \ref{lem:properties_of_induction} i)), where $w_\infty$ is any chosen prolongation of $w$ to $L_\infty$. Since archimedean places do not split in $L_\infty/L$, this sum of inductions amounts to $\Lambda(\Gamma)^{\abs{S_\infty(L)}}$.
            }
            \item{
                Once we understand the structure of $\Lambda(\cG)$ in some detail, it will become clear that whether a $\Lambda(\cG)$-module is $\Lambda(\Gamma)$-torsion or not is in fact independent of the choice of $\Gamma$ (with the conditions from setting \ref{sett:construction}). This is unnecessary for the moment and will be explained at the beginning of subsection \ref{subsec:the_choice_of_l}.
            }
        \end{enumerate}
    \end{rem}

    \section{An integral trivialisation}
    \label{sec:an_integral_trivialisation}

        The aim of this section is to construct a trivialisation for the complex $\cC_{S, T}\q$ and define the corresponding refined Euler characteristic\footnote{In our treatment of refined Euler characteristics, we follow \cite{bb} sections 5 and 6. Our conventions match those of the object $\chi^{\text{old}}$ therein, which is related to an alternative definition via theorem 6.2 from the article.} $\chi_{\Lambda(\cG), \cQ(\cG)}(\cC_{S, T}\q, t^\alpha)$ in the relative $K$-group $K_0(\Lambda(\cG), \cQ(\cG))$, which plays a major role in the formulation of the Main Conjecture. After introducing refined Euler characteristics in an abstract setting, we specialise to our objects of interest, present the fundamental ring $\cQ(\cG)$ and use the homomorphism $\alpha$ from setting \ref{sett:formulation} to obtain $\chi_{\Lambda(\cG), \cQ(\cG)}(\cC_{S, T}\q, t^\alpha)$.

        Let $R$ be a ring. Recall the notion of perfect complexes introduced in section \ref{sec:the_main_complex}: a (cochain) complex of (left) $R$-modules is called strictly perfect if it is bounded and consists of finitely generated projective modules, and perfect if it isomorphic to a strictly perfect complex in the derived category $\cD(R)$. Given a complex $C\q$ of $R$-modules, we denote its modules of $i$-th cocycles and $i$-th coboundaries by $Z^i(C\q)$ and $B^i(C\q)$ respectively, both inside $C^i$. We set
        \[
            H^\even(C\q) = \bigoplus_{i \in \ZZ}H^{2i}(C\q) \quad \text{and} \quad H^\odd(C\q) = \bigoplus_{i \in \ZZ} H^{2i + 1}(C\q)
        \]
        and define $C^\even$, $C^\odd$, $Z^\even(C\q)$, $Z^\odd(C\q)$, $B^\even(C\q)$ and $B^\odd(C\q)$ analogously. If $C\q$ is bounded (for instance, strictly perfect), all of these are finite sums. If $C\q$ is perfect, at least $H^\even(C\q)$ and $H^\odd(C\q)$ are.

        Let $\varphi \colon R \to S$ be a homomorphism of rings such that $S$ is flat as a right $R$-module via $\varphi$ and semisimple Artinian as a ring. Given an $R$-module $M$, we denote (in this section only) the $S$-module $S \otimes_R M$ by $M_S$. Analogously, $C_S\q$ denotes the result of applying $S \otimes_R -$ to the complex $C\q$ degree-wise.

        A \textbf{trivialisation}\index{trivialisation} (over $S$) for a perfect complex $C\q$ of $R$-modules is an isomorphism of $S$-modules
        \[
             t \colon H^\odd(C_S\q) \xrightarrow{\sim} H^\even(C_S\q).
        \]
        Note that $t$ need not come from extension of scalars of an $R$-homomorphism $H^\odd(C\q) \to H^\even(C\q)$. One may then refer to the pair $(C\q, t)$ as a \textbf{trivialised complex}. Such a trivialisation gives rise to a well-defined \textbf{refined Euler characteristic}\index{refined Euler characteristic} $\chi_{R, S}(C\q, t) \in K_0(R, S)$ as follows (cf. section \ref{sec:algebraic-k-theory} for the definition and properties of the relative $K_0$). Choose a strictly perfect representative $P\q$ of $C\q$ (i.e. $P\q \iso C\q$ in $\cD(R)$). The flatness of $S$ implies $H^i(P_S\q) = H^i(P\q)_S$ for all $i$, and analogously for $B^i$ and $Z^i$. Therefore, the sequences
        \[
            0 \to B^i(P_S\q) \to Z^i(P_S\q) \to H^i(P_S\q) \to 0
        \]
        are exact for all $i$ and the same is true of
        \[
            0 \to Z^i(P_S\q) \to P_S^i \to B^{i + 1}(P_S\q) \to 0.
        \]
        Furthermore, by semisimplicity of $S$, both kinds of sequences are split. A choice of splittings then gives rise to an isomorphism
        \begin{align}
        \label{eq:refined_euler_characteristic_map}
            \varphi_t \colon P_S^\odd & \iso B^\even(P_S\q) \oplus Z^\odd(P_S\q) \\
            & \iso B^\even(P_S\q) \oplus B^\odd(P_S\q) \oplus H^\odd(P_S\q) \nonumber \\
            & \iso B^\even(P_S\q) \oplus B^\odd(P_S\q) \oplus H^\even(P_S\q) \nonumber \\
            & \iso Z^\even(P_S\q) \oplus B^\odd(P_S\q) \nonumber \\
            & \iso P_S^\even \nonumber
        \end{align}
        of $S$-modules, where the third isomorphism is induced by $t$. This trivialisation was defined on the cohomology of $C\q$, but isomorphic complexes in $\cD(R)$ have isomorphic cohomology and therefore $t$ induces a trivialisation for $P\q$ as well. The \textbf{refined Euler characteristic} of the trivialised complex $(C\q, t)$ is defined as
        \begin{equation}
        \label{eq:definition_rec_general}
            \chi_{R, S}(C\q, t) = [P^\odd, \varphi_t, P^\even] \in K_0(R, S).
        \end{equation}
        Crucially, $\chi_{R, S}(C\q, t)$ is independent of the choices of a strictly perfect representative $P\q$ of $C\q$ and splittings for the short exact sequences - only the complex and the trivialisation matter. See \cite{bb}, sections 5 and 6 for a proof of this\footnote{As mentioned above, the Euler characteristic defined here corresponds to the $\chi^{\text{old}}$ in section 6 of the cited paper. Independence is really shown for a different construction in a more general setting in section 5. However, theorem 6.2 therein shows the difference between both constructions is measured by a certain $\delta_{\Lambda, \Sigma}^1(B^\odd(P_{\Sigma}), -\Id)$ which is indeed independent of the above choices.}. As an immediate consequence, if $D\q$ is another complex and $q \colon C\q \isoa D\q$ is an isomorphism in $\cD(R)$ under which $t$ becomes $t_D \colon H^\odd(D_S\q) \isoa H^\even(D_S\q)$, then
        \begin{equation}
        \label{eq:rec_independent_derived_iso}
            \chi_{R, S}(C\q, t) = \chi_{R, S}(D\q, t_D).
        \end{equation}

        In order to apply this construction to our context, we introduce the rings $\cQ(\Gamma)$ and $\cQ(\cG)$, which will in fact be essential for most of the sequel. While both were already present in the proof of proposition \ref{prop:existence_injectivity_alpha} (one explicitly and one implicitly), a closer look is in order now. In general, given a profinite group $G$, we denote by $\cQ(G)$ the total ring of fractions of the Iwasawa algebra $\Lambda(G) = \ZZ_p[[G]]$ (section \ref{sec:iwasawa_algebras_and_modules}) obtained by inverting all regular elements. For an $p$-adic field $E$ (i.e. a finite extension of $\QQ_p$), we set $\cQ^E(G) = E \otimes_{\QQ_p} \cQ(G)$. We also let
        \begin{equation}
        \label{eq:q_direct_limit_finite}
            \cQ^c(G) = \QQ_p^c \otimes_{\QQ_p} \cQ(G) = \varinjlim_E \cQ^E(G),
        \end{equation}
        with $E$ running over all $p$-adic fields and transition maps given by the natural inclusions.

        In the notation from setting \ref{sett:construction}, $\Lambda(\Gamma)$ is an integral domain which lies in the centre of  $\Lambda(\cG)$ and such that $[\Lambda(\cG) : \Lambda(\Gamma)] = [L : K]$ is finite. Therefore, for a $\Lambda(\cG)$-module, being finitely generated is equivalent to being so over $\Lambda(\Gamma)$. An essential property of $\cQ(\cG)$ is that it can in fact be obtained by adding inverses of all non-zero elements of $\Lambda(\Gamma)$ to $\Lambda(\cG)$. In other words, there is a canonical isomorphism
        \begin{equation}
        \label{eq:lambda_regular_elements}
            \cQ(\cG) \iisoo \cQ(\Gamma) \otimes_{\Lambda(\Gamma)} \Lambda(\cG)
        \end{equation}
        (see \cite{rwii} p. 551 for a proof\footnote{\label{foot:rw}We will resort to this article for multiple algebraic results. The setting of the article differs from ours in that $L_\infty$ is assumed to be totally real there, and $p$ odd. However, the algebraic results we cite are not affected by this difference.}). In particular a $\Lambda(\cG)$-module is torsion if and only if it is $\Lambda(\Gamma)$-torsion (which we may even take as a definition). It also follows that $\cQ(\cG)$ is a finite-dimensional $\cQ(\Gamma)$-algebra and $\Lambda(\cG)$ is a $\Lambda(\Gamma)$-order in it. Given the desirable properties of the ring ${\Lambda(\Gamma) \iso  \ZZ[[T]]}$ (Noetherian integral domain, regular local ring, easily characterised prime ideals, etc.), some questions concerning $\Lambda(\cG)$-modules will be tackled by restricting scalars to $\Lambda(\Gamma)$. As explained in section \ref{sec:iwasawa_algebras_and_modules}, the non-canonical isomorphism $\Lambda(\Gamma) \iso \ZZ_p[[T]]$ is the continuous $\ZZ_p$-linear map induced by
        \[
            \gamma \mapsto T + 1
        \]
        for any topological generator $\gamma$ of $\Gamma$, which in our case was fixed in setting $\ref{sett:construction}$. We sometimes refer to the elements of $\cQ(\Gamma) \iso \ffrac(\ZZ_p[[T]])$ as \textbf{series quotients}\index{series quotient}. It is worth noting that $\cQ(\Gamma)$ does not contain something like $\QQ_p[[T]]$ - we say the elements in $\cQ(\Gamma)$ have \textit{bounded denominators}. It also follows from definition \eqref{eq:q_direct_limit_finite} that every series quotient in $\cQ^c(\Gamma)$ has coefficients contained in a finite extension of $\QQ_p$ (in both the numerator and the denominator): it cannot, for instance, be of the form $\sum_{i = 0}^\infty x_i T^i$ with $\QQ_p(x_i \colon i \in \NN)$ of infinite degree over $\QQ_p$.

        In order to construct a refined Euler characteristic for the perfect complex $\cC_{S, T}\q$ of $\Lambda(\cG)$-modules from chapter \ref{chap:construction_of_the_complex}, we consider the embedding $\Lambda(\cG) \ia \cQ(\cG)$. The semisimplicity of $\cQ(\cG)$ is shown in the proof of \cite{rwii} proposition 5. The flatness of $\cQ(\cG)$ as a right $\Lambda(\cG)$-module is clear: on $\Lambda(\cG)$-modules, $\cQ(\cG) \otimes_{\Lambda(\cG)} - $ coincides with $\cQ(\Gamma) \otimes_{\Lambda(\Gamma)} -$ by \eqref{eq:lambda_regular_elements}, and fields of fractions of integral domains are flat.

        The cohomology of $\cC_{S, T}\q$ was established in theorem \ref{thm:cohomology_of_complex}: one has $H^\even(\cC_{S, T}\q) = H^0(\cC_{S, T}\q) \iso E_{S, T}$, and $H^\odd(\cC_{S, T}\q) = H^1(\cC_{S, T}\q)$ fits in a short exact sequence $X_{T, S}^{cs} \ia H^1(\cC_{S, T}\q) \sa \cX_S$. Therefore, a trivialisation amounts to an isomorphism
        \[
            t \colon \cQ(\cG) \otimes_{\Lambda(\cG)} H^1(\cC_{S, T}\q) \xrightarrow{\sim} \cQ(\cG) \otimes_{\Lambda(\cG)} E_{S, T}.
        \]
        Note here that, even though $H^\even(\cC_{S, T}\q)$ is not technically $E_{S, T}$ but rather isomorphic to it, we can indeed identify them  by \eqref{eq:rec_independent_derived_iso}. Our chosen trivialisation is given by extension of scalars ${t^\alpha = \cQ(\cG) \otimes_{\Lambda(\cG)} t_\iota^\alpha}$ of a map $t_\iota^\alpha \colon H^1(\cC_{S, T}\q) \to E_{S, T}$ on integral level. This $\Lambda(\cG)$-homomorphism is defined as the composition
        \begin{equation}
        \label{eq:integral_trivialisation}
            t_\iota^\alpha = \alpha\varphi\pi \colon H^1(\cC_{S, T}\q) \xrightarrow{\pi} \cX_S \xrightarrow{\varphi} \cY_{S_\infty} \xrightarrow{\alpha} E_{S, T},
        \end{equation}
        which we shall refer to as the \textbf{integral trivialisation}\index{integral trivialisation}\index{trivialisation!integral}. The notation $t_\iota^\alpha$ reflects the dependence on the choice of $\alpha$, which will be studied in subsection \ref{subsec:the_choice_of_alpha}. Let us consider each of the arrows separately:
        \begin{itemize}
            \item{
                $\pi$ is surjective with kernel $X^{cs}_{T, S}$, which is torsion: by sequence \eqref{eq:modified_five_term_sequence_alpha}, it is enough to show that $\bigoplus_{v \in T^p} \Ind_{\cG_v}^\cG \ZZ_p(1)$ and $X_S^{cs}$ are. The former has finite $\ZZ_p$-rank, which implies torsionness by the structure theorem \ref{thm:structure_theorem_iwasawa}; and the latter admits a surjection from the unramified Iwasawa module $X_{nr}$ (the limit of the $p$-Hilbert class groups), which is torsion by \cite{nsw} proposition 11.1.4.
            }
            \item{
                $\varphi$ is the inverse limit of the canonical projections $\cX_{L_n, S} \sa \cY_{L_n, S_\infty}$ - which simply discard the components at finite places - along the cyclotomic tower (cf. \eqref{eq:y_finite_level}). By the right-exactness of $\varprojlim_n$, it is surjective. Its kernel is $\cX_{S_f}$, which has finite $\ZZ_p$-rank (because so does $\cY_{S_f}$) and is therefore torsion.
            }
            \item{
                $\alpha$ as in setting \ref{sett:formulation} is injective and has torsion cokernel by definition.
            }
        \end{itemize}
        Since all maps have torsion kernels and cokernels, $t_\iota^\alpha$ becomes indeed an isomorphism after tensoring with $\cQ(\cG)$: our desired trivialisation\index{trivialisation} $t^\alpha = \cQ(\cG) \otimes_{\Lambda(\cG)} t_\iota^\alpha$. This induces a refined Euler characteristic\index{refined Euler characteristic}
        \begin{equation}
        \label{eq:definition_rec}
            \chi_{\Lambda(\cG), \cQ(\cG)}(\cC_{S, T}\q, t^\alpha) \in K_0(\Lambda(\cG), \cQ(\cG))
        \end{equation}
        as explained above.

        We conclude on the brief mention that $\cY_{S_\infty}$ and $E_{S, T}$ are known to be isomorphic in some cases. The classical example comes from Jannsen's work \cite{jann}, although the exact formulation that suits our setting appears in \cite{nsw} theorem 11.3.11 ii): if $p \nmid [L : K]$ and $L$ contains no primitive $p$-th root of unity, then $\cY_{S_\infty} \iso E_{S, \varnothing}$ as $\Lambda(\cG)$-modules. In particular, the $\alpha$ in setting \ref{sett:formulation} can be chosen to be this isomorphism. Note that $t_\iota^\alpha$ will rarely be an isomorphism itself, since $\pi$ is always surjective and $\varphi$ is never injective unless $S = S_p \cup S_\infty$ and $\abs{S_p} = 1$.

        The same result provides an example where $\cY_{S_\infty}$ and $E_{S, T}$ are certainly not isomorphic: if we instead require $L$ to contain a primitive $p$-th root of unity, then $E_{S, \varnothing}$ has $\ZZ_p(1)$ as a direct summand (the inverse limit of all $p$-power roots of unity), which could never occur in the torsion-free module $\cY_{S_\infty}$. In practice, one often imposes light conditions on the set $T$ precisely to do away with this torsion part.

    \section{Morphisms on finite level}
    \label{sec:morphisms_on_finite_level}

        As mentioned in the introduction to the present chapter, special $L$-values need to be regulated before they can stand a chance of being interpolated by quotients of $p$-adic power series. This regulation comes in the form of division by certain \textit{regulators}, ($p$-adic) complex numbers constructed as determinants of morphisms $\varphi_n^\alpha$ on finite level. While the definition of regulators is addressed in the next section, this one paves the way by defining the finite-level maps and proving the essential fact that they are isomorphisms on $\chi$-parts for almost all Artin characters $\chi$ - which will guarantee the resulting regulators to be non-zero. These maps should contain information about the homomorphism $\alpha$ from setting \ref{sett:formulation}, and choosing them in a suitable way will in fact ensure the independence of the Main Conjecture from the choice of $\alpha$ (subsection \ref{subsec:the_choice_of_alpha}). The passage to finite level is achieved by taking coinvariants of some of the $\Lambda(\cG)$-modules we have already encountered. However, some care has to be put into accounting for the difference between those coinvariants and their natural finite-level counterparts.

        We start by introducing Artin characters and proving some technical properties of modules of coinvariants. This is followed by a careful study of the kernel and cokernel of $\varphi_n^\alpha$, leading to the proof that they have the desired behaviour on $\chi$-parts (proposition \ref{prop:isomorphism_on_chi_parts}).

        By a ($p$-adic) \textbf{Artin character}\index{Artin character}\index{character!Artin} $\chi$ of a profinite group $G$ we always mean the trace of a representation $\rho_\chi \colon G \to \GLg_n(\QQ_p^c)$ with open kernel, where $n \geq 1$. Such a $\rho_\chi$ factors through the finite quotient $G/\ker(\rho_\chi)$ and takes values in $\GLg_n(E)$ for some $p$-adic field $E$, over which we say it \textit{realises}. We denote the set of irreducible Artin characters of $G$ by $\Irr_p(G)$. The usual operations on characters of finite groups (inflation, induction, restriction, sum, product, dual, etc.) can be performed on Artin characters as well in a natural way (conceptually, one can factor the characters through a common finite group first, then inflate the result of the operation back to $G$).

        Our main interest lies in the Artin characters of $\cG$. In this case, since any open subgroup of $\cG$ contains $\Gp{n}$ for some $n$ (because ${L_\infty = \bigcup_n L_n}$), $\rho_\chi$ and $\chi$ are, respectively, a representation and a character of ${\cG_n = \cG/\Gp{n} = \Gal(L_n/K)}$ in the sense of section \ref{sec:representations_of_finite_groups}. We will in general not distinguish $\chi$ from its projection to $\cG_n$ in the notation, but it will always be clear which layer $L_n$ we are projecting onto. The notation $\ker(\chi) = \ker(\rho_\chi)$ and terminology such as \textit{irreducible} and \textit{linear} carry over from the finite-group case.

        An Artin character of $\cG$ is said to be of \textbf{type $W$}\index{type-W character@type-$W$ character}\index{character!of type $W$} if it is linear and its kernel contains ${H = \Gal(L_\infty/K_\infty)}$. These characters identify bijectively with the group homomorphisms
        \[
            \rho \colon \Gamma_K \to \units{(\QQ_p^c)}
        \]
        with open kernel. Any character\footnote{We reserve the notation $\rho$ for type-$W$ characters, which will in practice not cause any confusion with the representation $\rho_\chi \colon \cG \to \GLg_n(\QQ_p^c)$ associated to an arbitrary Artin character $\chi$ of $\cG$.} $\rho$ of type $W$ factors through a finite layer $K_n/K$ and its image is the group of $p^m$-th roots of unity in $\QQ_p^c$ for some $m$. Indeed, there is a group isomorphism
        \begin{align}
        \label{eq:bijection_type_W_rou}
            \set{\text{Characters of} \ \cG \et{of type} W} & \to \mu_{p^\infty} \\
            \rho & \mapsto \rho(\gamma_K) \nonumber
        \end{align}
        where $\mu_{p^\infty}$ denotes the group of $p$-power roots of unity of $\QQ_p^c$ and the operation on the domain is character product. We point out that the multiplicative inverse $\rho^{-1}$ of any linear character $\rho$ is its dual $\check{\rho}$.

        Given two Artin characters $\chi$ and $\chi'$ of $\cG$, we write $\chi \sim_W \chi'$ and say they are \textbf{$W$-equivalent}\index{character!W-equivalent@$W$-equivalent} if $\chi' = \chi \otimes \rho$ for some character $\rho$ of type $W$. This can easily be verified to be an equivalence relation. Note that $\chi$ is irreducible if and only $\chi \otimes \rho$ is, as $\oplus$ distributes over $\otimes$ and linear characters are invertible.

        The first necessary result concerning Artin characters is the following lemma. Here and in the sequel, expressions of the form $e(\chi)S \otimes_R -$ denote $e(\chi)(S \otimes_R -)$ rather than $(e(\chi)S) \otimes_R -$.

        \begin{lem}
        \label{lem:torsion_modules_bounded_coinvariants}
            Let $\Lambda(\Gamma)$ be the Iwasawa algebra of a profinite group $\Gamma \iso \ZZ_p$.
            \begin{enumerate}[i)]
                \item{
                    Choose a topological generator $\gamma$ of $\Gamma$, which in particular induces an isomorphism ${\Lambda(\Gamma) \iso \ZZ_p[[T]]}$ as in section \ref{sec:iwasawa_algebras_and_modules}. Given a finitely generated torsion $\Lambda(\Gamma)$-module $M$, set
                    \[
                        n_0 = \max\set{n \in \NN \et{such that} \ideal{\xi_n} \in \supp(M)}
                    \]
                    with $\xi_n$ as defined in the same section, or $n_0 = 0$ if that set is empty. Then $n_0 \in \NN$ and, for all $n \geq n_0$, one has
                    \begin{equation}
                    \label{eq:rank_stabilisation}
                        \rank_{\ZZ_p} \ginv{M}{n} = \rank_{\ZZ_p} \gcoinv{M}{n} = \rank_{\ZZ_p} \gcoinv{M}{n_0}.
                    \end{equation}
                }
                \item{
                    Let $G$ be a profinite group which has $\Gamma$ as an open central subgroup, and $m < n$ two natural numbers. Furthermore, let $\chi$ be an irreducible Artin character of $G$ such that $\ker(\chi)$ contains $\Gp{n}$ but not $\Gp{m}$. If $\rank_{\ZZ_p} \gcoinv{M}{n} = \rank_{\ZZ_p} \gcoinv{M}{m}$ and $E$ is any extension of $\QQ_p$ over which $\chi$ has a realisation, then
                    \[
                        e(\chi) E \otimes_{\ZZ_p} \gcoinv{M}{n} = e(\chi) E \otimes_{\ZZ_p} \ginv{M}{n} = 0,
                    \]
                    where $e(\chi) \in E[G/\Gp{n}]$ is the primitive central idempotent given by \eqref{eq:definition_pci}.
                }
            \end{enumerate}
        \end{lem}

        \begin{proof}
            We start with part i). Finitely generated torsion $\Lambda(\Gamma)$-modules have finite support (which shows $n_0 \in \NN)$ and finite $\ZZ_p$-rank by the structure theorem \ref{thm:structure_theorem_iwasawa}, and hence the exact sequence
            \[
                0 \to \ginv{M}{n} \to M \xrightarrow{\gamma^{p^n} - 1} M \to \gcoinv{M}{n} \to 0
            \]
            shows that $\rank_{\ZZ_p} \ginv{M}{n} = \rank_{\ZZ_p} \gcoinv{M}{n}$ for all $n \in \NN$. We therefore only need to show \eqref{eq:rank_stabilisation} for one these, and we do so for coinvariants. Let
            \[
                f \colon M \xrightarrow{\approx} E_M = \bigoplus_{i = 1}^s \faktor{\Lambda(\Gamma)}{\ideal{p^{m_i}}} \oplus \bigoplus_{j = 1}^t \faktor{\Lambda(\Gamma)}{\ideal{F_j^{l_j}}}
            \]
            be a pseudo-isomorphism as in the structure theorem for some $m_i, l_j > 0$ and irreducible Weierstrass polynomials $F_j$ - all of which are unique up to order. Note that the height-one prime ideals in $\supp(M)$ are precisely the $\ideal{F_j}$ and, if $s > 0$, $\ideal{p}$ as well.

            We now show that it suffices to prove \eqref{eq:rank_stabilisation} for $E_M$ rather than $M$. Indeed, decomposing the exact sequence
            \begin{equation}
            \label{eq:reduction_to_elementary}
                \begin{tikzcd}[column sep=tiny]
                0 \arrow[rr] &  & \ker(f) \arrow[rr] &  & M \arrow[rr, "f"] \arrow[rd, two heads] &                          & E_M \arrow[rr] &  & \coker(f) \arrow[rr] &  & 0 \\
                             &  &                    &  &                                         & \img(f) \arrow[ru, hook] &                &  &                      &  &
                \end{tikzcd}
            \end{equation}
            into two short ones and taking invariants-coinvariants (cf. \eqref{eq:invariants_coinvariants}) on each of them yields new exact sequences
            \[
                \cdots \to \gcoinv{\ker(f)}{n} \to \gcoinv{M}{n} \to \gcoinv{\img(f)}{n} \to 0
            \]
            and
            \[
                \cdots \to \ginv{\coker(f)}{n} \to \gcoinv{\img(f)}{n} \to \gcoinv{(E_M)}{n} \to \gcoinv{\coker(f)}{n} \to 0
            \]
            for all $n \in \NN$. Since $\ker(f)$ and $\coker(f)$ are finite, so are their invariants and coinvariants, which in particular have $\ZZ_p$-rank equal to 0. It therefore follows from the last two sequences that $\rank_{\ZZ_p} \gcoinv{M}{n} = \rank_{\ZZ_p} \gcoinv{(E_M)}{n}$. The right-hand side is easily determined, since taking $\Gamma^{p^n}$-coinvariants amounts to forming the quotient by the submodule generated by $w_n = (T + 1)^{p^n} - 1$, and hence
            \[
                (E_M)_{\Gamma^{p^n}} = \bigoplus_{i = 1}^s \fzmod{p^{m_i}}[1, T, \ldots, T^{p^n - 1}] \oplus \bigoplus_{j = 1}^t \faktor{\Lambda(\Gamma)}{\ideal{F_j^{l_j}, w_n}}.
            \]
            In writing the first sum (the $\mu$-part), we are tacitly using \cite{nsw} corollary 5.3.3. All terms of that sum are finite and therefore do not contribute to the $\ZZ_p$-rank. For those in the second sum (the $\lambda$-part), we resort to the fact that, given any two non-zero elements $x, y \in \Lambda(\Gamma)$, the quotients $\Lambda(\Gamma)/\ideal{x, y}$ and $\Lambda(\Gamma)/\ideal{\gcd(x, y)}$ are pseudo-isomorphic (recall here that $\Lambda(\Gamma)$ is a unique factorisation domain) and in particular have the same $\ZZ_p$-rank.

            The polynomial $w_n$ decomposes into irreducibles as $w_n = \xi_n \cdot \ldots \cdot \xi_0$, with $\xi_i \neq \xi_{i'}$ whenever $i \neq i'$. The $F_j$ are irreducible by definition.  Furthermore, by the Weierstrass preparation theorem \cite{nsw} theorem 5.3.4, two Weierstrass polynomials which differ by a unit in $\Lambda(\Gamma)$ must in fact coincide. It follows that $\gcd(F_j^{l_j}, w_n)$ is $\xi_i$ if $F_j = \xi_i$ for some $0 \leq i \leq n$, and 1 otherwise. Therefore,
            \[
                \rank_{\ZZ_p} \faktor{\Lambda(\Gamma)}{\ideal{F_j^{l_j}, w_n}} =
                \begin{cases}
                    0, & F_j \neq \xi_i \ \et{for all} i \\
                    0, & F_j = \xi_i \et{and} n < i \\
                    \deg(F_j), & F_j = \xi_i \et{and} n \geq i. \\
                \end{cases}
            \]
            After fixing $F_j$, this rank is either identically 0 (first case) or independent of $n$ for $n \geq i$ (second and third cases). This shows i).

            Let us now prove part ii). The canonical surjection $\gcoinv{M}{n} \sa \gcoinv{M}{m}$ ($m < n$) induces an epimorphism
            \[
               q \colon E \otimes_{\ZZ_p} \gcoinv{M}{n} \sa E \otimes_{\ZZ_p} \gcoinv{M}{m}
            \]
            of $E[G/\Gp{n}]$-modules, where the codomain is regarded as such via the projection
            \[
               \varepsilon \colon E[G/\Gp{n}] \sa E[G/\Gp{m}].
            \]
            But $q$ is in fact an isomorphism (even after tensoring with $\QQ_p$ only) as both modules of coinvariants have the same $\ZZ_p$-rank. Since $\chi$ does not factor through $G/\Gp{m}$ by hypothesis, $\varepsilon(e(\chi)) = 0$ by \eqref{eq:pci_projection} (together with the argument in the last paragraph of that section) and hence
            \[
               e(\chi) E \otimes_{\ZZ_p} \gcoinv{M}{n} \iisoo \varepsilon(e(\chi)) E \otimes_{\ZZ_p} \gcoinv{M}{m} = 0.
            \]

            As for the invariants, we have
            \[
                \rank_{\ZZ_p} \ginv{M}{n} = \rank_{\ZZ_p} \gcoinv{M}{n} = \rank_{\ZZ_p} \gcoinv{M}{m} = \rank_{\ZZ_p} \ginv{M}{m}
            \]
            by the assumptions in ii) and the same argument as at the beginning of part i). Therefore, the canonical injection $\ginv{M}{m} \ia \ginv{M}{n}$ induces an embedding
            \[
               q \colon E \otimes_{\ZZ_p} \ginv{M}{m} \ia E \otimes_{\ZZ_p} \ginv{M}{n}
            \]
            of $E[G/\Gp{n}]$-modules, which turns out to be an isomorphism because of the above equality of $\ZZ_p$-ranks. This again implies
            \[
                0 = \varepsilon(e(\chi)) E \otimes_{\ZZ_p} \ginv{M}{m} \iso e(\chi) E \otimes_{\ZZ_p} \ginv{M}{n}.
            \]
        \end{proof}

        \begin{rem}
        \phantomsection
            \begin{enumerate}[i)]
                \item{
                    For a finitely generated $\Lambda(\Gamma)$-module, not only does the condition that it is torsion imply that its modules of $\Gp{n}$-coinvariants and $\Gp{n}$-invariants have bounded $\ZZ_p$-rank as $n \to \infty$ (as stated in part i) of the lemma), but the two are in fact equivalent. The converse can easily be proved by contradiction using the structure theorem for Iwasawa modules.
                }
                \item{
                    The proof shows that the $n_0$ in part i) is in fact minimal with property \eqref{eq:rank_stabilisation}. This invariant is the $d(M)$ in \cite{nsw} definition 5.3.12 in disguise, except for our convention that $n_0 = 0$ (as opposed to -1) when the support of $M$ contains no $\ideal{\xi_i}$.
                }
            \end{enumerate}
        \end{rem}

        The previous lemma will be useful when studying the vanishing of the kernel and cokernel of the finite-level maps $\varphi_n^\alpha$, which we shall construct by essentially taking coinvariants of the integral trivialisation ${t_\iota^\alpha = \alpha\varphi\pi}$ defined in \eqref{eq:integral_trivialisation}. Our aim is to compare these morphisms to the Dirichlet regulator map $\RR \otimes \units{\cO_{L_n, S}} \to \RR \otimes \cX_{L_n, S}^\ZZ$ \eqref{eq:dirichlet_regulator_map_real}, so only the last two arrows $\alpha\varphi \colon \cX_{S} \to E_{S, T}$ play a role. The difference between the coinvariants of $\cX_S$ and its finite-level counterpart $\cX_{L_n, S}$ (defined after \eqref{eq:y_finite_level}) is given by the following result:

        \begin{prop}
        \label{prop:isomoprhisms_y_x_modules}
            Setting \ref{sett:construction}. Let $S' \supseteq S_\infty$ be a finite set of places of $K$ containing all archimedean places and at least one non-archimedean place, and denote by $S_f' \neq \varnothing$ the set of non-archimedean places in $S'$. Then, for all $n \in \NN$, there exists a canonical isomorphism of $\Lambda(\cG_n)$-modules
            \[
                \gcoinv{(\cY_{S'})}{n} \iisoo \cY_{L_n, S'}
            \]
            and a canonical epimorphism of $\Lambda(\cG_n)$-modules
            \[
                \gcoinv{(\cX_{S'})}{n} \sa \cX_{L_n, S'}
            \]
            with finite kernel of order
            \[
                \min \set{[\Gp{n} : \Gp{n} \cap \cG_v] : v \in S_f'}.
            \]

            In particular, the above epimorphism is an isomorphism if and only if $\Gp{n} \subseteq \cG_v$ for some $v \in S_f'$, or in other words, if and only if some $w_n \in S_f'(L_n)$ is non-split in $L_\infty/L_n$.
        \end{prop}

        \begin{proof}
            As usual, we set $\cG_{n, v} = (\cG_n)_{v(L_n)}$. The isomorphism between the $\cY$-modules is
            \begin{align*}
                \gcoinv{(\cY_{S'})}{n} & \iso \bigoplus_{v \in S'} \gcoinv{(\Ind_{\cG_v}^\cG \ZZ_p)}{n} \\
                                      & \iso \bigoplus_{v \in S'} \Lambda(\cG_n) \otimes_{\Lambda(\cG)} \big(\Lambda(\cG) \otimes_{\Lambda(\cG_v)} \ZZ_p\big) \\
                                      & \iso \bigoplus_{v \in S'} \Lambda(\cG_n) \otimes_{\Lambda(\cG_v)} \ZZ_p \\
                                      & = \bigoplus_{v \in S'} \Lambda(\cG_n) \otimes_{\Lambda(\cG_{n, v})} \ZZ_p \\
                                      & = \bigoplus_{v \in S'} \Ind_{\cG_{n, v}}^\cG \ZZ_p \\
                                      & = \cY_{L_n, S'}.
            \end{align*}

            In the second isomorphism, we have used \eqref{eq:completed_tensor_product_equal} - the conditions to apply it are met as explained several times in the previous chapter (cf. for instance section \ref{sec:local-to-global_maps}).

            Consider now the $\Gp{n}$-invariants-coinvariants exact sequence \eqref{eq:invariants_coinvariants} induced by $\cX_{S'} \ia \cY_{S'} \sa \ZZ_p$:
            \[
                0 \to \ginv{\cX_{S'}}{n} \to \ginv{\cY_{S'}}{n} \xrightarrow{\varepsilon} \ZZ_p \to \gcoinv{(\cX_{S'})}{n} \to \gcoinv{(\cY_{S'})}{n} \to \ZZ_p \to 0.
            \]
            Since $\Gamma$ is central in $\cG$, all arrows are $\Lambda(\cG_n)$-homomorphisms. By exactness, there exists a canonical surjection
            \[
                \gcoinv{(\cX_{S'})}{n} \sa \ker(\gcoinv{(\cY_{S'})}{n} \sa \ZZ_p) \iso \cX_{L_n, S'},
            \]
            where the isomorphism follows from $\gcoinv{(\cY_{S'})}{n} \iso \cY_{L_n, S'}$. In order to determine its kernel, which is isomorphic to $\coker(\ginv{\cY_{S'}}{n} \xrightarrow{\varepsilon} \ZZ_p)$, we study the action of $\Gp{n}$ on $\cY_{S'}$ more closely. This module decomposes as $\cY_{S'} = \cY_{S_\infty'} \oplus \cY_{S_f'}$. We have already seen that $\cY_{S_\infty'}$ is free as a $\Lambda(\Gamma)$-module (cf. remark \ref{rem:existence_alpha} i)), so $\ginv{\cY_{S'}}{n}$ is trivial. For each non-archimedean $v \in S_f'$, we treat the direct summand $\Ind_{\cG_v}^\cG \ZZ_p$ of $\cY_{S_f'}$ separately.

            As explained in \eqref{eq:y_finite_level}, there is an isomorphism
            \[
                \Ind_{\cG_v}^\cG \ZZ_p \iso \varprojlim_m \Lambda(\cG_m) \otimes_{\Lambda(\cG_{m, v})} \ZZ_p \iso \varprojlim_m \bigoplus_{w_m \in {v}(L_m)} \ZZ_p \cdot w_m.
            \]
            This map is $\cG$-equivariant, with $\cG$ acting on the right-hand side by permuting the places $w_m$. Consider an arbitrary element $x = \varprojlim_m x_m \in \Ind_{\cG_v}^\cG \ZZ_p$, where
            \[
                x_m = \sum_{w_m} z_{w_m} w_m \in \bigoplus_{w_m} \ZZ_p \cdot w_m
            \]
            for all $m \in \NN$. Suppose $x$ is invariant under $\Gp{n}$, and therefore so is $x_m$ for all $m$. Let $v_n$ be any prolongation of $v$ to $L_n$. The group $\Gp{n}$ acts transitively on the set of places of $L_\infty$ above $v_n$ and thus, for any fixed $m \geq n$, the  coefficient $z_{w_m}$ must coincide for all $w_m \mid v_n$. Since $v_n$ splits into $p^{m - n}$ places in $L_m$ (because it is archimedean), the $p$-adic integer $z_{v_n}$ is infinitely divisible by $p$ and hence $0$. This holds for any $v_n \mid v$, so $x$ must in fact be trivial itself. This shows $\ginv{(\Ind_{\cG_v}^\cG \ZZ_p)}{n} = 0$.

            Let now $v \in S_f'$. Then $\cG_v$ is open in $\cG$ and
            \[
                \Ind_{\cG_v}^\cG \ZZ_p  \iisoo \bigoplus_{w_\infty \in v(L_\infty)} \ZZ_p \cdot w_\infty
            \]
            has finitely many summands. Suppose $x = \sum_{w_\infty} z_{w_\infty} w_\infty \in \Ind_{\cG_v}^\cG \ZZ_p$ is $\Gp{n}$-invariant. By the same transitivity argument as in the archimedean case, $z_{w_\infty}$ coincides with $z_{w_\infty'}$ whenever $w_\infty$ and $w_\infty'$ lie above the same place $v_n$ of $L_n$. That is, $x$ is of the form
            \[
                x = \sum_{v_n \mid v} z_{v_n}  \sum_{w_\infty \mid v_n} w_\infty.
            \]
            Conversely, any element of that form is $\Gp{n}$-invariant. Such an element has augmentation (i.e. sum of coefficients) given by
            \[
                \varepsilon(x) = \sum_{v_n \mid v} z_{v_n} [\Gp{n} : \Gp{n} \cap \cG_v] \in \ZZ_p.
            \]
            Here $[\Gp{n} : \Gp{n} \cap \cG_v]$ is the number of places of $L_\infty$ above $v(L_n)$, and therefore above $v_n$ (because $L_\infty/K$ is Galois). It is necessarily a power of $p$.

            All in all, the image of
            \[
                \ginv{\cY_{S'}}{n} = \ginv{\cY_{S_f'}}{n} \xrightarrow{\varepsilon} \ZZ_p
            \]
            is generated (as an ideal in $\ZZ_p$) by $\set{[\Gp{n} : \Gp{n} \cap \cG_v] : v \in S_f'}$. Since these are all $p$-powers, the same ideal is generated by their minimum, and it has index equal to that minimum. The result follows.
        \end{proof}

        In our case of interest, $S'$ will be the set $S$ from setting \ref{sett:construction}. The proposition shows that a sufficient condition for $\gcoinv{(\cX_{S})}{n} \isoa \cX_{L_n, S}$ is $n \geq n(S)$.

        \begin{ex}
            It can indeed happen that the map $\gcoinv{(\cX_S)}{n} \sa \cX_{L_n, S}$ fails to be injective. Let
            \[
                f_0(x) = x^3 - 30x - 1
            \]
            and
            \[
                f_1(x) = x^{9} - 3 x^{8} - 27 x^{7} + 90 x^{6} + 87 x^{5} - 336 x^{4} - 73 x^{3} + 288 x^{2} - 27 x - 27.
            \]
            For $i \in \set{0, 1}$, let $L_i$ be the number field formed by adjoining a root of $f_i$ to $\QQ$. These are  \cite[\href{https://www.lmfdb.org/NumberField/3.3.107973.1}{Number field 3.3.107973.1}]{lmfdb} and  \cite[\href{https://www.lmfdb.org/NumberField/9.9.1258767452176317.1}{Number field 9.9.1258767452176317.1}]{lmfdb}, respectively. As can be seen in the LMFDB, $L_1$ contains $L_0$ (for a suitable choice of roots) as well as ${\QQ(\rou{9})^+ = \QQ(\rou{9} + \rou{9}^{-1})}$. The field $\QQ(\rou{9})^+$ is the only subfield of $\QQ(\rou{9})$ of relative degree 2, and therefore necessarily the first layer of the cyclotomic $\ZZ_3$-extension of $\QQ$. Since $L_1$ is an extension of $L_0$ of degree 3 containing $\QQ(\rou{9})^+$ (and $L_0 \neq \QQ(\rou{9})^+$), it is precisely the first layer of the cyclotomic $\ZZ_3$-extension of $L_0$.

            Set $L = K = L_0$ and let $L_\infty$ be the cyclotomic $\ZZ_3$-extension of $L$. Let $\cG = \Gamma = \Gal(L_\infty/L) \iso \ZZ_3$. Consider the finite set of places $S = S_\infty \cup S_3$ of $K$ consisting of the archimedean places and the $3$-adic ones, and set $T = \varnothing$. This is an example of setting \ref{sett:construction} for $p = 3$.

            The referenced LMFDB pages show that the prime $3$ is totally ramified in $L$ but splits into three primes in $L_1$. Therefore, the only prime $v$ in $S_3(L)$ splits in in $L_1/L$, which implies
            \[
                \Abs{\ker((\cX_{S})_\Gamma \sa \cX_{L, S})} = [\Gamma : \Gamma_{w_\infty}] \geq 3
            \]
            by the previous proposition (with $w_\infty$ denoting any prolongation of $v$ to $L_\infty$). The next layer $L_2$ can be verified to contain exactly three places above 3, which means all $3$-adic places of $L_1$ are non-split in $L_2/L_1$ (and hence in $L_\infty/L_1$ by the structure of $\Gamma$). It follows that $\Gamma_{w_\infty} = \Gamma^3$ and the above inequality is in fact an equality.

            As an informal remark, random sampling using \textit{SageMath} seems to hint at this situation being rather infrequent. Specifically, the prime $p \in \set{3,  5}$ was already non-split in the first layer of the cyclotomic $\ZZ_p$-extension of most tested number fields (all of low degree). This would imply that $\gcoinv{(\cX_{S})}{n} \sa \cX_{L_n, S}$ is \textit{often} an isomorphism for all $n$. However, the testing was not systematic enough to make anything other than a passing comment out of it. \qedef
        \end{ex}

        Now that a relation between $\gcoinv{(\cX_S)}{n}$ and $\cX_{L_n, S}$ has been established, we address the same descent question for $E_{S, T}$. The key concept now is that of universal norms:
        \begin{defn}
            Let $F$ be an algebraic extension of a number field $E$. Consider two disjoint finite sets of places $S' \supseteq S_\infty$ and $T'$ of $E$. The group of $p$-adic \textbf{universal norms}\index{universal norm} of $(S', T')$-units in $E$ is defined as
            \[
                \cE_{E, S', T'}^F = \bigcap_{E'} N_{E'/E}(\ZZ_p \otimes \units{\cO_{E', S', T'}}) \subseteq \ZZ_p \otimes \units{\cO_{E, S', T'}},
            \]
            where $E'$ runs over the finite extensions of $E$ contained in $F$. Given a finite extension $E'$ of $E$ contained in $F$, $\cE_{E', S', T'}^F$ denotes $\cE_{E', S'(E'), T'(E')}^F$. \qedef
        \end{defn}

        If $E \subseteq F_0 \subseteq F_1 \subseteq \cdots$ is a family of finite extensions of $E$ such that $F = \bigcup_n F_n$, then one has
        \[
            \cE_{E, S', T'}^F = \bigcap_{n \in \NN} N_{F_n/E}(\ZZ_p \otimes \units{\cO_{F_n, S', T'}})
        \]
        by the transitivity of the norm maps. More generally, this works for any cofinal family in the set of all finite extensions of $E$ inside $F$. Also relevant to our purposes is the fact that $\cE_{E, S', T'}^F$ coincides with the image of the canonical projection
        \[
            \big(\varprojlim_{E'} \ZZ_p \otimes \units{\cO_{E', S', T'}}\big) \to \ZZ_p \otimes \units{\cO_{E, S', T'}}
        \]
        where $E'$ is as before and the inverse limit is taken with respect to the norm maps. This is not difficult to show using the classical result that the inverse limit of non-empty compact Hausdorff spaces is non-empty.

        Let us now specialise to our Iwasawa-theoretic setting \ref{sett:construction}. We omit the superscript $F$ from the notation of the universal norms whenever it is $L_\infty$. Since taking $\Gp{n}$-coinvariants yields the maximal quotient with trivial $\Gp{n}$-action, there is a canonical epimorphism of $\Lambda(\cG_n)$-modules
        \[
            \gcoinv{(E_{S, T})}{n} \sa \cE_{L_n, S, T}.
        \]
        by the previous displayed equation. This map can be shown to be injective, and therefore an isomorphism, using work of Fukaya and Kato. Namely, the short exact sequence ${\Lambda(\cG) \xhookrightarrow{1 - \gp{n}} \Lambda(\cG) \sa \Lambda(\cG_n)}$ induces an exact triangle
        \[
            \cC_{S, T}\q \to \cC_{S, T}\q \to \Lambda(\cG_n) \otimes_{\Lambda(\cG)}^\LL \cC_{S, T}\q \to
        \]
        in the derived category $\cD(\Lambda(\cG))$, and therefore a long exact sequence
        \[
            \cdots \to H^{-1}(\Lambda(\cG_n) \otimes_{\Lambda(\cG)}^\LL \cC_{S, T}\q) \to H^0(\cC_{S, T}\q) \xrightarrow{1 - \gp{n}} H^0(\cC_{S, T}\q) \to H^0(\Lambda(\cG_n) \otimes_{\Lambda(\cG)}^\LL \cC_{S, T}\q) \to \cdots.
        \]
        In order to determine the last term, we resort to a result whose proof we defer to chapter \ref{chap:properties_of_the_main_conjecture} for the sake of exposition: by lemma \ref{lem:complexes_coinvariants_finite_level}, there exist $\Lambda(\cG_n)$-isomorphisms
        \[
            H^i(\Lambda(\cG_n) \otimes_{\Lambda(\cG)}^\LL \cC_{S, T}\q) \iisoo H^i(\cB_{L_n, S, T}\q)
        \]
        for all $i$, where  $\cB_{L_n, S, T}\q$ is the complex of Burns, Kurihara and Sano on finite level (see definition \ref{defn:complexes_bks}). Its cohomology is trivial outside degrees 0 and 1, and $\ZZ_p \otimes \units{\cO_{L_n, S, T}}$ in degree 0 by \cite{bks} p. 1535. The exactness of the above sequence in cohomology yields an embedding ${\gcoinv{(E_{S, T})}{n} \iso \coker(H^0(\cC_{S, T}\q) \xrightarrow{1 - \gp{n}} H^0(\cC_{S, T}\q)) \ia \ZZ_p  \otimes \units{\cO_{L_n, S, T}}}$, which proves our injectivity claim. We point out that essentially the same question is treated using more direct methods in existing literature (see for instance theorem 7.3 from \cite{kuzmin} or section 1 of \cite{nguyen}), from which an alternative proof is likely to follow.

        We denote by $\iota_n$ the composition
        \[
            \iota_n \colon \gcoinv{(E_{S, T})}{n} \xrightarrow{\sim} \cE_{L_n, S, T} \ia \ZZ_p \otimes \units{\cO_{L_n, S, T}}.
        \]
        This concludes the preparations for the following definition:

        \begin{defn}
        \label{defn:finite_level_map}
            Setting \ref{sett:construction}, $\alpha$ as in setting \ref{sett:formulation}. For $n \geq n(S)$, the \textbf{finite-level map}\index{finite-level map} $\varphi_n^\alpha$ is the $\Lambda(\cG_n)$-homomorphism given by the composition
            \begin{center}
                \begin{tikzcd}
                    {\gcoinv{(\cX_S)}{n}} \arrow[r, "\gcoinv{(\alpha \varphi)}{n}"] & {\gcoinv{(E_{S, T})}{n}} \arrow[d, "\iota_n", hook] \\
                    {\cX_{L_n, S}} \arrow[u, "\rsim"] \arrow[r, "\varphi_n^\alpha", dashed]                   & {\ZZ_p \otimes \units{\cO_{L_n, S, T}}}
                \end{tikzcd}
            \end{center}
            where the left vertical arrow is the inverse of the canonical isomorphism from proposition \ref{prop:isomoprhisms_y_x_modules}, $\gcoinv{(\alpha \varphi)}{n}$ is the map induced by $\alpha \varphi \colon \cX_S  \to E_{S, T}$ (cf. \eqref{eq:integral_trivialisation}) on coinvariants and $\iota_n$ is as above. \qedef
        \end{defn}

        Note that $\varphi_n^\alpha$ can be defined for lower $n$ than $n(S)$: it is enough that \textit{one} (as opposed to \textit{each}) place above $S_f$ is non-split in $L_\infty/L_n$ by proposition \ref{prop:isomoprhisms_y_x_modules}. However, restricting to $n \geq n(S)$ will be necessary soon (cf. \ref{manualcond:kc} and proposition \ref{prop:isomorphism_on_chi_parts}) and causes no harm here. A feature of this critical layer $n(S)$ is that, for all $n \geq n(S)$, one has a natural bijection between $S_f(L_n)$ and $S_f(L_{n(S)})$, and hence canonical $\Lambda(\cG_n)$-(and even $\Lambda(\cG_{n(S)})$-)isomorphisms
        \begin{equation}
        \label{eq:xy_modules_coincide_above_ns}
            \cY_{L_n, S_f} \iso \cY_{L_{n(S)}, S_f} \quad \text{and} \quad \cX_{L_n, S_f} \iso \cX_{L_{n(S)}, S_f},
        \end{equation}
        which regard as identifications.

        The property of the finite-level maps which will ensure the regulators defined in the next section are non-zero is that they are isomorphisms on $\chi$-parts for almost all Artin characters $\chi$. The proof hinges on the following lemma, which allows us to control their kernels and cokernels:
        \begin{lem}
        \label{lem:zp_rank_bounds_finite level}
            Setting \ref{sett:construction}, $\alpha$ as in setting \ref{sett:formulation}. For $n \geq n(S)$, the kernel and cokernel of $\varphi_n^\alpha$ fit into the short exact sequences of $\Lambda(\cG_n)$-modules
            \begin{equation}
            \label{eq:ses_kernel_varphinalpha}
                0 \to \cX_{L_{n(S)}, S_f} \to \ker(\varphi_n^\alpha) \to \ginv{\coker(\alpha)}{n} \to 0
            \end{equation}
            and
            \begin{equation}
            \label{eq:ses_cokernel_varphinalpha}
                0 \to \gcoinv{\coker(\alpha)}{n} \to \coker(\varphi_n^\alpha) \to \coker(\iota_n) \to 0
            \end{equation}
            respectively. In particular,
            \[
                \rank_{\ZZ_p} \ker(\varphi_n^\alpha) = \rank_{\ZZ_p} \coker(\varphi_n^\alpha) = \rank_{\ZZ_p} \ginv{\coker(\alpha)}{n} + \abs{S_f(L_{n(S)})} - 1.
            \]
        \end{lem}

        \begin{proof}
            The homomorphism $\varphi_n^\alpha$ is given by a composition of three maps, the first of which is an isomorphism and therefore plays no role: $\ker(\varphi_n^\alpha) \iso \ker(\iota_n \gcoinv{(\alpha \varphi)}{n})$ and ${\coker(\varphi_n^\alpha) = \coker(\iota_n \gcoinv{(\alpha \varphi)}{n})}$. Using the injectivity of $\iota_n$, an easy application of the snake lemma yields a six-term exact sequence
            \begin{equation}
            \label{eq:alphaphi_kernel_cokernel}
                \ker(\gcoinv{(\alpha \varphi)}{n}) \ia  \ker(\iota_n \gcoinv{(\alpha \varphi)}{n}) \to 0 \to \coker(\gcoinv{(\alpha \varphi)}{n}) \to \coker(\iota_n \gcoinv{(\alpha \varphi)}{n}) \sa \coker(\iota_n).
            \end{equation}

            Let us consider $\gcoinv{(\alpha \varphi)}{n}$ first. Since $\alpha\varphi$ is neither injective nor surjective in general, it is convenient define the two short exact sequences
            \begin{center}
                \begin{tikzcd}[column sep=tiny]
                \ker(\alpha\varphi) \arrow[rr, hook] &  & \cX_S \arrow[rd, two heads] \arrow[rr, "\alpha\varphi"] &                                      & {E_{S, T}} \arrow[rr, two heads] &  & \coker(\alpha\varphi) \\
                                                     &  &                                                         & \img(\alpha\varphi) \arrow[ru, hook] &                                  &  &
                \end{tikzcd}
            \end{center}
            The map $\varphi$ is surjective with kernel $\cX_{S_f}$, so in particular $\img(\alpha\varphi) = \img(\alpha)$ and $\coker(\alpha\varphi) = \coker(\alpha)$. Since $\alpha$ is injective, we have $\img(\alpha) \iso \cY_{S_\infty}$ and $\ker(\alpha\varphi) = \ker(\varphi) = \cX_{S_f}$. We now take invariants-coinvariants of both short exact sequences above, the first of which yields
            \begin{equation}
            \label{eq:invariants-coinvariants_alphaphi}
                \ginv{\cX_{S_f}}{n} \ia \ginv{\cX_S}{n} \to \ginv{\img(\alpha)}{n} \to \gcoinv{(\cX_{S_f})}{n} \to \gcoinv{(\cX_S)}{n} \sa \gcoinv{\img(\alpha)}{n}.
            \end{equation}
            The $\Gp{n}$-coinvariants of $\cX_{S_f}$ and $\cX_S$ coincide with $\cX_{L_n, S_f}$and $\cX_{L_n, S}$, respectively, by proposition \ref{prop:isomoprhisms_y_x_modules}. The map between them is the canonical inclusion, which is injective and has cokernel ${\gcoinv{\img(\alpha)}{n} \iso \gcoinv{(\cY_{S_\infty})}{n} \iso \cY_{L_n, S_\infty}}$ by the same proposition.

            The second sequence induces
            \begin{equation}
            \label{eq:invariants-coinvariants_alphaphi_ii}
                \ginv{\img(\alpha)}{n} \ia \ginv{E_{S, T}}{n} \to \ginv{\coker(\alpha)}{n} \to \gcoinv{\img(\alpha)}{n} \to \gcoinv{(E_{S, T})}{n} \sa \gcoinv{\coker(\alpha)}{n}.
            \end{equation}
            The second term admits the description
            \[
                \ginv{E_{S, T}}{n} = \ginv{\big(\varprojlim_m \ZZ_p \otimes \units{\cO_{L_m, S, T}}\big)}{n} = \varprojlim_m \ZZ_p \otimes \ginv{(\units{\cO_{L_m, S, T}})}{n} = \varprojlim_m \ZZ_p \otimes (\units{\cO_{L_m, S, T}} \cap \units{L_n}),
            \]
            where the inverse limit is taken with respect to the norm maps (for the second equality, use the left exactness of $\varprojlim_m$, then apply the exact functor $\ZZ_p \otimes -$ to  $\ginv{\units{\cO_{L_m, S, T}}}{n} \ia \units{\cO_{L_m, S, T}} \xrightarrow{\gamma^{p^n} - 1} \units{\cO_{L_m, S, T}}$). But for $m \geq n$, a unit $x \in \units{\cO_{L_m, S, T}} \cap \units{L_n}$ is sent to $x^{p^{m -n}} \in \units{\cO_{L_n, S, T}}$ via the norm map $N_{L_m/L_n}$, which means that the $n$-th-layer component of any element of $\ginv{E_{S, T}}{n}$ must be infinitely $p$-divisible in $\units{L_n}$, i.e. trivial. The same argument applies to the component at any layer above $L_n$, and hence $\ginv{E_{S, T}}{n}$ is trivial and sequence \eqref{eq:invariants-coinvariants_alphaphi_ii} amounts to a four-term exact sequence.

            In order to recover some information about $\gcoinv{(\alpha \varphi)}{n}$, we note that this map is the composition
            \[
                \gcoinv{(\alpha \varphi)}{n} \colon \gcoinv{(\cX_S)}{n} \sa \gcoinv{(\img\alpha)}{n} \to \gcoinv{(E_{S, T})}{n}
            \]
            and apply the snake lemma incorporating the information from sequences \eqref{eq:invariants-coinvariants_alphaphi} and \eqref{eq:invariants-coinvariants_alphaphi_ii}:
            \begin{center}
                \begin{tikzcd}
                                & {\cX_{L_n, S_f}} \arrow[r, hook] \arrow[d, hook] & \ker(\gcoinv{(\alpha \varphi)}{n}) \arrow[d, hook] \arrow[r]                            & \ginv{\coker(\alpha)}{n} \arrow[d, hook] \arrow[llddd, out=0, in=178, looseness=1, overlay, dashed]                &   \\
                    0 \arrow[r] & {\cX_{L_n, S_f}} \arrow[d] \arrow[r]             & {\cX_{L_n, S}} \arrow[d, "\gcoinv{(\alpha \varphi)}{n}"'] \arrow[r]             & \gcoinv{\img(\alpha)}{n} \arrow[d] \arrow[r]          & 0 \\
                    0 \arrow[r] & 0 \arrow[d, two heads] \arrow[r]                 & {\gcoinv{(E_{S, T})}{n}} \arrow[d, two heads] \arrow[r, equals] & {\gcoinv{(E_{S, T})}{n}} \arrow[d, two heads] \arrow[r] & 0 \\
                                & 0 \arrow[r]                                      & \coker(\gcoinv{(\alpha\varphi)}{n}) \arrow[r, two heads]                     & \gcoinv{\coker(\alpha)}{n}                              &
                \end{tikzcd}
            \end{center}
            Together with \eqref{eq:alphaphi_kernel_cokernel}, this shows both short exact sequences in the proposition (recall also \eqref{eq:xy_modules_coincide_above_ns}). The equality $\rank_{\ZZ_p} \ker(\varphi_n^\alpha) = \rank_{\ZZ_p} \coker(\varphi_n^\alpha)$ follows from the exact sequence
            \[
                0 \to \ker(\varphi_n^\alpha) \to \cX_{L_n, S} \xrightarrow{\varphi_n^\alpha} \ZZ_p \otimes \units{\cO_{L_n, S, T}} \to \coker(\varphi_n^\alpha) \to 0,
            \]
            where the two middle terms have the same $\ZZ_p$-rank $\abs{S(L_n)} - 1$. The last equality in the proposition is then an immediate consequence of \eqref{eq:ses_kernel_varphinalpha}.
        \end{proof}

        Part ii) of lemma \ref{lem:torsion_modules_bounded_coinvariants} hints at the importance of the $\ZZ_p$-rank of the $\Gp{n}$-coinvariants of a torsion Iwasawa module becoming stable as $n$ grows. Let $n(\alpha) \in \NN$ be minimal with the property:
        \begin{equation}
        \label{eq:introduction_alpha_n}
            \rank_{\ZZ_p} \ginv{\coker(\alpha)}{n} = \rank_{\ZZ_p} \ginv{\coker(\alpha)}{n(\alpha)} \quad \text{for all} \ n \geq n(\alpha).
        \end{equation}
        Since $\coker(\alpha)$ is $\Lambda(\Gamma)$-torsion by definition, part i) of the same lemma shows $n(\alpha)$ is well defined. The other stability condition we need is satisfied by all layers starting at $n(S)$, which motivates the definition of
        \[
            n(S, \alpha)  = \max(n(S), n(\alpha)) \in \NN
        \]
        and the following \textbf{kernel condition}\index{kernel condition}:
        \begin{manualcond}{(KC)}
        \label{manualcond:kc}
            Setting \ref{sett:construction}, $\alpha$ as in setting \ref{sett:formulation}. We say an irreducible Artin character $\chi \in \Irr_p(\cG)$ of $\cG$ satisfies condition (KC) if $\Gp{n(S, \alpha)} \nsubseteq \ker(\chi)$. \qedef
        \end{manualcond}

        The fundamental property of the finite-level maps now follows immediately from lemmas \ref{lem:torsion_modules_bounded_coinvariants} and \ref{lem:zp_rank_bounds_finite level}:

        \begin{prop}
        \label{prop:isomorphism_on_chi_parts}
            Setting \ref{sett:construction}, $\alpha$ as in setting \ref{sett:formulation}. Let $\chi \in \Irr_p(\cG)$ satisfy \ref{manualcond:kc} and choose $n \in \NN$ such that $\Gp{n} \subseteq \ker(\chi)$ (in particular, $n > n(S, \alpha)$). Then $\varphi_n^\alpha$ induces an isomorphism
            \[
                e(\chi) \QQ_p^c \otimes_{\ZZ_p} \varphi_n^\alpha \colon e(\chi) \QQ_p^c \otimes_{\ZZ_p} \cX_{L_n, S} \xrightarrow{\ \sim \ } e(\chi) \QQ_p^c \otimes \units{\cO_{L_n, S, T}}
            \]
            of $\QQ_p^c[\cG_n]$-modules on $\chi$-parts.
        \end{prop}

        \begin{proof}
            We start by noting that $\QQ_p^c$ is flat as a (right) $\ZZ_p$-module, so $\QQ_p^c \otimes_{\ZZ_p} -$ takes exact sequences of (left)  $\Lambda(\cG_n) = \ZZ_p[\cG_n]$-modules into ones of $\QQ_p^c[\cG_n]$-modules. Furthermore, since $e(\chi) \in \QQ_p^c[\cG_n]$ is a central idempotent, multiplication by it is an exact endofunctor on the category of $\QQ_p^c[\cG_n]$-modules. Sequence \eqref{eq:ses_kernel_varphinalpha} therefore induces a short exact sequence
            \[
                0 \to e(\chi) \QQ_p^c \otimes_{\ZZ_p} \cX_{L_{n(S)}, S_f} \to e(\chi) \QQ_p^c \otimes_{\ZZ_p} \ker(\varphi_n^\alpha) \to e(\chi) \QQ_p^c \otimes_{\ZZ_p} \ginv{\coker(\alpha)}{n} \to 0.
            \]

            By definition of $n(S, \alpha)$, we have $\rank_{\ZZ_p} \ginv{\coker(\alpha)}{n} = \rank_{\ZZ_p} \ginv{\coker(\alpha)}{n(S, \alpha)}$ (and hence the same equality on ranks of coinvariants) and
            \[
                \rank_{\ZZ_p} \gcoinv{(\cX_{S_f})}{n} = \rank_{\ZZ_p} \gcoinv{(\cX_{S_f})}{n(S, \alpha)} = \abs{S_f(L_{n(S)}) - 1}
            \]
            (recall here proposition \ref{prop:isomoprhisms_y_x_modules}). In conjunction with lemma \ref{lem:torsion_modules_bounded_coinvariants} ii), this shows
            \[
                e(\chi) \QQ_p^c \otimes_{\ZZ_p} \cX_{L_n, S_f} = e(\chi) \QQ_p^c \otimes_{\ZZ_p} \ginv{\coker(\alpha)}{n} = 0,
            \]
            which in turn yields $e(\chi) \QQ_p^c \otimes_{\ZZ_p} \ker(\varphi_n^\alpha) = 0$ by the above exact sequence.

            In order to conclude the triviality of the cokernel, we use the fact that the two middle terms of the exact sequence
            \[
                e(\chi) \QQ_p^c \otimes_{\ZZ_p} \ker(\varphi_n^\alpha) \ia e(\chi) \QQ_p^c \otimes_{\ZZ_p} \cX_{L_n, S} \xrightarrow{e(\chi) \QQ_p^c \otimes_{\ZZ_p} \varphi_n^\alpha} e(\chi) \QQ_p^c \otimes \units{\cO_{L_n, S, T}} \sa e(\chi) \QQ_p^c \otimes_{\ZZ_p} \coker(\varphi_n^\alpha)
            \]
            have the same $\QQ_p^c$-dimension by virtue of the classical Dirichlet regulator map \eqref{eq:dirichlet_regulator_map_real} (extending scalars to $\CC_p$ in both cases).
        \end{proof}

        The non-vanishing of the regulator at a character $\chi$, to be defined in the next section, will only be guaranteed if $\varphi_n^\alpha$ is an isomorphism on $\chi$ parts. The previous proposition therefore highlights the importance of kernel condition \ref{manualcond:kc} for the analytic side of the Main Conjecture. Fortunately, \textit{almost all} $\chi \in \Irr_p(\cG)$ satisfy it: the ones which do not are in bijection with the finitely many irreducible characters of $\cG_{n(S, \alpha)}$. Since $\Irr_p(\cG)$ is infinite, this leaves an ample supply of characters to consider.

    \section{Regulators}
    \label{sec:regulators}

        One of the few missing pieces for the formulation of the Interpolation Conjecture is the \textit{Stark-Tate regulator}. Although definition \ref{defn:regulator} below is tailored to our specific needs, namely to include the map $\alpha$ from setting \ref{sett:formulation}, it is largely influenced by (and indeed a particular case of) Tate's work on Stark's conjectures (cf. \cite{tate}) - which justifies the above terminology. Before delving into the technicalities, we point out several factors which motivate the introduction of these objects. On its analytic side, the Main Conjecture will assert the existence of certain series quotients in $\cQ^c(\Gamma_\chi) \iso \QQ_p^c \otimes_{\QQ_p} \ffrac(\ZZ_p[[T_\chi]])$ which interpolate leading coefficients at 0 of Artin $L$-functions (cf. \eqref{eq:definition_leading_coefficient}). The latter we shall regard as $p$-adic complex via the (inverse of the) isomorphism $\beta \colon \CC_p \xrightarrow{\sim} \CC$ chosen in setting \ref{sett:formulation}. The \textbf{regulated special value}\index{regulated special value} is simply the quotient
        \[
            \frac{\beta^{-1}(L_{K, S, T}^\ast(\beta \check{\chi}, 0))}{R_S^\beta(\alpha, \chi)}
        \]
        by a non-zero $R_S^\beta(\alpha, \chi) \in \units{\CC_p}$: the \textit{regulator}.
        \begin{itemize}
            \item{
                Series quotients in $ \cQ^c(\Gamma_\chi)$, when evaluated at the points we will, converge to values in $\QQ_p^c$. However, the leading coefficients $\beta^{-1}(L_{K, S, T}^\ast(\beta (\check{\chi}), 0))$ are not known to be $p$-adic algebraic in general - only $p$-adic complex.
            }
            \item{
                Leading coefficients are not independent of the choice of $\beta$, which would be a desirable feature for the Main Conjecture to have. Stark's conjecture claims (essentially) that the above quotient is indeed so. This also has implications for the field where the regulated special $L$-value should be expected to live. These topics are the subject of section \ref{sec:starks_conjecture_and_the_choice_of_beta}.
            }
            \item{
                The Main Conjecture postulates a relation between the interpolating series quotients and the refined Euler characteristic $\chi_{\Lambda(\cG), \cQ(\cG)}(\cC_{S, T}\q, t^\alpha)$ of the main complex. However, the latter depends fundamentally on the map $\alpha$, whereas the former are characterised by an interpolation property which is oblivious to it. The introduction of regulators, which contain information about $\alpha$, corrects this disparity. This will become apparent when we prove the Main Conjecture is independent of the choice of $\alpha$ in subsection \ref{subsec:the_choice_of_alpha}.
            }
        \end{itemize}

        As discussed a few lines ago, the kernel condition is essential for the definition of non-zero regulators. To that end, we introduce the following set:
        \begin{defn}
        \label{defn:r_chi}
            Setting \ref{sett:construction}, $\alpha$ as in setting \ref{sett:formulation}. For an irreducible Artin character $\chi$ of $\cG$, we define \[
                \cK_S^\alpha(\chi) = \set{\rho \in \Irr_p(\cG) \et{of type} W \et{such that} \chi \otimes \rho \ \text{satisfies  \ref{manualcond:kc}}}.
            \]
            \qedef
        \end{defn}

        \begin{lem}
        \label{lem:properties_of_r_chi}
            Setting \ref{sett:construction}, $\alpha$ as in setting \ref{sett:formulation}. Let $\chi \in \Irr_p(\cG)$. Then:
            \begin{enumerate}[i)]
                \item{
                    $\cK_S^\alpha(\chi)$ contains almost all characters of $\cG$ of type $W$. In particular, it is infinite.
                }
                \item{
                    For any character $\rho$ of $\cG$ of type $W$ one has $\cK_S^\alpha(\chi \otimes \rho) = \rho^{-1} \otimes \cK_S^\alpha(\chi)$, where the right-hand side denotes $\set{\rho^{-1} \otimes \rho' \colon \rho' \in \cK_S^\alpha(\chi)}$.
                }
                \item{
                    One has $\cK_S^\alpha(\check{\chi}) = \cK_S^\alpha(\chi)^{-1}$, where the right-hand side denotes $\set{\rho^{-1} : \rho \in \cK_S^\alpha(\chi)}$.
                }
            \end{enumerate}
        \end{lem}

        \begin{proof}
            The only subtle point in i) is that multiplication of $\chi$ by type-$W$ characters is not an injective operation. In other words, it may happen that $\chi \otimes \rho = \chi$ with $\rho \neq \bbone$ (cf. remark \ref{rem:properties_of_ev} i)). However, we claim that every fibre of the map
            \begin{align*}
                t_\chi \colon \set{\text{Characters of} \ \cG \et{of type} W} & \to \Irr_p(\cG) \\
                \rho & \mapsto \chi \otimes \rho
            \end{align*}
            is finite. To see this, let $\chi' \in \Irr_p(\cG)$. If $\chi' \not\sim_W \chi$, the fibre of $\chi'$ is empty and we are done. Otherwise, choose $n \in \NN$ such that $\Gp{n} \subseteq \ker(\chi) \cap \ker(\chi')$. Then, for every $\rho$ in the fibre of $\chi'$, one has
            \[
                \chi(1) \rho(\gp{n}) = \chi(\gp{n}) \rho(\gp{n}) = \chi'(\gp{n}) = \chi'(1) = \chi(1).
            \]
            Since $\chi(1) \neq 0$, this implies $\rho$ is trivial on $\Gp{n}$. But only finitely many linear characters of $\cG$ have that property.

            Having shown the claim, part i) is trivial: as mentioned at the end of the preceding section, the set of irreducible Artin characters of $\cG$ which do not satisfy \ref{manualcond:kc} is finite. Therefore, so is its preimage under $t_\chi$, which consists precisely of the type-$W$ characters which do not belong to $\cK_S^\alpha(\chi)$.

            Parts ii) and iii) are immediate by
            \[
                \widetilde{\rho} \in \cK_S^\alpha(\chi \otimes \rho) \iff (\chi \otimes \rho) \otimes \widetilde{\rho} \ \; \text{satisfies \ref{manualcond:kc}} \iff \chi \otimes (\rho \otimes \widetilde{\rho}) \ \; \text{satisfies \ref{manualcond:kc}} \iff \rho \otimes \widetilde{\rho} \in \cK_S^\alpha(\chi)
            \]
            and
            \[
                \rho \in \cK_S^\alpha(\check{\chi}) \iff \check{\chi} \otimes \rho \ \; \text{satisfies \ref{manualcond:kc}} \iff \chi \otimes \rho^{-1}  \ \; \text{satisfies \ref{manualcond:kc}} \iff \rho^{-1} \in \cK_S^\alpha(\chi),
            \]
            where in the next-to-last equality we have used the fact that the kernel of a character coincides with that of its dual.
        \end{proof}

        The aim of the Stark-Tate regulator is to capture the difference between the finite-level maps $\varphi_n^\alpha$ and the Dirichlet regulator map. In order to compare the two, we bring them together to a common scalar field:

        \begin{defn}
        \label{defn:p-adic_dirichlet_regulator_map}
            Setting \ref{sett:construction}, $\beta$ as in setting \ref{sett:formulation}. For $n \in \NN$, we define the \textbf{$p$-adic Dirichlet regulator map}\index{Dirichlet regulator map!$p$-adic} $\lambda_{n, S}^\beta$ via the commutative diagram
            \begin{center}
                \begin{tikzcd}
                    {\CC_p \otimes \units{\cO_{L_n, S}}} \arrow[d, "\rsim"] \arrow[r, dashed, "\lambda_{n, S}^\beta"]       & \CC_p \otimes_{\ZZ_p} \cX_{L_n, S} \\
                    {\CC_p \otimes_{\beta^{-1}} \CC \otimes_{\RR} \RR \otimes \units{\cO_{L_n, S}}} \arrow[r, "\sim"] & \CC_p \otimes_{\beta^{-1}} \CC \otimes_{\RR} \RR \otimes \cX_{L_n, S}^\ZZ \arrow[u, "\rsim"]
                \end{tikzcd}
            \end{center}
            where the bottom arrow is the scalar extension of the classical regulator map \eqref{eq:dirichlet_regulator_map_real}. \qedef
        \end{defn}

        Since the classical regulator map is Galois-equivariant, $\lambda_{n, S}^\beta$ is an isomorphism of $\CC_p[\cG_n]$-modules. Let $\chi \in \Irr_p(\cG)$ be an irreducible Artin character and choose $n \geq n(S)$ such that $\chi$ factors through $\cG_n$. We consider the $\CC_p[\cG_n]$-endomorphism
        \[
            \lambda_{n, S}^\beta \circ (\CC_p \otimes_{\ZZ_p} \varphi_n^\alpha) \colon \CC_p \otimes_{\ZZ_p} \cX_{L_n, S} \to \CC_p \otimes_{\ZZ_p} \cX_{L_n, S}.
        \]
        Note that, although the codomain of $\varphi_n^\alpha$ is $\ZZ_p \otimes \units{\cO_{L_n, S, T}}$, the index $[\units{\cO_{L_n, S}} : \units{\cO_{L_n, S, T}}]$ is finite and thus $\CC_p \otimes \units{\cO_{L_n, S, T}}$ coincides with $\CC_p \otimes \units{\cO_{L_n, S}}$. If $V_\chi$ is the simple $\CC_p[\cG_n]$-module with character $\chi$, the above map induces a $\CC_p$-linear endomorphism $(\lambda_{n, S}^\beta \circ (\CC_p \otimes_{\ZZ_p} \varphi_n^\alpha))_\ast$ of $\Hom_{\CC_p[\cG_n]}(V_\chi, \CC_p \otimes_{\ZZ_p} \cX_{L_n, S})$ by post-composition. By analogy with existing literature, we denote the determinant of this last endomorphism by
        \begin{equation}
        \label{eq:notation_endomorphism_acting_on_hom}
            \det(\lambda_{n, S}^\beta \circ (\CC_p \otimes_{\ZZ_p} \varphi_n^\alpha) \mid \Hom_{\CC_p[\cG_n]}(V_\chi, \CC_p \otimes_{\ZZ_p} \cX_{L_n, S})) \in \CC_p.
        \end{equation}
        Suppose now that $\chi$ satisfies \ref{manualcond:kc}. Then $\lambda_{n, S}^\beta \circ (\CC_p \otimes_{\ZZ_p} \varphi_n^\alpha)$ restricts to an isomorphism
        \[
            e(\chi)\CC_p \otimes_{\ZZ_p} \cX_{L_n, S} \to e(\chi)\CC_p \otimes_{\ZZ_p} \cX_{L_n, S}
        \]
        on $\chi$-parts by proposition \ref{prop:isomorphism_on_chi_parts} (after extending scalars from $\QQ_p^c$ to $\CC_p$). Since
        \begin{equation}
        \label{eq:hom_is_taking_chi-parts}
            \Hom_{\CC_p[\cG_n]}(V_\chi, M) = \Hom_{\CC_p[\cG_n]}(V_\chi, e(\chi)M)
        \end{equation}
        for any $\CC_p[\cG_n]$-module $M$ (because $e(\chi)$ acts trivially on $V_\chi$), the map $(\lambda_{n, S}^\beta \circ (\CC_p \otimes_{\ZZ_p} \varphi_n^\alpha))_\ast$ is in fact an automorphism. This justifies the non-vanishing of the regulator:

        \begin{defn}
        \label{defn:regulator}
            Setting \ref{sett:formulation}. Let $\chi \in \Irr_p(\cG)$ and choose $n \geq n(S)$ such that $\chi$ factors through $\cG_n$. The \textbf{Stark-Tate regulator}\index{regulator}\index{Stark-Tate regulator} of $\alpha$ on $\chi$-parts\footnote{Some authors, most notably Tate, denote by $R(\chi, \ldots)$ a similar regulator constructed by looking at \textit{$\check{\chi}$-parts} instead (cf. \cite{tate}, p. 26). The reader should be aware of the different notational choice in this text, where $R_S^\beta(\alpha, \chi)$ really refers to the regulator on \textit{$\chi$-parts}.} is defined as
            \[
                R_S^\beta(\alpha, \chi) = \det(\lambda_{n, S}^\beta \circ (\CC_p \otimes_{\ZZ_p} \varphi_n^\alpha) \mid \Hom_{\CC_p[\cG_n]}(V_\chi, \CC_p \otimes_{\ZZ_p} \cX_{L_n, S})) \in \CC_p,
            \]
            which is non-zero if $\chi$ satisfies \ref{manualcond:kc}.
            \qedef
        \end{defn}

        This object behaves well with respect to character inflation:

        \begin{lem}
        \label{lem:r_chi_independent_of_n}
            In definition \ref{defn:regulator}, $R_S^\beta(\alpha, \chi)$ is independent of the choice of $n$.
        \end{lem}

        \begin{proof}
            Let $n$ be as in the definition and choose $m > n$. Denote the projections of $\chi$ to $\cG_n$ and $\cG_m$ by $\chi_n$ and $\chi_m$, respectively As usual, we write $\cG_{n, v}$ for the decomposition group of $v(L_n)$ in $L_n/K$ and analogously for $\cG_{m, v}$. We first compare the $\CC_p$-dimension of ${\Hom_{\CC_p[\cG_n]}(V_{\chi_n}, \CC_p \otimes_{\ZZ_p} \cX_{L_n, S})}$ and  $\Hom_{\CC_p[\cG_m]}(V_{\chi_m}, \CC_p \otimes_{\ZZ_p} \cX_{L_m, S})$. One could apply lemma \ref{lem:properties_of_L-functions} ii) and v), but we give a self-contained proof. The basic fact at play here is that, for an irreducible character $\psi \in \Irr_{\CC_p}(\cG_n)$, one has
            \[
                \dim_{\CC_p}(\Hom_{\CC_p[\cG_n]}(V_{\chi_n}, V_\psi)) =
                \begin{cases}
                    1, & \psi = \chi_n,\\
                    0, & \text{otherwise},
                \end{cases}
            \]
            the first case being a consequence of Wedderburn's theorem \eqref{eq:wedderburn} (regarding $V_\chi$ as a column vector). It follows that, if $\psi_n$ is the character of $\CC_p \otimes_{\ZZ_p} \cX_{L_n, S}$ in the sense of \eqref{eq:character_associated_module}, the  $\CC_p$-dimension of $\Hom_{\CC_p[\cG_n]}(V_{\chi_n}, \CC_p \otimes_{\ZZ_p} \cX_{L_n, S})$ is $\sprod{\chi_n, \psi_n}$. The short exact sequence
            \[
                0 \to \CC_p \otimes_{\ZZ_p} \cX_{L_n, S} \to \CC_p \otimes_{\ZZ_p} \cY_{L_n, S} \to \CC_p \to 0,
            \]
            together with the fact that $\CC_p \otimes_{\ZZ_p} \cY_{L_n, S} = \bigoplus_{v \in S} \Ind_{\cG_{n, v}}^{\cG_n} \CC_p$, implies $\psi_n = (\sum_{v \in S} \indu_{\cG_{n, v}}^{\cG_n} \bbone_{\cG_{n, v}}) - \bbone_{\cG_n}$. Therefore,
            \[
                \sprod{\chi_n, \psi_n}_{\cG_n} =
                \bigg(\sum_{v \in S} \sprod{\chi_n, \indu_{\cG_{n, v}}^{\cG_n} \bbone_{\cG_{n, v}}}_{\cG_n}\bigg) - \sprod{\chi_n, \bbone_{\cG_n}}_{\cG_n} = \bigg(\sum_{v \in S} \sprod{\rest_{\cG_{n, v}}^{\cG_n} \chi_n, \bbone_{\cG_{n, v}}}_{\cG_{n, v}}\bigg) - \sprod{\chi_n, \bbone_{\cG_n}}_{\cG_n}
            \]
            by Frobenius reciprocity \eqref{eq:frobenius_reciprocity}. An analogous argument yields
            \[
                \dim_{\CC_p}(\Hom_{\CC_p[\cG_m]}(V_{\chi_m}, \CC_p \otimes_{\ZZ_p} \cX_{L_m, S})) = \bigg(\sum_{v \in S} \sprod{\rest_{\cG_{m, v}}^{\cG_m} \chi_m, \bbone_{\cG_{m, v}}}_{\cG_{m, v}}\bigg) - \sprod{\chi_m, \bbone_{\cG_m}}_{\cG_m},
            \]
            after which a simple computation (for instance, directly using the definition of the scalar product of characters) shows that the two coincide.

            We may regard the canonical projection $\CC_p \otimes_{\ZZ_p} \cX_{L_m, S} \sa \CC_p \otimes_{\ZZ_p} \cX_{L_n, S}$ as a $\CC_p[\cG_m]$-homomorphism via $\CC_p[\cG_m] \sa \CC_p[\cG_n]$. By the same token, $V_{\chi_m}$ and $V_{\chi_n}$ are isomorphic $\CC_p[\cG_m]$-modules (observe here that $\chi_m = \infl_{\cG_n}^{\cG_m} \chi_n$). We therefore have a $\CC_p$-linear surjection
            \[
                \Hom_{\CC_p[\cG_m]}(V_{\chi_m}, \CC_p \otimes_{\ZZ_p} \cX_{L_m, S}) \sa \Hom_{\CC_p[\cG_m]}(V_{\chi_n}, \CC_p \otimes_{\ZZ_p} \cX_{L_n, S}) = \Hom_{\CC_p[\cG_n]}(V_{\chi_n}, \CC_p \otimes_{\ZZ_p} \cX_{L_n, S})
            \]
            by the semisimplicity of $\CC_p[\cG_m]$. But we have shown the first and last terms have the same $\CC_p$-dimensions, and hence this surjection is an isomorphism. The lemma now follows from the commutativity of the diagram
            \begin{center}
                \begin{tikzcd}
                    \Hom_{\CC_p[\cG_m]}(V_{\chi_m}, \CC_p \otimes_{\ZZ_p} \cX_{L_m, S}) \arrow[r, "\sim"] \arrow[d, "(\lambda_{m, S}^\beta \circ (\CC_p \otimes_{\ZZ_p} \varphi_m^\alpha))_\ast"] & \Hom_{\CC_p[\cG_n]}(V_{\chi_n}, \CC_p \otimes_{\ZZ_p} \cX_{L_n, S}) \arrow[d, "(\lambda_{n, S}^\beta \circ (\CC_p \otimes_{\ZZ_p} \varphi_n^\alpha))_\ast"] \\
                    \Hom_{\CC_p[\cG_m]}(V_{\chi_m}, \CC_p \otimes_{\ZZ_p} \cX_{L_m, S}) \arrow[r, "\sim"]                                                        &
                    \Hom_{\CC_p[\cG_n]}(V_{\chi_n}, \CC_p \otimes_{\ZZ_p} \cX_{L_n, S})
                \end{tikzcd}
            \end{center}
        \end{proof}

        As mentioned in \eqref{eq:hom_is_taking_chi-parts}, applying $\Hom_{\CC_p[\cG_n]}(V_\chi, -)$ to a $\CC_p[\cG_n]$-module is simply a way of taking $\chi$-parts. A more immediate method we have also encountered is to multiply the module by $e(\chi)$. Since $\lambda_{n, S}^\beta \circ (\CC_p \otimes_{\ZZ_p} \varphi_n^\alpha)$ is a $\CC_p[\cG_n]$-endomorphism of $e(\chi)\CC_p \otimes_{\ZZ_p} \cX_{L_n, S}$ (and hence $\CC_p$-linear), one could define the regulator as the determinant of this transformation instead. This approach is closely related to the above, but it yields a coarser magnitude: it can be shown that
        \[
            \det(\lambda_{n, S}^\beta \circ (\CC_p \otimes_{\ZZ_p} \varphi_n^\alpha) \mid e(\chi) \CC_p \otimes_{\ZZ_p} \cX_{L_n, S}) = R_S^\beta(\alpha, \chi)^{\chi(1)}.
        \]
        Subsequent sections (most notably \ref{sec:independence_of_the_choice_of_parameters}) will illustrate why $R_S^\beta(\alpha, \chi)$ is the right choice for our purposes.

    \section{Evaluation maps}
    \label{sec:evaluation_maps}

        As explained in the introduction to the present chapter, the Interpolation Conjecture will assert the existence of certain quotients of power series which interpolate regulated special $L$-values. More specifically, the series quotient associated to a character $\chi \in \Irr_p(\cG)$ is expected to interpolate the special values corresponding to almost all of its $\rho$-twists. By a $\rho$-twist of $\chi$ we mean its product with a character of type $W$ - in other words, any character which is $W$-equivalent to $\chi$ in the terminology of section \ref{sec:morphisms_on_finite_level}.

        In this section, we present the relation between evaluation at 0 of series quotients associated to $W$-equivalent characters. This is motivated by the structure of the centre of $\cQ^c(\cG)$ studied by Ritter and Weiss in \cite{rwii} and leads us to define \textit{twisted evaluation maps}, which amount to evaluation at certain $p$-adic algebraic points in the open unit ball. Requiring a single series quotient to interpolate $L$-values under infinitely many twisted evaluation maps therefore ensures its uniqueness - if it exists - which is important for the Main Conjecture.

        For the first definition, let $\tilde{\Gamma} \iso \ZZ_p$ be a profinite group and fix a topological generator $\tilde{\gamma}$ (the notation here has been chosen to avoid confusion with the specific $\Gamma$ and $\gamma$ from setting \ref{sett:construction}). The crucial isomorphism  \eqref{eq:iso_iwasawa_algebra} holds for more general rings of coefficients than $\ZZ_p$: given a $p$-adic field $E$, the homomorphism of topological $\cO_E$-algebras
        \begin{align*}
            \Lambda^{\cO_{E}}(\tilde{\Gamma}) & \to \cO_{E}[[\tilde{T}]] \\
            \tilde{\gamma} & \mapsto \tilde{T} + 1
        \end{align*}
        is an isomorphism, where $\cO_{E}[[\tilde{T}]]$ is equipped with the $\ideal{\pi, \tilde{T}}$-adic topology for any uniformiser $\pi$ of $E$. Under this identification, the augmentation map from section \ref{sec:iwasawa_algebras_and_modules} corresponds to evaluation at 0, i.e. the diagram
        \begin{center}
            \begin{tikzcd}[column sep=large]
            \Lambda^{\cO_{E}}(\tilde{\Gamma}) \arrow[r, "\tilde{\gamma} \: \mapsto \: \tilde{T} + 1"] \arrow[d, "\aug_{\tilde{\Gamma}}", two heads] & {\cO_E[[\tilde{T}]]} \arrow[d, "\tilde{T} = 0", two heads] \\
            \cO_E \arrow[r, equals]                                                       & \cO_E
            \end{tikzcd}
        \end{center}
        commutes. Recall that its kernel, the augmentation ideal $\Delta^{\cO_E}(\tilde{\Gamma}) = \ker(\aug_{\tilde{\Gamma}})$, is a height-one prime ideal of $\Lambda^{\cO_{E}}(\tilde{\Gamma})$ generated by $\tilde{\gamma} - 1$. The map $\aug_{\tilde{\Gamma}}$ has a natural extension to the  field of fractions $\cQ^E(\tilde{\Gamma}) = \ffrac(\Lambda^{\cO_E}(\tilde{\Gamma})) = E \otimes_{\QQ_p} \cQ(\tilde{\Gamma})$:
        \begin{defn}
        \label{defn:evaluation_at_zero}
            Let $E$ be a $p$-adic field and $\tilde{\gamma}$ a topological generator of a profinite group $\tilde{\Gamma} \iso \ZZ_p$. We define the evaluation-at-0 map with respect to $\tilde{\gamma}$ as
            \begin{align*}
                ev_{\tilde{\gamma}} \colon \cQ^E(\tilde{\Gamma}) & \to E \cup \set{\infty} \\
                q = \frac{f}{g} & \mapsto
                \begin{cases}
                    \frac{\aug_{\tilde{\Gamma}}(f)}{\aug_{\tilde{\Gamma}}(g)}, & q \in \Lambda^{\cO_{E}}(\tilde{\Gamma})_{\Delta^{\cO_{E}}(\tilde{\Gamma})} \\
                    \infty, & \text{otherwise}
                \end{cases}
            \end{align*}
            where the subscript in $\Lambda^{\cO_{E}}(\tilde{\Gamma})_{\Delta^{\cO_{E}}(\tilde{\Gamma})}$ denotes localisation and, in the first case, $g$ is chosen outside $\Delta^{\cO_{E}}(\tilde{\Gamma})$. Furthermore, we define
            \[
                ev_{\tilde{\gamma}} \colon \cQ^c(\tilde{\Gamma}) = \QQ_p^c \otimes_{\QQ_p} \cQ(\tilde{\Gamma}) = \varinjlim_E \cQ^E(\tilde{\Gamma}) \to \QQ_p^c \cup \set{\infty}
            \]
            as the direct limit of the above maps.\qedef
        \end{defn}

        \begin{rem}
        \phantomsection
        \label{rem:definition_of_evaluation}
            \begin{enumerate}[i)]
                \item{
                    The map $ev_{\tilde{\gamma}} \colon \cQ^E(\tilde{\Gamma}) \to E \cup \set{\infty}$ is well defined, i.e. independent of the choice of an expression $q = f/g$ for $q \in \Lambda^{\cO_{E}}(\tilde{\Gamma})_{\Delta^{\cO_{E}}(\tilde{\Gamma})}$, because $\aug_{\tilde{\Gamma}}$ is a ring homomorphism on $\Lambda^{\cO_E}(\tilde{\Gamma})$. The well-definedness of $ev_{\tilde{\gamma}} \colon \cQ^c(\tilde{\Gamma}) \to \QQ_p^c \cup \set{\infty}$ follows immediately from the compatibility of $\aug_{\tilde{\Gamma}}$ with extension of scalars - which also shows that, if $E'/E$ is an extension of $p$-adic fields, then $\Delta^{\cO_{E'}}(\tilde{\Gamma}) = \cO_{E'} \otimes_{\cO_E} \Delta^{\cO_E}(\tilde{\Gamma})$.
                }
                \item{
                    $E \cup \set{\infty}$ has no natural ring or $E$-vector-space structure, so $ev_{\tilde{\gamma}}$ is not a ring homomorphism and it is not $E$-linear. However, its restriction to the $\Lambda^{\cO_{E}}(\tilde{\Gamma})_{\Delta^{\cO_{E}}(\tilde{\Gamma})}$ is indeed a homomorphism of $E$-algebras for the same reasons as above. In other words, $ev_{\tilde{\gamma}}$ is compatible with sums, products and scaling as long as none of the involved series quotients evaluate to $\infty$.
                }
            \end{enumerate}
        \end{rem}

        The definition of $ev_{\tilde{\gamma}}$ as the \textit{evaluation-at-0 map with respect to $\tilde{\gamma}$} may appear superfluous, since it is independent of the choice of a topological generator $\tilde{\gamma}$ of $\tilde{\Gamma}$ (because so is $\aug_{\tilde{\Gamma}}$ by definition). Its relevance will become clear in a few lines, where $\tilde{\gamma}$ will be replaced not by another topological generator of the same group, but by a scaled version of itself. The inclusion of $\tilde{\gamma}$ in the notation will essentially allow us to track changes of variables in power series rings later on.

        The next two results are taken from \cite{rwii} (see the footnote\textsuperscript{\ref{foot:rw}} on page \pageref{foot:rw}) and stated here without proof. More specifically, they are a rearrangement of proposition 5 and the corollary to proposition 6 in the article. They concern structural properties of $\cQ^c(\cG)$ (cf. \eqref{eq:q_direct_limit_finite} and \eqref{eq:lambda_regular_elements}) which will be essential in the sequel. The first one characterises the primitive central idempotents of the ring:

        \begin{prop}
        \label{prop:rw_properties_q}
            Setting \ref{sett:construction}. The following hold:
            \begin{enumerate}[i)]
                \item{
                    The ring $\cQ^c(\cG)$ is semisimple Artinian.
                }
                \item{
                    Let $\chi \in \Irr_p(\cG)$. Then
                    \[
                        e_\chi = \sum_{\substack{\eta \in \Irr_{\QQ_p^c}(H)\\\eta \mid \rest_H^\cG \chi}} e(\eta)
                    \]
                    is a primitive central idempotent of $\cQ^c(\cG)$.\index{primitive central idempotent!e2@$e_\chi$}
                }
                \item{
                    Every primitive central idempotent of $\cQ^c(\cG)$ is of the form $e_\chi$ as above for some $\chi \in \Irr_p(\cG)$.
                }
                \item{
                    Let $\chi, \chi' \in \Irr_p(\cG)$. Then the following are equivalent:
                    \begin{itemize}
                        \item{
                            $e_\chi = e_{\chi'}$.
                        }
                        \item{
                            $\chi \sim_W \chi'$.
                        }
                        \item{
                            $\rest_H^\cG \chi = \rest_H^\cG \chi'$.
                        }
                    \end{itemize}
                }
            \end{enumerate}
        \end{prop}

        Consider $\chi$ and $\eta$ as in part ii). For $\sigma \in \cG$, let $\eta^\sigma$ be the conjugated character given by
        \[
            \eta^\sigma(\tau) = \eta(\sigma\tau\sigma^{-1})
        \]
        on $\tau \in H$. This is again an irreducible character of $H$ which divides $\rest_H^\cG \chi$, as $\chi^\sigma = \chi$. Thus, the elements of $\cG$ permute the irreducible constituents of $\rest_H^\cG \chi$, and this action is in fact transitive by Clifford theory. We define the stabiliser of $\eta$ as the group
        \[
            St(\eta) = \set{\sigma \in \cG \et{such that} \eta^\sigma = \eta} \leq \cG,
        \]
        so that when $\tau$ runs over a set of coset representatives of $St(\eta) \backslash \cG$, the character $\eta^\tau$ runs over the irreducible constituents of $\rest_H^\cG \chi$ exactly once. $St(\eta)$ is an open subgroup (as can be seen by projecting $\chi$ to a finite quotient $\cG_n$ of $\cG$) containing $H$, and hence of the form $\Gal(L_\infty/K_n)$ for some $n$. This shows that
        \begin{equation}
        \label{eq:definition_w_chi}
            w_\chi = [\cG : St(\eta)]
        \end{equation}
        is a power of $p$. Since all such subgroups are normal, Clifford theory shows that $St(\eta)$ does not depend of the choice of $\eta \mid \rest_H^\cG \chi$ and the notation of $w_\chi$ is thus justified. This index plays an important role in the next result, again taken from \cite{rwii}:

        \begin{prop}
        \label{prop:structure_gamma_chi}
            Setting \ref{sett:construction}. Let $\chi \in \Irr_p(\cG)$. Then the following hold:
            \begin{enumerate}[i)]
                \item{
                    There exists a unique $\gamma_\chi \in Z(\cQ^c(\cG)e_\chi)$ satisfying:
                    \begin{itemize}
                        \item{
                            $\gamma_\chi = gc$ for some $g \in \cG$ with projection $\gamma_K^{w_\chi} \in \Gamma_K$ (where $\gamma_K$ is the topological generator chosen in the setting) and some $c \in \punits{\QQ_p^c[H]e_\chi}$.
                        }
                        \item{
                            $\gamma_\chi$ acts trivially on the $\QQ_p^c$-vector space $V_\chi$ which affords $\chi$.
                        }
                    \end{itemize}
                    That $\gamma_\chi$ also satisfies $\gamma_\chi = cg$, and it topologically generates a subgroup $\Gamma_\chi \iso \ZZ_p$ of $\punits{\cQ^c(\cG)e_\chi}$.
                }
                \item{
                    Let $E$ be a $p$-adic field such that $\chi$ has a realisation over $E$. Then the group $\Gamma_\chi$ from i) is contained in $\units{Z(\cQ^E(\cG)e_\chi)}$, which induces an equality of fields ${\cQ^E(\Gamma_\chi) = Z(\cQ^E(\cG)e_\chi)}$.
                }
                \item{
                    With $E$ as in ii), the ring homomorphism
                    \[
                        j_\chi^E \colon Z(\cQ^E(\cG)) \sa Z(\cQ^E(\cG)e_\chi) = \cQ^E(\Gamma_\chi) \ia \cQ^E(\Gamma_K),
                    \]
                    where the last arrow is induced by $\gamma_\chi \mapsto \gamma_K^{w_\chi}$, is independent of the choice of $\gamma_K$.
                }
                \item{
                    If $\rho$ is a type-$W$ character of $\cG$, then $\gamma_{\chi \otimes \rho} = \rho(\gamma_K)^{-w_\chi} \gamma_\chi$.
                }
            \end{enumerate}
        \end{prop}

        Proposition \ref{prop:rw_properties_q} implies that $\cQ^c(\cG)$ decomposes as a finite product of simple rings
        \begin{equation}
        \label{eq:decomposition_algebra_chi-parts}
            \cQ^c(\cG) \iisoo \prod_{\chi/{\sim}_W} \cQ^c(\cG)e_\chi,
        \end{equation}
        where $\chi$ runs over any set of representatives of $\Irr_p(\cG)$ modulo $W$-equivalence. The reduced norm on $K_1$ introduced in section \ref{sec:algebraic-k-theory} then factors over these $\chi$-parts by definition, i.e. there is a commutative diagram
        \begin{equation}
        \label{eq:k1_reduced_norm_chi_parts}
            \begin{tikzcd}
                K_1(\cQ^c(\cG)) \arrow[d, "\nr"] \arrow[r, "\sim"] & \prod_{\chi/{\sim}_W} K_1(\cQ^c(\cG)e_\chi) \arrow[d, "\prod \nr"] \\
                \units{Z(\cQ^c(\cG))} \arrow[r, "\sim"]                  & \prod_{\chi/{\sim}_W} \units{Z(\cQ^c(\cG)e_\chi)}
            \end{tikzcd}
        \end{equation}
        of abelian groups. By part ii) of the last proposition, each term in the bottom-right corner is in fact the group of units of a field  $\cQ^c(\Gamma_\chi)$ of series quotients in the variable $T_\chi = \gamma_\chi - 1$. It is precisely there that the interpolating elements of the Main Conjecture are expected to live. However, to conjecture the existence of an element therein which evaluates at 0 to the special $L$-value associated to $\chi$ would be an empty statement in isolation - and also entail an evident lack of uniqueness. Both issues can be resolved through the fundamental relation iv) above, which hints at the notion that one should consider interpolation for all $\rho$-twists of $\chi$ simultaneously. The following lemma lays the foundations for that approach:

        \begin{lem}
        \label{lem:twisted_evaluation_maps}
            Setting \ref{sett:construction}. Let $\chi, \rho \in \Irr_p(\cG)$ with $\rho$ of type $W$ and choose a $p$-adic field $E$ such that $\chi$ and $\rho$ (and in particular $\chi \otimes \rho$) have a realisation over it. Then $\cQ^E(\Gamma_\chi) = \cQ^E(\Gamma_{\chi \otimes \rho}) \subseteq \cQ^E(\cG)$ and, for a series quotient $q \in \cQ^E(\Gamma_\chi) \iso \ffrac(\cO_E[[T_\chi]])$, one has
            \[
                ev_{\gamma_{\chi \otimes \rho}}(q) = \restr{q}{T_\chi = \rho(\gamma_K)^{w_\chi} - 1} \in E \cup \set{\infty},
            \]
            where the right-hand side denotes evaluation at $T_\chi = \rho(\gamma_K)^{w_\chi} - 1 \in E$ and we set $k/0 = \infty$ for $k \in \units{E}$.
        \end{lem}

        \begin{proof}
            We first show $\cQ^E(\Gamma_\chi) = \cQ^E(\Gamma_{\chi \otimes \rho})$. This follows immediately from propositions \ref{prop:rw_properties_q} iv) and \ref{prop:structure_gamma_chi} ii), but we present here a more explicit proof different from that in the reference. The Iwasawa algebras $\Lambda^{\cO_E}(\Gamma_\chi)$ and $\Lambda^{\cO_E}(\Gamma_{\chi \otimes \rho})$ are defined as
            \begin{equation}
            \label{eq:iwasawa_algebras_twists}
                \Lambda^{\cO_E}(\Gamma_\chi) = \varprojlim_{n} \cO_E\left[\faktor{\Gamma_\chi}{\Gamma_\chi^{p^n}}\right] \quad \text{and} \quad \Lambda^{\cO_E}(\Gamma_{\chi \otimes \rho}) = \varprojlim_{n} \cO_E\left[\faktor{\Gamma_{\chi \otimes \rho}}{\Gamma_{\chi \otimes \rho}^{p^n}}\right].
            \end{equation}
            Let us apply the relation $\gamma_\chi = \rho(\gamma_K)^{w_\chi} \gamma_{\chi \otimes \rho}$ from proposition \ref{prop:structure_gamma_chi}, denoting $\rho(\gamma_K)^{w_\chi}$ by $\zeta \in \cO_E$. Since $\zeta$ is a $p$-power root of unity (potentially 1), one has $\gamma_\chi^{p^n} = \gamma_{\chi \otimes \rho}^{p^n}$ (and hence $\Gamma_\chi^{p^n} = \Gamma_{\chi \otimes \rho}^{p^n}$) for all $n$ larger than some fixed $n_0$ satisfying $\zeta^{p^{n_0}} = 1$. This yields an equality
            \[
                 \sum_{i = 0}^{p^n - 1} a_i \cdot \gamma_\chi^i \Gamma_\chi^{p^n} = \sum_{i = 0}^{p^n - 1} (a_i \zeta^i) \cdot \gamma_{\chi \otimes \rho}^i \Gamma_{\chi \otimes \rho}^{p^n}
            \]
            in $\cO_E[\Gamma_\chi/\Gamma_\chi^{p^n}]$ for all such $n$, where the left-hand term is a generic element and the right-hand one clearly lies in $\cO_E[\Gamma_{\chi \otimes \rho}/\Gamma_{\chi \otimes \rho}^{p^n}]$. This equality commutes with the transition maps in \eqref{eq:iwasawa_algebras_twists} and therefore implies $\Lambda^{\cO_E}(\Gamma_\chi) = \Lambda^{\cO_E}(\Gamma_{\chi \otimes \rho})$, which extends to an equality of their fields of fractions $\cQ^E(\Gamma_\chi) = \cQ^E(\Gamma_{\chi \otimes \rho})$ inside the common ring $\cQ^E(\cG)$.

            The induced isomorphism
            \begin{align*}
                \cO_E[[T_\chi]] \xrightarrow{T_\chi \mapsto \gamma_\chi - 1} \Lambda^{\cO_E}(\Gamma_\chi) & = \Lambda^{\cO_E}(\Gamma_{\chi \otimes \rho}) \xrightarrow{\gamma_{\chi \otimes \rho} \mapsto T_{\chi \otimes \rho} + 1} \cO_E[[T_{\chi \otimes \rho}]] \\
                \gamma_\chi & = \rho(\gamma_K)^{w_\chi} \gamma_{\chi \otimes \rho}
            \end{align*}
            gives the relation $T_\chi = \rho(\gamma_K)^{w_\chi} (T_{\chi \otimes \rho} + 1) - 1$, which identifies evaluation at $T_{\chi \otimes \rho} = 0$ with evaluation at $T_\chi = \rho(\gamma_K)^{w_\chi} - 1$.

            Passing now to the field of fractions, let $q = f/g \in \cQ^E(\Gamma_{\chi \otimes \rho})$ be in reduced form (recall that $\cO_E[[T_{\chi \otimes \rho}]] \iso \Lambda^{\cO_E}(\Gamma_{\chi \otimes \rho})$ is a unique factorisation domain) with $fg \neq 0$. If $g \in \Delta^{\cO_E}(\Gamma_{\chi \otimes \rho})$, then the above argument shows
            \[
                \restr{g}{T_\chi = \rho(\gamma_K)^{w_\chi} - 1} = \restr{g}{T_{\chi \otimes \rho} = 0} = 0 \quad \text{and} \quad \restr{f}{T_\chi = \rho(\gamma_K)^{w_\chi} - 1} = \restr{f}{T_{\chi \otimes \rho} = 0} \neq 0
            \]
            (otherwise they would share the factor $T_{\chi \otimes \rho}$) and hence $\restr{q}{T_\chi = \rho(\gamma_K)^{w_\chi} - 1} = \infty$, which matches the definition of $ev_{\gamma_{\chi \otimes \rho}}(q)$. If $y \not \in \Delta^{\cO_E}(\Gamma_{\chi \otimes \rho})$, one has
            \[
                ev_{\gamma_{\chi \otimes \rho}}(q) =
                \frac{\restr{f}{T_{\chi \otimes \rho} = 0}}{\restr{g}{T_{\chi \otimes \rho} = 0}} =
                \frac{\restr{f}{T_\chi = \rho(\gamma_K)^{w_\chi} - 1}}{\restr{g}{T_\chi = \rho(\gamma_K)^{w_\chi} - 1}} =
                \restr{q}{T_\chi = \rho(\gamma_K)^{w_\chi} - 1}
            \]
            as well.
        \end{proof}

        \begin{rem}
        \phantomsection
        \label{rem:properties_of_ev}
            \begin{enumerate}[i)]
                \item{
                    The map $ev_{\gamma_\chi}$ depends on $\chi$ alone, rather than on its specific expression in terms of $\rho$-twists. As a sanity check, suppose that $\chi = \chi \otimes \rho$ for some $\rho$ of type $W$. The lemma then says that
                    \[
                        \restr{-}{T_\chi = \rho(\gamma_K)^{w_\chi} - 1} = ev_{\gamma_{\chi \otimes \rho}} = ev_{\gamma_{\chi}} = \restr{-}{T_\chi =0}.
                    \]
                    This can only happen if $\rho(\gamma_K)^{w_\chi} = 1$, but by proposition \ref{prop:structure_gamma_chi} iv), this is precisely the case.

                    For an example of $\chi = \chi \otimes \rho$ with non-trivial $\rho$, let $\rho$ be the inflation to $\cG$ of the linear character of $\Gamma_K$ which sends $\gamma_K$ to a primitive $p$-th root of unity. Denote by $G = \Gal(L_\infty/K_1)$ the preimage of $\Gamma_K^p$ in $\cG$ under the canonical projection, which is open and normal. Let $\psi$ be an Artin character of $G$ whose induction $\chi = \indu_G^\cG \psi$ is irreducible - such characters can be easily shown to exist in some cases. Then, for any $g \in \cG$, one has either $g \in G$, in which case $\rho(g) = 1$; or $g \not \in G$, in which case $\chi(g) = 0$. Therefore, $\chi = \chi \otimes \rho$.
                }
                \item{
                    It follows from the lemma that a series $x \in \cO_E[[T_\chi]]$ evaluated at $T_\chi = \zeta - 1$, where $\zeta$ is a $p$-power root of unity, yields an element in $\cO_E$ - in other words, it converges. This agrees with basic $p$-adic analysis: since the element $\zeta - 1$ is in the maximal ideal of $\ZZ[\zeta]$ and hence of $\cO_E$, any series
                    \[
                        \sum_{i = 0}^\infty a_i (\zeta - 1)^i
                    \]
                    in $\cO_E$ has terms tending to $0$ and is therefore convergent by the completeness of $E$.
                }
                \item{
                    Taking direct limits on the first statement of the lemma yields $\cQ^c(\Gamma_\chi) = \cQ^c(\Gamma_{\chi \otimes \rho})$.
                }
            \end{enumerate}
        \end{rem}

        It stands to reason to refer to $ev_{\gamma_{\chi \otimes \rho}}$ as a twist of $ev_{\gamma_\chi}$, and thus as a \textbf{twisted evaluation map}\index{twisted evaluation map}. Since type-$W$ characters can map $\gamma_K$ to roots unity of order an arbitrary power of $p$, fixing a series quotient $q \in  \cQ^c(\Gamma_\chi)$ and applying $ev_{\chi \otimes \rho}$ for all $\rho$ corresponds to evaluating it at infinitely many points of $\QQ_p^c$. These points lie in the open unit ball centred at 0 and become arbitrarily close to its boundary, as the $p$-adic valuation of $\zeta_{p^n} - 1$ (with $\zeta_{p^n}$ of order $p^n$)  is $1/(p^{n - 1}(p - 1))$ for $n \geq 1$.

        The twisted evaluation maps are closely related to the $\rho^\sharp$ in the cited article (cf. \cite{rwii}, definition on page 557). This automorphism of $\cQ^c(\Gamma_K)$ is induced by the continuous $\cO_E$-linear automorphism of $\Lambda^{\cO_E}(\Gamma_K)$ determined by $\rho^\sharp(\gamma_K) = \rho(\gamma_K) \gamma_K$. The connection to our maps is given by the fact that the composition
        \begin{equation}
        \label{eq:twisted_evaluation_maps_composition}
            \cQ^c(\Gamma_\chi) \xhookrightarrow{\gamma_\chi \mapsto \gamma_K^{w_\chi}} \cQ^c(\Gamma_K) \xrightarrow{\rho^\sharp} \cQ^c(\Gamma_K) \xrightarrow{ev_{\gamma_K}} \QQ_p^c \cup \set{\infty},
        \end{equation}
        is precisely $ev_{\gamma_{\chi \otimes \rho}}$. Indeed,
        \[
            ev_{\gamma_K}(\rho^\sharp(\gamma_K^{w_\chi})) = ev_{\gamma_K}(\rho(\gamma_K)^{w_\chi} \gamma_K^{w_\chi}) = \rho(\gamma_K)^{w_\chi} = ev_{\gamma_{\chi \otimes \rho}}(\gamma_\chi).
        \]
        The arrows in \eqref{eq:twisted_evaluation_maps_composition} play an important role in the $\Hom$-formulation of Ritter and Weiss, but we shall bypass them using the twisted evaluation maps. This will allow us to conjecture the existence of interpolating series quotients in $\cQ^c(\Gamma_\chi)$, rather than in the larger $\cQ^c(\Gamma_K)$.

        We now address the dependence of $ev_{\gamma_\chi}$ on the choice of $\gamma_K$:
        \begin{lem}
        \label{lem:twisted_evaluation_maps_independent_gamma_K}
            Setting \ref{sett:construction}. Let $\chi \in \Irr_p(\cG)$ and choose a finite extension $E$ of $\QQ_p$ over which $\chi$ has a realisation. Fix a topological generator $\tilde{\gamma}_K$ of $\Gamma_K$ in addition to the $\gamma_K$ in the setting, and let $\tilde{\gamma}_\chi$ and $\gamma_\chi$ be the corresponding elements of $Z(\cQ^E(\cG)e_\chi)$ constructed in proposition \ref{prop:structure_gamma_chi}. Denote by $\Gamma_\chi$ and $\widetilde{\Gamma}_\chi$ the groups topologically generated by $\gamma_\chi$ and $\tilde{\gamma}_\chi$, respectively. Then $\cQ^E(\Gamma_\chi) = \cQ^E(\widetilde{\Gamma}_\chi) = Z(\cQ^E(\cG)e_\chi)$ and the maps
            \[
                ev_{\gamma_\chi}, ev_{\gamma_{\tilde{\chi}}} \colon \cQ^E(\Gamma_\chi) = \cQ^E(\widetilde{\Gamma}_\chi) \to E \cup \set{\infty}
            \]
            coincide.
        \end{lem}

        \begin{proof}
            By proposition \ref{prop:structure_gamma_chi} ii), we have equalities $\cQ^E(\Gamma_\chi) = Z(\cQ^E(\cG)e_\chi) = \cQ^E(\widetilde{\Gamma}_\chi)$ (not mere isomorphisms), where the middle term - most notably, $e_\chi$ - does not depend on the choice of $\gamma_K$. Consider the diagram
            \begin{equation}
            \label{eq:diagram_ev_independent_gamma_K}
                \begin{tikzcd}
                    & \cQ^E(\Gamma_\chi) \arrow[rd, "\gamma_\chi \mapsto \gamma_K^{w_\chi}"'] \arrow[rr, "\gamma_\chi \mapsto 1"]             &                                       & E \cup \set{\infty} \arrow[dd, equal] \\
Z(\cQ^E(\cG)e_\chi) \arrow[ru, equal] \arrow[rd, equal] &                                                      & \cQ^E(\Gamma_K) \arrow[ru, "\gamma_K \mapsto 1"'] \arrow[rd, "\gamma_K \mapsto 1"] &                                                     \\
                    & \cQ^E(\widetilde{\Gamma}_\chi) \arrow[ru, "\tilde{\gamma}_\chi \mapsto \tilde{\gamma}_K^{w_\chi}"] \arrow[rr, "\tilde{\gamma}_\chi \mapsto 1"] &                                       & E \cup \set{\infty}
                \end{tikzcd}
            \end{equation}
            The diamond commutes by proposition \ref{prop:structure_gamma_chi} iii), as do clearly the three triangles. This shows the commutativity of the entire diagram, whose two horizontal arrows are precisely $ev_{\gamma_\chi}$ and $ev_{\gamma_{\tilde{\chi}}}$.
        \end{proof}

        \begin{rem}
        \label{rem:twisted_evaluation_maps_independent_gamma_K}
            The topological generators of $\ZZ_p$ are precisely the units $\units{\ZZ_p}$, and therefore any $\gamma_K$ and $\tilde{\gamma}_K$ as in the lemma satisfy $\tilde{\gamma}_K = \gamma_K^v$ for some $v \in \units{\ZZ_p}$. It can then be shown (cf. \cite{rwii} proof of proposition 6 (4)) that the resulting $\gamma_\chi$ and $\tilde{\gamma}_\chi$ are related by the equation $\tilde{\gamma}_\chi = \gamma_\chi^v$ as well, which implies $\Gamma_\chi = \tilde{\Gamma}_\chi$. Note the difference between this situation and that of $W$-equivalent characters (proposition \ref{prop:structure_gamma_chi}, lemma \ref{lem:twisted_evaluation_maps}), where $\Gamma_\chi$ and $\Gamma_{\chi \otimes \rho}$ may not in coincide even though $\cQ^E(\Gamma_\chi)$ and $\cQ^E(\Gamma_{\chi \otimes \rho})$ always do. \qedef
        \end{rem}

        Out last aim in this section is to show that series quotients are uniquely determined by their values at any infinite set of integral points where they converge and therefore, in our particular setting, by their values under the twisted evaluation maps $ev_{\gamma_{\chi \otimes \rho}}$. An important comment regarding evaluation of series quotients at points of $\QQ_p^c$ and $\CC_p$ is in order. We phrase it as a remark for subsequent reference:
        \begin{rem}
        \label{rem:evaluation_of_series_quotients}
            Consider a finite extension $E/\QQ_p$ and a series $f \in \cO_{E}[[T]]$. Let us denote the open unit ball in $\QQ_p^c$ by $B_1^{\QQ_p^c}(0) = \set{x \in \QQ_p^c \colon \abs{x}_p < 1}$ with $\abs{-}_p$ normalised to $\abs{p}_p = p^{-1}$, i.e. $v_p(p) = 1$. Any $x \in B_1^{\QQ_p^c}(0)$ lies some finite extension $F/\QQ_p$, which we can choose to contain $E$. In fact, $x$ belongs to the maximal ideal of $\cO_{F}$ and hence the series $f$ converges at $x$ to a value $f(x) \in \cO_{F}$ by completeness. The result is independent of the choice of $F$, and we can thus evaluate series in $\cO_{E}[[T]]$ at points of $B_1^{\QQ_p^c}(0)$.

            In the case of a series quotient $q \in \ffrac(\cO_{E}[[T]])$, evaluation is well defined by expressing ${q = f/g}$ with $f, g \in \cO_{E}[[T]]$ coprime and $fg \neq 0$ (if $f = 0$, the discussion is obvious). Namely, for $x \in B_1^{\QQ_p^c}(0)$, $x \in F$ as above, we set ${q(x) = f(x)/g(x) \in F \cup \set{\infty}}$ with the usual convention that $k/0 = \infty$ for $k \neq 0$. In order to prove that this is well defined, we only need to show one cannot simultaneously have $f(x) = g(x) = 0$. This is clear when $F = E$ but requires an argument otherwise. We resort to the Weierstrass preparation theorem (cf. \cite{nsw} theorem 5.3.4), by which $f$ has a (unique) decomposition
            \[
                f = \pi_E^s u P
            \]
            with $\pi_E$ a chosen uniformiser of $E$, $s \in \NN$, $u \in \units{\cO_{E}[[T]]}$ a unit and $P \in \cO_{E}[T]$ a Weierstrass polynomial (cf. section \ref{sec:iwasawa_algebras_and_modules}). Note that $u(x)$ cannot vanish, since crucially both $u$ and $u^{-1}$ converge at $x$ (because $\abs{x}_p < 1$). Hence $f(x) = 0$ implies $P(x) = 0$, and therefore the minimal polynomial $\min_E(x) \in E[T]$ of $x$ over $E$ divides $P$. An easy computation with the ultrametric inequality then shows that $\min_E(x) \in \cO_E[T]$. It follows that the decomposition $P = \min_E(x) P'$ in $E[T]$ takes place entirely over $\cO_E[T]$, as can be seen by multiplying $P'$ by a suitable power of $\pi_E$ passing to $\cO_E/\fm_E[T]$. This shows that $\min_E(x)$ divides the series $f$ in $\cO_E[[T]]$. By the same argument, $g(x) = 0$ implies $\min_E(x) \mid g$. Since we chose $q = f/g$ in reduced form, it cannot happen that $f(x) = g(x) = 0$, as desired.

            The key takeaway from this, which will be essential for the next two results, is that one can evaluate an element $h \in \QQ_p^c \otimes_{\QQ_p} \ffrac(\ZZ_p[[T]]) = \varinjlim_E \ffrac(\cO_E[[T]])$ at all points of $B_1^{\QQ_p^c}(0)$ by choosing a \textit{fixed}  representative $h = f/g$ independent of the point.

            Although it will not be relevant in the sequel, there is a natural way to extend the above discussion to $B_1^{\CC_p}(0) = \set{x \in \CC_p \colon \abs{x}_p < 1}$: for $h = f/g$ as above (with $fg \neq 0$) and $x \in B_1^{\CC_p}(0) \setminus B_1^{\QQ_p^c}(0)$, simply set $h(x) = f(x)/g(x)$. Note that, in this case, neither $f(x)$ nor $g(x)$ can be zero, as $x$ would then be a root of the polynomial appearing in the corresponding Weierstrass decomposition, and hence algebraic over $\QQ_p$. If we approximate $x$ by a Cauchy sequence $\set{x_n}_{n \in \NN} \subseteq B_1^{\QQ_p^c}(0)$, then
            \[
                h(x) = \frac{f(x)}{g(x)} = \frac{\lim\limits_{n \to \infty} f(x_n)}{\lim\limits_{n \to \infty} g(x_n)} = \lim_{n \to \infty} \frac{f(x_n)}{g(x_n)} = \lim_{n \to \infty} h(x_n),
            \]
            where the second and third equalities use the continuity of power series and that of a quotient of continuous functions. The last equality is immediate by definition if $f$ and $g$ are chosen coprime to begin with, but it holds in fact without this restriction as well: since the sequence $\set{x_n}$ does not converge in $\QQ_p^c$, any $p$-adic field - and in particular the splitting fields of the polynomials in the Weierstrass decompositions of $f$ and $g$ - contains only finitely many of its terms. Therefore, $f(x_n) \neq 0 \neq g(x_n)$ for large enough $n$.\qedef
        \end{rem}

        The following lemma extends a classical application of the Weierstrass preparation theorem to the case of series quotients:
        \begin{lem}
        \label{lem:series_quotient_uniqueness}
            Let $S$ be an infinite subset of $B_1^{\QQ_p^c}(0)$. If ${q, q' \in \QQ_p^c \otimes_{\QQ_p} \ffrac(\ZZ_p[[T]])}$ satisfy ${q(x) = q'(x)}$ for all $x \in S$ (including possibly $q(x) = q'(x) = \infty$), then $q = q'$.
        \end{lem}

        \begin{proof}
            There exists a $p$-adic field $E$ such that
            \[
                q = \frac{f}{g}, \ q' = \frac{f'}{g'}
            \]
            with $f, g, f', g' \in \cO_E[[T]]$, $g g' \neq 0$ and both fractions irreducible in $\cO_E[[T]]$. By remark \ref{rem:evaluation_of_series_quotients}, $q$ and $q'$ can be evaluated (potentially to $\infty$) at all $x \in B_1^{\QQ_p^c}(0)$ as $q(x) = f(x)/g(x)$ (i.e. $0/0$ does not occur) and analogously for $q'$. Therefore, one has
            \begin{equation}
                (f g' - g f')(x) = 0
            \end{equation}
            for all $x \in S$. This illustrates why $q(x) = q'(x) = \infty$ is admitted: it is tantamount to ${g(x) = g'(x) = 0}$, which makes the above equation hold all the same (it can only happen at finitely many $x$ at any rate). By a consequence of the Weierstrass preparation theorem (see for instance \cite{wash} corollary 7.4), the series $f g' - g f' \in \cO_E[[T]]$ must be identically 0, as it vanishes at infinitely many points of $ B_1^{\QQ_p^c}(0)$. In other words, $q = f/g = f'/g' = q'$.
        \end{proof}

        The following immediate consequence, which hints at the sets $\cK_S^\alpha(\chi)$ from definition \ref{defn:r_chi}, will be relevant to the formulation of the Main Conjecture:
        \begin{cor}
        \label{cor:uniqueness_series_type_w}
            Setting \ref{sett:construction}. Let $\chi \in \Irr_p(\cG)$ be an irreducible Artin character of $\cG$, $\cK$ an infinite set of type-$W$ characters, and $q, q' \in \cQ^c(\Gamma_\chi)$ two series quotients. If $ev_{\gamma_{\chi \otimes \rho}}(q) = ev_{\gamma_{\chi \otimes \rho}}(q')$ for all $\rho \in \cK$, then $q = q'$.
        \end{cor}

        \begin{proof}
            We regard $q$ and $q'$ as elements of $\QQ_p^c \otimes_{\QQ_p} \ffrac(\ZZ_p[[T_\chi]])$ via the usual identification of $\gamma_\chi - 1$ with $T_\chi$. A character $\rho$ of type $W$ corresponds uniquely to the choice $\rho(\gamma_K) = \zeta$ of a $p$-power root of unity $\zeta\in  \mu_{p^\infty} \subseteq \QQ_p^c$ (cf. \eqref{eq:bijection_type_W_rou}). In particular, the set $S = \set{\rho(\gamma_K)^{w_\chi} - 1 \colon \rho \in \cK} \subseteq B_1^{\QQ_p^c}(0)$ (with $w_\chi$ as in \eqref{eq:definition_w_chi}) is infinite because the homomorphism $-^{w_\chi} \colon \mu_{p^\infty} \to \mu_{p^\infty}$ has finite kernel. Since $ev_{\gamma_{\chi \otimes \rho}}$ corresponds to evaluation at $T_\chi = \rho(\gamma_K)^{w_\chi} - 1$ by lemma \ref{lem:twisted_evaluation_maps}, $q$ and $q'$ coincide at all points of $S$ by assumption. The result now follows from lemma \ref{lem:series_quotient_uniqueness}.
        \end{proof}

    \section{The Main Conjecture}
    \label{sec:the_main_conjecture}

        We are now in a position to formulate the equivariant Main Conjecture. As explained at the beginning of the chapter, this conjecture consists of two parts: interpolation and algebraicity. The former is of analytic nature and claims the existence of certain series quotients which interpolate regulated special values of Artin $L$-series. The latter postulates a $K$-theoretic relation between these series quotients and the refined Euler characteristic of the main complex $\cC_{S, T}\q$. After stating both conjectures, we shall discuss some immediate consequences of their formulation and introduce a refined version which incorporates a uniqueness claim.

        The Interpolation Conjecture concerns the irreducible Artin characters of $\cG$, which we have defined as $p$-adic in nature. Since Artin $L$-functions are constructed from complex characters a priori, we must relate the two spaces via the $\beta \colon \CC_p \xrightarrow{\sim} \CC$ fixed in setting \ref{sett:formulation}. Namely, given $\chi \in \Irr_p(\cG)$, the composition $\beta \chi$ is a complex character of $\cG$ which factors through some finite Galois extension $K'$ of $K$. This gives rise to an Artin $L$-function $L_{K, S, T}(\beta \chi, s)$ for $K'/K$ as explained in section \ref{sec:artin_l-series}, and one is free to enlarge $K'$ without altering the function by lemma \ref{lem:properties_of_L-functions} ii) - in practice, we will always take $K = L_n$ for some large enough $n$. The special value of interest to us is the leading coefficient $L_{K, S, T}^\ast(\beta \chi, 0)$ of the power series expansion of this function at $s = 0$ (cf. \eqref{eq:definition_leading_coefficient}), which we bring back to the $p$-adic setting by considering $\beta^{-1}(L_{K, S, T}^\ast(\beta \chi, 0)) \in \CC_p^\ast$. The resulting $p$-adic complex value is not independent of the choice of $\beta$, but it should be so when regulated - this is precisely the content of Stark's conjecture, which we explore in section \ref{sec:starks_conjecture_and_the_choice_of_beta}. The \textbf{Interpolation Conjecture}\index{Interpolation Conjecture} is the following assertion about the leading coefficients:

        \begin{conje*}[IC($L_\infty/K, \chi, L, S, T, \alpha, \beta$)]
        \phantomsection
        \label{conje:ic}
            Setting \ref{sett:formulation}. Let $\chi \in \Irr_p^c(\cG)$.

            There exists a series quotient $F_{S, T, \chi}^{\alpha, \beta} \in \units{\cQ^c(\Gamma_\chi)}$ such that
            \begin{equation}
            \label{eq:ic}
                ev_{\gamma_{\chi \otimes \rho}}(F_{S, T, \chi}^{\alpha, \beta}) = \frac{\beta^{-1}(L_{K, S, T}^\ast(\beta (\check{\chi} \otimes \rho^{-1}), 0))}{R_S^\beta(\alpha, \chi \otimes \rho)}
            \end{equation}
            for almost all $\rho \in \cK_S^\alpha(\chi)$.\qedef
        \end{conje*}

        \begin{rem}
        \phantomsection
        \label{rem:ic}
            \begin{enumerate}[i)]
                \item{
                    $\check{\chi} \otimes \rho^{-1}$ is the dual character $(\chi \otimes \rho)^{\check{}}$ of $\chi \otimes \rho$.
                }
                \item{
                    If the Interpolation Conjecture holds, the predicted element $F_{S, T, \chi}^{\alpha, \beta} \in \units{\cQ^c(\Gamma_\chi)}$ is unique by corollary \ref{cor:uniqueness_series_type_w}. Note that $\rho \in \cK_S^\alpha(\chi)$ contains infinitely many characters by lemma \ref{lem:properties_of_r_chi}.
                }
                \item{
                    The regulator $R_S^\beta(\alpha, \chi \otimes \rho)$ is non-zero for $\rho \in \cK_S^\alpha(\chi)$ (cf. definitions \ref{defn:r_chi} and \ref{defn:regulator}). Recall that, even though it relies on the implicit choice of an $n \geq n(S)$ to define the finite level maps $\varphi_n^\alpha$ over, it is independent of that choice by lemma \ref{lem:r_chi_independent_of_n}.
                }
                \item{
                    If Stark's conjecture holds, the field where the coefficients of ${F_{S, T, \chi}^{\alpha, \beta} \in \units{\cQ^c(\Gamma_\chi)}}$ lie can be narrowed down to $\QQ_{p, \chi} = \QQ_p(\chi(h) \colon h \in H)$ with some principal unit adjoined to it, and the same is true if one assumes the equivariant Main Conjecture formulated below instead. Both facts are explained at the end of section \ref{sec:starks_conjecture_and_the_choice_of_beta}.
                }
                \item{
                    The topological generator $\gamma_K$ of $\Gamma_K$ is not listed as a parameter of the conjecture, since both $\cQ^c(\Gamma_\chi)$ and the twisted evaluation maps $ev_{\gamma_{\chi \otimes \rho}}$ are independent of its choice by lemma \ref{lem:twisted_evaluation_maps_independent_gamma_K}.
                }
                \item{
                    The specific nature of $\Gamma_\chi$ does not play a role in the Interpolation Conjecture: one could have considered $\units{\cQ^c(\tilde{\Gamma})}$ for any abstract $\tilde{\Gamma} = \overline{\ideal{\tilde{\gamma}}} \iso \ZZ_p$ and replaced the twisted evaluation maps by evaluation at some specific values of $\tilde{T} = \tilde{\gamma} - 1$ as in lemma \ref{lem:twisted_evaluation_maps}. It is however convenient to expressly consider $\units{\cQ^c(\Gamma_\chi)}$ with a view towards the second part of the Main Conjecture, where a \textit{zeta element} is expected to be mapped to the $F_{S, T, \chi}^{\alpha, \beta}$ under certain maps $K_1(\cQ(\cG)) \to \units{\cQ^c(\Gamma_\chi)}$ where the meaning of the latter becomes relevant.
                }
                \item{
                    The left-hand side of \eqref{eq:ic} always lies in $\QQ_p^c \cup \set{\infty}$ (see for instance remark \ref{rem:evaluation_of_series_quotients}). Therefore, if the Interpolation Conjecture holds, the regulated special values (which are $p$-adic complex by definition) must be in fact algebraic over $\QQ_p$.
                }
            \end{enumerate}
        \end{rem}

        The interpolation property \eqref{eq:ic} of the predicted series quotient $F_{S, T, \chi}^{\alpha, \beta}$ is limited to $\rho$-twists of $\chi$ which do not factor through too low a layer, namely $n(S, \alpha)$ - and even then, only to almost all. The aim of this part of the conjecture is not to interpolate values of \textit{all} $L$-functions of intermediate extensions of $L_\infty/K$, but rather \textit{enough} values that the interpolating element is uniquely determined by them, since the essence of the Main Conjecture is the fact that these elements should be algebraic in nature. Whereas the restriction to $\rho \in \cK_S^\alpha(\chi)$ is necessary for the non-vanishing of the regulator, the reason interpolation is conjectured at \textit{almost all} such $\rho$, rather than \textit{all}, is that this will enable us to prove functoriality (that is, good behaviour under change of $K$ or $L_\infty$) in section \ref{sec:functoriality}.

        By virtue of our definition of the twisted evaluation maps, the Interpolation Conjecture does not refer to a single $\chi \in \Irr_p(\cG)$, but rather to its entire $W$-equivalence class:
        \begin{prop}
        \label{prop:ic_rho_twist}
            Setting \ref{sett:formulation}. Let $\widetilde{\rho}$ be a type-$W$ character of $\cG$ and $\chi \in \Irr_p(\cG)$. Set $\widetilde{\chi} = \chi \otimes \widetilde{\rho}$. Then \hyperref[conje:ic]{IC($L_\infty/K, \chi, L, S, T, \alpha, \beta$)} holds if and only if \hyperref[conje:ic]{IC($L_\infty/K, \widetilde{\chi}, L, S, T, \alpha, \beta$)} does, in which case the predicted elements $F_{S, T, \chi}^{\alpha, \beta}$ and $F_{S, T, \tilde{\chi}}^{\alpha, \beta}$ coincide.
        \end{prop}

        \begin{proof}
            We first point out that the parameters $L, S, T, \alpha$ and $\beta$ are independent of $\chi$, so one can indeed formulate the two Interpolation Conjectures by changing the character alone. Recall that $\cQ^c(\Gamma_\chi) = Z(\cQ^c(\cG)e_\chi) = \cQ^c(\Gamma_{\widetilde{\chi}})$ by propositions \ref{prop:rw_properties_q} iv) and \ref{prop:structure_gamma_chi} ii) - this was also shown explicitly in the proof of lemma \ref{lem:twisted_evaluation_maps}.

            Let $\rho \in \cK_S^\alpha(\chi)$. The characters $\chi \otimes \rho$ and $\widetilde{\chi} \otimes (\widetilde{\rho}^{-1} \otimes \rho)$ coincide, and therefore any $F_{S, T, \chi}^{\alpha, \beta} \in  \units{\cQ^c(\Gamma_\chi)}$ satisfying the interpolation property \eqref{eq:ic} for $\chi \otimes \rho$ will also do so for $\widetilde{\chi} \otimes (\widetilde{\rho}^{-1} \otimes \rho)$:
            \begin{align*}
                ev_{\gamma_{\widetilde{\chi} \otimes (\widetilde{\rho}^{-1} \otimes \rho)}}(F_{S, T, \chi}^{\alpha, \beta})
                & = ev_{\gamma_{\chi \otimes \rho}}(F_{S, T, \chi}^{\alpha, \beta}) \\
                & = \frac{\beta^{-1}(L_{K, S, T}^\ast(\beta (\check{\chi} \otimes \rho^{-1}), 0))}{R_S^\beta(\alpha, \chi \otimes \rho)} \\
                & = \frac{\beta^{-1}(L_{K, S, T}^\ast(\beta ((\widetilde{\chi})^{\check{}} \otimes (\widetilde{\rho}^{-1} \otimes \rho)^{-1}), 0))}{R_S^\beta(\alpha, \widetilde{\chi} \otimes (\widetilde{\rho}^{-1} \otimes \rho))}
            \end{align*}
            (cf. remark \ref{rem:properties_of_ev} i) for the first equality).

            Since the map $\widetilde{\rho}^{-1} \otimes - \colon \cK_S^\alpha(\chi) \to \cK_S^\alpha(\widetilde{\chi})$ is a bijection by lemma \ref{lem:properties_of_r_chi} ii), interpolation for almost all characters in the former implies that for almost all characters in the latter. This shows one implication in the statement, with the converse following from an analogous argument or by regarding $\chi$ as $\widetilde{\chi} \otimes \widetilde{\rho}^{-1}$.
        \end{proof}

        This proposition aligns perfectly with diagram \eqref{eq:k1_reduced_norm_chi_parts}, which shows that our algebraic objects of interest decompose into $\chi$-parts \textit{up to $W$-equivalence}. Indeed, the last step before the formulation of the equivariant Main Conjecture is to define maps from $K_1(\cQ(\cG))$ to the $\units{\cQ^c(\Gamma_\chi)}$:

        \begin{defn}
        \label{defn:k1_to_gamma_chi}
            Setting \ref{sett:construction}. For $\chi \in \Irr_p(\cG)$, we define the group homomorphism $\psi_\chi$ as the composition
            \[
                \psi_\chi \colon K_1(\cQ(\cG)) \xrightarrow{\nr} \units{Z(\cQ(\cG))} \ia \units{Z(\cQ^c(\cG))} \sa \units{Z(\cQ^c(\cG) e_\chi)} = \units{\cQ^c(\Gamma_\chi)}.
            \]
            where the first arrow is the reduced norm introduced in section \ref{sec:algebraic-k-theory} and the equality is proposition \ref{prop:structure_gamma_chi} ii).\qedef
        \end{defn}

        It should be noted that, even though ring homomorphisms have poor behaviour with respect to taking centres in general, the second and third arrow defining $\psi_\chi$ are well defined because they come from extension of scalars (from one field to another) and multiplication by a central element, respectively. An immediate consequence of the above definition is that $\psi_\chi$ and $\psi_{\chi'}$ coincide whenever $\psi_\chi \sim_W \psi_{\chi'}$. These maps are closely related to the $j_\chi$ of Ritter and Weiss (cf. proposition \ref{prop:structure_gamma_chi} iii)), which plays a prominent role in their paper \cite{rwii}. Namely, let us extend scalars by setting $j_\chi^c = \QQ_p^c \otimes_E j_\chi^E$. Then $j_\chi^c$ and $\psi_\chi$ fit into the commutative diagram

        \begin{equation}
        \label{eq:psi_and_j}
            \begin{tikzcd}
                K_1(\cQ(\cG)) \arrow[d, "\nr"] \arrow[rr, "\psi_\chi"]              &                                                 & \units{\cQ^c(\Gamma_\chi)} \arrow[dd, "\gamma_\chi \mapsto \gamma_K^{w_\chi}"] \\
                \units{Z(\cQ(\cG))} \arrow[d, hook]                                 & \units{Z(\cQ^c(\cG) e_\chi)} \arrow[ru, "\sim"] &                                                                                \\
                \units{Z(\cQ^c(\cG))} \arrow[ru, two heads] \arrow[rr, "j_\chi^c"'] &                                                 & \units{\cQ^c(\Gamma_K)}
            \end{tikzcd}
        \end{equation}
        Note that the map $\gamma_\chi \mapsto \gamma_K^{w_\chi}$ is injective and commutes with the evaluation maps (see for instance the proof of lemma \ref{lem:twisted_evaluation_maps_independent_gamma_K}).

        We now state the \textbf{equivariant Main Conjecture}\index{equivariant Main Conjecture}\index{Main Conjecture|see{equivariant Main Conjecture}}, which relates the refined Euler characteristic \eqref{eq:definition_rec} to the series quotients predicted in the Interpolation Conjecture by means of the localisation sequence of $K$-theory \eqref{eq:exact_sequence_k-theory}:

        \begin{conje*}[eMC($L_\infty/K, L, S, T, \alpha, \beta$)]
        \phantomsection
        \label{conje:emc}
            Setting \ref{sett:formulation}.

            Conjecture \hyperref[conje:ic]{IC($L_\infty/K, \chi, L, S, T, \alpha, \beta$)} holds for all $\chi \in \Irr_p(\cG)$ and there exists a $\zeta_{S, T}^{\alpha, \beta} \in K_1(\cQ(\cG))$ such that $\partial(\zeta_{S, T}^{\alpha, \beta}) = - \chi_{\Lambda(\cG), \cQ(\cG)}(\cC_{S, T}\q, t^\alpha) \in K_0(\Lambda(\cG), \cQ(\cG))$ and
            \[
                \psi_\chi(\zeta_{S, T}^{\alpha, \beta}) = F_{S, T, \chi}^{\alpha, \beta}
            \]
            for all $\chi$ as above, where $F_{S, T, \chi}^{\alpha, \beta} \in \units{\cQ^c(\Gamma_\chi)}$ is the unique series quotient satisfying the Interpolation Conjecture at $\chi$.
        \end{conje*}

        A natural refinement is the \textbf{equivariant Main Conjecture with uniqueness}\index{equivariant Main Conjecture!with uniqueness}:

        \begin{conje*}[eMC\textsuperscript{u}($L_\infty/K, L, S, T, \alpha, \beta$)]
        \phantomsection
        \label{conje:emcu}
            Setting \ref{sett:formulation}.

            Conjecture \hyperref[conje:ic]{IC($L_\infty/K, \chi, L, S, T, \alpha, \beta$)} holds for all $\chi \in \Irr_p(\cG)$ and there exists a \emph{unique} element ${\zeta_{S, T}^{\alpha, \beta} \in K_1(\cQ(\cG))}$ such that
            \[
                \psi_\chi(\zeta_{S, T}^{\alpha, \beta}) = F_{S, T, \chi}^{\alpha, \beta}
            \]
            for all $\chi$ as above, where $F_{S, T, \chi}^{\alpha, \beta} \in \units{\cQ^c(\Gamma_\chi)}$ is the unique series quotient satisfying the Interpolation Conjecture at $\chi$. Furthermore, one has $\partial(\zeta_{S, T}^{\alpha, \beta}) = - \chi_{\Lambda(\cG), \cQ(\cG)}(\cC_{S, T}\q, t^\alpha) \in K_0(\Lambda(\cG), \cQ(\cG))$.
        \end{conje*}

        \begin{rem}
        \phantomsection
        \label{rem:emcu}
            \begin{enumerate}[i)]
                \item{
                    \hyperref[conje:emcu]{eMC\textsuperscript{u}($L_\infty/K, L, S, T, \alpha, \beta$)} can be reformulated as: \hyperref[conje:emc]{eMC($L_\infty/K, L, S, T, \alpha, \beta$)} holds and the reduced norm ${\nr \colon K_1(\cQ(\cG)) \to \units{Z(\cQ(\cG))}}$ is injective.

                    The object $SK_1(R) = \ker(\nr \colon K_1(R) \to \units{Z(R)})$ for a general semisimple Artinian ring $R$ has considerable algebraic interest in itself and has been the subject of research for the last few decades. In our case of interest, Merkurjev reports (cf. \cite{merkurjev} p. 250) that Suslin conjectured the vanishing of $SK_1(R)$ for a class of rings which $\cQ(\cG)$ is an instance of (see \cite{rwii} p. 565 for more details). If this is the case, the equivariant Main Conjectures with and without uniqueness are equivalent.
                }
                \item{
                    The connecting homomorphism $\partial$ is known to be surjective in our setting by work of Witte (cf. \cite{witte} corollary 3.8), and thus $- \chi_{\Lambda(\cG), \cQ(\cG)}(\cC_{S, T}\q, t^\alpha)$ always has a preimage under $\partial$.
                }
            \end{enumerate}
        \end{rem}


\newpage
\chapter{Basic properties}
\label{chap:properties_of_the_main_conjecture}

    The construction presented in chapter \ref{chap:formulation_of_the_main_conjecture}, culminating in the formulation of the Interpolation Conjecture and equivariant Main Conjecture, gives rise to some natural questions. Arguably the most immediate one is that of the dependence on the many parameters which have made an appearance so far. In one form or another, that is the subject of most results in this chapter, which has three main aims distributed across its three sections.

    The first one is to study the dependence of the conjectures on the choice of $\beta \colon \CC_p \xrightarrow{\sim} \CC$. As already mentioned, this question is closely linked to Stark's conjecture on leading coefficients of Artin $L$-functions. This conjecture and its various refinements have been extensively studied by many authors - among them Tate in \cite{tate}, who shaped its modern formulation. It will be necessary for us to incorporate the set $T$ into conjecture, which a priori only features one set of places $S$. After showing the two variants to be equivalent, we shall prove that our Main Conjecture is independent of $\beta$ if Stark's holds, and unconditionally independent of $\beta$ as long as its restriction to $\QQ_p$ is fixed. Another consequence of Stark's conjecture will be explored, namely where the coefficients of the series quotients predicted by the Interpolation Conjecture should be expected to lie. This provides some measure of support for, and is also implied by, the Main Conjecture.

    The second aim is to establish that, once $L_\infty/K$ has been fixed, the Main Conjecture is unconditionally independent of the choices of all parameters other than $\beta$. Section \ref{sec:independence_of_the_choice_of_parameters} (as well as \ref{sec:functoriality}) has been exceptionally divided into subsections for the convenience of the reader. The following theorem encapsulates the main results in the former:
    \begin{thm}
    \label{thm:independence_of_parameters}
        Setting \ref{sett:formulation}. Let $\tilde{L}, \tilde{S}, \tilde{T}$ and $\tilde{\alpha}$ be another valid set of choices of the corresponding parameters in the setting. Then:
        \begin{enumerate}[i)]
            \item{
                For any $\chi \in \Irr_p(\cG)$, \hyperref[conje:ic]{IC($L_\infty/K, \chi, L, S, T, \alpha, \beta$)} holds if and only if \hyperref[conje:ic]{IC($L_\infty/K, \chi, \tilde{L}, \tilde{S}, \tilde{T}, \tilde{\alpha}, \beta$)} does.
            }
            \item{
                \hyperref[conje:emc]{eMC($L_\infty/K, L, S, T, \alpha, \beta$)} holds if and only if \hyperref[conje:emc]{eMC($L_\infty/K, \tilde{L}, \tilde{S}, \tilde{T}, \tilde{\alpha}, \beta$)} does.
            }
            \item{
                \hyperref[conje:emcu]{eMC\textsuperscript{u}($L_\infty/K, L, S, T, \alpha, \beta$)} holds if and only if \hyperref[conje:emcu]{eMC\textsuperscript{u}($L_\infty/K, \tilde{L}, \tilde{S}, \tilde{T}, \tilde{\alpha}, \beta$)} does.
            }
        \end{enumerate}
    \end{thm}

    As a consequence, one is justified in writing \hyperref[conje:ic]{IC($L_\infty/K, \chi, \beta$)}, \hyperref[conje:emc]{eMC($L_\infty/K, \beta$)} and \hyperref[conje:emcu]{eMC\textsuperscript{u}($L_\infty/K, \beta$)} - which we nonetheless avoid in almost all formal mathematical claims. The proof of the independence of $L$ will be completely straightforward, unlike that of $\alpha, S$ and $T$. By remark \ref{rem:emcu} i), it will suffice to prove i) and ii): once the equivariant Main Conjectures are shown to be equivalent, uniqueness is a purely $K$-theoretic question. Note that we can only argue this way because $L_\infty/K$, and hence $\Lambda(\cG)$ and $\cQ(\cG)$, are the same for all conjectures in the theorem.

    The third and last goal of this chapter is to study the functoriality of the conjectures, that is, how modifying $L_\infty$ and $K$ affects their validity. In the case of a finite extension $L_\infty'/L_\infty$ such that $L_\infty/K'$ is Galois, we will prove that all three conjectures for $L_\infty'/K$ imply their counterparts for $L_\infty/K$. If one instead replaces $K$ by a finite extension $K'$ contained in $L_\infty$, it is necessary to assume the Interpolation Conjecture for several characters of $\Gal(L_\infty/K)$ in order to conclude that for one of $\Gal(L_\infty/K')$. As for the equivariant Main Conjecture, only the version without uniqueness will be treated. As an immediate consequence, we deduce that the general case of those two conjectures for $p$ odd follows from the case $K = \QQ$. We end the chapter by addressing the converse problem: whether one can deduce the conjectures for $L_\infty/K$ from their counterparts for a suitable family of subextensions. Two families $\cS$ and $\cE_p$ will be described where this is the case, consisting of the subextensions determined by certain elementary subgroups and certain $p$-elementary subquotients of $\cG$, respectively. In both instances, the discussion will be limited to the Interpolation Conjecture and the equivariant Main Conjecture without uniqueness. For the latter, we essentially use the argument from \cite{johnston_nickel} section 10 together with some additional verifications which are not necessary in the totally real case.

    The justification for many specific features of the formulation of the Main Conjecture will become apparent in the following sections, such as the duality between the characters in the numerator and denominator of the regulated $L$-value \eqref{eq:ic}, the \textit{almost all} quantifier in the Interpolation Conjecture, and the element $\zeta_{S, T}^{\alpha, \beta}$ being mapped to the \textit{inverse} of the refined Euler characteristic $\chi_{\Lambda(\cG), \cQ(\cG)}(\cC_{S, T}\q, t^\alpha)$.

    \section{\texorpdfstring{Stark's conjecture and the choice of $\beta$}{Stark's conjecture and the choice of beta}}
    \label{sec:starks_conjecture_and_the_choice_of_beta}

        Stark's conjecture, as formulated by Tate in \cite{tate}, asserts a certain Galois-invariance property of the quotient of the leading coefficient of an Artin $L$-series by a regulator constructed using the classical Dirichlet map. The Stark-Tate regulator from section \ref{sec:regulators} is a particular instance of such objects. The classical formulation only considers $S$-depleted Artin $L$-functions, so the purpose of this section is to introduce a $T$-smoothed version of it - that is, a conjecture for $(S, T)$-modified Artin $L$-functions - which will be equivalent to the original $S$-version; and to study its implications for the Main Conjecture.

        We work in the generality in which Artin $L$-functions were introduced in section \ref{sec:artin_l-series}, which is essentially the same as in our applications save for the fact that $S$ is not required to contain all ramified places here. Tate defines the following regulators: suppose given a Galois extension of number fields $L/K$ with Galois group $G$ and two finite sets of places $S$ and $T$ of $K$ such that $S \supseteq S_\infty$ and $S \cap T = \varnothing$. Consider an embedding\footnote{It is well known that $\CC_p \iso \CC$ as abstract fields by a transcendence-degree argument. However, this isomorphism is far from unique.} $\beta \colon \CC_p \ia \CC$ and a $\ZZ[G]$-homomorphism
        \[
            f \colon \cX_{L, S}^\ZZ \to \CC_p \otimes \units{\cO_{L, S}}
        \]
        in the notation of section \ref{sec:artin_l-series}. These maps induce a $\CC[G]$ homomorphism
        \begin{equation}
        \label{eq:f_beta}
            f^\beta \colon \CC \otimes \cX_{L, S}^\ZZ \to \CC \otimes \units{\cO_{L, S}}
        \end{equation}
        by extending scalars of $\cX_{L, S}^\ZZ \xrightarrow{f} \CC_p \otimes \units{\cO_{L, S}} \xrightarrow{\beta \otimes \Id} \CC \otimes \units{\cO_{L, S}}$. Finally, let ${\lambda_{L, S}^\CC \colon \CC \otimes \units{\cO_{L, S}} \to \CC \otimes \cX_{L, S}^\ZZ}$ be the extension of scalars of the Dirichlet regulator map \eqref{eq:dirichlet_regulator_map_real} from $\RR$ to $\CC$.

        We are interested in the character-wise determinant of $\lambda_{L, S}^\CC \circ f^\beta$. Specifically, if $\chi \in \CChars{\CC_p}{G}$ and $V_\chi$ is a $\CC_p[G]$-module which affords $\chi$, then $V_{\beta \chi} = \CC \otimes_\beta V_\chi$ is a $\CC[G]$-module which affords $\beta \chi$ and one may therefore define\footnote{Tate uses the notation $R(\chi^\beta, f^\beta)$ for the determinant on $\check{\chi}$-parts instead (cf. \cite{tate} p.28).}
        \begin{equation}
        \label{eq:regulator_tate}
            R_S(\beta \chi, f^\beta) = \det(\lambda_{L, S}^\CC \circ f^\beta \mid \Hom_{\CC[G]}(V_{\beta\chi}, \CC \otimes \cX_{L, S}^\ZZ)) \in \CC
        \end{equation}
        in the notation of \eqref{eq:notation_endomorphism_acting_on_hom}. In other words, $R_S(\beta \chi, f^\beta)$ is the determinant of the $\CC$-linear endomorphism $(\lambda_{L, S}^\CC \circ f^\beta)_\ast$ of $\Hom_{\CC[G]}(V_{\beta\chi}, \CC \otimes \cX_{L, S}^\ZZ)$ given by postcomposition with $\lambda_{L, S}^\CC \circ f^\beta$.

        Our $T$-smoothed version of Stark's conjecture\index{Stark's conjecture} is the following assertion:
        \begin{conje*}[Stark\textsuperscript{T}($L/K, \chi, f, S, T$)]
        \phantomsection
        \label{conje:stark}
            Consider:
            \begin{itemize}
                \item{
                    A Galois extension of number fields $L/K$.
                }
                \item{
                    A character $\chi \in \CChars{\CC_p}{\Gal(L/K)}$.
                }
                \item{
                    Two finite sets of places $S$ and $T$ of $K$ such that $S \supseteq S_\infty$ and $S \cap T = \varnothing$.
                }
                \item{
                    A homomorphism of $\ZZ[G]$-modules $f \colon \cX_{L, S}^\ZZ \to \CC_p \otimes \units{\cO_{L, S}}$.
                }
            \end{itemize}

            There exists an $\cL_{S, T}(\chi, f) \in \CC_p$ such that, for any ring homomorphism $\beta \colon \CC_p \ia \CC$, one has
            \begin{equation}
            \label{eq:starks_conjecture}
                \beta(\cL_{S, T}(\chi, f)) = \frac{R_S(\beta \check{\chi}, f^\beta)}{L_{K, S, T}^\ast(\beta \chi, 0)},
            \end{equation}
        \end{conje*}

        The original conjecture is the above claim for $T = \varnothing$. However, our $T$-smoothed version is a \textit{reformulation}, rather than a strengthening:
        \begin{prop}
        \label{prop:stark_independent_of_t}
            \hyperref[conje:stark]{Stark\textsuperscript{T}($L/K, \chi, f, S, T$)} holds if and only if Stark's conjecture as formulated\footnote{\label{foot:stark_coefficient_field}It should be pointed out that Tate formulates this conjecture for any field $E$ which can be embedded into $\CC$, whereas we specialise to the case $E = \CC_p$. However, the choice of $E$ is irrelevant (cf. section I\S6 in the reference).} in \cite{tate} 5.4 does for the same set of parameters minus $T$.
        \end{prop}

        \begin{proof}
            Conjecture \cite{tate} 5.4 asserts the existence of some $\cL_S(\chi, f) \in \CC_p$ such that
            \begin{equation}
            \label{eq:starks_conjecture_tate}
                \beta(\cL_S(\chi, f)) = \frac{R_S(\beta \check{\chi}, f^\beta)}{L_{K, S, \varnothing}^\ast(\beta \chi, 0)}
            \end{equation}
            for any $\beta \colon \CC_p \ia \CC$. The $S$-depleted and $(S, T)$-modified $L$-functions differ by the $\delta$-factors (cf. \eqref{eq:local_delta-factors}) at $T$
            \begin{equation}
            \label{eq:l_functions_t-smoothed}
                L_{K, S, T}(\beta \chi, s) = L_{K, S, \varnothing}(\beta \chi, s) \cdot \prod_{v \in T} \delta_v(\beta\chi, s)
            \end{equation}
            in the half-plane $\re(s) > 1$. This relation extends to all of $\CC$, since Artin $L$-functions are by definition the meromorphic continuation of their Euler-product expression and two meromorphic functions which coincide on an open subset of $\CC$ (where they are analytic) are identical. Recall that $\delta$-factors are entire functions by \eqref{eq:euler_factors_analytic}.

            We now show that $\delta$-factors do not vanish at $s = 0$. They obey the same formalism as $L$-factors (cf. lemma \ref{lem:properties_of_L-functions}), since $\delta_v(\beta\chi, s) = L_v(\beta\chi, 1 - s)^{-1}$. Thus, Brauer's induction theorem \ref{thm:brauer_induction} allows one to express $\delta_v(\beta\chi, s)$ as a (finite) product of factors for the form $\delta_w(\lambda, s)$, where $\lambda$ is a linear \textit{complex} character of some $U \leq \Gal(L/K)$ and $w$ is a prolongation of $v$ to $L^U$. In the expression
            \begin{equation}
            \label{eq:delta_factors_linear_case}
                \delta_w(\lambda, s) = \det(1 - \fN(w)^{1 - s} \varphi_w \mid V_\lambda^{I_w}),
            \end{equation}
            $V_\lambda^{I_w}$ is either trivial (if $\lambda$ is not trivial on $I_w$), in which case the above determinant takes the value 1 at all $s$ by definition; or a 1-dimensional $\CC$-vector space (if $\lambda$ is trivial on $I_w$) upon which $\varphi_w$ acts as $\lambda(\varphi_w) \in \units{\CC}$. In this latter case,
            \[
                \delta_w(\lambda, 0) = 1 - \fN(w) \lambda(\varphi_w)
            \]
            is different from 0, as $\fN(w) > 1$ and $\lambda(\varphi_w)$ is a complex root of unity. Since the original $\delta_v(\beta\chi, 0)$ is a product of such linear factors, it does not vanish either.

            Equation \eqref{eq:l_functions_t-smoothed} now yields the relation
            \[
                L_{K, S, T}^\ast(\beta \chi, 0) = L_{K, S, \varnothing}^\ast(\beta \chi, 0) \cdot \prod_{v \in T} \delta_v(\beta\chi, 0).
            \]
            on leading coefficients. It therefore remains to show that $\delta$-factors behave well with respect to $\beta$. For $v \in T$, consider the \textit{$p$-adic $\delta$-factor}
            \[
                \delta_v^{\CC_p}(\chi, 0) = \det(1 - \fN(v) \varphi_w \mid V_\chi^{I_w}) \in \CC_p.
            \]
            Then
            \begin{equation}
            \label{eq:delta-factors_galois-equivariant}
                \beta(\delta_v^{\CC_p}(\chi, 0)) = \det(1 - \fN(v) \beta(\varphi_w \mid V_\chi^{I_w})),
            \end{equation}
            where $\beta(\varphi_w \mid V_\chi^{I_w})$ denotes $\beta$ applied entry-wise to the matrix describing the action of $\varphi_w$ on $V_\chi^{I_w}$. This coincides with the matrix describing the action of $\varphi_w$ on $V_{\beta\chi}^{I_w} = (\CC \otimes_{\beta} V_\chi)^{I_w} = \CC \otimes_{\beta} (V_\chi^{I_w})$ by definition, and hence
            \[
                \beta(\delta_v^{\CC_p}(\chi, 0)) = \det(1 - \fN(v) \varphi_w \mid V_{\beta\chi}^{I_w}) = \delta_v(\beta\chi, 0).
            \]
            In particular, $\delta_v^{\CC_p}(\chi, 0) \neq 0$. Therefore, assuming the existence of $\cL_S(\chi, f) \in \CC_p$ as in the beginning of the proof and setting $\cL_{S, T}(\chi, f) = \cL_S(\chi, f) \prod_{v \in T} \delta_v^{\CC_p}(\chi, 0)^{-1} \in \CC_p$, one has
            \begin{align*}
                \beta(\cL_{S, T}(\chi, f)) & = \frac{R_S(\beta \check{\chi}, f^\beta)}{L_{K, S, \varnothing}^\ast(\beta \chi, 0)} \cdot \frac{1}{\prod_{v \in T} \beta(\delta_v^{\CC_p}(\chi, 0))} \\
                & = \frac{R_S(\beta \check{\chi}, f^\beta)}{L_{K, S, \varnothing}^\ast(\beta \chi, 0) \cdot \prod_{v \in T} \delta_v(\beta\chi, 0)} \\
                & = \frac{R_S(\beta \check{\chi}, f^\beta)}{L_{K, S, T}^\ast(\beta \chi, 0)}
            \end{align*}
            for any $\beta \colon \CC_p \to \CC$, as claimed by Stark\textsuperscript{T}($L/K, \chi, f, S, T$).

            The converse is proved analogously by multiplying by $\prod_{v \in T} \delta_v^{\CC_p}(\chi, 0)$ at the end, rather than dividing by it.
        \end{proof}

        In view of the proposition, the known properties of Stark's conjecture carry over to the smoothed version. \cite{tate} constitutes a systematic study of these - here we limit ourselves to citing a few basic results from sections I\S6 and I\S7:
        \begin{enumerate}[i)]
            \item{
                If Stark's conjecture (or its $T$-smoothed version) holds for all characters of $\Gal(K/\QQ)$ for any finite Galois extension $K/\QQ$ (and all $S$, $T$ and $f$), then it holds in general. The proof relies on the invariance of the conjecture under inflation and induction.
            }
            \item{
                If Stark's conjecture (or its $T$-smoothed version) holds for all linear characters of $\Gal(L/K)$ for any Galois extension $L/K$ of number fields (and all $S$, $T$ and $f$), then it holds in general. The proof relies on Brauer's induction theorem \ref{thm:brauer_induction} and the behaviour of the conjecture with respect to character addition and induction.
            }
            \item{
                If Stark's conjecture (or its $T$-smoothed version) holds for a choice of $L/K$, $\chi$ and $S$ and an $f_0 \colon \cX_{L, S}^\ZZ \to \CC_p \otimes \units{\cO_{L, S}}$ such that $\CC_p \otimes f_0 \colon \CC_p \otimes \cX_{L, S}^\ZZ \xrightarrow{\sim} \CC_p \otimes \units{\cO_{L, S}}$ is an isomorphism, then it holds for any $f \colon \cX_{L, S}^\ZZ \to \CC_p \otimes \units{\cO_{L, S}}$ and the same remaining parameters. Note that the introduction of $\CC_p$ as the coefficient field is unnecessary here - one can work over $\CC$ instead (see also the footnote\textsuperscript{\ref{foot:stark_coefficient_field}} on page \pageref{foot:stark_coefficient_field}).
            }
            \item{
                Stark's conjecture (and its $T$-smoothed version) is independent of the choice of the set $S \supseteq S_\infty$. The proof is similar to that of proposition \ref{prop:stark_independent_of_t}, but one needs to treat the case where the local $L$-factors vanish (which, as we have seen, cannot occur to the $\delta$-factors).
            }
        \end{enumerate}

        Most cases of Stark's conjecture are unknown. Exceptions to this are:
        \begin{itemize}
            \item{
                If $L/K$ is a Galois extension of number fields with Galois group $G$ and $\chi$ is a character of $G$ such that $\chi(\sigma) \in \QQ$ for all $\sigma \in G$, then Stark's conjecture is known for all choices of $S$ and $f$. This is due to Tate and Artin (cf. \cite{tate} chapter II) and its proof can be reduced to the case of permutation characters, which is settled by virtue of the analytic class number formula. Under the above assumption on $\chi$, the element $\cL_{S, \emptyset}(\chi, f)$ from equation \eqref{eq:starks_conjecture} lies in $\QQ$.
            }
            \item{
                If $L/K$ is a finite abelian extension and $K$ is either $\QQ$ or an imaginary quadratic field, then Stark's conjecture is known for all choices of $\chi$, $S$ and $f$. These cases date back to Stark.
            }
        \end{itemize}

        We now study the implications of Stark's conjecture for the Interpolation Conjecture, which stem from the following fact:
        \begin{lem}
        \label{lem:stark_invariance_interpolated}
            Setting \ref{sett:formulation}. Let $\chi \in \Irr_p(\cG)$ satisfy \ref{manualcond:kc} and choose $n \in \NN$ large enough that $\Gp{n} \subseteq \ker(\chi)$. Define $f$ as the composition
            \[
                f \colon \cX_{L_n, S}^\ZZ \ia \cX_{L_n, S} \xrightarrow{\varphi_n^\alpha} \ZZ_p \otimes \units{\cO_{L_n, S, T}} \ia \CC_p \otimes \units{\cO_{L_n, S, T}} = \CC_p \otimes \units{\cO_{L_n, S}}
            \]
            with $\varphi_n^\alpha$ as in definition \ref{defn:finite_level_map}. If \hyperref[conje:stark]{Stark\textsuperscript{T}($L_n/K, \check{\chi}, f, S, T$)} holds, then
            \begin{equation}
            \label{eq:interpolation_conjecture_expression}
                 \frac{\beta^{-1}(L_{K, S, T}^\ast(\beta \check{\chi}, 0))}{R_S^\beta(\alpha, \chi)} \in \units{\CC_p}
            \end{equation}
            is independent of the choice of $\beta \colon \CC_p \xrightarrow{\sim} \CC$ and lies in $\QQ_p(\chi) = \QQ_p(\chi(g): g \in \cG)$. In particular, it is algebraic over $\QQ_p$.
        \end{lem}

        \begin{proof}
            Let $\beta$ be as in the statement. Since $\cX_{L_n, S} = \ZZ_p \otimes \cX_{L_n, S}^\ZZ$, an immediate computation shows that the map $f^\beta$ from \eqref{eq:f_beta} coincides with $\CC \otimes_{\beta} (\CC_p \otimes_{\ZZ_p} \varphi_n^\alpha)$. As for the Dirichlet regulator, we have
            \[
                \CC \otimes_\beta \lambda_{n, S}^\beta = \lambda_{L_n, S}^\CC
            \]
            by definition \ref{defn:p-adic_dirichlet_regulator_map}, where $\lambda_{L_n, S}^\CC$ is again the scalar extension to $\CC$ of the classical map.

            Let $M$ be the matrix of the $\CC_p$-linear action of $\lambda_{n, S}^\beta \circ (\CC_p \otimes_{\ZZ_p} \varphi_n^\alpha)$ on ${\Hom_{\CC_p[\cG_n]}(V_\chi, \CC_p \otimes_{\ZZ_p} \cX_{L_n, S})}$ (by postcomposition) with respect to an arbitrary basis $\set{f_1, \ldots, f_n}$. Then $\set{1 \otimes f_1, \ldots, 1 \otimes f_n}$ is a $\CC$-basis of
            \[
                \CC \otimes_\beta \Hom_{\CC_p[\cG_n]}(V_\chi, \CC_p \otimes_{\ZZ_p} \cX_{L_n, S}) \iisoo \Hom_{\CC[\cG_n]}(V_{\beta \chi}, \CC \otimes \cX_{L_n, S}^\ZZ)
            \]
            (canonically) with respect to which $\lambda_{L_n, S}^\CC \circ f^\beta = \CC \otimes_\beta (\lambda_{n, S}^\beta \circ (\CC_p \otimes_{\ZZ_p} \varphi_n^\alpha))$ acts by $\beta(M)$. Hence
            \begin{equation}
            \label{eq:relation_regulator_tate_mine}
                R_S(\beta \chi, f^\beta) = \det(\beta(M)) = \beta(\det(M)) = \beta(R_S^\beta(\alpha, \chi)).
            \end{equation}

            Assume Stark\textsuperscript{T}($L_n/K, \check{\chi}, f, S, T$) and let $\cL_{S, T}(\check{\chi}, f) \in \CC_p$ be the element predicted therein. Then
            \[
                \frac{R_S^\beta(\alpha, \chi)}{\beta^{-1}(L_{K, S, T}^\ast(\beta \check{\chi}, 0))} = \beta^{-1}\left(\frac{R_S(\beta \chi, f^\beta)}{L_{K, S, T}^\ast(\beta \check{\chi}, 0)} \right) = \cL_{S, T}(\check{\chi}, f),
            \]
            which is independent of the choice of $\beta$. We can then take inverses since $R_S^\beta(\alpha, \chi) \neq 0$ (because $\chi$ satisfies  \ref{manualcond:kc}), which proves the first part of the lemma.

            The field $\QQ_p(\chi)$, which coincides with $\QQ_p(\check{\chi})$, is a finite (and in fact abelian) extension of $\QQ_p$. The action of the absolute Galois group $G_{\QQ_p(\chi)}$ on $\QQ_p^c$ extends uniquely to a continuous action on $\CC_p$ (recall that every element of $\CC_p$ is by definition a limit of elements of $\QQ_p^c$) and one thus has ${G_{\QQ_p(\chi)} \ia \Aut_{\QQ_p(\chi)}(\CC_p)}$. Let $\sigma \in G_{\QQ_p(\chi)}$, so in particular $\sigma \chi = \chi$ and $\sigma \check{\chi} = \check{\chi}$. Then $\tilde{\beta} = \beta \sigma^{-1}$ is another isomorphism $\CC_p \xrightarrow{\sim} \CC$, and the first part of the lemma now yields
            \begin{align*}
                \frac{\beta^{-1}(L_{K, S, T}^\ast(\beta \check{\chi}, 0))}{R_S^\beta(\alpha, \chi)} & =
                \frac{\tilde{\beta}^{-1}(L_{K, S, T}^\ast(\tilde{\beta} \check{\chi}, 0))}{R_S^{\tilde{\beta}}(\alpha, \chi)} \\
                &= \frac{\sigma(\beta^{-1}(L_{K, S, T}^\ast(\beta \check{\chi}, 0))}{R_S^{\tilde{\beta}}(\alpha, \chi)} \\
                &= \frac{\sigma(\beta^{-1}(L_{K, S, T}^\ast(\beta \check{\chi}, 0))}{\sigma(R_S^\beta(\alpha, \sigma^{-1} \chi))} =
                \sigma \left( \frac{\beta^{-1}(L_{K, S, T}^\ast(\beta \check{\chi}, 0))}{R_S^\beta(\alpha, \chi)} \right),
            \end{align*}
            where the third equality follows easily from \eqref{eq:relation_regulator_tate_mine} noting that $f^\beta = f^{\tilde{\beta}}$ (they are scalar extensions of a map $\varphi_n^\alpha$ with coefficients in $\ZZ_p$, which is fixed by $G_{\QQ_p}$). But this shows that the regulated $L$-value lies in $(\CC_p)^{G_{\QQ_p(\chi)}}$, which is known to coincide $\QQ_p(\chi)$ (see for instance \cite{tate_p-divisible} theorem 1 on p. 176).
        \end{proof}

        Requiring Stark's conjecture might seem exceedingly restrictive in that it concerns all ring homomorphisms $\beta \colon \CC_p \ia \CC$, whereas only isomorphisms $\beta$ are relevant to us. However, it is not difficult to see that both claims - for all embeddings and for isomorphisms only - are equivalent.

        The main result of this section is the independence of the Main Conjecture of $\beta$ contingent on Stark's conjecture:
        \begin{prop}
        \label{prop:independence_of_beta}
            Setting \ref{sett:formulation}. Let $\tilde{\beta} \colon \CC_p \isoa \CC$. Then:
            \begin{enumerate}[i)]
                \item{
                    Given $\chi \in \Irr_p(\cG)$, assume that \hyperref[conje:stark]{Stark\textsuperscript{T}($L_n/K, \check{\chi} \otimes \rho^{-1}, f, S, T$)} holds for almost all $\rho \in \cK_S^\alpha(\chi)$, where $n$ is large enough that $\Gp{n} \subseteq \ker(\check{\chi} \otimes \rho^{-1})$ and $f$ is constructed as in lemma \ref{lem:stark_invariance_interpolated} for each character. Then \hyperref[conje:ic]{IC($L_\infty/K, \chi, L, S, T, \alpha, \beta$)} holds if and only if \hyperref[conje:ic]{IC($L_\infty/K, \chi, L, S, T, \alpha, \tilde{\beta}$)} does.
                }
                \item{
                    Assume that \hyperref[conje:stark]{Stark\textsuperscript{T}($L_n/K, \chi, f, S, T$)} holds for almost all $\chi \in \Irr_p(\cG)$, where $n$ is large enough that $\Gp{n} \subseteq \ker(\chi)$ and $f$ is constructed as in lemma \ref{lem:stark_invariance_interpolated} for each character. Then the conjecture \hyperref[conje:emc]{eMC($L_\infty/K, L, S, T, \alpha, \beta$)} holds if and only if \hyperref[conje:emc]{eMC($L_\infty/K, L, S, T, \alpha, \tilde{\beta}$)} does.
                }
                \item{
                    Under the same assumption as in ii), the conjecture \hyperref[conje:emcu]{eMC\textsuperscript{u}($L_\infty/K, L, S, T, \alpha, \beta$)} holds if and only if \hyperref[conje:emcu]{eMC\textsuperscript{u}($L_\infty/K, L, S, T, \alpha, \tilde{\beta}$)} does.
                }
            \end{enumerate}
        \end{prop}

        \begin{proof}
            The only aspect of the Interpolation Conjecture which references $\beta$ is the regulated $L$-value $\beta^{-1}(L_{K, S, T}^\ast(\beta (\check{\chi} \otimes \rho^{-1}), 0))/R_S^\beta(\alpha, \chi \otimes \rho)$. Under the hypothesis in i), it is independent of $\beta$ for almost all $\rho \in \cK_S^\alpha(\chi)$ by lemma \ref{lem:stark_invariance_interpolated}. Hence, an $F_{S, T, \chi}^{\alpha, \beta} \in \units{\cQ^c(\Gamma_\chi)}$ satisfying the interpolation property \eqref{eq:ic} for almost all such $\rho$ will also satisfy the corresponding property for IC($L_\infty/K, \chi, L, S, T, \alpha, \tilde{\beta}$).

            As for ii) and iii), the only element of the equivariant Main Conjecture (with or without uniqueness) involving $\beta$ is $F_{S, T, \chi}^{\alpha, \beta}$, which has just been shown to be independent of it if Stark's conjecture holds for enough characters.
        \end{proof}

        We now illustrate a way in which Stark's conjecture supports the equivariant Main Conjecture by bringing attention to an issue we have tacitly ignored so far: while the interpolating series quotient $F_{S, T, \chi}^{\alpha, \beta} \in \units{\cQ^c(\Gamma_\chi)}$ has coefficients in some $p$-adic field, conjecture eMC($L_\infty/K, L, S, T, \alpha, \beta$) claims it is the image of a zeta element $\zeta_{S, T}^{\alpha, \beta} \in K_1(\cQ(\cG))$ - which seemingly has coefficients in $\QQ_p$.

        In order to understand this discrepancy, we endow $\cQ^c(\cG) = \QQ_p^c \otimes_{\QQ_p} \cQ(\cG)$ with the natural $G_{\QQ_p}$-action on the first component. It is known (cf. for instance \cite{nickel_conductor} equation (4)) that, given $\sigma \in G_{\QQ_p}$ and an irreducible Artin character $\chi$ of $\cG$, one has $\sigma(e_\chi) = e_{\sigma\chi}$ and $\sigma(\gamma_\chi) = \gamma_{\sigma\chi}$ in the notation of section \ref{sec:evaluation_maps}. In particular,
        \begin{equation}
        \label{eq:galois_automorphisms_t}
            \sigma(T_\chi) =  \sigma(\gamma_\chi - 1) = \sigma(\gamma_\chi) - 1 = \gamma_{\sigma\chi} - 1 = T_{\sigma\chi}
        \end{equation}
        under the usual identification. These relations imply that $G_{\QQ_p}$ not only acts on each individual component of
        \[
            Z(\cQ^c(\cG)) \iisoo \prod_{\chi / \sim_W} \cQ^c(\Gamma_\chi)
        \]
        (cf. propositions \ref{prop:rw_properties_q} and \ref{prop:structure_gamma_chi}), but it might also permute them. This translates into the following relation between interpolating series quotients associated to Galois-conjugated characters:
        \begin{prop}
        \label{prop:stark_galois_invariance_f}
            Setting \ref{sett:formulation}. Let $\sigma \in G_{\QQ_p}$ and $\chi \in \Irr_p(\cG)$. Suppose $F_{S, T, \chi}^{\alpha, \beta} \in \units{\cQ^c(\Gamma_\chi)}$ satisfies \hyperref[conje:ic]{IC($L_\infty/K, \chi, L, S, T, \alpha, \beta$)}. Assume furthermore that \hyperref[conje:stark]{Stark\textsuperscript{T}($L_n/K, \check{\chi} \otimes \rho^{-1}, f, S, T$)} holds for almost all $\rho \in \cK_S^\alpha(\chi)$ in the notation of proposition \ref{prop:independence_of_beta} i). Then $\sigma(F_{S, T, \chi}^{\alpha, \beta}) \in  \units{\cQ^c(\Gamma_{\sigma \chi})}$ satisfies \hyperref[conje:ic]{IC($L_\infty/K, \sigma \chi, L, S, T, \alpha, \beta$)}.
        \end{prop}

        \begin{proof}
            Let us abbreviate $F_{S, T, \chi}^{\alpha, \beta}$ to $F_\chi$. We first note that, for $x \in B_1^{\QQ_p^c}(0)$ in the notation of remark \ref{rem:evaluation_of_series_quotients}, one has
            \[
                \restr{\sigma(F_\chi)}{T_{\sigma \chi} = x} = \sigma\big(\restr{F_\chi}{T_\chi = \sigma^{-1}(x)}\big),
            \]
            where $\sigma$ acts on $F_\chi$ by regarding the latter as an element of $\cQ^c(\cG)$ rather than an abstract series quotient. In other words, if we express $F_\chi$ as a quotient of power series in $T_\chi$, then $\sigma$ is applied to both the coefficients and the variable - which explains why the result is a series quotient in $T_{\sigma\chi}$ by \eqref{eq:galois_automorphisms_t}. The above equality then follows from the continuity of the $G_{\QQ_p}$-action on any $p$-adic field and implies
            \begin{equation}
            \label{eq:evaluation_galois_automorphisms}
                ev_{\gamma_{\sigma \chi \otimes \rho}}(\sigma(F_\chi)) = \restr{\sigma(F_\chi)}{T_{\sigma\chi} = \rho(\gamma_K)^{w_{\sigma \chi}} - 1} =
                \sigma\big(\restr{F_\chi}{T_\chi = \sigma^{-1}(\rho(\gamma_K)^{w_\chi} - 1)}\big) = \sigma(ev_{\gamma_{\chi \otimes \sigma^{-1}\rho}}(F_\chi))
            \end{equation}
            for any $\rho$ of type $W$. Here we have used the fact that $w_\chi = w_{\sigma \chi}$, as $St(\eta) = St(\sigma \eta)$ for any $\QQ_p^c$-valued character $\eta$ of $H$.

            Let $\stackrel{\cdot}{=}$ denote equality for almost all $\rho \in \cK_S^\alpha(\chi)$. We have
            \begin{align*}
                ev_{\gamma_{\sigma\chi \otimes \rho}}(\sigma(F_\chi)) &= \sigma\big(ev_{\gamma_{\chi \otimes \sigma^{-1}\rho}}(F_\chi)\big) \\
                &\stackrel{\cdot}{=} \sigma\bigg(\frac{\beta^{-1}(L_{K, S, T}^\ast(\beta (\check{\chi} \otimes \sigma^{-1}\rho^{-1}), 0))}{R_S^\beta(\alpha, \chi \otimes \sigma^{-1}\rho)}\bigg)
            \end{align*}
            \begin{align*}
                \qquad \qquad \qquad \qquad \quad \ &\stackrel{\cdot}{=} \sigma\bigg(\frac{(\beta \sigma)^{-1}(L_{K, S, T}^\ast((\beta \sigma) (\check{\chi} \otimes \sigma^{-1}\rho^{-1}), 0))}{R_S^{\beta \sigma}(\alpha, \chi \otimes \sigma^{-1}\rho)}\bigg) \\
                &= \frac{\beta^{-1}(L_{K, S, T}^\ast(\beta (\sigma\check{\chi} \otimes \rho^{-1}), 0))}{\sigma(R_S^{\beta \sigma}(\alpha, \chi \otimes \sigma^{-1}\rho))}
            \end{align*}
            by, in order: equation \eqref{eq:evaluation_galois_automorphisms}; the Interpolation Conjecture for $F_\chi$; lemma \ref{lem:stark_invariance_interpolated} (note that composition with $\sigma$ induces a bijection $\cK_S^\alpha(\chi) \to \cK_S^\alpha(\sigma \chi) $); and a simple manipulation. The argument used at the end of the proof of the same lemma also yields $\sigma(R_S^{\beta \sigma}(\alpha, \chi \otimes \sigma^{-1}\rho)) = R_S^{\beta}(\alpha, \sigma \chi \otimes \rho)$. Hence,
            \[
                ev_{\gamma_{\sigma\chi \otimes \rho}}(\sigma(F_\chi)) \stackrel{\cdot}{=} \frac{\beta^{-1}(L_{K, S, T}^\ast(\beta (\sigma\check{\chi} \otimes \rho^{-1}), 0))}{R_S^{\beta}(\alpha, \sigma \chi \otimes \rho)}
            \]
            and $\sigma(F_\chi)$ satisfies IC($L_\infty/K, \sigma \chi, L, S, T, \alpha, \beta$).
        \end{proof}

        This has the following immediate Galois-invariance consequence for the Main Conjecture:
        \begin{cor}
        \label{cor:tuple_fchi_invariant}
            Setting \ref{sett:formulation}. Suppose \hyperref[conje:ic]{IC($L_\infty/K, \chi, L, S, T, \alpha, \beta$)} holds for all $\chi \in \Irr_p(\cG)$ and \hyperref[conje:stark]{Stark\textsuperscript{T}($L_n/K, \chi, f, S, T$)} (in the notation of proposition \ref{prop:independence_of_beta}) does for almost all $\chi \in \Irr_p(\cG)$. Then
            \[
                \prod_{\chi / \sim_W} F_{S, T, \chi}^{\alpha, \beta} \in \bigg(\prod_{\chi / \sim_W} \units{\cQ^c(\Gamma_\chi)} \bigg)^{G_{\QQ_p}} \subseteq \prod_{\chi / \sim_W} \units{\cQ^c(\Gamma_\chi)}.
            \]
        \end{cor}

        \begin{proof}
            Given $\sigma \in G_{\QQ_p}$, one has
            \[
                \sigma \bigg( \prod_{\chi / \sim_W} F_{S, T, \chi}^{\alpha, \beta} \bigg) = \prod_{\chi / \sim_W} \sigma(F_{S, T, \sigma^{-1}\chi}^{\alpha, \beta}) = \prod_{\chi / \sim_W} F_{S, T, \chi}^{\alpha, \beta},
            \]
            where the first equality reflects how $\sigma$ permutes the $\chi$-parts of $Z(\cQ^c(\cG))$ and the second one is proposition \ref{prop:stark_galois_invariance_f} together with uniqueness of the interpolating elements (cf. remark \ref{rem:ic} ii)).
        \end{proof}

        The Galois invariants $\big(\prod_{\chi/{\sim}_W}\units{\cQ^c(\Gamma_\chi)}\big)^{G_{\QQ_p}} \iso (\units{Z(\cQ^c(\cG))})^{G_{\QQ_p}}$ are precisely $\units{Z(\cQ(\cG))}$ (see the proof of \cite{rwii} theorem 8) and therefore fit into the commutative diagram
        \begin{equation}
        \label{diag:centres_galois_invariants}
            \hspace{-2.2cm}
            \begin{tikzcd}
                K_1(\cQ(\cG)) \arrow[r, "\nr"] \arrow[d] \arrow[rrd, "\prod_{\chi/{\sim}_W} \psi_\chi"', out=195, in=210, looseness=1.5] & \units{Z(\cQ(\cG))} \arrow[d, hook] \arrow[r, "\sim"] & { \bigg(\prod_{\chi/{\sim}_W}\units{\cQ^c(\Gamma_\chi)}\bigg)^{G_{\QQ_p}} } \arrow[d, hook] \\
                K_1(\cQ^c(\cG)) \arrow[r, "\nr"]           & \units{Z(\cQ^c(\cG))} \arrow[r, "\sim"]               & \prod_{\chi/{\sim}_W} \units{\cQ^c(\Gamma_\chi)}
            \end{tikzcd}
            \vspace{-2em}
        \end{equation}
        (cf. \eqref{eq:k1_reduced_norm_chi_parts}). Both isomorphisms are canonical, and we regard them as identifications. Suppose now that \hyperref[conje:ic]{IC($L_\infty/K, \chi, L, S, T, \alpha, \beta$)} holds for every $\chi \in \Irr_p(\cG)$ and that Stark's conjecture does for almost all $\chi$ as in proposition \ref{prop:independence_of_beta} ii). Then corollary \ref{cor:tuple_fchi_invariant} shows that the tuple $\prod_{\chi/{\sim}_W} F_{S, T, \chi}^{\alpha, \beta}$ lies in fact in the image of the rightmost vertical arrow above. This supports the existence of a $\zeta_{S, T}^{\alpha, \beta} \in K_1(\cQ(\cG))$ which is mapped to it under $\prod_{\chi/{\sim}_W} \psi_\chi$ (which is precisely the analytic part of the equivariant Main Conjecture) despite the considerable difference between $\cQ(\cG)$ and $\cQ^c(\cG)$.

        The same reasoning shows that the equivariant Main Conjecture is \textit{weakly} independent of $\beta$. By this we mean that it is well defined on equivalence classes of isomorphisms $\CC_p \isoa \CC$ under the relation: $\beta \sim \tilde{\beta}$ if $\restr{\beta}{\QQ_p} = \restr{\tilde{\beta}}{\QQ_p}$.
        \begin{cor}
        \label{cor:emc_implies_weak_stark}
            Setting \ref{sett:formulation}. Let $\tilde{\beta} \colon \CC_p \isoa \CC$ be a ring isomorphism whose restriction to $\QQ_p$ coincides with that of $\beta$. Then \hyperref[conje:emc]{eMC($L_\infty/K, L, S, T, \alpha, \beta$)} holds if and only if \hyperref[conje:emc]{eMC($L_\infty/K, L, S, T, \alpha, \tilde{\beta}$)} does.
        \end{cor}

        \begin{proof}
            The equivariant Main Conjecture for $\beta$ asserts that $\prod_{\chi/{\sim}_W} F_{S, T, \chi}^{\alpha, \beta} \in \prod_{\chi/{\sim}_W}\units{\cQ^c(\Gamma_\chi)}$ lives in the image of $K_1(\cQ(\cG))$ under $\prod_{\chi/{\sim}_W} \psi_\chi$ and is therefore $G_{\QQ_p}$-invariant by diagram \eqref{diag:centres_galois_invariants}. In particular, \eqref{eq:galois_automorphisms_t} implies $F_{S, T, \chi}^{\alpha, \beta} = \sigma(F_{S, T, \sigma^{-1}\chi}^{\alpha, \beta})$ for any $\sigma \in G_{\QQ_p}$. The composition $\sigma = \tilde{\beta}^{-1} \beta$ is an automorphism of $\CC_p$ which fixes $\QQ_p$ and thus restricts to an element of $G_{\QQ_p}$. Hence the same argument as in \eqref{eq:evaluation_galois_automorphisms} together with the relation $R_S^{\tilde{\beta}}(\alpha, \chi \otimes \rho) = \sigma(R_S^\beta(\alpha, \sigma^{-1}(\chi \otimes \rho)))$ (see the proof of lemma \ref{lem:stark_invariance_interpolated}) yields
            \begin{align*}
                ev_{\gamma_{\chi \otimes \rho}}(F_{S, T, \chi}^{\alpha, \beta}) =
                ev_{\gamma_{\chi \otimes \rho}}(\sigma(F_{S, T, \sigma^{-1}\chi}^{\alpha, \beta}))
                & = \sigma(ev_{\gamma_{\sigma^{-1}(\chi \otimes \rho)}}(F_{S, T, \sigma^{-1}\chi}^{\alpha, \beta})) \\
                & = \sigma\left(\frac{\beta^{-1}(L_{K, S, T}^\ast(\beta \sigma^{-1}(\check{\chi} \otimes \rho^{-1}), 0))}{R_S^\beta(\alpha, \sigma^{-1}(\chi \otimes \rho))}\right) \\
                & = \frac{\tilde{\beta}^{-1}(L_{K, S, T}^\ast(\tilde{\beta} (\check{\chi} \otimes \rho^{-1}), 0))}{R_S^{\tilde{\beta}}(\alpha, \chi \otimes \rho)}
            \end{align*}
            for almost all $\sigma^{-1}\rho \in \cK_S^\alpha(\sigma^{-1} \chi)$. Since composition with $\sigma$ defines a bijection ${\cK_S^\alpha(\sigma^{-1} \chi) \to \cK_S^\alpha(\chi)}$, $F_{S, T, \chi}^{\alpha, \beta}$ satisfies IC($L_\infty/K, \chi, L, S, T, \alpha, \tilde{\beta}$) as well. It follows that any $\zeta_{S, T}^{\alpha, \beta} \in K_1(\cQ(\cG))$ satisfying eMC($L_\infty/K, L, S, T, \alpha, \beta$) will also satisfy eMC($L_\infty/K, L, S, T, \alpha, \tilde{\beta}$), as only the analytic part of the Main Conjecture depends on $\beta$. A symmetric argument proves the converse.
        \end{proof}

        \begin{rem}
            \begin{enumerate}[i)]
                \item{
                    By remark \ref{rem:emcu} i), this corollary automatically implies the equivalence of the Main Conjectures with uniqueness \hyperref[conje:emcu]{eMC\textsuperscript{u}($L_\infty/K, L, S, T, \alpha, \beta$)} and \hyperref[conje:emcu]{eMC\textsuperscript{u}($L_\infty/K, L, S, T, \alpha, \tilde{\beta}$)}.
                }
                \item{
                     It follows from the proof (essentially, from the displayed equation) that \hyperref[conje:emc]{eMC($L_\infty/K, L, S, T, \alpha, \tilde{\beta}$)} implies a weaker version of \hyperref[conje:stark]{Stark\textsuperscript{T}($L_n/K, \chi, f, S, T$)} (in the notation of proposition \ref{prop:independence_of_beta}) for almost all $\chi \in \Irr_p(\cG)$, where one fixes a $\tilde{\beta} \colon \CC_p \isoa \CC$ and replaces ``for any ring homomorphism $\beta \colon \CC_p \ia \CC$'' by ``for any isomorphism $\beta \colon \CC_p \isoa \CC$ such that $\restr{\beta}{\QQ_p} = \restr{\tilde{\beta}}{\QQ_p}$''.
                }
            \end{enumerate}
        \end{rem}

        We end on a brief discussion of how Stark's conjecture  and the equivariant Main Conjecture narrow down the field where the coefficients of $F_{S, T, \chi}^{\alpha, \beta}$ should be expected to lie. In order to understand this, we introduce the following notation from \cite{nickel_conductor}: for an Artin character $\chi$ of $\cG$, set
        \[
            \QQ_{p, \chi} = \QQ_p(\chi(h) : h \in H).
        \]
        This is a finite abelian extension of $\QQ_p$, as it is contained in some cyclotomic field $\QQ_p(\zeta_n)$. Note that $\chi \sim_W \chi'$ implies $\QQ_{p, \chi} = \QQ_{p, \chi'}$. Given an automorphism $\sigma \in G_{\QQ_{p, \chi}}$, one has $\rest_H^{\cG} \chi = \rest_H^{\cG} \sigma\chi$ and hence $\sigma(e_\chi) = e_{\sigma \chi} = e_{\chi}$. Therefore, the $G_{\QQ_p}$-action on $\cQ^c(\cG)$ restricts to a $G_{\QQ_{p, \chi}}$-action on $\cQ^c(\cG)e_\chi = \cQ^c(\Gamma_\chi)$. It then follows from proposition \ref{prop:stark_galois_invariance_f} that, assuming the Interpolation Conjecture for $\chi$ and Stark's for suitable characters, one has $\sigma(F_{S, T, \chi}^{\alpha, \beta}) = F_{S, T, \sigma \chi}^{\alpha, \beta} = F_{S, T, \chi}^{\alpha, \beta}$ for all $\sigma \in G_{\QQ_{p, \chi}}$. In other words,
        \begin{equation}
        \label{eq:narrow_coefficient_field_interpolating_element}
            F_{S, T, \chi}^{\alpha, \beta} \in (\units{(\cQ^c(\Gamma_\chi))})^{G_{\QQ_{p, \chi}}}.
        \end{equation}
        The invariants in question have been explicitly determined in \cite{nickel_conductor}. Namely, there is a field isomorphism
        \[
            (\cQ^c(\Gamma_\chi))^{G_{\QQ_{p, \chi}}} = \cQ^{\QQ_{p, \chi}}(\Gamma_\chi'),
        \]
        where $\Gamma_\chi' \iso \ZZ_p$ is a subgroup of $\units{(\cQ^c(\Gamma_\chi))}$ which is topologically generated by $\gamma_\chi' = u \gamma_\chi$ for a certain principal unit $u$ in some $p$-adic field. In particular, $F_{S, T, \chi}^{\alpha, \beta}$ is a quotient of power series in $T_\chi$ with coefficients in $\QQ_{p, \chi}(u)$.

        By diagram \eqref{diag:centres_galois_invariants}, another way to ensure $\sigma(F_{S, T, \chi}^{\alpha, \beta}) = F_{S, T, \sigma \chi}^{\alpha, \beta} = F_{S, T, \chi}^{\alpha, \beta}$ for all $\sigma \in G_{\QQ_{p, \chi}}$, and therefore \eqref{eq:narrow_coefficient_field_interpolating_element}, is to assume \hyperref[conje:emc]{eMC($L_\infty/K, L, S, T, \alpha, \beta$)}.

    \section{Independence of the choice of parameters}
    \label{sec:independence_of_the_choice_of_parameters}

        Unlike in the case of $\beta$, where the assumption of Stark's conjecture is necessary to show the independence of the Main Conjecture, the rest of the parameters admit unconditional proofs. This concerns $L$, $\alpha$, $S$ and $T$ - the change of $L_\infty$ or $K$ is treated in the next section. We divide these results into three subsections devoted to $L$, $\alpha$, and $S$ and $T$, respectively. Although the idea is to modify these parameters one at a time, a change in one of them may have implications for some of the remaining ones, which will be accounted for in each subsection.

        \subsection{\texorpdfstring{The choice of $L$}{The choice of L}}
        \label{subsec:the_choice_of_l}

            We first prove the independence of a choice of $L$ (or equivalently, $\Gamma$). All other parameters of \hyperref[conje:ic]{IC($L_\infty/K, \chi, L, S, T, \alpha, \beta$)} - aside from $L_\infty/K$, which is fixed beforehand - are independent of this choice:
            \begin{itemize}
                \item{
                    $\chi \in \Irr_p(\cG)$ is an Artin character whose definition does not involve $\Gamma$. However, we will need to be mindful of related objects which do depend on $\Gamma$, most notably $\cK_S^\alpha(\chi)$ (see the point on $\alpha$ below).
                }
                \item{
                    $S$ and $T$ are defined over $K$ in terms of ramification in $L_\infty$. The set $T^p \subseteq T$, which a priori could depend on $L$, is not an arbitrary parameter (and it is in fact independent, as will become apparent).
                }
                \item{
                    $\alpha \colon \cY_{S_\infty} \ia E_{S, T}$ is $\cG$-equivariant. It should be noted that, although many objects have been defined as limits of Galois modules along the cyclotomic tower $L_\infty/L$, they do not truly depend on the choice of $L$: the same can be obtained as the inverse limit over all finite (Galois) extensions of $K$ contained in $L_\infty$ instead. This applies, for instance, to $\cY_{S_\infty}$ and $E_{S, T}$.

                    Another relevant fact, which already made an appearance earlier, is that the $\Lambda(\Gamma)$-torsionness of $\Lambda(\cG)$-modules (which is required of $\coker(\alpha)$) is independent of the choice of $\Gamma$. This is a direct consequence of \eqref{eq:lambda_regular_elements}: if $\widetilde{\Gamma}$ is another such choice and $M$ is a $\Lambda(\cG)$-module, then
                    \begin{equation}
                    \label{eq:lambda_torsionness_independent_gamma}
                        \cQ(\Gamma) \otimes_{\Lambda(\Gamma)} M = 0 \iff \cQ(\cG) \otimes_{\Lambda(\cG)} M = 0 \iff \cQ(\widetilde{\Gamma}) \otimes_{\Lambda(\widetilde{\Gamma})} M = 0.
                    \end{equation}
                    Explicitly, suppose $m \in M$ satisfies $\widetilde{\lambda} m = 0$ for some non-zero $\widetilde{\lambda} \in \Lambda(\widetilde{\Gamma})$, which is in particular central and regular in $\Lambda(\cG)$. By the cited equation, $\widetilde{\lambda}^{-1} = l/\lambda$ in $\cQ(\cG)$ for some $l \in \Lambda(\cG)$ and $\lambda \in \Lambda(\Gamma)$. Therefore, $\lambda m = l\widetilde{\lambda} m = 0$.

                     Although this shows that $\alpha$ is indeed independent of the choice of $L$, the related $n(S, \alpha) \in \NN$ (introduced immediately before \ref{manualcond:kc}) is not so.
                }
                \item{
                    $\beta \colon \CC_p \xrightarrow{\sim} \CC$ is unrelated to $L$.
                }
            \end{itemize}

            \begin{prop}
            \label{prop:indep_ic_l}
                Setting \ref{sett:formulation}. Let $\tilde{L}$ be another valid choice for the parameter $L$. Then:
                \begin{enumerate}[i)]
                    \item{
                        For any $\chi \in \Irr_p(\cG)$, \hyperref[conje:ic]{IC($L_\infty/K, \chi, L, S, T, \alpha, \beta$)} holds if and only if \hyperref[conje:ic]{IC($L_\infty/K, \chi, \tilde{L}, S, T, \alpha, \beta$)} does, in which case the interpolating elements $F_{S, T, \chi}^{\alpha, \beta}$ coincide.
                    }
                    \item{
                        \hyperref[conje:emc]{eMC($L_\infty/K, L, S, T, \alpha, \beta$)} holds if and only if \hyperref[conje:emc]{eMC($L_\infty/K, \tilde{L}, S, T, \alpha, \beta$)} does, in which case the zeta elements $\zeta_{S, T}^{\alpha, \beta}$ coincide.
                    }
                \end{enumerate}
            \end{prop}

            \begin{proof}
                Let us consider the Interpolation Conjecture first. The compositum $L\widetilde{L}$ is an intermediate extension of $L_\infty/L$ and $L_\infty/\widetilde{L}$, and hence equal to $L_m$ and $\widetilde{L}_{\widetilde{m}}$ for some $m, \widetilde{m} \in \NN$. Therefore, it suffices to show that the conjecture is independent of the choice of a layer along the cyclotomic tower, and for this, in turn, that
                \[
                    IC(L_\infty/K, \chi, L, S, T, \alpha, \beta) \iff IC(L_\infty/K, \chi, L_1, S, T, \alpha, \beta)
                \]
                (note that $L_m$ is the first layer of the cyclotomic $\ZZ_p$-extension of $L_{m - 1}$), from which the same assertion for $\widetilde{L}$ and $\widetilde{L}_1$ will follow by analogy.

                The elements involved in the definition of $\cQ^c(\Gamma_\chi)$ (cf. proposition \ref{prop:structure_gamma_chi}), most notably $\gamma_\chi$, are oblivious to $\Gamma$, as are the evaluation maps $ev_{\gamma_{\chi \otimes \rho}} \colon \cQ^c(\Gamma_\chi) \to \QQ_p^c \cup \set{\infty}$ (see for instance lemma \ref{lem:twisted_evaluation_maps} and use ${T_\chi = \gamma_\chi - 1}$). In other words, they are the same in the formulation of IC($L_\infty/K, \chi, L, S, T, \alpha, \beta$) and IC($L_\infty/K, \chi, L_1, S, T, \alpha, \beta$).

                Let $\cK_{S, 0}^\alpha(\chi)$ be the set $\cK_S^\alpha(\chi)$ defined as in $\ref{defn:r_chi}$ for $L = L_0$, and analogously for $\cK_{S, 1}^\alpha(\chi)$. Then $\cK_{S, 0}^\alpha(\chi) \cap \cK_{S, 1}^\alpha(\chi)$ contains almost all type-$W$ characters of $\cG$, since each set does (it is in fact easy to see that $\cK_{S, 1}^\alpha(\chi) \subseteq \cK_{S, 0}^\alpha(\chi)$ and they very frequently coincide).

                Suppose $F_{S, T, \chi}^{\alpha, \beta} \in \units{\cQ^c(\Gamma_\chi)}$ satisfies IC($L_\infty/K, \chi, L, S, T, \alpha, \beta$).  Then, for almost all characters ${\rho \in \cK_{S, 0}^\alpha(\chi) \cap \cK_{S, 1}^\alpha(\chi)}$ (and therefore almost all $\rho \in \cK_{S, 1}^\alpha(\chi)$), one has
                \begin{equation}
                \label{eq:interpolation_L_L1}
                    ev_{\gamma_{\chi \otimes \rho}}(F_{S, T, \chi}^{\alpha, \beta}) = \frac{\beta^{-1}(L_{K, S, T}^\ast(\beta(\check{\chi} \otimes \rho^{-1}), 0))}{R_S^\beta(\alpha, \chi \otimes \rho)},
                \end{equation}
                where the numerator and the denominator are the leading coefficient and Stark-Tate regulator computed at some large enough layer $L_n$, which we can choose to contain $L_1$ by lemmas \ref{lem:properties_of_L-functions} ii) and \ref{lem:r_chi_independent_of_n}. It follows that $F_{S, T, \chi}^{\alpha, \beta}$ satisfies IC($L_\infty/K, \chi, L_1, S, T, \alpha, \beta$), which concludes the proof of i).

                In order to prove ii), we first point out that the $K$-groups $K_1(\cQ(\cG))$ and $K_0(\Lambda(\cG), \cQ(\cG))$ are independent of the choice of $\Gamma$ by definition (recall \eqref{eq:lambda_torsionness_independent_gamma} regarding $\Lambda(\Gamma)$-torsionness), as is the connecting homomorphism $\partial$ between them. The main complex $\cC_{S, T}\q$ was shown to be isomorphic in the derived category $\cD(\Lambda(\cG))$ to an inverse limit of complexes $\cB_{L_n, S, T}\q$ along the cyclotomic tower (cf. definition \ref{defn:complexes_bks} and theorem \ref{thm:iso_bks_complex}), but it has been established above that such limits do not depend on $L$. As a consequence, neither does the refined Euler characteristic $\chi_{\Lambda(\cG), \cQ(\cG)}(\cC_{S, T}\q, t^\alpha)$.

                We are done now, as the field $L$ does not appear in (or otherwise affect) the definition \ref{defn:k1_to_gamma_chi} of ${\psi_\chi \colon K_1(\cQ(\cG)) \to \units{\cQ^c(\Gamma_\chi)}}$, and part i) showed that the Interpolation Conjectures for $L$ and $\widetilde{L}$ are equivalent and the series quotients $F_{S, T, \chi}^{\alpha, \beta}$ coincide whenever they hold. Therefore, any ${\zeta_{S, T}^{\alpha, \beta} \in K_1(\cQ(\cG))}$ satisfying eMC($L_\infty/K, L, S, T, \alpha, \beta$) will also satisfy eMC($L_\infty/K, \widetilde{L}, S, T, \alpha, \beta$) and vice versa.
            \end{proof}

        \subsection{\texorpdfstring{The choice of $\alpha$}{The choice of alpha}}
        \label{subsec:the_choice_of_alpha}

            The next goal is to show the independence of the Main Conjecture of the choice of $\alpha \colon \cY_{S_\infty} \ia E_{S, T}$. Although this requires a more technical proof than that in the previous subsection, we have chosen not to split it due to how intertwined the impact of a change in $\alpha$ on the analytic and algebraic sides of the conjecture is. Indeed, the way to adjust the interpolating elements $F_{S, T, \chi}^{\alpha, \beta}$ to a new $\alpha'$ will be to construct a certain element in the group $K_0(\Lambda(\cG), \cQ(\cG))$ and transfer it to $\units{\cQ^c(\Gamma_\chi)}$.

            As in the case of $L$, the homomorphism $\alpha$ does not affect any of the other parameters of the conjectures ($\chi, S, T$ and $\beta$). It does have an influence on $n(S, \alpha)$, and therefore on $\cK_S^\alpha(\chi)$ (cf. definition \ref{defn:r_chi}), but this will again not be of any consequence. An object which is fundamentally linked to $\alpha$ is trivialisation $t^\alpha$ (\eqref{eq:integral_trivialisation} and preceding lines), and consequently the refined Euler characteristic $\chi_{\Lambda(\cG), \cQ(\cG)}(\cC_{S, T}\q, t^\alpha) \in K_0(\Lambda(\cG), \cQ(\cG))$. Precisely this will inform how to prove the equivalence of the Interpolation Conjectures.

            \begin{prop}
            \label{prop:independence_of_alpha}
                Setting \ref{sett:formulation}. Let $\alpha'$ be another valid choice for the parameter $\alpha$. Then:
                \begin{enumerate}[i)]
                    \item{
                        For any $\chi \in \Irr_p(\cG)$, \hyperref[conje:ic]{IC($L_\infty/K, \chi, L, S, T, \alpha, \beta$)} holds if and only if \hyperref[conje:ic]{IC($L_\infty/K, \chi, L, S, T, \alpha', \beta$)} does.
                    }
                    \item{
                        \hyperref[conje:emc]{eMC($L_\infty/K, L, S, T, \alpha, \beta$)} holds if and only if \hyperref[conje:emc]{eMC($L_\infty/K, L, S, T, \alpha', \beta$)} does.
                    }
                \end{enumerate}
            \end{prop}

            \begin{proof}
                Assume that IC($L_\infty/K, \chi, L, S, T, \alpha', \beta$) holds and let $F_{S, T, \chi}^{\alpha', \beta} \in \units{\cQ^c(\Gamma_\chi)}$ be the element predicted therein. It suffices to show the existence of a series quotient $f \in \units{\cQ^c(\Gamma_\chi)}$ such that, for almost all type-$W$ characters $\rho$, one has
                \begin{equation}
                \label{eq:interpolation_quotient_of_regulators}
                    ev_{\gamma_{\chi \otimes \rho}}(f) = \frac{R_S^\beta(\alpha', \chi \otimes \rho)}{R_S^\beta(\alpha, \chi \otimes \rho)},
                \end{equation}
                since then the element $F_{S, T, \chi}^{\alpha', \beta} \cdot f$ satisfies the interpolation property for IC($L_\infty/K, \chi, L, S, T, \alpha, \beta$). Recall the multiplicativity of evaluation (under the logical restrictions) explained in remark \ref{rem:definition_of_evaluation} ii) and the fact that $\cK_S^\alpha(\chi) \cap \cK_S^{\alpha'}(\chi)$ contains almost all $\rho$.

                For a $\varsigma \in \Irr_p(\cG)$ satisfying the kernel condition \ref{manualcond:kc} for both $\alpha$ and $\alpha'$ (for instance, $\varsigma = \chi \otimes \rho$ as above) and any $n$ such that $\Gp{n} \subseteq \ker(\varsigma)$ (so $n \geq n(S, \alpha)$), consider the commutative diagram of $\CC_p [\cG_n]$-homomorphisms
                \begin{center}
                    \begin{tikzcd}[column sep=large, row sep=large]
                        {\CC_p \otimes_{\ZZ_p} \cX_{L_n, S}} \arrow[r, "\sim"] \arrow[rrdd, dashed, shift left=0.6] \arrow[rrdd, dashed, shift right=0.6, "{\CC_p \otimes_{\ZZ_p} \varphi_n^\alpha}"', "{\CC_p \otimes_{\ZZ_p}  \varphi_n^{\alpha'}}"] & \CC_p \otimes_{\ZZ_p} \gcoinv{(\cX_S)}{n} \arrow[r, "\CC_p \otimes_{\ZZ_p} \gcoinv{\varphi}{n}"] & \CC_p \otimes_{\ZZ_p} \gcoinv{(\cY_{S_\infty})}{n} \arrow[d, shift right=.50ex] \arrow[d, shift left=.50, "{\CC_p \otimes_{\ZZ_p} \gcoinv{\alpha}{n}}\ "', "{\CC_p \otimes_{\ZZ_p} \gcoinv{\alpha'}{n}}"] &                 \\
                                                                                                            &                               & \CC_p \otimes_{\ZZ_p} \gcoinv{(E_{S, T})}{n} \arrow[d, "\CC_p \otimes_{\ZZ_p} \iota_n"]                                                    &                 \\
                                                                                                            &                               & {\CC_p \otimes \cO_{L_n, S, T}^\ast} \arrow[r, "\lambda^\beta_{n, S}"]                                                  & {\CC_p \otimes_{\ZZ_p} \cX_{L_n, S}}
                    \end{tikzcd}
                \end{center}
                (cf. definition \ref{defn:finite_level_map}). By proposition \ref{prop:isomorphism_on_chi_parts}, all arrows become $\CC_p$-isomorphisms after applying the covariant functor $\Hom_{\CC_p[\cG_n]}(V_\varsigma, -)$, where $V_\varsigma$ is any simple $\CC_p[\cG_n]$-module affording $\varsigma$. The regulators $R_S^\beta(\alpha, \varsigma)$ and $R_S^\beta(\alpha', \varsigma)$ are defined as the determinants of the two resulting compositions of maps on $\Hom$-spaces. But all arrows involved coincide except for the first vertical one, and therefore
                \begin{equation}
                \label{eq:difference_between_regulators}
                    \frac{R_S^\beta(\alpha', \varsigma)}{R_S^\beta(\alpha, \varsigma)} = \det(\Hom_{\CC_p[\cG_n]}(V_\varsigma, (e(\varsigma) \CC_p \otimes_{\ZZ_p} \gcoinv{\alpha}{n})^{-1} \circ (e(\varsigma) \CC_p \otimes_{\ZZ_p} \gcoinv{\alpha'}{n}))),
                \end{equation}
                which denotes the determinant of the $\CC_p$-linear automorphism of $\Hom_{\CC_p[\cG_n]}(V_\varsigma, \CC_p \otimes_{\ZZ_p} \cY_{L_n, S_\infty})$ (recall proposition \ref{prop:isomoprhisms_y_x_modules}) given by postcomposition with $(e(\varsigma) \CC_p \otimes_{\ZZ_p} \gcoinv{\alpha}{n})^{-1} \circ (e(\varsigma) \CC_p \otimes_{\ZZ_p} \gcoinv{\alpha'}{n})$. Since $\gcoinv{\alpha}{n}$ and $\gcoinv{\alpha'}{n}$ are already isomorphisms on $\varsigma$-parts after extending scalars to $\QQ_p^c$, we may replace $\CC_p$ by $\QQ_p^c$ and $V_\varsigma$ by $V_\varsigma'$ in the above equation, with $V_\varsigma'$ a $\QQ_p^c[\cG]$-module with character $\varsigma$.

                In order to look for series quotients which interpolate that difference of regulators, we turn to $K$-theory. This will also naturally lead to the equivalence of the equivariant Main Conjectures for $\alpha$ and $\alpha'$. We shall make use of (an intermediate result in the proof of)  \cite{nickel_fitting} theorem 6.4 to compare the augmentation of the reduced norm of a certain element in $K_1(\cQ(\cG))$ to the quotient \eqref{eq:difference_between_regulators}. The obstacle is that the element in question should essentially describe $\alpha^{-1} \circ \alpha'$, but $\alpha$ cannot be inverted on integral level. We use the following trick: let
                \[
                    g_0 = p^{\mu(\coker(\alpha))} \Char(\coker(\alpha)),
                \]
                where the $\mu$-invariant and the characteristic polynomial refer to the structure of $\coker(\alpha)$ as a $\Lambda(\Gamma)$-module and were defined after theorem \ref{thm:structure_theorem_iwasawa}. Then there exists a non-zero $l \in \NN$ such that $g = l g_0 \in \Lambda(\Gamma)$ annihilates $\coker(\alpha)$, or in other words, $gx \in \img(\alpha)$ for all $x \in E_{S, T}$. Since $\alpha$ is injective, such a $gx$ has a unique pre-image $\alpha^{-1}(gx) \in \cY_{S_\infty}$. We define $a_g$ as the composition
                \[
                    a_g \colon \cY_{S_\infty} \xrightarrow{m_g} \cY_{S_\infty} \xrightarrow{\alpha'} E_{S, T}  \stackrel{\alpha^{-1}}{\dashrightarrow} \cY_{S_\infty},
                \]
                where $m_g$ denotes multiplication by $g$ and the last arrow is dashed to emphasise the fact that it does not exist in isolation, but it makes sense as part of the entire composition (note that $\alpha'$ commutes with multiplication by $g$). The map $a_g$ is a homomorphism of $\Lambda(\cG)$-modules because $g$ is central in $\Lambda(\cG)$. The idea here is that multiplication by $g$ allows us to construct an integral homomorphism containing information about the inverse of $\alpha$, and the perturbations thus introduced into $K_1(\cQ(\cG))$ and $\Hom_{\CC_p[\cG_n]}(V_\varsigma, \CC_p \otimes_{\ZZ_p} \cY_{L_n, S_\infty})$ cancel out.

                For the remainder of this proof, let $-_{\cQ}$ denote $\cQ(\cG) \otimes_{\Lambda(\cG)} -$. Then $(a_g)_\cQ$ is a $\cQ(\cG)$-automorphism of $(\cY_{S_\infty})_\cQ$ which arises as the (now honest) composition
                \begin{equation}
                \label{eq:triple_composition_after_tensoring_q}
                    (a_g)_\cQ = (\alpha_\cQ)^{-1} \circ \alpha'_\cQ \circ (m_g)_\cQ,
                \end{equation}
                with the inverse of $\alpha_\cQ$ mapping $\lambda^{-1} e$ to $(\lambda g)^{-1} \alpha^{-1}(ge)$ for $\lambda \in \Lambda(\cG), e \in E_{S, T}$. Both $(\alpha_\cQ)^{-1} \circ \alpha'_\cQ$ and $(m_g)_\cQ$ are $\cQ(\cG)$-automorphisms of $(\cY_{S_\infty})_\cQ$.

                We now turn our attention to finite level again. With $\varsigma$ and $n$ as above, let $\overline{x}$ denote the class of $x \in M$ in $\gcoinv{M}{n}$ for any $\Lambda(\cG)$-module $M$. Then $\gcoinv{(m_g)}{n} = m_{\overline{g}}$, i.e. multiplication by $\overline{g} \in \Lambda(\cG_n)$. The map $\gcoinv{\alpha}{n}$ is no longer injective in general ($\gcoinv{-}{n}$ is only right-exact), so one cannot describe $\gcoinv{(a_g)}{n}$ as a composition of three maps. This is solved by passing to $\varsigma$-parts, where everything is an isomorphism. Namely, let $\overline{y} \in \cY_{L_n, S_\infty} = \gcoinv{(\cY_{S_\infty})}{n}$ and choose any lift $y \in \cY_{S_\infty}$. Then
                \[
                    \gcoinv{\alpha}{n}\big(\overline{a_g(y)}\big) = \overline{\alpha(a_g(y))} = \overline{\alpha'(m_g(y))} = \gcoinv{\alpha'}{n} (m_{\overline{g}}(\overline{y}))
                \]
                and hence
                \[
                    \gcoinv{\alpha}{n}\big(e(\varsigma) (\QQ_p^c \otimes_{\ZZ_p} \gcoinv{(a_g)}{n})(\overline{y})\big) = e(\varsigma) (\QQ_p^c \otimes_{\ZZ_p} \gcoinv{\alpha'}{n})(m_{\overline{g}}(\overline{y})).
                \]
                We thus obtain a finite-level decomposition
                \begin{equation}
                \label{eq:triple_decomposition_on_sigma_parts}
                    e(\varsigma) \QQ_p^c \otimes_{\ZZ_p} \gcoinv{(a_g)}{n} =
                    \big(e(\varsigma) \QQ_p^c \otimes_{\ZZ_p} \gcoinv{\alpha}{n}\big)^{-1} \circ
                    \big(e(\varsigma) \QQ_p^c \otimes_{\ZZ_p} \gcoinv{\alpha'}{n}\big) \circ
                    \big(e(\varsigma) \QQ_p^c \otimes_{\ZZ_p} m_{\overline{g}}\big),
                \end{equation}
                where the first map on the right-hand side is well defined because $\gcoinv{\alpha}{n}$ is an isomorphism on $\varsigma$-parts (this follows from  proposition \ref{prop:isomorphism_on_chi_parts} and \ref{manualcond:kc}).

                In order to see that $m_{\overline{g}}$ also induces an isomorphism on $\varsigma$-parts, consider the exact sequence of $\Lambda(\cG_n)$-modules
                \begin{equation}
                \label{eq:multiplication_by_scalar_isomorphism_on_chi_parts}
                    \ginv{\coker(m_g)}{n} \to \cY_{L_n, S_\infty} \xrightarrow{m_{\overline{g}}} \cY_{L_n, S_\infty} \to \gcoinv{\coker(m_g)}{n} \to 0,
                \end{equation}
                where we have simply dropped the first two terms of the invariants-coinvariants sequence \eqref{eq:invariants_coinvariants} associated to the injective homomorphism $m_g$. It suffices to show that $e(\varsigma) \QQ_p^c \otimes_{\ZZ_p} \gcoinv{\coker(m_g)}{n}$ and $e(\varsigma) \QQ_p^c \otimes_{\ZZ_p} \ginv{\coker(m_g)}{n}$ are trivial (note that the latter surjects onto the kernel of $m_{\overline{g}}$ on $\varsigma$-parts by the above sequence). Since $\cY_{S_\infty}$ is a free $\Lambda(\Gamma)$-module of rank $\abs{S_\infty(L)}$, it is enough in turn to prove $e(\varsigma) \QQ_p^c \otimes_{\ZZ_p} \gcoinv{(\Lambda(\Gamma)/\ideal{g})}{n} = e(\varsigma) \QQ_p \otimes_{\ZZ_p} \ginv{(\Lambda(\Gamma)/\ideal{g})}{n} = 0$. But this follows from \eqref{eq:introduction_alpha_n} and lemma \ref{lem:torsion_modules_bounded_coinvariants}, as $\supp(\Lambda(\Gamma)/\ideal{g})$ and $\supp(\coker(\alpha))$ coincide by definition of $g$ except possibly for the prime $\ideal{p}$ (due to the $p$-power $l$ introduced into $g$).

                Now that the relevant properties of $\gcoinv{(a_g)}{n}$ and $m_{\overline{g}}$ have been determined, a connection must be established to infinite level. The heavy lifting is done by equation (8) from \cite{nickel_fitting}. In our notation, it states that, given a finitely generated projective $\Lambda(\cG)$-module\footnote{Technically, the equation in the article only refers to the specific case where $P = \Lambda(\cG)$ and the endomorphism $f$ is given by right multiplication by an $\alpha \in \Lambda(\cG)$. However, as the reference states, the same relation holds when those objects are replaced by general $P$ and $f$ as above by essentially the same proof - only with more involved indexing.} $P$ and an endomorphism ${f \in \End_{\Lambda(\cG)}(P)}$ with torsion kernel and cokernel, one has
                \begin{equation}
                \label{eq:relation_reduced_norm_and_invariants}
                    \det(\Hom_{\QQ_p^c[\cG_n]}(V_\varsigma', \QQ_p^c \otimes_{\ZZ_p} \gcoinv{f}{n})) = ev_{\gamma_K}(j_\varsigma(\nr([P_\cQ, f_\cQ]))),
                \end{equation}
                where $[P_\cQ, f_\cQ] \in K_1^{\det}(\cQ(\cG))$ and the reduced norm $\nr$ are as in section \ref{sec:algebraic-k-theory} and $j_\varsigma$ was introduced in proposition \ref{prop:structure_gamma_chi} (extending scalars from $E$ to $\QQ_p^c$ now). Recall that we identify $K_1^{\det}(\cQ(\cG))$ with $K_1(\cQ(\cG))$. The left-hand side of the above equation denotes, as usual, the determinant of the $\QQ_p^c$-linear endomorphism of $\Hom_{\QQ_p^c[\cG_n]}(V_\varsigma', \QQ_p^c \otimes_{\ZZ_p} \gcoinv{P}{n})$ induced by postcomposition with $\QQ_p^c \otimes_{\ZZ_p} \gcoinv{f}{n}$.

                Diagram \eqref{eq:psi_and_j}, together with the fact that $\gamma_\varsigma \mapsto \gamma_K^{w_\varsigma}$ commutes with evaluation (see for instance the last diagram in the proof of lemma \ref{lem:twisted_evaluation_maps_independent_gamma_K}), implies
                \begin{equation}
                \label{eq:evaluation_j_and_psi}
                    ev_{\gamma_K}(j_\varsigma(\nr([P_\cQ, f_\cQ]))) = ev_{\gamma_\varsigma}(\psi_\varsigma([P_\cQ, f_\cQ])).
                \end{equation}
                Consider the two choices $f = a_g$ and $f = m_g$, which are indeed endomorphisms of the finitely generated projective $\Lambda(\cG)$-module $\cY_{S_\infty}$ with torsion kernel and cokernel. Equation \eqref{eq:relation_reduced_norm_and_invariants} becomes
                \[
                    \det(\Hom_{\QQ_p^c[\cG_n]}(V_\varsigma', \QQ_p^c \otimes_{\ZZ_p} \gcoinv{(a_g)}{n})) = ev_{\gamma_\varsigma}(\psi_\varsigma([(\cY_{S_\infty})_\cQ, (a_g)_\cQ]))
                \]
                and
                \[
                    \det(\Hom_{\QQ_p^c[\cG_n]}(V_\varsigma', \QQ_p^c \otimes_{\ZZ_p} m_{\overline{g}})) = ev_{\gamma_\varsigma}(\psi_\varsigma([(\cY_{S_\infty})_\cQ, (m_g)_\cQ])).
                \]
                We now rewrite both sides of the first equation of the two using decompositions \eqref{eq:triple_composition_after_tensoring_q} and \eqref{eq:triple_decomposition_on_sigma_parts}, as well as the multiplicativity of determinants and the relations of $K_1^{\det}(\cQ(\cG))$:
                \begin{align*}
                    & \det(\Hom_{\QQ_p^c[\cG_n]}(V_\varsigma', (e(\varsigma) \QQ_p^c \otimes_{\ZZ_p} \gcoinv{\alpha}{n})^{-1} \circ (e(\varsigma) \QQ_p^c \otimes_{\ZZ_p} \gcoinv{\alpha'}{n}))) \cdot \det(\Hom_{\QQ_p^c[\cG_n]}(V_\varsigma', \QQ_p^c \otimes_{\ZZ_p} m_{\overline{g}})) \\
                    = \ & ev_{\gamma_\varsigma}(\psi_\varsigma([(\cY_{S_\infty})_\cQ, (\alpha_\cQ)^{-1} \circ \alpha_\cQ'])) \cdot ev_{\gamma_\varsigma}(\psi_\varsigma([(\cY_{S_\infty})_\cQ, (m_g)_\cQ])).
                \end{align*}
                The second factors cancel out by the second of the two equations above, since $e(\varsigma) \QQ_p^c \otimes_{\ZZ_p} m_{\overline{g}}$ has already been shown to be an isomorphism. This results in the fundamental relation
                \[
                \det\nolimits(\Hom_{\QQ_p^c[\cG_n]}(V_\varsigma', (e(\varsigma) \QQ_p^c \otimes_{\ZZ_p} \gcoinv{\alpha}{n})^{-1} \circ (e(\varsigma) \QQ_p^c \otimes_{\ZZ_p} \gcoinv{\alpha'}{n}))) = ev_{\gamma_\varsigma}(\psi_\varsigma([(\cY_{S_\infty})_\cQ, (\alpha_\cQ)^{-1} \circ \alpha_\cQ'])),
                \]
                where the left-hand side is precisely the quotient $R_S^\beta(\alpha', \varsigma) / R_S^\beta(\alpha, \varsigma)$ by equation \eqref{eq:difference_between_regulators} (after restring scalars to $\QQ_p^c$ as explained immediately after it).

                The upshot from the above argument is that the element $[(\cY_{S_\infty})_\cQ, (\alpha_\cQ)^{-1} \circ \alpha_\cQ'] \in K_1^{\det}(\cQ(\cG))$, which is equivariant (i.e. independent of the choice of a particular character), has the following property: for every $\varsigma \in \Irr_p(\cG)$ satisfying \ref{manualcond:kc}, one has
                \begin{equation}
                \label{eq:evaluation_of_psi_alphas}
                    ev_{\gamma_\varsigma}(\psi_\varsigma([(\cY_{S_\infty})_\cQ, (\alpha_\cQ)^{-1} \circ \alpha_\cQ'])) = \frac{R_S^\beta(\alpha', \varsigma)}{R_S^\beta(\alpha, \varsigma)}.
                \end{equation}
                Let now $\chi \in \Irr_p(\cG)$ be as in the beginning of the proof. Since $\psi_{\chi \otimes \rho} = \psi_\chi$ for all $\rho$ of type $W$, the element $f = \psi_\chi([(\cY_{S_\infty})_\cQ, (\alpha_\cQ)^{-1} \circ \alpha_\cQ']) \in \cQ^c(\Gamma_\chi)$ satisfies \eqref{eq:interpolation_quotient_of_regulators}, from which the equivalence of the Interpolation Conjectures follows.

                Most of the work required to prove part ii) has already been done. Suppose $\zeta_{S, T}^{\alpha', \beta} \in K_1(\cQ(\cG))$ satisfies eMC($L_\infty/K, L, S, T, \alpha', \beta$) and set
                \[
                    \zeta_{S, T}^{\alpha, \beta} = \zeta_{S, T}^{\alpha', \beta} \cdot [(\cY_{S_\infty})_\cQ, (\alpha_\cQ)^{-1} \circ \alpha_\cQ'].
                \]
                Let $\chi \in \Irr_p(\cG)$. By assumption, $F_{S, T, \chi}^{\alpha', \beta} = \psi_\chi(\zeta_{S, T}^{\alpha', \beta})$ satisfies IC($L_\infty/K, \chi, L, S, T, \alpha', \beta$), which implies that
                \begin{equation}
                \label{eq:change_alpha_interpolating_element}
                    \psi_\chi(\zeta_{S, T}^{\alpha, \beta}) = F_{S, T, \chi}^{\alpha', \beta} \cdot \psi_\chi([(\cY_{S_\infty})_\cQ, (\alpha_\cQ)^{-1} \circ \alpha_\cQ'])
                \end{equation}
                satisfies IC($L_\infty/K, \chi, L, S, T, \alpha, \beta$) by part i).

                It remains to study the algebraic side of the conjecture. For this we resort to \cite{bb}, which shows that $(\alpha_\cQ)^{-1} \circ \alpha_\cQ'$ measures exactly the difference between the refined Euler characteristics defined via $\alpha$ and $\alpha'$. As mentioned in section \ref{sec:an_integral_trivialisation}, the cited article features two versions of the refined Euler characteristic: one denoted by $\chi^{\text{old}}$, which coincides with our $\chi_{\Lambda(\cG), \cQ(\cG)}$; and one denoted by $\chi$, which we write here as $\chi^{\text{new}}_{\Lambda(\cG), \cQ(\cG)}$. Choosing a strictly perfect representative $P\q$ of $\cC_{S, T}\q$, one has
                \begin{align*}
                    & \ \partial([H^\odd(\cC_{S, T}\q)_\cQ, (t^\alpha)^{-1} \circ t^{\alpha'}]) \\
                    = & \ \chi^{\text{new}}_{\Lambda(\cG), \cQ(\cG)}(\cC_{S, T}\q, (t^\alpha)^{-1}) - \chi^{\text{new}}_{\Lambda(\cG), \cQ(\cG)}(\cC_{S, T}\q, (t^{\alpha'})^{-1}) \\
                    = & \ \left(- \chi_{\Lambda(\cG), \cQ(\cG)}(\cC_{S, T}\q, t^\alpha) - \partial([B^\odd(P\q)_\cQ, -\Id])\right) \\
                     & \quad - \left(- \chi_{\Lambda(\cG), \cQ(\cG)}(\cC_{S, T}\q, t^{\alpha'}) - \partial([B^\odd(P\q)_\cQ, -\Id])\right) \\
                    = & \ \chi_{\Lambda(\cG), \cQ(\cG)}(\cC_{S, T}\q, t^{\alpha'}) - \chi_{\Lambda(\cG), \cQ(\cG)}(\cC_{S, T}\q, t^\alpha),
                \end{align*}
                where the first two equalities are \cite{bb} proposition 5.6 (2) and theorem 6.2, respectively, and $H^\odd$ and $B^\odd$ are as in section \ref{sec:an_integral_trivialisation}. As can be seen in the definition of the integral trivialisation $t_\iota^\alpha = \alpha \varphi \pi$ (equation \eqref{eq:integral_trivialisation}), the difference between $H^\odd(\cC_{S, T}\q)$ and $\cY_{S_\infty}$ is independent of $\alpha$. Since $t^\alpha = (t_\iota^\alpha)_\cQ$ by definition, a simple computation using the group law of $K_0(\Lambda(\cG), \cQ(\cG))$ yields
                \begin{align*}
                    & \partial([H^\odd(\cC_{S, T}\q)_\cQ, (t^\alpha)^{-1} \circ t^{\alpha'}]) \\
                    = & \ [H^\odd(\cC_{S, T}\q)_\cQ, (t^\alpha)^{-1} \circ t^{\alpha'}, H^\odd(\cC_{S, T}\q)_\cQ] \\
                    = & \ [(\cY_{S_\infty})_\cQ, (\alpha_\cQ)^{-1} \circ \alpha_\cQ', (\cY_{S_\infty})_\cQ] \\
                    = & \ \partial([(\cY_{S_\infty})_\cQ, (\alpha_\cQ)^{-1} \circ \alpha_\cQ']),
                \end{align*}
                where the first and last equalities are \eqref{eq:partial_on_k1det}. It follows that $\zeta_{S, T}^{\alpha, \beta}$ satisfies
                \begin{align*}
                    \partial(\zeta_{S, T}^{\alpha, \beta})
                    = \ & \partial(\zeta_{S, T}^{\alpha'}) + \partial([H^\odd(\cC_{S, T}\q)_\cQ, (t^\alpha)^{-1} \circ t^{\alpha'}]) \\
                    = \ & \left(- \chi_{\Lambda(\cG), \cQ(\cG)}(\cC_{S, T}\q, t^{\alpha'})\right) +  \left(\chi_{\Lambda(\cG), \cQ(\cG)}(\cC_{S, T}\q, t^{\alpha'}) - \chi_{\Lambda(\cG), \cQ(\cG)}(\cC_{S, T}\q, t^\alpha)\right) \\
                    = \ & - \chi_{\Lambda(\cG), \cQ(\cG)}(\cC_{S, T}\q, t^\alpha),
                \end{align*}
                and therefore also eMC($L_\infty/K, L, S, T, \alpha, \beta$) by \eqref{eq:change_alpha_interpolating_element}.
            \end{proof}

        \subsection{The choice of $S$ and $T$}
        \label{subsec:the_choice_of_s_and_t}

            When dealing with $L$-functions, groups of units and other objects of arithmetical nature, it is often very convenient to have some flexibility in the choice of the sets of places $S$ and $T$. By way of example, two common techniques are to enlarge $S$ enough that the $S$-class group of some number field vanishes; and to enlarge $T$ enough that the $(S, T)$-units become torsion-free as an abelian group, whether on finite or infinite level. In this subsection, we show one is free to modify these sets without affecting the validity of the Main Conjecture.

            We will treat the case of $S$ first, although most of the preparatory work will overlap with that for $T$. In both instances, the strategy is to enlarge the set in question one place at a time and construct an analogue of the corresponding Euler factor in $K_1(\cQ(\cG))$. In terms of distinguishing features, changing $S$ requires studying the regulator, which is not affected by $T$; whereas modifications in $T$ will entail changes in $\alpha$ (if small), unlike those in $S$. This is in fact the content of the first result:

            \begin{lem}
            \label{lem:difference_galois_modules_s}
                Setting \ref{sett:construction}. Let $v_0 \not \in S \cup T$ be a place of $K$ and set $S' = S \cup \set{v_0}$. Then, the canonical embeddings and surjections of $\Lambda(\cG)$-modules
                \[
                    E_S \ia E_{S'}, \quad E_{S, T} \ia E_{S', T}, \quad X_S^{cs} \sa X_{S'}^{cs} \quad \et{and} \quad X_{T, S}^{cs} \sa X_{T, S'}^{cs}
                \]
                in the notation of section \ref{sec:the_main_complex} (see especially proposition \ref{prop:kernel_cokernel_h0_alpha}) are all equalities.
            \end{lem}

            \begin{proof}
                It is enough to prove the claim for the second and fourth arrows, since the rest correspond to the particular case $T = \varnothing$. We start by considering the exact sequence of $\cG_n$-modules
                \[
                    1 \to \units{\cO_{L_n, S, T}} \to \units{\cO_{L_n, S', T}} \to \bigoplus_{w_n \in \set{v_0}(L_n)} \ZZ \cdot w_n \to Cl_{L_n, S, T} \to Cl_{L_n, S', T} \to 1,
                \]
                where $Cl_{L_n, S, T}$ denotes the $(S, T)$-ray class group of $L_n$ (see the proof of proposition \ref{prop:kernel_cokernel_h0_alpha}), and analogously for  $Cl_{L_n, S', T}$. Here the third arrow is given by the valuation and the next one sends $z \cdot w_n$ to the class $[w_n]^z \in Cl_{L_n, S, T}$. By the same argument which allowed us to deduce \eqref{eq:modified_five_term_sequence_alpha} from \eqref{eq:classical_five_term_sequence_units}, the inverse limit of the previous sequence with respect to the norm maps along the cyclotomic tower $L_\infty/L$ results in the exact sequence of $\Lambda(\cG)$-modules
                \begin{equation}
                \label{eq:five_term_sequence_enlarge_s}
                    0 \to E_{S, T} \to E_{S', T} \to \varprojlim_n \bigoplus_{w_n \in \set{v_0}(L_n)} \ZZ_p \cdot w_n \to X_{T, S}^{cs} \to X_{T, S'}^{cs} \to 0.
                \end{equation}

                We therefore need to show the vanishing of the middle term $\varprojlim_n \bigoplus_{w_n \in \set{v_0}(L_n)} \ZZ_p \cdot w_n$, which coincides with $\Ind_{\cG_v}^\cG \varprojlim_n \ZZ_p \cdot v(L_n)$ (in the notation of setting \ref{sett:construction}) by the commutativity of inverse limits with induction (cf. \eqref{eq:induction_inverse_limit}). Hence, it suffices to consider $\varprojlim_n \ZZ_p \cdot v(L_n)$. The transition maps, induced by the norm on units, are given by $v(L_{n + 1}) \mapsto f_{n + 1 \mid n} \cdot v(L_n)$ (see for instance \cite{ant} proposition III.1.2 (iv)), where $f_{n + 1 \mid n}$ is the inertia degree of $v(L_{n + 1})$ in $L_{n + 1}/L_n$. But the properties of the cyclotomic $\ZZ_p$-extension imply that, for a high enough layer $L_{n_0}$, $v(L_{n_0})$ is inert (i.e. unramified and non-split) in $L_\infty/L_{n_0}$, as $v_0 \in S_f \setminus S_p$. Thus, the transition maps become multiplication by $p$ from the layer $L_{n_0}$ onwards, from which $\varprojlim_n \ZZ_p \cdot v(L_n) = 0$ follows.
            \end{proof}

            \begin{rem}
            \label{rem:difference_galois_modules_s}
                Even though we always assume $S$ to contain all ramified places in $L_\infty/K$, all this proof requires is $v_0 \in S_f \setminus S_p$. In other words, the objects in the statement all coincide with their counterparts defined on the set $\Sigma = S_\infty \cup S_p$, which holds a distinguished place in Iwasawa theory. Another path to the same conclusion is the realisation that it suffices to consider the $\Lambda(\Gamma)$-structure of those modules (i.e. the case $L_\infty = K_\infty, L = K$), and the ramified places in the cyclotomic $\ZZ_p$-extension are precisely the $p$-adic ones.\qedef
            \end{rem}

            Let us explore how the parameters of the Main Conjecture are affected by the addition of a place $v_0  \not \in S \cup T$ to $S$:
            \begin{itemize}
                \item{
                    $L_\infty/K$ and $L$ are defined before introducing $S$ and $T$. The condition $S \supseteq S_\infty \cup S_{\ram}(L_\infty/K)$ and $S \cap T = \varnothing$ automatically implies $S \cup \set{v_0} \supseteq S_\infty \cup S_{\ram}(L_\infty/K)$ and $(S \cup \set{v_0}) \cap T = \varnothing$.
                }
                \item{
                    By the previous lemma, the canonical embedding $E_{S, T} \ia E_{S \cup \set{v_0}, T}$ is in fact an equality. Therefore, the choice of a $\Lambda(\cG)$-homomorphism $\alpha \colon \cY_{S_\infty} \to E_{S, T}$ is exactly the the same as the choice of an $\alpha \colon \cY_{S_\infty} \to E_{S \cup \set{v_0}, T}$.
                }
                \item{
                    The choice of $\chi \in \Irr_p(\cG)$ is independent of $S$. Note that $n(S \cup \set{v_0}) \geq n(S)$, and therefore $n(S \cup \set{v_0}, \alpha) \geq n(S, \alpha)$ if the same $\alpha$ is chosen as explained in the previous point. This means $\cK_{S \cup \set{v_0}}^\alpha(\chi) \subseteq \cK_S^\alpha(\chi)$ in general, but the difference only consists of finitely many characters by lemma \ref{lem:properties_of_r_chi}.
                }
                \item{
                    $\beta \colon \CC_p \xrightarrow{\sim} \CC$ is unrelated to $S$.
                }
            \end{itemize}

            The question of how the Main Conjecture varies under enlargement of $S$ must be addressed on both the analytic side, concerning the existence of series quotients interpolating regulated special $L$-values; and the algebraic one, in terms of the refined Euler characteristic of the arithmetic complex $\cC_{S, T}\q$. We start with the latter.

            The key homological result in our study of refined Euler characteristics is their additivity in exact triangles of perfect complexes with compatible trivialisations. We are thus tasked with finding an exact triangle measuring the difference between $\cC_{S, T}\q$ and $\cC_{S \cup \set{v_0}, T}\q$ for $v_0 \not \in S \cup T$. The proof will involve answering the analogous question for $\cC_{S, T}\q$ and $\cC_{S, T \cup \set{v_0}}\q$. We note that, although the argument below relies on the explicit construction of the complex from chapter \ref{chap:construction_of_the_complex}, it is likely an alternative (and potentially simpler) proof exists using the $R\Gamma$-formulation of section \ref{sec:description_in_terms_of_rgamma_complexes}.

            \begin{prop}
            \label{prop:exact_triangle_difference_s_and_t}
                Setting \ref{sett:construction}. Let $v_0 \not \in S \cup T$ be a place of $K$. Then there exist exact triangles\footnote{The second triangle may seem obvious in light of definition \ref{defn:complexes_bks} and theorem \ref{thm:iso_bks_complex} - but this theorem was proved precisely on the promise that we would show such a relation later on.}
                \begin{equation}
                \label{eq:exact_triangle_difference_s_and_t}
                    \cC_{S, T}\q \to \cC_{S \cup \set{v_0}, T}\q \to \Ind_{\cG_{v_0}}^{\cG} \ZZ_p [-1] \to
                \end{equation}
                and
                \begin{equation}
                    \label{eq:exact_triangle_difference_t}
                    \cC_{S, T \cup \set{v_0}}\q \to \cC_{S, T}\q \to \overline{\Ind}_{\cG_{v_0}}^\cG \ZZ_p(1)[0] \to
                \end{equation}
                in the derived category $\cD(\Lambda(\cG))$, where
                \[
                    \overline{\Ind}_{\cG_{v_0}}^\cG \ZZ_p(1) =
                    \begin{cases}
                        \Ind_{\cG_{v_0}}^\cG \ZZ_p(1), \quad &L_{v_0(L)} \et {contains a primitive} p\text{-th root of unity} \\
                        0, &\text{otherwise}.
                    \end{cases}
                \]
            \end{prop}

            \begin{proof}
                Let us denote $S \cup \set{v_0}$ and $T \cup \set{v_0}$ by $S'$ and $T'$, respectively.  We will make use of the fact that $S \cup T' = S' \cup T$, and in particular, $\cT_{S \cup T'}\q = \cT_{S' \cup T}\q$. There are no obvious non-trivial morphisms between $\cC_{S, T}\q$ and $\cC_{S', T}\q$ on the level of complexes. On the global side, the natural map goes in the direction $\cT_{S' \cup T}\q\to \cT_{S \cup T}\q$ (this induces, for instance, the projection ${X_{S' \cup T} \sa X_{S \cup T}}$ on $H^0$). In the case of the $\cC\q$-complexes, however, the canonical maps on cohomology (cf. theorem \ref{thm:cohomology_of_complex}) are reversed - most notably $\cX_S \ia \cX_{S'}$.

                We solve this issue by considering the auxiliary complex
                \[
                    \cC_{S, T'}\q = \Cone(\alpha_{S, T'} \colon \cL_S\q \to \cT_{S \cup T'}\q)[-1]
                \]
                in the notation of section \ref{sec:the_main_complex}. Recall that $\alpha_{S, T'}$ is induced by $G_{K_v} = (G_K)_{v^c} \ia G_K \sa G_{S \cup T}$ in degree 0 (equation \eqref{eq:local_to_global_maps_galois}) and given by the natural map $\big(\bigoplus_{v \in S} \Lambda(\cG)\big) \sa \Lambda(\cG)$ in degree 1. Our aim is to define \textit{morphisms of morphisms of complexes} from $\alpha_{S, T'}$ to $\alpha_{S, T}$ and $\alpha_{S', T}$ - that is, four morphisms of complexes $\varepsilon_\cL, \varepsilon_\cT, \mu_\cL$ and $\mu_\cT$ making the two diagrams
                \begin{equation}
                \label{eq:diagrams_morphisms_of_morphisms}
                    \begin{tikzcd}
                        \cL_S\q \arrow[r, "{\alpha_{S, T'}}"] \arrow[d, "\varepsilon_\cL"'] & \cT_{S \cup T'}\q \arrow[d, "\varepsilon_\cT"'] &  & \cL_S\q \arrow[r, "{\alpha_{S, T'}}"] \arrow[d, "\mu_\cL"'] & \cT_{S \cup T'}\q \arrow[d, "\mu_\cT"'] \\
                        \cL_S\q \arrow[r, "{\alpha_{S, T}}"]                                & \cT_{S \cup T}\q                                &  & \cL_{S'}\q \arrow[r, "{\alpha_{S', T}}"]                              & \cT_{S' \cup T}\q
                    \end{tikzcd}
                \end{equation}
                commute. Set $\varepsilon_\cL = \mu_\cT = \Id$. As $\mu_\cL$, we choose the natural morphism
                \[
                    \mu_\cL \colon \cL_S\q = \bigoplus_{v \in S} \Ind_{\cG_v}^\cG \cL_v\q \to \bigoplus_{v \in S'} \Ind_{\cG_v}^\cG \cL_v\q = \cL_{S'}\q.
                \]
                In the case of $\varepsilon_\cT$, we specify the morphism on each degree, of which there are two non-trivial ones. In degree 0, we let $\varepsilon_\cT^0 \colon Y_{S' \cup T} \sa Y_{S \cup T}$ be the canonical surjection induced by $G_{S' \cup T} \sa G_{S \cup T}$ (cf. section \ref{sec:the_global_complex}). As for degree 1, take $\varepsilon_\cT^1 = \Id \colon \Lambda(\cG) \to \Lambda(\cG)$.

                A number of verifications are in order, namely whether those maps are indeed morphisms of complexes and whether \eqref{eq:diagrams_morphisms_of_morphisms} commutes. In the case of $\varepsilon$, this follows from the commutativity of
                \begin{center}
                    \begin{tikzcd}[column sep=tiny, row sep=small]
                        {[Y_{S \cup T'}} \arrow[rd, "\varepsilon_\cT^0", two heads] \arrow[rr]                                                  &                                                                                      & {\Lambda(\cG)]} \arrow[rd, "\varepsilon_{\cT}^1", equals]                                                    &                                                                     \\
                                                                                                                                                & {[Y_{S \cup T}} \arrow[rr]                                                             &                                                                                                      & {\Lambda(\cG)]}                                                     \\
                        {\bigg[\bigoplus_{v \in S}} \Ind_{\cG_v}^\cG Y_v \arrow[uu, "{\alpha_{S, T'}^0}"] \arrow[rr] \arrow[rd, "\varepsilon_\cL^0", equals] &                                                                                      & {\bigoplus_{v \in S} \Lambda(\cG)\bigg]} \arrow[uu, "{\alpha_{S, T'}^1}", pos=0.82] \arrow[rd, "\varepsilon_\cL^1", equals] &                                                                     \\
                                                                                                                                                & {\bigg[\bigoplus_{v \in S} \Ind_{\cG_v}^\cG Y_v} \arrow[uu, "{\alpha_{S, T}^0}", pos=0.85] \arrow[rr] &                                                                                                      & {\bigoplus_{v \in S} \Lambda(\cG)\bigg]} \arrow[uu, "{\alpha_{S, T}^1}"]
                    \end{tikzcd}
                \end{center}
                The only two squares which deserve mention are the top and left ones, which hinge on the compatibility of Galois restriction $G_K \sa G_{S \cup T'} \sa G_{S \cup T} \sa \cG$. A similar diagram proves the analogous claims for $\mu_\cL$ and $\mu_\cT$. Although immediate, these checks should not be neglected - for instance, commutativity fails for the very natural choice $\nu_\cL \colon \cL_{S'}\q \to \cL_S\q$ and $\nu_\cT = \varepsilon_\cT$.

                A morphism of morphisms of complexes induces a morphism between their cones, which in our case implies the existence of maps
                \[
                    \varepsilon_\cC \colon \Cone(\alpha_{S, T'}) \to \Cone(\alpha_{S, T}) \quad \text{and} \quad \mu_\cC \colon  \Cone(\alpha_{S, T'}) \to \Cone(\alpha_{S', T}).
                \]
                Our interest now lies in the cones of these new maps. Since the horizontal arrows in \eqref{eq:diagrams_morphisms_of_morphisms} are more complicated than the vertical ones, we resort to the following property of triangulated categories: the four cones resulting from a commutative square of morphisms fit into a nine-term diagram with two new exact triangles connecting them. In the case of $\varepsilon$, this translates into
                \begin{equation}
                \label{eq:nine_diagram_change_t}
                    \begin{tikzcd}
                        \cL_S\q \arrow[r, "{\alpha_{S, T'}}"] \arrow[d, "\varepsilon_\cL"'] & \cT_{S \cup T'}\q \arrow[d, "\varepsilon_\cT"'] \arrow[r] & {\Cone(\alpha_{S, T'})} \arrow[d, dashed, "\varepsilon_\cC"] \arrow[r]                   & {} \\
                        \cL_S\q \arrow[r, "{\alpha_{S, T}}"] \arrow[d]                      & \cT_{S \cup T}\q \arrow[d] \arrow[r]                      & {\Cone(\alpha_{S, T})} \arrow[d, dashed] \arrow[r] \arrow[d, dashed] & {} \\
                        \Cone(\varepsilon_\cL) \arrow[d] \arrow[r, dashed]                  & \Cone(\varepsilon_\cT) \arrow[d] \arrow[r, dashed, "\iota"]        & \Cone(\varepsilon_\cC) \arrow[d, dashed] \arrow[r, dashed]            & {} \\
                        {}                                                                  & {}                                                        & {}                                                                    &
                    \end{tikzcd}
                \end{equation}
                where all arrows exist on level of complexes except possibly $\iota$. The map $\varepsilon_\cL = \Id$ has trivial cone, and hence $\iota \colon \Cone(\varepsilon_\cT) \isoa \Cone(\varepsilon_\cC)$ is an isomorphism in $\cD(\Lambda(\cG))$. Consider the cohomology of $\Cone(\varepsilon_\cT)$, which fits into the exact sequence
                \begin{center}
                    \begin{tikzcd}
                        0 \arrow[r] & H^{-1}(\Cone(\varepsilon_\cT)) \arrow[r] & H^0(\cT_{S \cup T'}\q) \arrow[r, "H^0(\varepsilon_\cT)"] & H^0(\cT_{S \cup T}\q) \arrow[r] & H^0(\Cone(\varepsilon_\cT)) \arrow[lld, out=0, in=180, looseness=1.5, overlay] &   \\
                                    &             & H^1(\cT_{S \cup T'}\q) \arrow[r, "H^1(\varepsilon_\cT)"] & H^1(\cT_{S \cup T}\q) \arrow[r] & H^1(\Cone(\varepsilon_\cT)) \arrow[r]   & 0
                    \end{tikzcd}
                \end{center}
                and is trivial elsewhere because $\cT_{S \cup T'}\q$ and $\cT_{S \cup T}\q$ are acyclic outside degrees 0 and 1 (proposition \ref{prop:global_complex}). The maps induced by $\varepsilon_\cT$ on cohomology are the canonical projection ${H^0(\varepsilon_\cT) \colon X_{S \cup T'} \sa X_{S \cup T}}$ and the identity $H^1(\varepsilon_\cT) \colon \ZZ_p \xrightarrow{\Id} \ZZ_p$. Thus, $H^0(\Cone(\varepsilon_\cT)) = H^1(\Cone(\varepsilon_\cT)) = 0$ and $H^{-1}(\Cone(\varepsilon_\cT))$ is isomorphic to $\ker(X_{S \cup T'} \sa X_{S \cup T})$. This kernel already appeared in the proof of proposition \ref{prop:kernel_cokernel_h0_alpha} as the term $\ker(\pi)$, where it was shown to be precisely $\overline{\Ind}_{\cG_{v_0}}^\cG \ZZ_p(1)$ in our newly introduced notation. It follows that $\Cone(\varepsilon_\cC)$ itself has cohomology concentrated in at most one degree, which in turn implies
                \begin{equation}
                \label{eq:cone_varepsilon_c}
                    \Cone(\varepsilon_\cC) \iisoo \overline{\Ind}_{\cG_{v_0}}^\cG \ZZ_p(1)[1]
                \end{equation}
                in $\cD(\Lambda(\cG))$. Shifting the triangle $\Cone(\alpha_{S, T'}) \to \Cone(\alpha_{S, T}) \to \Cone(\varepsilon_\cC) \to$ by $-1$ yields \eqref{eq:exact_triangle_difference_t}.

                We use the same technique to compute $\Cone(\mu_\cC)$. An analogous 9-term diagram shows
                \[
                    \Cone(\mu_\cC) \iso \Cone\big(\Cone(\mu_\cL) \to \Cone(\mu_\cT)\big) \iso \Cone(\mu_\cL)[1],
                \]
                where the term $\Cone(\mu_\cL)$ can be determined directly using the obvious short exact sequence of complexes
                \[
                    0 \to \cL_S\q \xrightarrow{\mu_\cL} \cL_{S'}\q \to \Ind_{\cG_{v_0}}^\cG \cL_{v_0}\q \to 0.
                \]
                This results in an isomorphism $\Cone(\mu_\cC) \iso \Ind_{\cG_{v_0}}^\cG \cL_{v_0}\q[1] = [\stackrel{-1}{\Ind_{\cG_{v_0}}^\cG Y_{v_0}} \to \stackrel{0}{\Lambda(\cG)}]$ in $\cD(\Lambda(\cG))$ in the notation of section \ref{sec:the_local_complexes}. The cohomology of the local complexes was established in proposition \ref{prop:local_complexes} and commutes with induction by \eqref{eq:local_cohomology_compatible_induction}, so we have
                \[
                    H^i(\Cone(\mu_\cC)) \iisoo \Ind_{\cG_{v_0}}^\cG H^{i+1}(\cL_{v_0}\q) =
                    \begin{cases}
                        \Ind_{\cG_{v_0}}^\cG G_{L_\infty, v_0}^{ab}(p), & i = -1 \\
                        \Ind_{\cG_{v_0}}^\cG \ZZ_p, & i = 0 \\
                        0 & \text{otherwise.}
                    \end{cases}
                \]

                Now class field theory brings $\Cone(\varepsilon_\cC)$ and $\Cone(\mu_\cC)$ together by virtue of the $\Lambda(\cG)$-isomorphism
                \[
                    \overline{\Ind}_{\cG_{v_0}}^\cG \ZZ_p(1) \iso \Ind_{\cG_{v_0}}^\cG G_{L_\infty, v_0}^{ab}(p) = \ker(\Ind_{\cG_{v_0}}^\cG Y_{v_0} \to \Lambda(\cG))
                \]
                (cf. \cite{nsw} theorem 11.2.3 (ii)), which allows us to define a morphism
                \[
                    \varphi \colon \Cone(\varepsilon_\cC) \isoa
                    \overline{\Ind}_{\cG_{v_0}}^\cG \ZZ_p(1)[1] \to
                    [\stackrel{-1}{\Ind_{\cG_{v_0}}^\cG Y_{v_0}} \ \to \ \stackrel{0}{\Lambda(\cG)}]  =
                    \Ind_{\cG_{v_0}}^\cG \cL_{v_0}\q[1] \isoa
                    \Cone(\mu_\cC)
                \]
                in $\cD(\Lambda(\cG))$. The cone of $\varphi$ is acyclic outside degree 0 and thus isomorphic in the derived category to
                \[
                    \Cone(\varphi) \iso H^0(\Cone(\varphi))[0] \iso H^0(\Ind_{\cG_{v_0}}^\cG \cL_{v_0}\q)[0] = \Ind_{\cG_{v_0}}^\cG \ZZ_p[0].
                \]

                We conclude the argument by considering the commutative diagram
                \begin{center}
                    \begin{tikzcd}
                        {\Cone(\varepsilon_\cC)[-1]} \arrow[r] \arrow[d, "{\varphi[-1]}"] & {\Cone(\alpha_{S, T'})} \arrow[r, "\varepsilon_\cC"] \arrow[d, equals] & {\Cone(\alpha_{S, T})} \arrow[r]  & {} \\
                        {\Cone(\mu_\cC)[-1]} \arrow[r]                                           & {\Cone(\alpha_{S, T'})} \arrow[r, "\mu_\cC"]           & {\Cone(\alpha_{S', T})} \arrow[r] & {}
                    \end{tikzcd}
                \end{center}
                where the rows are rotated versions of natural exact triangles in $\cD(\Lambda(\cG))$. This induces a third vertical morphism $\kappa \colon \Cone(\alpha_{S, T}) \to \Cone(\alpha_{S', T})$ whose cone we can compute using the same technique as before:
                \[
                    \Cone(\kappa) \iisoo \Cone(\varphi[-1])[1] = \Cone(\varphi) \iso \Ind_{\cG_{v_0}}^\cG \ZZ_p[0].
                \]
                The desired triangle \eqref{eq:exact_triangle_difference_s_and_t} is now a shift of $\Cone(\alpha_{S, T}) \xrightarrow{\kappa} \Cone(\alpha_{S', T}) \to \Cone(\kappa) \to$.
            \end{proof}

            In order to relate the refined Euler characteristics of $\cC_{S, T}\q$ and $\cC_{S \cup \set{v_0}, T}\q$ using the above proposition, we must trivialise them by choosing a map $\alpha \colon \cY_{S_\infty} \to E_{S, T} = E_{S \cup \set{v_0}, T}$ as in setting \ref{sett:formulation}. Note that, even though the same $\alpha$ can be employed for both complexes, the trivialisations $t^\alpha$ themselves will differ (cf. \eqref{eq:integral_trivialisation}), since the cohomology in degree 1 does not coincide.

            \begin{cor}
            \label{cor:difference_euler_characteristics_change_s}
                Setting \ref{sett:formulation}. Let $v_0 \not \in S \cup T$ be a place of $K$ and denote by $t^\alpha$ and $t_0^\alpha$ the trivialisations constructed in section \ref{sec:an_integral_trivialisation} for the complexes $\cC_{S, T}\q$ and $\cC_{S \cup \set{v_0}, T}\q$, respectively. Then one has
                \[
                    \chi_{\Lambda(\cG), \cQ(\cG)}(\cC_{S \cup \set{v_0}, T}\q, t_0^\alpha) = \chi_{\Lambda(\cG), \cQ(\cG)}(\cC_{S, T}\q, t^\alpha) - [\Lambda(\cG), (1 - \varphi_{v_0(L_\infty)}^{-1})_r, \Lambda(\cG)]
                \]
                in $K_0(\Lambda(\cG), \cQ(\cG))$, where $\varphi_{v_0(L_\infty)} \in \cG_{v_0}$ is the Frobenius automorphism at $v_0(L_\infty)$ and $(1 - \varphi_{v_0(L_\infty)}^{-1})_r$ denotes right multiplication by $1 - \varphi_{v_0(L_\infty)}^{-1}$.
            \end{cor}

            \begin{proof}
                The first step is to show that $[\stackrel{0}{\Lambda(\cG)} \xrightarrow{(1 - \varphi_{v_0(L_\infty)}^{-1})_r} \stackrel{1}{\Lambda(\cG)}]$ is a strictly perfect representative of $\Ind_{\cG_{v_0}}^{\cG} \ZZ_p [-1]$. Consider the endomorphism $(1 - \varphi_{v_0(L_\infty)}^{-1})_r$ of the regular left $\Lambda(\cG_{v_0})$-module, which maps $x \in \Lambda(\cG_{v_0})$ to $x (1 - \varphi_{v_0(L_\infty)}^{-1})$. Recall that $\cG_{v_0}$ is notation for $\cG_{v_0(L_\infty)}$, which is a procyclic group topologically generated by $\varphi_{v_0(L_\infty)}$ (because $v_0 \not \in S_{\ram}(L_\infty/K)$), or equivalently, by $\varphi_{v_0(L_\infty)}^{-1}$. In particular, $\Lambda(\cG_{v_0})$ is a commutative ring. The sequence of (topological left) $\Lambda(\cG_{v_0})$-modules
                \begin{equation}
                \label{eq:strictly_perfect_representative_local}
                    0 \to \Lambda(\cG_{v_0}) \xrightarrow{(1 - \varphi_{v_0(L_\infty)}^{-1})_r} \Lambda(\cG_{v_0}) \to \ZZ_p \to 0
                \end{equation}
                is exact: the third arrow is simply the augmentation map and injectivity follows from taking inverse limits on the exact sequence
                \[
                    0 \to \ZZ_p \cdot \Tr_{\cG_{n, v_0}} \to \Lambda(\cG_{n, v_0}) \xrightarrow{(1 - \varphi_{v_0(L_n)}^{-1})_r} \Lambda(\cG_{n, v_0})
                \]
                along the cyclotomic tower. Here $\cG_{n, v_0}$ is the decomposition group of $v_0(L_n)$ in $L_n/K$, which is a finite cyclic group, and the transition maps on the first term send $\Tr_{\cG_{n + 1, v_0}} \in \Lambda(\cG_{n + 1, v_0})$ to $p \cdot \Tr_{\cG_{n, v_0}} \in \Lambda(\cG_{n, v_0})$ for $n$ large enough (because $v_0$ splits into finitely many places in $L_\infty$). The left-exactness of the inverse limit implies $\ker((1 - \varphi_{v_0(L_\infty)}^{-1})_r) = \varprojlim_n \ZZ_p \cdot \Tr_{\cG_{n, v_0}} = 0$.

                This yields a quasi-isomorphism of complexes of $\Lambda(\cG_{v_0})$-modules
                \[
                    [\stackrel{0}{\Lambda(\cG_{v_0})} \xrightarrow{(1 - \varphi_{v_0(L_\infty)}^{-1})_r} \stackrel{1}{\Lambda(\cG_{v_0})}] \to \ZZ_p [-1]
                \]
                which can then be induced to $\cG$ by applying compact induction (introduced in section \ref{sec:iwasawa_algebras_and_modules}). We have ${\Ind_{\cG_{v_0}}^\cG \Lambda(\cG_{v_0}) = \Lambda(\cG)}$ and $\Ind_{\cG_{v_0}}^\cG (1 - \varphi_{v_0(L_\infty)}^{-1})_r = (1 - \varphi_{v_0(L_\infty)}^{-1})_r$. Note that $\Lambda(\cG)$ is no longer abelian and thus the left-right distinction becomes relevant. The functor $\Ind_{\cG_{v_0}}^\cG -$ is exact on \eqref{eq:strictly_perfect_representative_local} for the reasons described in section \ref{sec:local-to-global_maps}, which shows that
                \[
                    [\stackrel{0}{\Lambda(\cG)} \xrightarrow{(1 - \varphi_{v_0(L_\infty)}^{-1})_r} \stackrel{1}{\Lambda(\cG)}] \iisoo \Ind_{\cG_{v_0}}^{\cG} \ZZ_p [-1]
                \]
                in $\cD(\Lambda(\cG))$ - and, in particular, that the right-hand side is perfect. Since $\Ind_{\cG_{v_0}}^{\cG} \ZZ_p$ is $\Lambda(\cG)$-torsion, the zero map constitutes a trivialisation for $\Ind_{\cG_{v_0}}^{\cG} \ZZ_p[-1]$, with associated refined Euler characteristic
                \begin{equation}
                \label{eq:difference_s_rec}
                    \chi_{\Lambda(\cG), \cQ(\cG)}(\Ind_{\cG_{v_0}}^{\cG} \ZZ_p[-1], 0) = - [\Lambda(\cG), (1 - \varphi_{v_0(L_\infty)}^{-1})_r, \Lambda(\cG)]
                \end{equation}
                according to \eqref{eq:definition_rec_general}. The negative sign appears because the map $(1 - \varphi_{v_0(L_\infty)}^{-1})_r$ goes from even to odd degree.

                We must now combine the exact triangle \eqref{eq:exact_triangle_difference_s_and_t} with the additivity of refined Euler characteristics studied in \cite{bb}. Although corollary 6.6 therein is specific to the semisimple case, it cannot be applied here, as it requires injectivity of the reduced norm $\nr \colon K_1(\cQ(\cG)) \to \units{Z(\cQ(\cG))}$ (recall at this point remark \ref{rem:emcu} i)). Instead, we rely on theorem 5.7 from the same source. As in the proof of proposition $\ref{prop:independence_of_alpha}$, two versions of refined Euler characteristics must be considered, which we again denote by $\chi_{\Lambda(\cG), \cQ(\cG)}$ (the one in section \ref{sec:an_integral_trivialisation}) and $\chi^{\text{new}}_{\Lambda(\cG), \cQ(\cG)}$. Of particular importance is the fact that trivialisations for $\chi^{\text{new}}_{\Lambda(\cG), \cQ(\cG)}$ are maps from even to odd degree, as opposed to those for $\chi_{\Lambda(\cG), \cQ(\cG)}$.

                Let $S' = S \cup \set{v_0}$. In order to apply theorem 5.7 to \eqref{eq:exact_triangle_difference_s_and_t}, we need to verify the commutativity hypothesis on diagram (15) from the cited article. The square brackets therein denote a certain \textit{universal determinant functor} $[ - ] \colon \cD(\cQ(\cG)) \to \cV(\cQ(\cG))$ from the derived category of left $\cQ(\cG)$-modules to a Picard category $\cV(\cQ(\cG))$. The following facts will be used:
                \begin{itemize}
                    \item{
                        $\cQ(\cG) \otimes_{\Lambda(\cG)} H^\odd(\Ind_{\cG_{v_0}}^{\cG} \ZZ_p [-1]) = \cQ(\cG) \otimes_{\Lambda(\cG)} H^\even(\Ind_{\cG_{v_0}}^{\cG} \ZZ_p [-1]) = 0$. Therefore, the last term in each row of the diagram is $[0]$, the unit object $\bbone \in \cV(\cQ(\cG))$. The vertical arrow between them is $[0_{\Hom}]$, where $0_{\Hom}$ denotes the zero homomorphism.
                    }
                    \item{
                        The homomorphisms induced by the arrow $\cC_{S, T}\q \to \cC_{S', T}$ in cohomology are injective: in degree 0, it is simply the embedding $H^{0}(\cC_{S, T}\q) \iso E_{S, T} \ia E_{S', T} \iso H^{0}(\cC_{S', T}\q)$; and in degree 1, it fits into the diagram
                        \begin{center}
                            \begin{tikzcd}
                                0 \arrow[r] & {X_{T, S}^{cs}} \arrow[d, two heads] \arrow[r] & {H^1(\cC_{S, T}\q)} \arrow[r] \arrow[d] & \cX_S \arrow[d, hook] \arrow[r] & 0 \\
                                0 \arrow[r] & {X_{T, S'}^{cs}} \arrow[r]                     & {H^1(\cC_{S', T}\q)} \arrow[r]          & \cX_{S'} \arrow[r]              & 0
                            \end{tikzcd}
                        \end{center}
                        where the left vertical arrow is in fact an equality by lemma \ref{lem:difference_galois_modules_s}.
                    }
                \end{itemize}
                This shows that the second and third terms in each row of diagram (15) are again $[0] = \bbone$, and the vertical arrows $[\Id]$ and $[-\Id]$ are both equal to $[0_{\Hom}]$. Since functors preserve commutativity, it suffices to verify that of
                \begin{center}
                    \begin{tikzcd}
                        {\cQ(\cG) \otimes_{\Lambda(\cG)} H^\even(\cC_{S, T}\q)} \arrow[r] \arrow[d, "(t^\alpha)^{-1}"] & {\cQ(\cG) \otimes_{\Lambda(\cG)} H^\even(\cC_{S', T}\q)} \arrow[d, "(t_0^\alpha)^{-1}"] \\
                        {\cQ(\cG) \otimes_{\Lambda(\cG)} H^\odd(\cC_{S, T}\q)} \arrow[r]          & {\cQ(\cG) \otimes_{\Lambda(\cG)} H^\odd(\cC_{S', T}\q)}
                    \end{tikzcd}
                \end{center}
                which follows immediately from the commutative diagram
                \begin{center}
                    \begin{tikzcd}
                        {H^1(\cC_{S, T}\q)} \arrow[d] \arrow[r, two heads] & \cX_S \arrow[r, two heads] \arrow[d, hook] & \cY_{S_\infty} \arrow[d, equals] \arrow[r, "\alpha", hook] & {{E_{S, T}}} \arrow[d, equals] \\
                        {H^1(\cC_{S', T}\q)} \arrow[r, two heads]          & \cX_{S'} \arrow[r, two heads]              & \cY_{S_\infty} \arrow[r, "\alpha", hook]           & {E_{S', T}}
                    \end{tikzcd}
                \end{center}
                where all arrows become isomorphisms after $\cQ(\cG) \otimes_{\Lambda(\cG)} -$ (cf. \eqref{eq:integral_trivialisation}). We can therefore apply \cite{bb} theorem 5.7 and conclude that
                \begin{equation}
                     \chi_{\Lambda(\cG), \cQ(\cG)}^{\text{new}}(\cC_{S', T}\q, (t_0^\alpha)^{-1}) = \chi_{\Lambda(\cG), \cQ(\cG)}^{\text{new}}(\cC_{S, T}\q, (t^\alpha)^{-1}) + \chi_{\Lambda(\cG), \cQ(\cG)}^{\text{new}}(\Ind_{\cG_{v_0}}^{\cG} \ZZ_p [-1], 0).
                \end{equation}

                Now the same argument as in the proof of proposition \ref{prop:independence_of_alpha} yields the corresponding relation on our usual refined Euler characteristics:
                \[
                \chi_{\Lambda(\cG), \cQ(\cG)}(\cC_{S', T}\q, t_0^\alpha) = \chi_{\Lambda(\cG), \cQ(\cG)}(\cC_{S, T}\q, t^\alpha) + \chi_{\Lambda(\cG), \cQ(\cG)}(\Ind_{\cG_{v_0}}^{\cG} \ZZ_p [-1], 0).
                \]
                This concludes the proof by \eqref{eq:difference_s_rec}.
            \end{proof}

            The above results offer a homological comparison of the Main Conjectures defined on $S$ and $S \cup \set{v_0}$. Over on the analytic side, one needs to determine the change in leading coefficients and regulators caused by the addition of the prime $v_0$. The difference between the Artin $L$-series is given by the local Euler factor $L_{v_0}$ defined in \eqref{eq:local_L-factors}, which \textit{almost never} vanishes at $s = 0$:

            \begin{lem}
            \label{lem:change_s_vanishing_euler_factors}
                Setting \ref{sett:construction}, $\beta$ as in setting \ref{sett:formulation}. Let $v_0 \not\in S_\ram(L_\infty/K)$ be a place of $K$ and $\chi$ an Artin character of $\cG$. Then, for almost all $\rho$ of type $W$, one has
                \[
                    L_{v_0}(\chi \otimes \rho, 0)^{-1} = \det(1 - \varphi_{v_0(L_n)} \mid V_{\beta (\chi \otimes \rho)}) \neq 0
                \]
                in the notation of \eqref{eq:local_L-factors}, where $n$ is chosen large enough that $\Gp{n} \subseteq \ker(\chi \otimes \rho)$.
            \end{lem}

            \begin{proof}
                The proof is almost immediate and can be done entirely over $\CC$. To see this, note that there is an equality of sets
                \[
                    \set{\beta (\chi \otimes \rho)}_\rho = \set{(\beta \chi) \otimes \tilde{\rho}}_{\tilde{\rho}},
                \]
                where $\rho$ runs over the type-$W$ characters of $\cG$ ($p$-adic by definition) and $\tilde{\rho}$ runs over the $\CC$-valued linear characters of $\cG$ with open kernel which are trivial on $H$.

                For $\rho$ as above, set $\tilde{\chi} = \beta \chi$ and $\tilde{\rho} = \beta \rho$ and choose $n = n(\rho) \in \NN$ such that $\tilde{\chi}$ and $\tilde{\rho}$ factor over $\cG_n$. Then
                \[
                    \det(1 - \varphi_{v_0(L_n)} \mid V_{\tilde{\chi} \otimes \tilde{\rho}}) = \det(1 - \tilde{\rho}(\varphi_{v_0(L_n)}) M_{\tilde{\chi}}(\varphi_{v_0(L_n)})),
                \]
                where $M_{\tilde{\chi}}(\varphi_{v_0(L_n)})$ is the matrix describing the action of $\varphi_{v_0(L_n)}$ on the $\CC[\cG_n]$-module $V_{\tilde{\chi}}$. We now note that the right-hand side is the value of the polynomial $f(t) = \det(1 - t M_{\tilde{\chi}}(\varphi_{v_0(L_n)}))$ at $t = \tilde{\rho}(\varphi_{v_0(L_n)}) \in \units{\CC}$. This $f$ (which is almost the characteristic polynomial of $M_{\tilde{\chi}}(\varphi_{v_0(L_n)})$) has degree $\chi(1) = \tilde{\chi}(1)$ and is independent of $\tilde{\rho}$ (even though $n$ may change, the matrix of the action of $\varphi_{v_0(L_n)}$ on $V_{\tilde{\chi}}$ does not).

                Now the result follows from the fact that $f(t)$ has finitely many roots $x \in \CC$ and, for any such $x$, only finitely many $\tilde{\rho}$ as above satisfy $\tilde{\rho}(\varphi_{v_0(L_\infty)}) = x$. To see why, fix one such $\tilde{\rho}$ and note that every other type-$W$ (and hence irreducible) character with that property factors through the finite group $\cG/(\cG_{v_0} \cap \ker(\tilde{\rho}))$.
            \end{proof}

            The last preparatory step to show the independence of the conjectures of the choice of $S$ concerns the Stark-Tate regulator:
            \begin{lem}
            \label{lem:change_s_regulators}
                Setting \ref{sett:formulation}. Let $v_0 \not\in S \cup T$ be a place of $K$ and $\chi \in \Irr_p(\cG)$. Then one has
                \begin{equation}
                \label{eq:change_s_equality_of_regulators}
                    R_S^\beta(\alpha, \chi \otimes \rho) = R_{S \cup \set{v_0}}^\beta(\alpha, \chi \otimes \rho)
                \end{equation}
                for almost all $\rho \in \cK_{S \cup \set{v_0}}^\alpha(\chi) \subseteq \cK_S^\alpha(\chi)$.
            \end{lem}

            \begin{proof}
                Set $S' = S \cup \set{v_0}$. The canonical short exact sequence of $\Lambda(\cG)$-modules
                \[
                    0 \to \cX_S \to \cX_{S'} \to \cY_{\set{v_0}} \to 0,
                \]
                which remains exact on finite level, induces an embedding
                \[
                    \Hom_{\CC_p[\cG_n]}(V_{\chi \otimes \rho}, \CC_p \otimes_{\ZZ_p} \cX_{L_n, S}) \ia \Hom_{\CC_p[\cG_n]}(V_{\chi \otimes \rho}, \CC_p \otimes_{\ZZ_p} \cX_{L_n, S'})
                \]
                for all $\rho \in \cK_{S \cup \set{v_0}}^\alpha(\chi)$ and sufficiently large $n$ depending on $\rho$. We claim this is an equality, that is, $\Hom_{\CC_p[\cG_n]}(V_{\chi \otimes \rho}, \CC_p \otimes_{\ZZ_p} \cY_{L_n, \set{v_0}}) = 0$, for almost all such $\rho$. As explained in the proof of lemma \ref{lem:r_chi_independent_of_n}, this amounts to $\ideal{\chi \otimes \rho, \Ind_{\cG_{n, v_0}}^{\cG_n} \bbone_{\cG_{n, v_0}}}_{\cG_n} = 0$, since  $\Ind_{\cG_{n, v_0}}^{\cG_n} \bbone_{\cG_{n, v_0}}$ is the character afforded by the $\CC_p[\cG_n]$-module $\CC_p \otimes_{\ZZ_p} \cY_{L_n, \set{v_0}}$. Recall at this point that $\cG_{n, v_0}$ denotes $(\cG_n)_{v_0(L_n)}$. By \eqref{eq:frobenius_reciprocity} and \eqref{eq:scalar_product_dual}, one has
                \[
                    \ideal{\chi \otimes \rho, \indu_{\cG_{n, v_0}}^{\cG_n} \bbone_{\cG_{n, v_0}}}_{\cG_n} = \ideal{\rest_{\cG_{n, v_0}}^{\cG_n} \chi \otimes \rho,  \bbone_{\cG_{n, v_0}}}_{\cG_{n, v_0}} = \ideal{\rest_{\cG_{n, v_0}}^{\cG_n} \rho, \rest_{\cG_{n, v_0}}^{\cG_n} \check{\chi}}_{\cG_{n, v_0}}.
                \]
                The restriction $\rest_{\cG_{n, v_0}}^{\cG_n} \rho$ is a linear, and therefore irreducible, character of $\cG_{n, v_0}$. This implies that the last scalar product is non-zero if and only if $\rest_{\cG_{n, v_0}}^{\cG_n} \rho$ is one of the finitely many irreducible divisors of $\rest_{\cG_{n, v_0}}^{\cG_n} \check{\chi}$ (note that these do not depend on the layer $n$). But only finitely many type-$W$ characters of $\cG$ have a particular restriction to $\cG_{v_0}$, as explained at the end of the proof of lemma \ref{lem:change_s_vanishing_euler_factors}.

                Let us now turn to the finite-level maps $\varphi_n^\alpha$ induced by $\alpha$ for $n \geq n(S)$ (cf. definition \ref{defn:finite_level_map}), which we denote here by $\varphi_{n, S}^\alpha$ and $\varphi_{n, S'}^\alpha$ depending on the set of places used. These two homomorphisms arise as the composition of the rows of the commutative diagram
                \begin{center}
                    \begin{tikzcd}
                        {\cX_{L_n, S}} \arrow[r, "\sim"] \arrow[d, hook] & \gcoinv{(\cX_S)}{n} \arrow[r, "\gcoinv{(\varphi_S)}{n}"] \arrow[d]           & \gcoinv{(\cY_{S_\infty})}{n} \arrow[r, "\gcoinv{\alpha}{n}"] \arrow[d, equals] & {\gcoinv{(E_{S, T})}{n}} \arrow[r, "(\iota_S)_n", hook] \arrow[d, equals]      & {\ZZ_p \otimes \units{\cO_{L_n, S, T}}} \arrow[d, hook] \\

                        {\cX_{L_n, S'}} \arrow[r, "\sim"]  & {\gcoinv{(\cX_{S'})}{n}} \arrow[r, "\gcoinv{(\varphi_{S'})}{n}"] & \gcoinv{(\cY_{S_\infty})}{n} \arrow[r, "\gcoinv{\alpha}{n}"]                                & {\gcoinv{(E_{S', T})}{n}} \arrow[r, "(\iota_{S'})_n", hook] & {\ZZ_p \otimes \units{\cO_{L_n, S', T}}}
                    \end{tikzcd}
                \end{center}
                in the notation of section \ref{sec:morphisms_on_finite_level}, where we have added the subscripts $S$ and $S'$ to some of the morphisms for distinction purposes. It follows that $\varphi_{n, S'}^\alpha$ restricted to $\cX_{L_n, S}$ coincides with $\varphi_{n, S}^\alpha$ (under the canonical inclusion of their codomains). A similar statement holds for the Dirichlet regulator map: the commutativity of
                \begin{center}
                    \begin{tikzcd}
                        {\RR \otimes \units{\cO_{L_n, S}}} \arrow[d, hook] \arrow[r, "\sim"] & {\RR \otimes \cX_{L_n, S}^\ZZ} \arrow[d, hook] \\
                        {\RR \otimes \units{\cO_{L_n, S'}}} \arrow[r, "\sim"] & {\RR \otimes \cX_{L_n, S'}^\ZZ}
                    \end{tikzcd}
                \end{center}
                implies $\restr{\lambda_{n, S'}^\beta}{\CC_p \otimes \units{\cO_{L_n, S}}} = \lambda_{n, S}^\beta$ (cf. definition \ref{defn:p-adic_dirichlet_regulator_map}), and this in turn
                \[
                    \restr{\lambda_{n, S'}^\beta \circ (\CC_p \otimes_{\ZZ_p} \varphi_{n, S'}^\alpha)}{\CC_p \otimes_{\ZZ_p} \cX_{L_n, S}} = \lambda_{n, S}^\beta \circ (\CC_p \otimes_{\ZZ_p} \varphi_{n, S}^\alpha).
                \]

                We are essentially done, since $R_S^\beta(\alpha, \chi \otimes \rho)$ and $R_{S'}^\beta(\alpha, \chi \otimes \rho)$ are given by the determinant of the $\CC_p$-linear vertical maps in the diagram
                \begin{center}
                    \begin{tikzcd}
                        {\Hom_{\CC_p[\cG_n]}(V_{\chi \otimes \rho}, \CC_p \otimes_{\ZZ_p} \cX_{L_n, S})} \arrow[d, "{(\lambda_{n, S}^\beta \circ (\CC_p \otimes_{\ZZ_p} \varphi_{n, S}^\alpha))_\ast}"] \arrow[r, hook] & {\Hom_{\CC_p[\cG_n]}(V_{\chi \otimes \rho}, \CC_p \otimes_{\ZZ_p} \cX_{L_n, S'})} \arrow[d, "{(\lambda_{n, S'}^\beta \circ (\CC_p \otimes_{\ZZ_p} \varphi_{n, S'}^\alpha))_\ast}"] \\
                        {\Hom_{\CC_p[\cG_n]}(V_{\chi \otimes \rho}, \CC_p \otimes_{\ZZ_p} \cX_{L_n, S})} \arrow[r, hook]                                                                                  & {\Hom_{\CC_p[\cG_n]}(V_{\chi \otimes \rho}, \CC_p \otimes_{\ZZ_p} \cX_{L_n, S'})}
                    \end{tikzcd}
                \end{center}
                which commutes by the above. The first part of the proof shows that the horizontal arrows are equalities for almost all $\rho \in \cK_{S'}^\alpha(\chi)$, and this concludes the argument.
            \end{proof}

            The first main result of this subsection now follows easily:
            \begin{prop}
            \label{prop:independence_of_s}
            Setting \ref{sett:formulation}. Let $v_0 \not\in S \cup T$ be a place of $K$ and set $S' = S \cup \set{v_0}$. Then:
                \begin{enumerate}[i)]
                    \item{
                        For any $\chi \in \Irr_p(\cG)$, \hyperref[conje:ic]{IC($L_\infty/K, \chi, L, S, T, \alpha, \beta$)} holds if and only if \hyperref[conje:ic]{IC($L_\infty/K, \chi, L, S', T, \alpha, \beta$)} does.
                    }
                    \item{
                        \hyperref[conje:emc]{eMC($L_\infty/K, L, S, T, \alpha, \beta$)} holds if and only if \hyperref[conje:emc]{eMC($L_\infty/K, L, S', T, \alpha, \beta$)} does.
                    }
                \end{enumerate}
            \end{prop}

            \begin{proof}
                Let $\varphi_{v_0(L_\infty)} \in \cG_{v_0}$ denote the Frobenius at $v_0(L_\infty)$, and $[(1 - \varphi_{v_0(L_\infty)}^{-1})] \in K_1(\cQ(\cG))$ the class of the 1-by-1 invertible matrix ${(1 - \varphi_{v_0(L_\infty)}^{-1}) \in \GLg_1(\cQ(\cG))}$. Here left invertibility follows from the proof of corollary \ref{cor:difference_euler_characteristics_change_s}, and right invertibility from a symmetric argument.

                Fix first $\chi \in \Irr_p(\cG)$ and assume that $F_{S, T, \chi}^{\alpha, \beta} \in \units{\cQ^c(\Gamma_\chi)}$ satisfies IC($L_\infty/K, \chi, L, S, T, \alpha, \beta$). Set
                \[
                    F' = F_{S, T, \chi}^{\alpha, \beta} \cdot \psi_\chi([(1 - \varphi_{v_0(L_\infty)}^{-1})]) \in \units{\cQ^c(\Gamma_\chi)}.
                \]

                Recall that $\psi_\chi = \psi_{\chi \otimes \rho} \colon K_1(\cQ(\cG)) \to \units{\cQ^c(\Gamma_{\chi})} = \units{\cQ^c(\Gamma_{\chi \otimes \rho})}$ for all $\rho$ of type $W$. Let $\rho$ be such a character and choose $n$ such that $\chi \otimes \rho$ factors through $\cG_n$. Equations \eqref{eq:relation_reduced_norm_and_invariants} and \eqref{eq:evaluation_j_and_psi} imply
                \[
                    ev_{\gamma_{\chi \otimes \rho}}(\psi_\chi([(1 - \varphi_{v_0(L_\infty)}^{-1})])) = \det(\QQ_p^c \otimes_{\ZZ_p} \gcoinv{((1 - \varphi_{v_0(L_\infty)}^{-1})_r)}{n} \mid \Hom_{\QQ_p^c[\cG_n]}(V_{\chi \otimes \rho}, \QQ_p^c[\cG_n])),
                \]
                where $\QQ_p^c \otimes_{\ZZ_p} \gcoinv{((1 - \varphi_{v_0(L_\infty)}^{-1})_r)}{n}$ coincides with the endomorphism of $\QQ_p^c[\cG_n]$ given by right multiplication by $1 - \varphi_{v_0(L_n)}^{-1}$. Denote by $M_{\chi \otimes \rho}(\varphi_{v_0(L_n)}^{-1}) \in M_{\chi(1)}(\QQ_p^c)$ the image of $\varphi_{v_0(L_n)}^{-1}$ under the representation associated to $\chi \otimes \rho$. Then an easy verification shows that, under a suitable $\QQ_p^c$-basis, postcomposition with $(1 - \varphi_{v_0(L_n)}^{-1})_r$ acts on $\Hom_{\QQ_p^c[\cG_n]}(V_{\chi \otimes \rho}, \QQ_p^c[\cG_n])$ via the transpose $M_{\chi \otimes \rho}(\varphi_{v_0(L_n)}^{-1})^t$. Thus, the above determinant coincides with $\det(\Id - M_{\chi \otimes \rho}(\varphi_{v_0(L_n)}^{-1})^t) = \det(\Id - M_{\check{\chi} \otimes \rho^{-1}}(\varphi_{v_0(L_n)}))$ and, as a consequence,
                \begin{equation}
                \label{eq:twisted_evaluation_frobenius_factor}
                    ev_{\gamma_{\chi \otimes \rho}}(\psi_\chi([(1 - \varphi_{v_0(L_\infty)}^{-1})])) = \det(1 - \varphi_{v_0(L_n)} \mid V_{\check{\chi} \otimes \rho^{-1}}) = \beta^{-1}(\det(1 - \varphi_{v_0(L_n)} \mid V_{\beta (\check{\chi} \otimes \rho^{-1})})).
                \end{equation}
                This provides us with the necessary information about leading coefficients, since lemma \ref{lem:change_s_vanishing_euler_factors} shows that
                \[
                    L_{K, S', T}^\ast(\beta (\check{\chi} \otimes \rho^{-1}), 0) = L_{K, S, T}^\ast(\beta (\check{\chi} \otimes \rho^{-1}), 0) \cdot
                     \det(1 - \varphi_{v_0} \mid V_{\beta (\check{\chi} \otimes \rho^{-1})})
                \]
                for almost all $\rho$. Here we are implicitly using definition \eqref{eq:artin_l_series}, which applies to the half-plane $\re(s) > 1$, at the point $s = 0$. The reason we can do so was already outlined after equation \eqref{eq:l_functions_t-smoothed}.

                It now follows that the series quotient $F'$ satisfies
                \begin{align*}
                    ev_{\gamma_{\chi \otimes \rho}}(F') & = ev_{\gamma_{\chi \otimes \rho}}(F_{S, T, \chi}^{\alpha, \beta})\cdot ev_{\gamma_{\chi \otimes \rho}}(\psi_\chi([(1 - \varphi_{v_0(L_\infty)}^{-1})])) \\
                    & = \frac{\beta^{-1}(L_{K, S, T}^\ast(\beta (\check{\chi} \otimes \rho^{-1}), 0))}{R_S^\beta(\alpha, \chi \otimes \rho)} \cdot ev_{\gamma_{\chi \otimes \rho}}(\psi_\chi([(1 - \varphi_{v_0(L_\infty)}^{-1})])) \\
                    & = \frac{\beta^{-1}(L_{K, S', T}^\ast(\beta (\check{\chi} \otimes \rho^{-1}), 0))}{R_{S}^\beta(\alpha, \chi \otimes \rho)} \\
                    & = \frac{\beta^{-1}(L_{K, S', T}^\ast(\beta (\check{\chi} \otimes \rho^{-1}), 0))}{R_{S'}^\beta(\alpha, \chi \otimes \rho)}
                \end{align*}
                for almost all $\rho \in \cK_{S'}^\alpha(\chi) \subseteq \cK_S^\alpha(\chi)$, where the first equation relies on remark \ref{rem:evaluation_of_series_quotients} ii) together with \eqref{eq:twisted_evaluation_frobenius_factor}, and the last one is lemma  \ref{lem:change_s_regulators}. In other words, IC($L_\infty/K, \chi, L, S', T, \alpha, \beta$) holds.

                An analogous argument shows the converse implication: if the element $F_{S', T, \chi}^{\alpha, \beta} \in \units{\cQ^c(\Gamma_\chi)}$ satisfies IC($L_\infty/K, \chi, L, S', T, \alpha, \beta$), then $F = F_{S', T, \chi}^{\alpha, \beta} \cdot \psi_\chi([(1 - \varphi_{v_0(L_\infty)}^{-1})])^{-1}$ has the interpolation property required by IC($L_\infty/K, \chi, L, S, T, \alpha, \beta$).

                Part ii) is a direct consequence of part i) and corollary \ref{cor:difference_euler_characteristics_change_s}, as $[(1 - \varphi_{v_0(L_\infty)}^{-1})] \in K_1(\cQ(\cG))$ does not depend on the choice of a character. Assume first eMC($L_\infty/K, L, S, T, \alpha, \beta$) and let $\zeta_{S, T}^{\alpha, \beta} \in K_1(\cQ(\cG))$ be the zeta element predicted therein. By assumption, IC($L_\infty/K, \chi, L, S, T, \alpha, \beta$) holds for all $\chi \in \Irr_p(\cG)$, and hence so does IC($L_\infty/K, \chi, L, S', T, \alpha, \beta$).

                Consider $\zeta' = \zeta_{S, T}^{\alpha, \beta} \cdot [(1 - \varphi_{v_0(L_\infty)}^{-1})]$, which satisfies
                \begin{align*}
                    \partial(\zeta') & = \partial(\zeta_{S, T}^{\alpha, \beta}) + \partial([(1 - \varphi_{v_0(L_\infty)}^{-1})]) \\
                    & = - \chi_{\Lambda(\cG), \cQ(\cG)}(\cC_{S, T}\q, t^\alpha) + [\Lambda(\cG), (1 - \varphi_{v_0(L_\infty)}^{-1})_r, \Lambda(\cG)] \\
                    & = - \chi_{\Lambda(\cG), \cQ(\cG)}(\cC_{S', T}\q, t^\alpha)
                \end{align*}
                by corollary \ref{cor:difference_euler_characteristics_change_s}. Together with the relation $F_{S', T, \chi}^{\alpha, \beta} = F_{S, T, \chi}^{\alpha, \beta} \cdot \psi_\chi([(1 - \varphi_{v_0(L_\infty)}^{-1})])$ established in part i), this proves that $\zeta'$ satisfies eMC($L_\infty/K, L, S, T, \alpha, \beta$). The reverse implication follows from the choice $\zeta = \zeta_{S', T}^{\alpha, \beta} \cdot [(1 - \varphi_{v_0(L_\infty)}^{-1})]^{-1}$ in the obvious notation.
            \end{proof}

            We now turn our attention to the change of $T$. Much of the necessary work has already been done in the treatment of $S$, and the elements which do differ follow similar arguments. The following aspects distinguish the two cases:
            \begin{itemize}
                \item{
                    The unit module $E_{S, T} = \varprojlim_n \ZZ_p \otimes \units{\cO_{L_n, S, T}}$ does become smaller as we enlarge $T$, so the homomorphism $\alpha$ needs to be modified now. The next lemma provides a natural way to fix this discrepancy.
                }
                \item{
                    The Stark-Tate regulator $R_S^\beta(\alpha, \chi)$ is independent of $T$.
                }
                \item{
                    Unlike $L$-factors, the $\delta$-factors of Artin $L$-functions never vanish at $s = 0$, as shown at the beginning of the proof of proposition \ref{prop:stark_independent_of_t}.
                }
            \end{itemize}

            \begin{lem}
            \label{lem:alpha_change_of_t}
                 Setting \ref{sett:construction}. Let $v_0 \not\in S \cup T$ be a place of $K$. If $\alpha \colon \cY_{S_\infty} \ia E_{S, T \cup \set{v_0}}$ is an injective $\Lambda(\cG)$-homomorphism with $\Lambda(\Gamma)$-torsion cokernel, then the same is true of the composition $\tilde{\alpha} \colon \cY_{S_\infty} \xhookrightarrow{\alpha} E_{S, T \cup \set{v_0}} \ia E_{S, T}$.
            \end{lem}

            \begin{proof}
                Let $T'$ denote $ T \cup \set{v_0}$. Taking inverse limits of the exact sequence
                \[
                    1 \to \units{\cO_{L_n, S, T'}} \to \units{\cO_{L_n, S, T}} \to \sum_{w_n \in \set{v_0}(L_n)} \units{\kappa(w_n)} \to Cl_{L_n, S, T'} \to Cl_{L_n, S, T} \to 1
                \]
                along the cyclotomic tower $L_\infty/L$ gives rise to an exact sequence of $\Lambda(\cG)$-modules
                \begin{equation}
                \label{eq:galois_sequence_change_t}
                    0 \to E_{S, T'} \to E_{S, T} \to \overline{\Ind}_{\cG_{v_0}}^\cG \ZZ_p(1) \to X_{T', S}^{cs} \to X_{T, S}^{cs} \to 0
                \end{equation}
                by the same argument used on \eqref{eq:classical_five_term_sequence_units}, where $\overline{\Ind}_{\cG_v}^\cG \ZZ_p(1)$ is as in proposition \ref{prop:exact_triangle_difference_s_and_t}. In particular,  $\coker(E_{S, T'} \ia E_{S, T}) \ia \overline{\Ind}_{\cG_{v_0}}^\cG \ZZ_p(1)$ is $\Lambda(\Gamma)$-torsion, as it has finite $\ZZ_p$-rank. Now apply the snake lemma to
                \begin{center}
                    \begin{tikzcd}
                        0 \arrow[r] & \cY_{S_\infty} \arrow[d, "\tilde{\alpha}", hook] \arrow[r, "\alpha"] & {E_{S, T'}} \arrow[d, hook] \arrow[r] & \coker(\alpha) \arrow[r] \arrow[d] & 0 \\
                        0 \arrow[r] & {E_{S, T}} \arrow[r, equals]                            & {E_{S, T}} \arrow[r]                                         & 0 \arrow[r]                      & 0
                    \end{tikzcd}
                \end{center}
            \end{proof}

            It only remains to find a strictly perfect representative of the complex measuring the difference between $\cC_{S, T}\q$ and $\cC_{S, T'}\q$ in the sense of \eqref{eq:exact_triangle_difference_t}. This result is likely well known among experts.

            \begin{lem}
            \label{lem:change_t_perfect_representative}
                Setting \ref{sett:construction}, $\alpha$ as in setting \ref{sett:formulation}. Let $v_0 \not\in S_{\ram}(L_\infty/K)$ be a place of $K$. Then there exists a quasi-isomorphism of complexes of $\Lambda(\cG)$-modules
                \[
                    [\stackrel{-1}{\Lambda(\cG)} \xrightarrow{(1 - \fN(v_0)\varphi_{v_0(L_\infty)}^{-1})_r} \stackrel{0}{\Lambda(\cG)}] \to \overline{\Ind}_{\cG_{v_0}}^\cG \ZZ_p(1)[0],
                \]
                where $\varphi_{v_0(L_\infty)}$ is the Frobenius automorphism at $v_0(L_\infty)$ and $(1 - \varphi_{v_0(L_\infty)}^{-1} \fN(v_0))_r$ denotes multiplication by $1 - \fN(v_0)\varphi_{v_0(L_\infty)}^{-1}$ on the right.
            \end{lem}

            \begin{proof}
                 We will be concise, since the argument is very similar to that of corollary \ref{cor:difference_euler_characteristics_change_s}. Taking the inverse limit of the finite-level sequences
                \[
                    0 \to \Lambda(\cG_{n, v_0}) \xrightarrow{(\varphi_{v_0(L_n)} - \fN(v_0))_r} \Lambda(\cG_{n, v_0}) \to \faktor{\Lambda(\cG_{n, v_0})}{\ideal{\varphi_{v_0(L_n)} -  \fN(v_0)}} \to 0
                \]
                along the cyclotomic tower yields an exact sequence of $\Lambda(\cG_{v_0})$-modules. If $L_{v_0(L)}$ does not contain a primitive $p$-th root of unity (that is, $p \nmid 1 - \fN(v_0(L))$), the cokernel is trivial on all finite levels and hence in the limit. Otherwise, the limit of the cokernels is isomorphic to $\ZZ_p$ with the $\cG_{v_0}$-action $\varphi_{v_0(L_\infty)} z = \fN(v_0) z$, which is precisely $\ZZ_p(1)$. Multiplying the arrow by the unit $\varphi_{v_0(L_n)}^{-1}$ and applying induction results in an exact sequence
                \[
                    0 \to \Lambda(\cG) \xrightarrow{(1 - \fN(v_0) \varphi_{v_0(L_n)}^{-1})_r} \Lambda(\cG) \to \Ind_{\cG_{v_0}}^\cG \faktor{\ZZ_p}{\ideal{\varphi_{v_0(L_n)}^{-1} - \fN(v_0)}} \to 0,
                 \]
                 from which the lemma follows.
             \end{proof}

             This allows us to prove the independence of the choice of $T$:

            \begin{prop}
            \label{prop:independence_of_t}
             Setting \ref{sett:construction}, $\beta$ as in setting \ref{sett:formulation}. Let $v_0 \not\in S \cup T$ be a place of $K$ and set ${T' = T \cup \set{v_0}}$. Choose $\alpha \colon \cY_{S_\infty} \ia E_{S, T'}$ as in setting \ref{sett:formulation} and define $\tilde{\alpha}$ as in lemma \ref{lem:alpha_change_of_t}. Then:
                \begin{enumerate}[i)]
                    \item{
                         For any $\chi \in \Irr_p(\cG)$, \hyperref[conje:ic]{IC($L_\infty/K, \chi, L, S, T, \tilde{\alpha}, \beta$)} holds if and only if \hyperref[conje:ic]{IC($L_\infty/K, \chi, L, S, T', \alpha, \beta$)} does.
                    }
                    \item{
                        \hyperref[conje:emc]{eMC($L_\infty/K, L, S, T, \tilde{\alpha}, \beta$)} holds if and only if \hyperref[conje:emc]{eMC($L_\infty/K, L, S, T \cup \set{v_0}, \alpha, \beta$)} does.
                    }
                \end{enumerate}
            \end{prop}

            \begin{proof}
                We proceed along very similar lines to the proof of proposition \ref{prop:independence_of_s} and omit many of the details. The element $1 - \fN(v_0)\varphi_{v_0(L_\infty)}^{-1}$ is a unit in $\cQ(\cG)$ for the same reasons as $1 - \varphi_{v_0(L_\infty)}^{-1}$ before and we may therefore consider its class $[(1 - \fN(v_0)\varphi_{v_0(L_\infty)}^{-1})] \in K_1(\cQ(\cG))$.

                In order to show i), we first study the difference between the $L$-values and regulators associated to $T$ and $T'$. Given $\varsigma \in \Irr_p(\cG)$, choose $n \in \NN$ large enough that $\varsigma$ factors through $\cG_n$. Since $\delta$-factors never vanish at 0 (see for instance the proof of proposition \ref{prop:stark_independent_of_t}), one has
                \[
                    L_{K, S, T'}^\ast(\varsigma, 0) = L_{K, S, T}^\ast(\varsigma, 0) \cdot \delta_{v_0}(\varsigma, 0).
                \]
                As for the regulators, suppose $n \geq n(S)$ and consider the commutative diagram
                \begin{center}
                    \begin{tikzcd}
                    {\cX_{L_n, S}} \arrow[r, "\sim"] \arrow[d, equals] & \gcoinv{(\cX_S)}{n} \arrow[r, two heads] \arrow[d, equals] & \gcoinv{(\cY_{S_\infty})}{n} \arrow[d, equals] \arrow[r, "\gcoinv{\alpha}{n}"] & {\gcoinv{(E_{S, T'})}{n}} \arrow[d, "\gcoinv{\iota}{n}"] \arrow[r, "\sim"] & {\cE_{L_n, S, T'}} \arrow[d, hook] \arrow[r, hook] & {\ZZ_p \otimes \units{\cO_{L_n, S, T'}}} \arrow[d, hook] \\
                    {\cX_{L_n, S}} \arrow[r, "\sim"]                       & \gcoinv{(\cX_S)}{n} \arrow[r, two heads]                       & \gcoinv{(\cY_{S_\infty})}{n} \arrow[r, "\gcoinv{\tilde{\alpha}}{n}"]   & {\gcoinv{(E_{S, T})}{n}} \arrow[r, "\sim"]                                 & {\cE_{L_n, S, T}} \arrow[r, hook]                  & {\ZZ_p \otimes \units{\cO_{L_n, S, T}}}
                    \end{tikzcd}
                \end{center}
                in the notation of section \ref{sec:morphisms_on_finite_level}. After applying $\CC_p \otimes_{\ZZ_p} -$, the last vertical arrow becomes an equality (in fact, all vertical arrows do) because the embedding is canonical. It follows that $R_{S}^\beta(\alpha, \varsigma) = R_{S}^\beta(\tilde{\alpha}, \varsigma)$.

                Assume now that $F_{S, T, \chi}^{\tilde{\alpha}, \beta} \in \units{\cQ^c(\Gamma_\chi)}$ satisfies IC($L_\infty/K, \chi, L, S, T, \tilde{\alpha}, \beta$) for some $\chi \in \Irr_p(\cG)$. Then
                \[
                    F' = F_{S, T, \chi}^{\tilde{\alpha}, \beta} \cdot \psi_\chi([(1 - \fN(v_0)\varphi_{v_0(L_\infty)}^{-1})]) \in \units{\cQ^c(\Gamma_\chi)}
                \]
                satisfies IC($L_\infty/K, \chi, L, S, T', \alpha, \beta$) by the same argument as in proposition \ref{prop:independence_of_s}, which also yields the reverse implication.

                We now address the equivalence of the equivariant Main Conjectures, which involves comparing the refined Euler characteristics of $\cC_{S, T}\q$ and $\cC_{S, T'}\q$ induced by the trivialisations $t^{\tilde{\alpha}}$ and $t^{\alpha}$, respectively. The difference between the complexes themselves is given by the exact triangle \eqref{eq:exact_triangle_difference_t}. Applying theorem 5.7 from \cite{bb}, we obtain
                \[
                    \chi_{\Lambda(\cG), \cQ(\cG)}(\cC_{S, T}\q, t^{\tilde{\alpha}}) = \chi_{\Lambda(\cG), \cQ(\cG)}(\cC_{S, T'}\q, t^\alpha) + \chi_{\Lambda(\cG), \cQ(\cG)}(\overline{\Ind}_{\cG_{v_0}}^\cG \ZZ_p(1)[0], 0)
                \]
                this time. Note that 0 is a valid trivialisation for $\overline{\Ind}_{\cG_{v_0}}^\cG \ZZ_p(1)[0]$, as its only non-trivial module is torsion. The second summand can be computed explicitly using the strictly perfect representative from lemma \ref{lem:change_t_perfect_representative}. This results in
                \begin{align*}
                    \chi_{\Lambda(\cG), \cQ(\cG)}(\cC_{S, T}\q, t^{\tilde{\alpha}})
                    & = \chi_{\Lambda(\cG), \cQ(\cG)}(\cC_{S, T'}\q, t^\alpha) + [\Lambda(\cG), (1 - \fN(v_0) \varphi_{v_0(L_\infty)}^{-1})_r, \Lambda(\cG)] \\
                    & = \chi_{\Lambda(\cG), \cQ(\cG)}(\cC_{S, T'}\q, t^\alpha) + \partial([(1 - \fN(v_0) \varphi_{v_0(L_\infty)}^{-1})]).
                \end{align*}

                Suppose that $\zeta_{S, T}^{\tilde{\alpha}, \beta} \in K_1(\cQ(\cG))$ satisfies eMC($L_\infty/K, L, S, T, \tilde{\alpha}, \beta$). Then $\zeta' = \zeta_{S, T}^{\tilde{\alpha}, \beta} \cdot [(1 - \fN(v_0) \varphi_{v_0}^{-1})]$ satisfies eMC($L_\infty/K, L, S, T', \alpha, \beta$) by part i) and the above equation. The converse is proved by setting $\zeta = \zeta_{S, T'}^{\alpha, \beta} \cdot [(1 - \fN(v_0) \varphi_{v_0}^{-1})]^{-1}$.
            \end{proof}

            This proposition, together with \ref{prop:independence_of_s}, shows that the Main Conjecture is unaffected by the enlargement of $S$ or $T$ one place at a time. Therefore, it is independent of the choice $S$ and $T$ altogether. As we have seen, the homomorphism $\alpha$ needs to be modified when $T$ changes. The previous subsection showed that changes in $\alpha$ do not affect the conjecture, but one can also formulate the following result without relying on that fact:
            \begin{cor}
            \label{cor:independence_of_s_and_t}
                Setting \ref{sett:construction}, $\beta$ as in setting \ref{sett:formulation}. Let $\widetilde{S}, \widetilde{T}$ be another valid choice of the corresponding sets of places. In particular, so is $S \cap \widetilde{S}, T \cup \widetilde{T}$. Consider an injective $\Lambda(\cG)$-homomorphism  $\alpha_0 \colon \cY_{S_\infty} \ia E_{S \cap \widetilde{S}, T \cup \widetilde{T}}$ with $\Lambda(\Gamma)$-torsion cokernel and set
                \[
                    \alpha \colon \cY_{S_\infty} \xhookrightarrow{\alpha_0} E_{S \cap \widetilde{S}, T \cup \widetilde{T}} \ia E_{S \cap \widetilde{S}, T} = E_{S, T}
                    \quad \text{and} \quad
                     \widetilde{\alpha} \colon \cY_{S_\infty} \xhookrightarrow{\alpha_0} E_{S \cap \widetilde{S}, T \cup \widetilde{T}} \ia E_{S \cap \widetilde{S}, \widetilde{T}} = E_{\widetilde{S}, \widetilde{T}}.
                \]
                (cf. lemma \ref{lem:difference_galois_modules_s}). Then:
                \begin{enumerate}[i)]
                    \item{
                        For all $\chi \in \Irr_p(\cG)$, \hyperref[conje:ic]{IC($L_\infty/K, \chi, L, S, T, \alpha, \beta$)} holds if and only if \hyperref[conje:ic]{IC($L_\infty/K, \chi, L, \widetilde{S}, \widetilde{T}, \widetilde{\alpha}, \beta$)} does.
                    }
                    \item{
                        \hyperref[conje:emc]{eMC($L_\infty/K, L, S, T, \alpha, \beta$)} holds if and only if \hyperref[conje:emc]{eMC($L_\infty/K, L, \widetilde{S}, \widetilde{T}, \widetilde{\alpha}, \beta$)} does.
                    }
                \end{enumerate}
            \end{cor}

            \begin{proof}
                In part i), it suffices to show that both conjectures are equivalent (for each fixed $\chi$) to IC($L_\infty/K, \chi, L, S \cap \widetilde{S}, T \cup \widetilde{T}, \alpha_0, \beta$). The same is true of ii) and eMC(${L_\infty/K, L, S \cap \widetilde{S}, T \cup \widetilde{T}, \alpha_0, \beta}$). This is achieved by propositions \ref{prop:independence_of_s} and \ref{prop:independence_of_t}, since the embedding $E_{S \cap \widetilde{S}, T \cup \widetilde{T}} \ia E_{S \cap \widetilde{S}, T}$ is the composition
                \[
                    E_{S \cap \widetilde{S}, T \cup \widetilde{T}} \ia E_{S \cap \widetilde{S}, T \cup \widetilde{T} \setminus \set{v_0}} \ia \cdots \ia E_{S \cap \widetilde{S}, T \cup \widetilde{T} \setminus \set{v_0, \ldots, v_n}} = E_{S \cap \widetilde{S}, T}
                \]
                and analogously for $E_{S \cap \widetilde{S}, T \cup \widetilde{T}} \ia E_{S \cap \widetilde{S},\widetilde{T}}$.
            \end{proof}

            This corollary, together with propositions \ref{prop:indep_ic_l} and \ref{prop:independence_of_alpha}, proves theorem \ref{thm:independence_of_parameters}: one may first change $L$ without affecting any other parameters; then replace $\alpha$ by the homomorphism induced by $\alpha_0$ in corollary \ref{cor:independence_of_s_and_t}, also without altering the remaining parameters; and finally modify $S$ and $T$. Recall that Stark's conjecture \hyperref[conje:stark]{Stark\textsuperscript{T}($L/K, \chi, f, S, T$)} is independent of $S$ and $T$ as well, as explained in section \ref{sec:starks_conjecture_and_the_choice_of_beta}.

    \section{Functoriality}
    \label{sec:functoriality}

        It has now been established that, for a fixed extension $L_\infty/K$, the Main Conjecture is independent of all choices of parameters except possibly $\beta$, and conjecturally of this last one too. The next natural step is to study changes in $L_\infty$ and $K$. As we shall see, the conjecture for $\Gal(L_\infty/K)$ will imply that for \textit{smaller} groups - be it subgroups or quotients. The challenge lies in the impact modifying the Galois group has on the algebraic machinery developed so far. In the last subsection, we apply $K$-theoretic techniques to briefly address the converse question: whether one can deduce the Main Conjecture for a given extension by assuming it for sufficiently many smaller ones.

        \subsection{\texorpdfstring{Change of $L_\infty$}{Change of L\_infinity}}
        \label{subsec:functoriality_1}

            We first discuss the top field by considering $L_\infty'/L_\infty/K$. The new situation can be formalised as follows:
            \begin{sett}
            \label{sett:functoriality_1}
            \addcontentsline{toc}{subsubsection}{Setting C}
                The objects $p$, $L_\infty'/K_\infty/K$, $\Gamma'$, $S$ and $T$ are fixed as in setting \ref{sett:construction} in the obvious notation. We set $\Gamma_K = \Gal(K_\infty/K) = \overline{{\ideal{\gamma_K}}}$ and $H' = \Gal(L_\infty'/K_\infty) \trianglelefteq_c \cG' = \Gal(L_\infty'/K)$. Additionally, we consider:
                \begin{itemize}
                    \item{
                        A finite normal subgroup $\widetilde{H} \trianglelefteq_c \cG'$. In particular, one has $\widetilde{H} \trianglelefteq H'$ and we can define ${L_\infty = (L_\infty')^{\widetilde{H}} \supseteq K_\infty,} \, H = \Gal(L_\infty/K_\infty) = H'/\widetilde{H}$ and $\cG = \Gal(L_\infty/K) = \cG'/\widetilde{H}$.
                    }
                    \item{
                        The projection of $\Gamma' \subseteq \cG'$ to $\cG$, which we denote by $\Gamma$. Since $\Gamma' \cap \widetilde{H} \subseteq \Gamma' \cap H' = 1$, the group $\Gamma'$ identifies with $\Gamma$, and $\Lambda(\Gamma) \subseteq \Lambda(\cG)$ with $\Lambda(\Gamma') \subseteq \Lambda(\cG')$. It follows immediately that $\Gamma \subseteq Z(\cG)$ is open in $\cG$ and $\Gamma \cap H = 1$. Therefore, $L = (L_\infty)^\Gamma$ is a finite Galois extension of $K$ such that $L_\infty$ is the cyclotomic $\ZZ_p$-extension of $L$.

                        As in setting \ref{sett:construction}, we use the notation $L_n'$ and $L_n$ for the layers of $L_\infty'/K$ and $L_\infty/K$, respectively, and set $\cG_n' = \Gal(L_n'/K)$ and $\cG_n = \Gal(L_n/K)$. Note that, for all $n \in \NN$, the field $L_n'$ is a Galois extension of $L_n$ such that $\Gal(L_n'/L_n) = \ker(\cG_n' \sa \cG_n) \iso \widetilde{H}$.
                    }
                    \item{
                        The natural numbers $n'(S)$ and $n(S)$ given by
                        \[
                            (\Gamma')^{p^{n'(S)}} = \Gamma' \cap \bigcap_{w_\infty' \in S_f(L_\infty')} \cG_w' \qquad \text{and} \qquad \Gp{n(S)} = \Gamma \cap \bigcap_{w_\infty \in S_f(L_\infty)} \cG_w.
                        \]
                        This implies that $L_{n(S)}$ is the smallest layer of $L_\infty/L$ such that all finite places in $S$ are non-split in $L_\infty/L_{n(S)}$, and analogously for $n'(S)$. Although it will not play a role, it is easy to show that $n'(S) \geq n(S)$. As usual, we denote the decomposition groups $\cG_{v(L_\infty')}'$ and $\cG_{v(L_\infty)}$ of the distinguished prolongations of $v$ by $\cG_{v}$ and $\cG_{v}'$, respectively.
                    }
                \end{itemize}
            \end{sett}
            With the above conventions, $(L_\infty'/K, \Gamma', S, T)$ and $(L_\infty/K, \Gamma, S, T)$ are two valid choices of the corresponding parameters in setting \ref{sett:construction}. The following diagram illustrates the relations between the relevant fields:
            
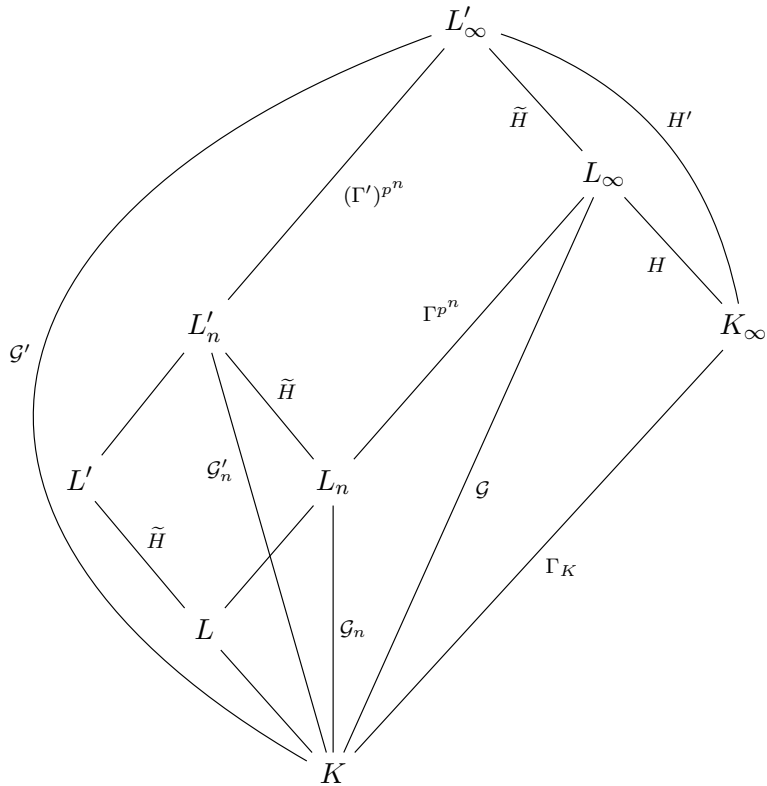
\begin{center}
            \hspace{-1.6cm}
                \begin{tikzcd}[row sep=huge]
                                                                            &                                                                            &                                                                                                                  & L_\infty' \arrow[rd, "\widetilde{H}"', no head] \arrow[rrdd, "H'", no head, bend left] \arrow[lddddd, "\cG'"', no head, in=150, out=200, looseness=1.7] &                                    &          \\
                                                                            &                                                                            &                                                                                                                  &                                                                                                                                        & L_\infty \arrow[rd, "H"', no head] &          \\
                                                                            & L_n' \arrow[rruu, "(\Gamma')^{p^n}"', no head] \arrow[rd, "\widetilde{H}", no head] &                                                                                                                  &                                                                                                                                        &                                    & K_\infty \\
                L' \arrow[ru, no head] \arrow[rd, "\widetilde{H}", no head] &                                                                            & L_n \arrow[rruu, "\Gp{n}", no head] \arrow[dd, "\cG_n", no head]                                                 &                                                                                                                                        &                                    &          \\
                                                                            & L \arrow[ru, no head] \arrow[rd, no head]                                  &                                                                                                                  &                                                                                                                                        &                                    &          \\
                                                                            &                                                                            & K \arrow[rrruuu, "\Gamma_K"', no head] \arrow[luuu, "\cG_n'", no head, near end] \arrow[rruuuu, "\cG"', no head] &                                                                                                                                        &                                    &
                \end{tikzcd}
            \end{center}

            Our first aim is to study the relation between the semisimple algebras $\cQ^c(\cG')$ and $\cQ^c(\cG)$, and in particular the fields $\cQ^c(\Gamma_\chi)$ contained in them. In general, it is not true that a continuous surjection of profinite groups induces a ring homomorphism between the total rings of fractions of their Iwasawa algebras, but the properties of $\cG'$ and $\cG$ ensure this is the case in our situation. The following discussion makes repeated use of the obvious identification
            \begin{align}
            \label{eq:functoriality_1_artin_characters}
                \Irr_p(\cG) & \leftrightarrow \set{\chi \in \Irr_p(\cG') \et{such that} \widetilde{H} \subseteq \ker(\chi)} \\
                \overline{\chi} & \leftrightarrow \chi = \infl_{\cG}^{\cG'} \overline{\chi}, \nonumber
            \end{align}
            which restricts to an identification of type-$W$ characters.

            \newpage
            \begin{lem}
            \label{lem:functoriality_1_chi_parts}
                Setting \ref{sett:functoriality_1}. The following hold:
                \begin{enumerate}[i)]
                    \item{
                        The projection $\cG' \sa \cG$ induces a canonical surjection of $\cQ(\Gamma')$-algebras ${\varepsilon \colon \cQ(\cG') \sa \cQ(\cG)}$ (where the codomain is regarded as such via $\cQ(\Gamma') \isoa \cQ(\Gamma) \ia Z(\cQ(\cG))$) which extends to one of $\cQ^c(\Gamma')$-algebras ${\varepsilon \colon \cQ^c(\cG') \sa \cQ^c(\cG)}$.
                    }
                    \item{
                        Let $\chi \in \Irr_p(\cG')$. Then
                        \[
                            \varepsilon(e_\chi) =
                            \begin{cases}
                                e_{\overline{\chi}}, & \widetilde{H} \subseteq \ker(\chi) \\
                                0,  & \text{otherwise},
                            \end{cases}
                        \]
                        where $\overline{\chi}$ denotes the projection of $\chi$ to $\cG$ and $e_\chi$ and $e_{\overline{\chi}}$ are defined as in proposition \ref{prop:rw_properties_q}. Furthermore, the map $\varepsilon$ from part i) restricts to an isomorphism $\cQ^c(\cG')e_\chi \iso \cQ^c(\cG)e_{\overline{\chi}}$ whenever $\chi$ factors through $\cG$ as $\overline{\chi}$.
                    }
                    \item{
                        Let $\overline{\chi} \in \Irr_p(\cG)$ and set $\chi = \infl_{\cG}^{\cG'} \overline{\chi}$. Then $\varepsilon(\gamma_\chi) = \gamma_{\overline{\chi}}$ (cf. proposition \ref{prop:structure_gamma_chi}) and $\varepsilon$ induces a canonical isomorphism $\cQ^c(\Gamma_{\chi}) \iso \cQ^c(\Gamma_{\overline{\chi}})$.
                    }
                \end{enumerate}
            \end{lem}

            \begin{proof}
                \begin{enumerate}[i)]
                    \item{
                        The augmentation sequence \eqref{eq:general_augmentation_ideal} associated to $\widetilde{H} \ia \cG' \sa \cG$ features the continuous surjection $\aug_{\widetilde{H}} \colon \Lambda(\cG') \sa \Lambda(\cG)$, which in this case restricts to an isomorphism $\Lambda(\Gamma') \iso \Lambda(\Gamma)$. Therefore, the $\Lambda(\Gamma')$-algebra structure of $\Lambda(\cG')$ is compatible with that of $\Lambda(\cG)$ as a $\Lambda(\Gamma)$-algebra. Since $\cQ(\cG') = \cQ(\Gamma') \otimes_{\Lambda(\Gamma')} \Lambda(\cG')$ and $\cQ(\cG) = \cQ(\Gamma) \otimes_{\Lambda(\Gamma)} \Lambda(\cG)$ by \eqref{eq:lambda_regular_elements}, the map $\varepsilon = \cQ(\Gamma') \otimes_{\Lambda(\Gamma')} \aug_{\widetilde{H}}$ satisfies the first claim. The second one is then obtained by applying $\QQ_p^c \otimes_{\QQ_p} -$.
                    }
                    \item{
                        This follows from the analogous result for group algebras over finite groups, since the primitive central idempotents of $\cQ^c(\cG')$ and $\cQ^c(\cG)$ only depend on the restriction of characters to $H$. Recall that $e_\chi$ is defined as
                        \[
                            e_\chi = \sum_{\substack{\eta \in \Irr_{\QQ_p^c}(H')\\\eta \mid \rest_{H'}^{\cG'} \chi}} e(\eta)
                        \]
                        where $e(\eta)$ is the primitive central idempotent of $\QQ_p^c[H']$ corresponding to $\eta$ (cf. \eqref{eq:definition_pci}). By \eqref{eq:pci_projection}, the projection $\varepsilon$ (which restricts to $\QQ_p^c[H'] \sa \QQ_p^c[H]$) sends $e(\eta)$ to $e(\overline{\eta})$ if $\eta$ factors through $H$ (as $\overline{\eta}$), and to 0 otherwise.

                        Suppose first that $\widetilde{H} \subseteq \ker(\chi)$ and let $\overline \chi$ be the projection of $\chi$ to $\cG$. Then  $\widetilde{H} \subseteq \ker(\rest_{H'}^{\cG'} \chi)$ and $\rest_{H'}^{\cG'} \chi = \rest_{H'}^{\cG'} \infl_{\cG}^{\cG'} \overline{\chi} = \infl_H^{H'} \rest_H^{\cG} \overline{\chi}$. In particular, an $\eta \in \Irr_{\QQ_p^c}(H')$ divides $\rest_{H'}^{\cG'} \chi$ if and only if it factors through $H$ and its projection $\overline{\eta}$ divides $\rest_H^{\cG} \overline{\chi}$. Here we are using the general fact that if a character factors through a quotient group, then all of its divisors do as well. Note also that a character of a quotient group is irreducible if and only if its inflation to the entire group is so. We therefore have:
                        \[
                            \varepsilon(e_\chi) = \sum_{\substack{\eta \in \Irr_{\QQ_p^c}(H')\\\eta \mid \rest_{H'}^{\cG'} \chi}} \varepsilon(e(\eta)) = \sum_{\substack{\eta \in \Irr_{\QQ_p^c}(H')\\\eta \mid \rest_{H'}^{\cG'} \chi}} e(\overline{\eta}) = \sum_{\substack{\overline{\eta} \in \Irr_{\QQ_p^c}(H)\\\overline{\eta} \mid \rest_H^{\cG} \overline{\chi}}} e(\overline{\eta}) = e_{\overline{\chi}}.
                        \]

                        Conversely, assume that $\widetilde{H} \nsubseteq \ker(\chi)$ and hence $\widetilde{H} \nsubseteq \ker(\rest_{H'}^{\cG'}\chi)$. Then at least one irreducible constituent $\eta$ of $\rest_{H'}^{\cG'}\chi$ satisfies $\widetilde{H} \nsubseteq \ker(\eta)$. But Clifford theory shows that all such divisors are conjugate to one another by elements of $\cG'$, so $\widetilde{H}$ cannot be contained in the kernel of any of them. Thus, one has
                        \[
                            \varepsilon(e_\chi) = \sum_{\eta \mid \rest_{H'}^{\cG'} \chi} \varepsilon(e(\eta)) = 0,
                        \]
                        which concludes the proof of the first claim.

                        Let now $\chi \in \Irr_p(\cG')$ factor through $\cG$ as $\overline{\chi}$. By the above argument, $\varepsilon$ restricts to a surjective ring homomorphism $\cQ^c(\cG')e_\chi \sa \cQ^c(\cG)e_{\overline{\chi}}$ (note that the unities of $\cQ^c(\cG')e_\chi$ and $\cQ^c(\cG)e_{\overline{\chi}}$ are $e_\chi$ and $e_{\overline{\chi}}$ respectively). The kernel of a ring homomorphism is a two-sided ideal, but the only such ideals of $\cQ^c(\cG')e_\chi$ are $\set{0}$ and the entire ring by simplicity.
                    }
                    \item{
                        Part ii) already implies that $\varepsilon$ restricts to an isomorphism
                        \[
                            \cQ^c(\Gamma_{\chi}) = Z(\cQ^c(\cG')e_\chi) \iso Z(\cQ^c(\cG)e_{\overline{\chi}}) = \cQ^c(\Gamma_{\overline{\chi}})
                        \]
                        by proposition \ref{prop:structure_gamma_chi}. This does not a priori mean that $\Gamma_\chi$ is mapped to $\Gamma_{\overline{\chi}}$ (see for instance the proof of lemma \ref{lem:twisted_evaluation_maps}), which we prove now. The generator $\gamma_\chi$ of $\Gamma_\chi \subseteq \units{Z(\cQ^c(\cG')e_\chi)}$ is given by $\gamma_\chi = gc$ as in proposition \ref{prop:structure_gamma_chi} i). Here $g \in \cG'$ projects to $\gamma_K^{w_\chi} \in \Gamma_K$, where $w_\chi$ coincides with the number of irreducible constituents $\eta$ of $\rest_{H'}^{\cG'} \chi$. We have shown these correspond bijectively with the irreducible constituents $\overline{\eta}$ of $\rest_H^{\cG} \overline{\chi}$, and thus $w_\chi = w_{\overline{\chi}}$.

                        The element $\varepsilon(\gamma_\chi) = \varepsilon(g)\varepsilon(c) \in Z(\cQ^c(\cG)e_{\overline{\chi}})$ satisfies the following:
                        \begin{itemize}
                            \item{
                                $\varepsilon(g)$ projects to $\gamma_K^{w_\chi} = \gamma_K^{w_{\overline{\chi}}}$ by the above.
                            }
                            \item{
                                $\varepsilon(c) \in \varepsilon(\units{(\QQ_p^c[H']e_\chi)}) = \units{(\QQ_p^c[H]e_{\overline{\chi}})}$ (observe that $\varepsilon$ restricts to a ring isomorphism $\QQ_p^c[H']e_\eta \iso \QQ_p^c[H]e_{\overline{\eta}}$ for all $\eta$ which factor through $H$).
                            }
                            \item{
                                $\varepsilon(\gamma_\chi)$ acts trivially on $V_{\overline{\chi}}$ (the $\QQ_p^c$-vector space afforded by $\overline{\chi}$), since so does $\gamma_\chi$ on $V_\chi$.
                            }
                        \end{itemize}
                        By proposition $\ref{prop:structure_gamma_chi}$, there exists only one $\gamma_{\overline{\chi}} \in Z(\cQ^c(\cG)e_{\overline{\chi}})$ with those three properties, which must therefore coincide with $\varepsilon(\gamma_\chi)$.
                    }
                \end{enumerate}
            \end{proof}

            \begin{rem}
            \label{rem:projection_gamma_chi_independent}
                Although the elements $\gamma_\chi$ and $\gamma_{\overline{\chi}}$ depend on the choice of $\gamma_{K}$, the lemma shows that one always has $\varepsilon(\gamma_\chi) = \gamma_{\overline{\chi}}$ as long as the same $\gamma_{K}$ is chosen in both cases. In fact, as already mentioned in remark \ref{rem:twisted_evaluation_maps_independent_gamma_K}, replacing $\gamma_{K}$ by another generator $\gamma_K^v$ (where $v \in \units{\ZZ_p}$ necessarily) results in the substitution of $\gamma_\chi$ by $\gamma_\chi^v$ and $\gamma_{\overline{\chi}}$ by $\gamma_{\overline{\chi}}^v$. In particular, $\gamma_\chi \mapsto \gamma_{\overline{\chi}}$ and $\gamma_\chi^v \mapsto \gamma_{\overline{\chi}}^v$ define the same map on $\cQ^c(\Gamma_{\chi})$: the canonical isomorphism induced by $\varepsilon$. \qedef
            \end{rem}

            The above result sheds light on the structural relation between $\cQ^c(\cG')$ and $\cQ^c(\cG)$. If $\chi$ and $\rho$ are irreducible Artin characters of $\cG'$ with $\rho$ of type $W$, then $\chi \otimes \rho$ factors through $\cG$ if and only if $\chi$ does. Furthermore, characters of $\cG'$ which factor through $\cG$ are $W$-equivalent (cf. section \ref{sec:evaluation_maps}) if and only if their projections to $\cG$ are. This allows one to interpret the lemma as follows: for a chosen set of representatives $\Xi_\cG$ of $\Irr_p(\cG) / {\sim}_W$, set $\Xi_{\cG'} = \set{\infl_\cG^{\cG'} \overline{\chi} : \overline{\chi} \in \Xi_\cG}$. Let $\widetilde{\Xi}_{\cG'}$ be an arbitrary set of $\sim_W$-representatives of the irreducible Artin characters of $\cG'$ which do not factor through $\cG$ (well defined by the above argument). Then the diagram
            \begin{center}
                \begin{tikzcd}[column sep=tiny]
                    \cQ^c(\cG') \arrow[d, "\varepsilon"] \arrow[rrr, "\sim"] &  &  & \prod_{\chi \in \widetilde{\Xi}_{\cG'}} \cQ^c(\cG')e_\chi \arrow[rrd, "0"'] & \hspace{-1.5em} \times \hspace{-1.5em} & \prod_{\chi \in \Xi_{\cG'}} \cQ^c(\cG')e_\chi \arrow[d, "\prod \restr{\varepsilon}{e_\chi}"]             &  &  & \prod_{\chi \in \Xi_{\cG'}} \cQ^c(\Gamma_{\chi}) \arrow[lll, hook'] \arrow[d, "\gamma_\chi \mapsto \gamma_{\overline{\chi}}"]     \\
                    \cQ^c(\cG) \arrow[rrrrr, "\sim"]                 &  &  &                                                                &        & \prod_{\overline{\chi} \in \Xi_\cG} \cQ^c(\cG)e_{\overline{\chi}} &  &  & \prod_{\overline{\chi} \in \Xi_\cG} \cQ^c(\Gamma_{\overline{\chi}}) \arrow[lll, hook']
                \end{tikzcd}
            \end{center}
            commutes, where the first map in each row is \eqref{eq:decomposition_algebra_chi-parts}. Although this decomposition only exists after extending scalars to a large enough $p$-adic field, the element
            \[
                e_{\widetilde{H}} = \sum_{\chi \in \Xi_{\cG'}} e_\chi = \frac{1}{\abs{\widetilde{H}}} N_{\widetilde{H}}
            \]
            (where $N_{\widetilde{H}} = \sum_{\sigma \in \widetilde{H}} \sigma \in \ZZ[\widetilde{H}] \subseteq{\ZZ[H']}$ is the norm\footnote{In this section, we will refer to the \textit{norm} $N_G$ of a finite group $G$, rather than its \textit{trace} $Tr_G$ - but the distinction is purely conventional. What is meant by this is the element $N_G = \sum_{\sigma \in G} \sigma \in \ZZ[G]$ of any ring where $\ZZ[G]$ can be naturally embedded. If $m$ is an element of a $G$-module (that is, a $\ZZ[G]$-module) $M$, then $N_G m$ is either $\prod_{\sigma \in G} (\sigma m)$ or $\sum_{\sigma \in G} (\sigma m)$ depending on whether $M$ uses multiplicative or additive notation.} of $\widetilde{H}$) is a central idempotent of $\cQ(\cG')$ which induces a decomposition
            \begin{equation}
            \label{eq:decomposition_trace_idempotent}
                \cQ(\cG') \iisoo \cQ(\cG')e_{\widetilde{H}} \times \cQ(\cG')(1 - e_{\widetilde{H}})
            \end{equation}
            of $\cQ(\cG)$ as a product of rings. The restriction of $\varepsilon$ to the first factor yields an isomorphism ${\cQ(\cG')e_{\widetilde{H}} \iso \cQ(\cG)}$.

            So far we have only addressed the elements in setting \ref{sett:construction}. The two other necessary parameters for the Main Conjecture are the $\alpha$ and $\beta$ from setting \ref{sett:formulation}, the second of which is manifestly unrelated to the rest. As for $\alpha$, we can use the one choice for $L_\infty'/K$ to construct its counterpart for $L_\infty/K$ in a way which will make the Main Conjectures easy to compare:

            \begin{lem}
            \label{lem:change_of_alpha_quotient}
                Setting \ref{sett:functoriality_1}. The following hold:
                \begin{enumerate}[i)]
                    \item{
                        Let $\cY_{S_\infty}$ and $E_{S, T}$ be the usual modules for the extension $L_\infty/K$ (cf. section \ref{sec:the_main_complex}), and $\cY_{S_\infty}'$ and $E_{S, T}'$ their analogues for $L_\infty'/K$, that is,
                        \[
                            \cY_{S_\infty} = \bigoplus_{v \in S_\infty} \Ind_{\cG_v}^\cG \ZZ_p, \quad \cY_{S_\infty}' = \bigoplus_{v \in S_\infty} \Ind_{\cG_v'}^{\cG'} \ZZ_p
                        \]
                        and
                        \[
                            E_{S, T} = \varprojlim_n \ZZ_p \otimes \units{\cO_{L_n, S, T}}, \quad E_{S, T}' = \varprojlim_n \ZZ_p \otimes \units{\cO_{L_n', S, T}}.
                        \]
                        Suppose $\alpha' \colon \cY_{S_\infty}' \ia E_{S, T}'$ is an injective $\Lambda(\cG')$-homomorphism with torsion cokernel. Then the map
                        \[
                            \alpha \colon \cY_{S_\infty} \isoa (\cY_{S_\infty}')_{\widetilde{H}} \xrightarrow{\alpha_{\widetilde{H}}'} (E_{S, T}')_{\widetilde{H}} \xrightarrow{N_{\widetilde{H}}} (E_{S, T}')^{\widetilde{H}} = E_{S, T}
                        \]
                        is an injective $\Lambda(\cG)$-homomorphism with torsion cokernel.
                    }
                    \item{
                        Let $\overline{\chi} \in \Irr_p(\cG)$ and denote its inflation to $\cG'$ by $\chi$. If $\alpha'$ and $\alpha$ are as in part i) and $\beta \colon \CC_p \to \CC$ is a field isomorphism, the Stark-Tate regulators
                        \[
                            R_S^\beta(\alpha', \chi) = R_S^\beta(\alpha, \overline{\chi})
                        \]
                        coincide.
                    }
                \end{enumerate}
            \end{lem}

            \begin{proof}
                \begin{enumerate}[i)]
                    \item{
                        We first verify some implicit claims in the definition of $\alpha$. The canonical isomorphism $\cY_{S_\infty} \iso (\cY_{S_\infty}')_{\widetilde{H}}$ is proved in the exact same way as the first part of proposition \ref{prop:isomoprhisms_y_x_modules}. The norm of $\widetilde{H}$ defines a homomorphism from coinvariants to invariants of any $\widetilde{H}$ module, and in particular $N_{\widetilde{H}} \colon (E_{S, T}')_{\widetilde{H}} \to (E_{S, T}')^{\widetilde{H}}$. This map is Galois-equivariant and hence a $\Lambda(\cG)$-homomorphism too. Lastly, one has
                        \[
                            (E_{S, T}')^{\widetilde{H}} = (\varprojlim_n \ZZ_p \otimes \units{\cO_{L_n', S, T}})^{\widetilde{H}} = \varprojlim_n \ZZ_p \otimes (\units{\cO_{L_n', S, T}})^{\widetilde{H}} = \varprojlim_n \ZZ_p \otimes \units{\cO_{L_n, S, T}} = E_{S, T}.
                        \]
                        It remains to show that $\alpha$ is injective and has torsion cokernel. The Tate cohomology groups
                        \[
                            \hat{H}_0(\widetilde{H}, E_{S, T}') = \ker(N_{\widetilde{H}}) \quad \text{and} \quad \hat{H}^0(\widetilde{H}, E_{S, T}') = \coker(N_{\widetilde{H}}),
                        \]
                        are $\abs{\widetilde{H}}$-torsion (cf. \cite{nsw} propositions 1.2.6 and 1.6.1) and hence $\Lambda(\cG)$-torsion. As for $\alpha_{\widetilde{H}}'$, we consider the usual long exact cohomology sequence
                        \[
                            \cdots \to H_1(\widetilde{H}, \coker(\alpha')) \to (\cY_{S_\infty}')_{\widetilde{H}} \xrightarrow{\alpha_{\widetilde{H}}'} (E_{S, T}')_{\widetilde{H}} \to \coker(\alpha')_{\widetilde{H}} \to 0
                        \]
                        associated to $\cY_{S_\infty}' \xhookrightarrow{\alpha'} E_{S, T}' \sa \coker(\alpha')$. The term $H_1(\widetilde{H}, \coker(\alpha')) = \hat{H}_1(\widetilde{H}, \coker(\alpha'))$, which surjects onto $\ker(\alpha_{\widetilde{H}}')$, is $\abs{\widetilde{H}}$-torsion for the same reasons as above. The cokernel of $\alpha_{\widetilde{H}}'$ coincides with $\coker(\alpha')_{\widetilde{H}}$, which is $\Lambda(\cG)$-torsion (i.e. $\Lambda(\Gamma)$-torsion) by assumption. It follows that $\alpha$ itself has torsion kernel and cokernel, which implies injectivity by proposition \ref{prop:existence_injectivity_alpha}.
                    }
                    \item{
                        The diagram
                        \begin{center}
                            \begin{tikzcd}
                                \cY_{S_\infty}' \arrow[r, equals] \arrow[d, phantom] \arrow[d, two heads] & \cY_{S_\infty}' \arrow[d, two heads] \arrow[rr, "\alpha'"]                 &                                                      & {E_{S, T}'} \arrow[d, "N_{\widetilde{H}}"] \arrow[r, equals] & {E_{S, T}'} \arrow[d, "N_{\widetilde{H}}"] \\
                                \cY_{S_\infty} \arrow[r, "\sim"]                                 & (\cY_{S_\infty}')_{\widetilde{H}} \arrow[r, "\alpha_{\widetilde{H}}'"] \arrow[r] & {(E_{S, T}')_{\widetilde{H}}} \arrow[r, "N_{\widetilde{H}}"] & {(E_{S, T}')^{\widetilde{H}}} \arrow[r, equals]              & {E_{S, T}}
                            \end{tikzcd}
                        \end{center}
                        commutes because, by definition, the norm map $N_{\widetilde{H}}$ on coinvariants is given by first choosing a lift to the original module (and it is independent of that choice). The composition of the bottom row is precisely $\alpha$.

                        Let $n \in \NN$ be  larger than $n(S)$ and $n(S')$, and large enough that $\overline{\chi}$ factors through $\cG_n$ (in particular, $\chi$ factors through $\cG_n'$). Recall that the projection $\Lambda(\cG') \sa \Lambda(\cG)$ restricts to an isomorphism $\Lambda(\Gamma') \iso \Lambda(\Gamma)$. Therefore, taking $(\Gamma')^{p^n}$- and $\Gp{n}$-coinvariants on the previous diagram yields a commutative square which fits into
                        \begin{center}
                        \vspace{0.5em}
                            \begin{small}
                                \begin{tikzcd}[row sep=large, column sep=small]
                                {\cX_{L_n', S}} \arrow[d, two heads] \arrow[r, "\sim"] & (\cX_S')_{(\Gamma')^{p^n}} \arrow[d, two heads] \arrow[r, two heads] & (\cY_{S_\infty}')_{(\Gamma')^{p^n}} \arrow[d, phantom] \arrow[d, two heads] \arrow[r, "\alpha_{(\Gamma')^{p^n}}'"] & (E_{S, T}')_{(\Gamma')^{p^n}} \arrow[d, "N_{\widetilde{H}}"] \arrow[r, "\sim"] & {\cE_{L_n', S, T}} \arrow[d, "N_{\widetilde{H}}"] \arrow[r, hook] & {\ZZ_p \otimes \units{\cO_{L_n', S, T}}} \arrow[d, "N_{\widetilde{H}}"] \arrow[r] & {\CC_p \otimes \cX_{L_n', S}} \arrow[d, two heads] \\
                                {\cX_{L_n, S}} \arrow[r, "\sim"]                       & \gcoinv{(\cX_S)}{n} \arrow[r, two heads]                       & \gcoinv{(\cY_{S_\infty})}{n} \arrow[r, "\gcoinv{\alpha}{n}"]                                           & {\gcoinv{(E_{S, T})}{n}} \arrow[r, "\sim"]                                 & {\cE_{L_n, S, T}} \arrow[r, hook]                                 & {\ZZ_p \otimes \units{\cO_{L_n, S, T}}} \arrow[r]                                 & {\CC_p \otimes \cX_{L_n, S}}
                                \end{tikzcd}
                            \end{small}
                            \vspace{0.5em}
                        \end{center}
                        The last arrow of the bottom row is given by
                        \[
                            \ZZ_p \otimes \units{\cO_{L_n, S, T}} \ia \CC_p \otimes \units{\cO_{L_n, S, T}} \xrightarrow{\lambda_{n,S}^\beta} \CC_p \otimes_{\ZZ_p} \cX_{L_n, S},
                        \]
                        where $\lambda_{n,S}^\beta$ denotes the $p$-adic Dirichlet regulator map (cf. definition \ref{defn:p-adic_dirichlet_regulator_map}), and analogously for the arrow above it. The commutativity of the rightmost square boils down to the general fact that $\abs{N_{L_n'/L_n}(u)}_v = \prod_{w \mid v} \abs{u}_w$ in the obvious notation. The remaining squares are manifestly commutative.

                        We now apply the functor $\Hom_{\CC_p[\cG_n']}(V_\chi, \CC_p \otimes_{\ZZ_p} -)$ to the above diagram except for the last column, where we apply $\Hom_{\CC_p[\cG_n']}(V_\chi, -)$ instead. On the bottom row, these functors coincide with $\Hom_{\CC_p[\cG_n]}(V_{\overline{\chi}}, \CC_p \otimes_{\ZZ_p} -)$ and $\Hom_{\CC_p[\cG_n]}(V_{\overline{\chi}}, -)$, respectively. Both the first and last vertical arrows of the resulting diagram are the $\CC_p$-linear homomorphism
                        \[
                            \Hom_{\CC_p[\cG_n']}(V_\chi, \CC_p \otimes_{\ZZ_p} \cX_{L_n', S}) \sa \Hom_{\CC_p[\cG_n]}(V_{\overline{\chi}}, \CC_p \otimes_{\ZZ_p} \cX_{L_n, S})
                        \]
                        (note that $\CC_p[\cG_n']$ is semisimple), which is in fact an isomorphism by the same argument used at the beginning of the proof of lemma \ref{lem:r_chi_independent_of_n}. This shows that the $\CC_p$-determinants of the two rows of the new diagram coincide - but these are precisely the regulators $R_S^\beta(\alpha', \chi)$ and $R_S^\beta(\alpha, \overline{\chi})$.
                    }
                \end{enumerate}
            \end{proof}

            The equality of regulators in the previous lemma covers the analytic side of the functoriality under change of $L_\infty$. On the homological side, we will resort to the good behaviour of refined Euler characteristics with respect to extension of scalars (lemma \ref{lem:functoriality_rec}). It is therefore necessary to show that this is precisely the relation between the complexes constructed from $L_\infty'/K$ and $L_\infty/K$. The key result to do so is due to Fukaya and Kato \cite{fk}.

            The following piece of notation will be useful: if $\Lambda(G)$ is the usual Iwasawa algebra of a profinite group $G$, we denote by $\iota \colon \Lambda(G) \to \Lambda(G)$ the continuous isomorphism of $\ZZ_p$-modules which sends an element $\sigma \in G \subseteq \Lambda(G)$ to $\sigma^{-1}$. Given a topological \textit{right} $\Lambda(G)$-module $M$, we denote by $M^\iota$ the topological \textit{left} $\Lambda(G)$-module consisting of the topological abelian group $M$ endowed with the $\Lambda(G)$-action
            \begin{equation}
            \label{eq:involution_action}
                \lambda \cdot^\iota m = m \cdot \iota(\lambda)
            \end{equation}
            for $m \in M$, $\lambda \in \Lambda(G)$.

            \begin{prop}
            \label{prop:isomorphism_complexes_derived_extension}
                Setting \ref{sett:functoriality_1}. Let $(\cC_{S, T}')\q$ and $\cC_{S, T}\q$ be the complexes given by definition \ref{defn:main_complex} for $L_\infty'/K$ and $L_\infty/K$, respectively. Denote by $\Lambda(\cG) \otimes_{\Lambda(\cG')}^\LL -$ the derived extension of scalars  induced by $\Lambda(\cG') \sa \Lambda(\cG)$. Then
                \[
                    \cC_{S, T}\q \iisoo \Lambda(\cG) \otimes_{\Lambda(\cG')}^\LL (\cC_{S, T}')\q
                \]
                in $\cD(\Lambda(\cG))$.
            \end{prop}

            \begin{proof}
                Consistently with previous notation, let $M_S$ denote the maximal $S$-ramified extension of $K$, which coincides with that of $L_\infty'$ and any intermediate field. Set $G_S = \Gal(M_S/K)$ and ${H_{L_n, S} = \Gal(M_S/L_n)}$ for $n \in \NN$, and analogously for $H_{L_n', S}$. Note that $H_{L_n, S}$ and $H_{L_n', S}$ are open and normal in $G_S$ for all $n$.

                By the same argument as in proposition \ref{prop:exact_triangle_difference_s_and_t}, the complex $(\cC_{S, T}')\q$ lies in an exact triangle
                \[
                    (\cC_{S, T}')\q \to (\cC_{S, \varnothing}')\q \to \bigoplus_{v \in T} \overline{\Ind}_{\cG_v'}^{\cG'} \ZZ_p(1)[0] \to.
                \]
                in the notation introduced therein. Since derived functors are triangulated, extension of scalars yields
                \begin{equation}
                \label{eq:exact_triangle_derived_tensor}
                    \Lambda(\cG) \otimes_{\Lambda(\cG')}^\LL (\cC_{S, T}')\q \to \Lambda(\cG) \otimes_{\Lambda(\cG')}^\LL (\cC_{S, \varnothing}')\q \to \Lambda(\cG) \otimes_{\Lambda(\cG')}^\LL \bigoplus_{v \in T} \overline{\Ind}_{\cG_v'}^{\cG'} \ZZ_p(1)[0] \to
                \end{equation}
                in $\cD(\Lambda(\cG))$. We consider the middle term first. As shown in \eqref{eq:iso_inverse_zp(1)}, $(\cC_{S, \varnothing}')\q$ is isomorphic in $\cD(\Lambda(\cG'))$ to $\varprojlim_n R\Gamma(H_{L_n', S}, \ZZ_p(1))[1]$. By Shapiro's lemma (see for instance \cite{lim} lemma 5.2.3), we have
                \[
                   R\Gamma(H_{L_n', S}, \ZZ_p(1)) \iisoo  R\Gamma(G_S, \ZZ_p[G_S/H_{L_n', S}]^\iota \otimes_{\ZZ_p} \ZZ_p(1)),
                \]
                where $\ZZ_p[G_S/H_{L_n', S}]$ is regarded as a right $G_S$-module so that $\ZZ_p[G_S/H_{L_n', S}]^\iota$ is a left $G_S$-module. This endows $\ZZ_p[G_S/H_{L_n', S}]^\iota \otimes_{\ZZ_p} \ZZ_p(1)$ with a left $\Lambda(G_S)$-module structure via
                \[
                    \sigma (\lambda \otimes z) = (\sigma \cdot^\iota \lambda) \otimes \sigma z = (\lambda \sigma^{-1}) \otimes \sigma z
                \]
                for $\sigma \in G_S$, under which it is clearly isomorphic to $\ZZ_p[\cG_n']^\iota(1)$. Therefore, one has
                \begin{equation}
                \label{eq:iso_c_rgamma_infinite}
                    (\cC_{S, \varnothing}')\q \iso \varprojlim_n R\Gamma(G_S, \ZZ_p[\cG_n']^\iota(1))[1] \iso R\Gamma(G_S, \varprojlim_n \ZZ_p[\cG_n']^\iota(1))[1] = R\Gamma(G_S, \Lambda(\cG')^\iota(1))[1]
                \end{equation}
                in $\cD(\Lambda(\cG'))$, where the second isomorphism is \cite{lim} lemma 5.2.3. By the same token, $\cC_{S, \varnothing}\q$ is isomorphic to $R\Gamma(G_S, \Lambda(\cG)^\iota(1))[1]$ in $\cD(\Lambda(\cG))$.

                We now apply proposition 1.6.5 from \cite{fk} (case 1). The ring $\Lambda = \Lambda(\cG')$ satisfies the necessary hypotheses by 1.4.1 and 1.4.2 therein, and $T = \Lambda(\cG')^\iota(1)$ is a free $\Lambda$-module with a compatible continuous $G = G_S$-action as above. The reason $G$ satisfies condition (ii) in the reference was explained at the beginning of the proof of proposition \ref{prop:complexes_are_perfect}. Condition (i) can be verified with some work using \cite{nsw} propositions 1.6.7, 8.6.10 and 10.11.3. The upshot is that
                \[
                    \Lambda(\cG) \otimes_{\Lambda(\cG')}^\LL R\Gamma(G_S, \Lambda(\cG')^\iota(1))[1] \iso R\Gamma(G_S, \Lambda(\cG) \otimes_{\Lambda(\cG')} \Lambda(\cG')^\iota(1))[1] \iso R\Gamma(G_S, \Lambda(\cG)^\iota(1))[1],
                \]
                and thus $\Lambda(\cG) \otimes_{\Lambda(\cG')}^\LL (\cC_{S, \varnothing}')\q \iso \cC_{S, \varnothing}\q$.

                Let us now address the last term in \eqref{eq:exact_triangle_derived_tensor}. By the additivity of $\Lambda(\cG) \otimes_{\Lambda(\cG')} -$, it suffices to determine $\Lambda(\cG) \otimes_{\Lambda(\cG')}^\LL \overline{\Ind}_{\cG_v'}^{\cG'} \ZZ_p(1)[0]$ for $v \in T$. The strictly perfect representative of $\overline{\Ind}_{\cG_v'}^{\cG'} \ZZ_p(1)[0]$ obtained in lemma \ref{lem:change_t_perfect_representative} doubles as a projective resolution to compute the derived tensor product:
                \begin{align*}
                    \Lambda(\cG) \otimes_{\Lambda(\cG')}^\LL \overline{\Ind}_{\cG_v'}^{\cG'} \ZZ_p(1)[0] & \iso [\stackrel{-1}{\Lambda(\cG) \otimes_{\Lambda(\cG')} \Lambda(\cG')} \xrightarrow{\Lambda(\cG) \otimes_{\Lambda(\cG')} (1 - \fN(v_0)\varphi_{v_0(L_\infty')}^{-1})_r} \stackrel{0}{\Lambda(\cG) \otimes_{\Lambda(\cG')} \Lambda(\cG')}] \\
                    & \iso [\stackrel{-1}{\Lambda(\cG)} \xrightarrow{(1 - \fN(v_0)\varphi_{v_0(L_\infty)}^{-1})_r} \stackrel{0}{\Lambda(\cG)}] \\
                    & \iso \overline{\Ind}_{\cG_v}^{\cG} \ZZ_p(1)[0].
                \end{align*}

                The result now follows from the diagram of exact triangles
                \begin{center}
                    \begin{tikzcd}[column sep=small]
                        \Lambda(\cG) \otimes_{\Lambda(\cG')}^\LL (\cC_{S, T}')\q \arrow[r] \arrow[d, dashed, "\rsim"] & \Lambda(\cG) \otimes_{\Lambda(\cG')}^\LL (\cC_{S, \varnothing}')\q \arrow[r] \arrow[d, "\rsim"] & \Lambda(\cG) \otimes_{\Lambda(\cG')}^\LL \bigoplus_{v \in T} \overline{\Ind}_{\cG_v'}^{\cG'} \ZZ_p(1)[0] \arrow[r] \arrow[d, "\rsim"] & {} \\
                        \cC_{S, T}\q \arrow[r]                   & \cC_{S, \varnothing}\q \arrow[r]                   & \bigoplus_{v \in T} \overline{\Ind}_{\cG_v}^{\cG} \ZZ_p(1)[0] \arrow[r]                   & {}
                    \end{tikzcd}
                \end{center}
                where the second square commutes because both horizontal arrows come from the same maps on the level of Galois modules.
            \end{proof}

            The same technique allows us to prove a fact which was necessary for the construction of the finite-level maps $\varphi_n^\alpha$ in section \ref{sec:morphisms_on_finite_level}, but whose proof we deferred to a later time. It concerns passage to finite level in setting \ref{sett:construction}, rather than the change of groups $\cG' \sa \cG$ from setting \ref{sett:functoriality_1} - the choice of placement here is motivated by the close similarity to the previous argument. Note that none of the objects involved in the above proposition rely on $\varphi_n^\alpha$, and hence no circular dependencies arise.

            \begin{lem}
            \label{lem:complexes_coinvariants_finite_level}
                Setting \ref{sett:construction}. For all $n \in \NN$, one has
                \[
                    \cB_{L_n, S, T}\q \iso \Lambda(\cG_n) \otimes_{\Lambda(\cG)}^\LL \cC_{S, T}\q
                \]
                in the derived category $\cD(\Lambda(\cG_n))$, where $\cB_{L_n, S, T}\q$ is defined as in \ref{defn:complexes_bks}.
            \end{lem}

            \begin{proof}
                It suffices to prove the result for $T = \varnothing$ for the same reasons as above. We again have
                \[
                    \Lambda(\cG_n) \otimes_{\Lambda(\cG)}^\LL \cC_{S, \varnothing}\q \iso R\Gamma(G_S, \Lambda(\cG_n) \otimes_{\Lambda(\cG)} \Lambda(\cG)^\iota(1))[1] \iso R\Gamma(H_{L_n, S}, \ZZ_p(1))[1]
                \]
                in $\cD(\Lambda(\cG_n))$ by \cite{fk} proposition 1.6.5 (whose hypotheses are still satisfied) and Shapiro's lemma. Since $R\Gamma(H_{L_n, S}, \ZZ_p(1))[1]$ is isomorphic to $\cB_{L_n, S, T}\q$ by \eqref{eq:rgamma_finite_level}, we are done.
            \end{proof}

            We can now prove the first functoriality result:
            \begin{prop}
            \label{prop:functoriality_1}
                Setting \ref{sett:functoriality_1}. Let $\alpha'$ and $\alpha$ be as in lemma \ref{lem:change_of_alpha_quotient} and $\beta$ as in setting \ref{sett:formulation}. Then:
                \begin{enumerate}[i)]
                    \item{
                        Given $\overline{\chi} \in \Irr_p(\cG)$, set $\chi = \infl_{\cG}^{\cG'} \overline{\chi} \in \Irr_p(\cG')$. Then \hyperref[conje:ic]{IC($L_\infty'/K, \chi, L', S, T, \alpha', \beta$)} holds if and only if \hyperref[conje:ic]{IC($L_\infty/K, \overline{\chi}, L, S, T, \alpha, \beta$)} does.
                    }
                    \item{
                        \hyperref[conje:emc]{eMC($L_\infty'/K, L', S, T, \alpha', \beta$)} implies \hyperref[conje:emc]{eMC($L_\infty/K, L, S, T, \alpha, \beta$)}.
                    }
                    \item{
                        \hyperref[conje:emcu]{eMC\textsuperscript{u}($L_\infty'/K, L', S, T, \alpha', \beta$)} implies  \hyperref[conje:emcu]{eMC\textsuperscript{u}($L_\infty/K, L, S, T, \alpha, \beta$)}.
                    }
                \end{enumerate}
            \end{prop}

            \begin{proof}
                Suppose $F_{S, T, \chi}^{\alpha', \beta} \in \units{\cQ^c(\Gamma_{\chi})}$ satisfies IC($L_\infty'/K, \chi, L', S, T, \alpha', \beta$) and let ${F = \varepsilon(F_{S, T, \chi}^{\alpha', \beta}) \in \units{\cQ^c(\Gamma_{\overline{\chi}})}}$ (cf. lemma \ref{lem:functoriality_1_chi_parts}). Recall that we identify the type-$W$ characters of $\cG$ and $\cG'$ by virtue of \eqref{eq:functoriality_1_artin_characters}. Thus, for almost all $\rho \in \cK_S^{\alpha'}(\chi) \cap \cK_S^{\alpha}(\overline{\chi})$, one has
                \begin{align*}
                    ev_{\gamma_{\overline{\chi} \otimes \rho}}(F) = ev_{\gamma_{\chi \otimes \rho}}(F_{S, T, \chi}^{\alpha', \beta})  & = \frac{\beta^{-1}(L_{K, S, T}^\ast(\beta (\check{\chi} \otimes \rho^{-1}), 0))}{R_S^\beta(\alpha', \chi \otimes \rho)} \\
                    & = \frac{\beta^{-1}(L_{K, S, T}^\ast(\beta ((\overline{\chi}){\check{}} \otimes \rho^{-1}), 0))}{R_S^\beta(\alpha', \chi \otimes \rho)} \\
                    & = \frac{\beta^{-1}(L_{K, S, T}^\ast(\beta ((\overline{\chi}){\check{}} \otimes \rho^{-1}), 0))}{R_S^\beta(\alpha, \overline{\chi} \otimes \rho)},
                \end{align*}
                where the equalities are, in order: lemmas \ref{lem:functoriality_1_chi_parts} iii) and \ref{lem:twisted_evaluation_maps}; the interpolation property of $F_{S, T, \chi}^{\alpha', \beta}$; lemma \ref{lem:properties_of_L-functions} ii); and lemma \ref{lem:change_of_alpha_quotient} ii). This shows $F$ satisfies IC($L_\infty/K, \chi, L, S, T, \alpha, \beta$). Since $\varepsilon$ is an isomorphism on $\chi$-parts by lemma \ref{lem:functoriality_1_chi_parts} ii), the same argument shows the converse.

                In order to prove ii), assume eMC($L_\infty'/K, L', S, T, \alpha', \beta$) holds and let $\zeta_{S, T}^{\alpha', \beta} \in K_1(\cQ(\cG'))$ be the element predicted therein. In particular, IC($L_\infty'/K, \chi, L', S, T, \alpha', \beta$) holds for all $\chi \in \Irr_p(\cG')$ and therefore so does IC($L_\infty/K, \overline{\chi}, L, S, T, \alpha, \beta$) for all $\overline{\chi} \in \Irr_p(\cG)$ by i). This uniquely determines interpolating elements $F_{S, T, \chi}^ {\alpha', \beta} \in \units{\cQ^c(\Gamma_{\chi})}$ and $F_{S, T, \overline{\chi}}^ {\alpha, \beta} \in \units{\cQ^c(\Gamma_{\overline{\chi}})}$.

                Let $\varepsilon \colon \cQ(\cG') \sa \cQ(\cG)$ be the projection map from lemma \ref{lem:functoriality_1_chi_parts}. Then the diagram
                \begin{center}
                    \begin{tikzcd}
                        K_1(\cQ(\cG')) \arrow[d, "K_1(\varepsilon)"] \arrow[r, "\psi_\chi"] & \units{\cQ^c(\Gamma_{\chi})} \arrow[d, "\varepsilon"] \\
                        K_1(\cQ(\cG)) \arrow[r, "\psi_{\overline{\chi}}"]             & \units{\cQ^c(\Gamma_{\overline{\chi}})}
                    \end{tikzcd}
                \end{center}
                commutes for any $\overline{\chi}$ as above and $\chi = \infl_\cG^{\cG'}\overline{\chi}$. Part i) now yields
                \[
                    \psi_{\overline{\chi}}(K_1(\varepsilon)(\zeta_{S, T}^{\alpha, \beta})) = \varepsilon(\psi_\chi(\zeta_{S, T}^{\alpha', \beta})) = F_{S, T, \overline{\chi}}^{\alpha, \beta},
                \]
                and hence the element $\zeta_{S, T}^{\alpha, \beta} = K_1(\varepsilon)(\zeta_{S, T}^{\alpha', \beta})$ satisfies the analytic part of eMC($L_\infty/K, L, S, T, \alpha, \beta$).

                On the homological side, let $(\cC_{S, T}')\q$ and $\cC_{S, T}\q$ be the main complexes constructed from $L_\infty'/K$ and $L_\infty/K$, respectively. The maps $R = \Lambda(\cG') \xrightarrow{\rho = \varepsilon} \Lambda(\cG) = R'$ and $S = \cQ(\cG') \xrightarrow{\sigma = \varepsilon} \cQ(\cG) = S'$ constitute an example of setup \eqref{eq:functoriality_relative_k0_setup} together with the canonical inclusions $\Lambda(\cG') \ia \cQ(\cG')$ and $\Lambda(\cG) \ia \cQ(\cG)$. As a consequence, diagram \eqref{eq:functoriality_relative_k0_payoff} implies
                \[
                    \partial(\zeta_{S, T}^{\alpha, \beta}) = K_0(\varepsilon, \varepsilon)(\partial(\zeta_{S, T}^{\alpha', \beta})) = K_0(\varepsilon, \varepsilon)(- \chi_{\Lambda(\cG'), \cQ(\cG')}((\cC_{S, T}')\q, t^{\alpha'})) \in K_0(\Lambda(\cG), \cQ(\cG)).
                \]
                By lemma \ref{lem:functoriality_rec} i),  the the last term coincides with $- \chi_{\Lambda(\cG), \cQ(\cG)}(\Lambda(\cG) \otimes_{\Lambda(\cG')}^\LL (\cC_{S, T}')\q, t')$ for a certain trivialisation $t'$. Proposition \ref{prop:isomorphism_complexes_derived_extension} provides an isomorphism $\varphi \colon \Lambda(\cG) \otimes_{\Lambda(\cG')}^\LL (\cC_{S, T}')\q \iso \cC_{S, T}\q$ in $\cD(\Lambda(\cG))$, and thus it only remains to show that $t'$ becomes $t^{\alpha}$ under $\varphi$. To see this, recall \eqref{eq:integral_trivialisation} and consider the commutative diagram
                \begin{center}
                    \begin{tikzcd}
                        {H^1((\cC_{S, T}')\q)_{\widetilde{H}}} \arrow[d, "\rsim"] \arrow[r] & (\cX_S')_{\widetilde{H}} \arrow[d, two heads] \arrow[r, two heads] & (\cY_{S_\infty}')_{\widetilde{H}} \arrow[d, "\rsim"] \arrow[r, "\alpha'_{\widetilde{H}}"] & {(E_{S, T}')_{\widetilde{H}}} \arrow[d, "N_{\widetilde{H}}"] \arrow[r, "\sim"] & {H^0((\cC_{S, T}')\q)_{\widetilde{H}}} \arrow[d, "N_{\widetilde{H}}"] \\
                        {H^1(\cC_{S, T})} \arrow[r]                                      & \cX_S \arrow[r, two heads]                                     & \cY_{S_\infty} \arrow[r, "\alpha"]                                               & {E_{S, T}} \arrow[r, "\sim"]                                 & {H^0(\cC_{S, T})}
                    \end{tikzcd}
                \end{center}
                where the first vertical arrow is induced by $\varphi$ (note that $\Lambda(\cG) \otimes_{\Lambda(\cG')} -$ is right exact) and $\cX_S', \cY_{S_\infty}'$ and $E_{S, T}'$ are the $L_\infty'/K$-analogues of their $L_\infty/K$ counterparts. Apply now $\cQ(\cG) \otimes_{\Lambda(\cG)} -$ to both rows. The composition of the bottom row then becomes the trivialisation $t^\alpha$ by definition. On the top row, we use the fact that $\cQ(\cG) \otimes_{\Lambda(\cG)} (-)_{\widetilde{H}} = \cQ(\cG) \otimes_{\Lambda(\cG)} \Lambda(\cG) \otimes_{\Lambda(\cG')} -$ is the same functor as (or more formally, canonically naturally isomorphic to) $\cQ(\cG) \otimes_{\cQ(\cG')} \cQ(\cG') \otimes_{\Lambda(\cG')} -$, which turns that row into $\cQ(\cG) \otimes_{\cQ(\cG')} t^{\alpha'}$. But this is precisely the trivialisation $t'$ constructed from $t = t^{\alpha'}$ in lemma \ref{lem:functoriality_rec} i), which concludes the proof of ii).

                Lastly, we address iii). By part ii) and remark \ref{rem:emcu} ii), it suffices to show that the injectivity of $K_1(\cQ(\cG')) \xrightarrow{\nr}   \units{Z(\cQ(\cG'))}$ implies that of $K_1(\cQ(\cG)) \xrightarrow{\nr} \units{Z(\cQ(\cG))}$. This follows immediately from the decomposition \eqref{eq:decomposition_trace_idempotent} and the compatibility of $K_1$ and the reduced norm with products:
                \begin{center}
                    \begin{tikzcd}
                        K_1(\cQ(\cG')) \arrow[r, "\sim"] \arrow[d, "\nr"] & K_1(\cQ(\cG')e_{\widetilde{H}}) \oplus K_1(\cQ(\cG')(1 - e_{\widetilde{H}})) \arrow[r, "\sim"] \arrow[d, "\nr \oplus \nr"] & K_1(\cQ(\cG)) \oplus K_1(\cQ(\cG')(1 - e_{\widetilde{H}}) \arrow[d, "\nr \oplus \nr"] \\
                        \units{Z(\cQ(\cG'))} \arrow[r, "\sim"]                      & \units{Z(\cQ(\cG')e_{\widetilde{H}})} \oplus \units{Z(\cQ(\cG')(1 - e_{\widetilde{H}}))}    \arrow[r, "\sim"]                                                                               & \units{Z(\cQ(\cG))} \oplus \units{Z(\cQ(\cG')(1 - e_{\widetilde{H}}))}
                    \end{tikzcd}
                \end{center}
                where the last horizontal arrow of each horizontal arrow is induced by functoriality on $\varepsilon$.
            \end{proof}

            The decomposition $\cQ(\cG') \iso \cQ(\cG) \times \cQ(\cG')(1 - e_{\widetilde{H}})$ illustrates why we do not have equivalence of the Main Conjectures for $L_\infty'/K$ and $L_\infty/K$, but rather implication in one direction. Any potential zeta element $\zeta_{S, T}^{\alpha, \beta}$ for $L_\infty/K$ has no information about the part of $\cQ(\cG')$ orthogonal to $e_{\widetilde{H}}$.

        \subsection{Change of $K$}
        \label{subsec:functoriality_2}

            We now address the situation of a tower of extensions $L_\infty/K'/K$, where $K'$ and $K$ are number fields and over both of which $L_\infty$ satisfies our usual hypotheses. This is equivalent to considering $L_\infty/K$ as in setting \ref{sett:construction} and choosing an open (not necessarily normal) subgroup $\cG'$ of $\cG = \Gal(L_\infty/K)$. We summarise the new situation as follows:

            \begin{sett}
            \label{sett:functoriality_2}
            \addcontentsline{toc}{subsubsection}{Setting D}
                The objects $p$, $L_\infty/K_\infty/K$, $S$ and $T$ are fixed as in setting \ref{sett:construction}. We also denote ${\Gamma_K = \Gal(K_\infty/K) = \overline{{\ideal{\gamma_K}}}}$ and $H = \Gal(L_\infty/K_\infty) \trianglelefteq_c \cG = \Gal(L_\infty/K)$. Additionally, we consider:
                \begin{itemize}
                    \item{
                        $\cG'$, an arbitrary open subgroup of $\cG$. Let $K' = L_\infty^{\cG'}$, which is a finite (but not necessarily Galois) extension of $K$. Then $L_\infty$ is a Galois extension of the cyclotomic $\ZZ_p$-extension $K_\infty'$ of $K'$ and $H' = \Gal(L_\infty/K_\infty') = \cG' \cap H$. Set $\Gamma_{K'} = \Gal(K_\infty'/K)$.

                        We now fix an open central subgroup $\widetilde{\Gamma} \leq_o \cG$ isomorphic to $\ZZ_p$ as in setting \ref{sett:construction} and let $\Gamma = \widetilde{\Gamma} \cap \cG'$. In particular, $\Gamma$ is open and central in both $\cG$ and $\cG'$. As usual, we set $L = L_\infty^\Gamma$ and define $L_n = L_\infty^{\Gamma^{p^n}}, \cG_n = \Gal(L_n/K) = \cG/\Gp{n}, \cG_n' = \Gal(L_n/K') = \cG'/\Gp{n} \leq \cG_n$.

                        The restriction $\Gamma_{K'} \ia \Gamma_K$ is injective (because $K_\infty' = K_\infty K'$) and its image is an open subgroup $\Gamma_K^{p^M} \leq_o \Gamma_K$. We define the topological generator $\gamma_{K'}$ of $\Gamma_{K'}$ as the only preimage of $\gamma_K^{p^M}$ under the restriction map.
                    }
                    \item{
                        For each place $w$ of $K'$, a distinguished prolongation $w^c$ to $(K')^c = \QQ^c$. Then, for each place $v$ of $K$, we choose a place $v^c$ among $\set{w^c : w \in \set{v}(K')}$. This is only tangentially relevant, namely for the isomorphisms in \eqref{eq:isomorphism_y_functoriality_2} below.
                    }
                \end{itemize}
            \end{sett}
            With the above conventions, $(L_\infty/K, \Gamma, S, T)$ and $(L_\infty/K', \Gamma, S(K'), T(K'))$ are two valid choices of the corresponding parameters in setting \ref{sett:construction}. The relations between the relevant fields and Galois groups are depicted in the diagram
            
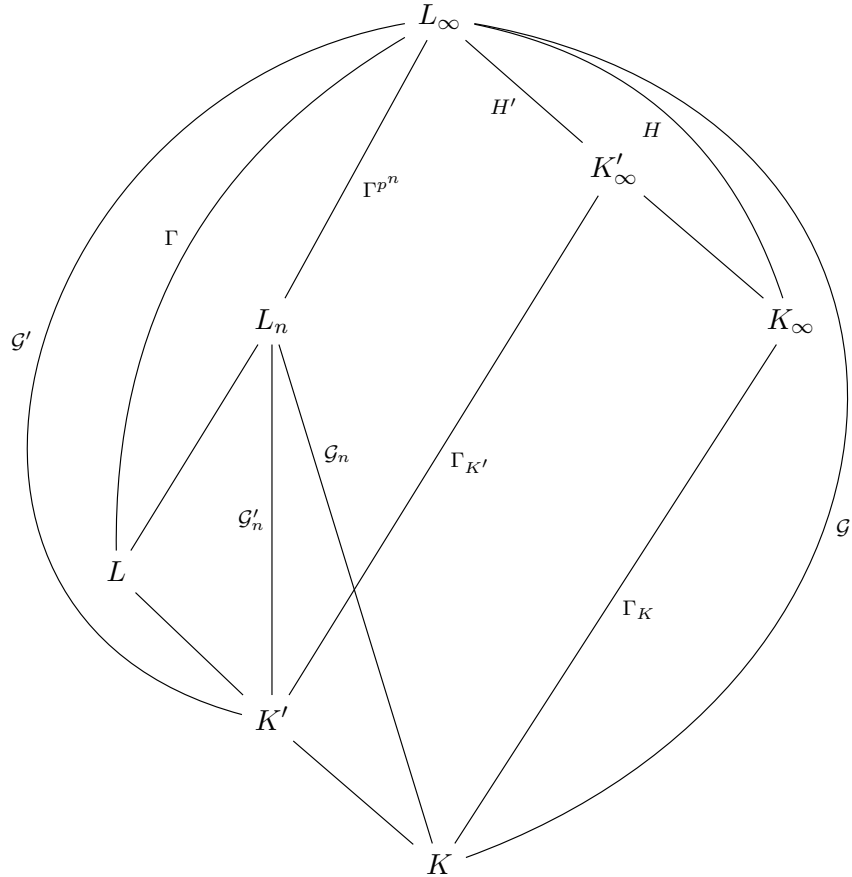
\begin{center}
                \begin{tikzcd}[row sep=huge, column sep=large]
                                                           &                                                                                                                     & L_\infty \arrow[rd, "H'"', no head] \arrow[rrdd, "H"', no head, bend left] \arrow[lldddd, "\Gamma"', no head, bend right] \arrow[dddddd, "\cG", no head, in=20, out=350, looseness=1.6, pos=0.6] \arrow[lddddd, "\cG'"', no head, in=165, out=190, looseness=1.4] &                               &          \\
                                                           &                                                                                                                     &                                                                                                                                                                                                                   & K_\infty' \arrow[rd, no head] &          \\
                                                           & L_n \arrow[ruu, "\Gamma^{p^n}"', no head] \arrow[ddd, "\cG_n'"', no head] \arrow[rdddd, "\cG_n", no head, near start] &                                                                                                                                                                                                                   &                               & K_\infty \\
                                                           &                                                                                                                     &                                                                                                                                                                                                                   &                               &          \\
                L \arrow[ruu, no head] \arrow[rd, no head] &                                                                                                                     &                                                                                                                                                                                                                   &                               &          \\
                                                           & K' \arrow[rruuuu, "\Gamma_{K'}"', no head] \arrow[rd, no head]                                            &                                                                                                                                                                                                                   &                               &          \\
                                                           &                                                                                                                     & K \arrow[rruuuu, "\Gamma_K"', no head]                                                                                                                                                                            &                               &
                \end{tikzcd}
            \end{center}

            The restriction of scalars induced by the embedding of topological rings $\Lambda(\cG') \ia \Lambda(\cG)$ allows the parameter $\alpha$ from setting \ref{sett:formulation} to remain unchanged when passing from $K$ to $K'$:
            \begin{itemize}
                \item{
                    The $\Lambda(\cG')$-module structure of
                    \[
                        E_{S(K'), T(K')} = \varprojlim_n \ZZ_p \otimes \units{\cO_{L_n, S(K'), T(K)}} = \varprojlim_n \ZZ_p \otimes \units{\cO_{L_n, S, T}} = E_{S, T}
                    \]
                    is simply the restriction of its $\Lambda(\cG)$-structure.
                }
                \item{
                    There is an isomorphism of $\Lambda(\cG')$-modules
                    \begin{align}
                    \label{eq:isomorphism_y_functoriality_2}
                        \cY_{S_\infty} =
                        \varprojlim_n \cY_{S_\infty, L_n} & \iso
                        \varprojlim_n \bigoplus_{w_n \in S(L_n)} \ZZ_p \cdot w_n \\
                        & = \varprojlim_n \bigoplus_{w_n \in S(K')(L_n)} \ZZ_p \cdot w_n \iso
                        \varprojlim_n \cY_{S_\infty(K'), L_n} =
                        \cY_{S_\infty(K')}.
                    \end{align}
                }
            \end{itemize}

            Therefore, any $\Lambda(\cG)$-homomorphism $\alpha \colon \cY_{S_\infty} \to E_{S, T}$ as in setting \ref{sett:formulation} constitutes a valid $\Lambda(\cG')$-homomorphism $E_{S(K'), T(K')} \to \cY_{S_\infty(K')}$. Note that conditions on $\Lambda(\Gamma)$-torsionness also remain unchanged, as the same $\Gamma$ has been chosen in $\cG'$ and $\cG$. As usual, the isomorphism $\beta \colon \CC_p \isoa \CC$ is independent of all other parameters.

            A few words are in order regarding the relation between Artin characters of $\cG$ and those of $\cG'$. We will repeatedly write expressions like $\indu_{\cG'}^\cG \chi'$ for $\chi' \in \Irr_p(\cG')$ even though induction was only introduced for characters of finite groups in section \ref{sec:representations_of_finite_groups}. This simply refers to the character
            \[
                \indu_{\cG'}^\cG \chi' = \infl_{\cG_n}^\cG \indu_{\cG_n'}^{\cG_n} \overline{\chi'}
            \]
            of the representation
            \[
                \indu_{\cG'}^\cG V_{\chi'} = \infl_{\cG_n}^\cG \indu_{\cG_n'}^{\cG_n} V_{\overline{\chi'}},
            \]
            where $\overline{\chi'}$ is the projection of $\chi'$ to any quotient $\cG_n'$ through which it factors. The result is independent of the choice of $\cG_n'$ and one has
            \[
                 \big(\indu_{\cG'}^\cG \chi'\big)(\sigma) = \sum_{\substack{\tau \in \cG/\cG'\\\tau^{-1} \sigma \tau \in \cG'}} \chi'(\tau^{-1} \sigma \tau)
            \]
            with $\tau$ running over any family of coset representatives of $\cG/\cG'$. This is essentially a reformulation of \eqref{eq:induction_character} which accommodates both the finite and infinite cases. Similarly, scalar products and divisibility of Artin characters can be addressed by passing to any finite quotient over which they factor.

            As for type-$W$ characters, it is clear that the restriction $\rest_{\cG'}^\cG \rho$ of such a character $\rho$ of $\cG$ is again a type-$W$ character of $\cG'$. Conversely, given any type-$W$ character $\rho'$ of $\cG'$, there exists a $\rho$ as above such that $\rho' = \rest_{\cG'}^\cG \rho$. To see this, note that such a $\rho'$ corresponds uniquely to the choice of a $p$-power root of unity $\rho'(\gamma_{K'}) \in \mu_{p^\infty}$ and $\gamma_{K'}$ identifies with $\gamma_K^{p^M}$ as in setting \ref{sett:functoriality_2}. One may therefore define $\rho$ by mapping $\gamma_K$ to a $p^M$-th root of $\rho'(\gamma_{K'})$.

            Character induction is fundamental to the relation of the Interpolation Conjectures for $\cG'$ and $\cG$. In this regard, we have the following fact, which is one of the reasons the interpolation property \eqref{eq:ic} is only claimed for almost all $\rho \in \cK_S^\alpha(\chi)$: each $\chi \in \Irr_p(\cG)$ divides $\indu_{\cG'}^\cG \chi'$ for at most finitely many $\chi' \in \Irr_p(\cG')$. This follows immediately from Frobenius reciprocity \eqref{eq:frobenius_reciprocity}, as
            \begin{equation}
            \label{eq:division_induction_finitely_many}
                \chi \mid \indu_{\cG'}^\cG \chi' \iff \sprod{\chi, \indu_{\cG'}^\cG \chi'}_\cG > 0 \iff \sprod{\rest_{\cG'}^\cG \chi, \chi'}_{\cG'} > 0 \iff \chi' \mid \rest_{\cG'}^\cG \chi,
            \end{equation}
            where the last condition is satisfied by finitely many $\chi'$ only.

            Unlike in the previous subsection, there is no natural ring homomorphism $\Lambda(\cG) \to \Lambda(\cG')$. $K$-theory does provide us, however, with a restriction map on $K_1$ groups: if $\iota \colon \Lambda(\cG') \ia \Lambda(\cG)$ denotes the canonical embedding, which extends to $\cQ(\cG') \ia \cQ(\cG)$ and $\cQ^c(\cG') \ia \cQ^c(\cG)$, one can define a group homomorphism $K_1^\rest(\iota) \colon K_1(\cQ(\cG)) \to K_1(\cQ(\cG'))$ as in section \ref{sec:algebraic-k-theory}. Together with some results from \cite{rwii}, this will be enough to prove the relation of the equivariant Main Conjectures for $L_\infty/K$ and $L_\infty/K'$. However, directly relating the Interpolation Conjectures demands a slightly finer analysis (which also relies on \cite{rwii}) and the construction of homomorphisms $\rest_{\chi'}^{\chi} \colon \units{\cQ^c(\Gamma_\chi)} \to \units{\cQ^c(\Gamma_{\chi'})}$ for $\chi \in \Irr_p(\cG), \chi' \in \Irr_p(\cG')$ in a compatible manner with $K_1^\rest(\iota)$. Given such a $\chi$, consider the injective ring homomorphism $j_\chi \colon Z(\cQ^c(\cG)e_\chi) = \cQ^c(\Gamma_\chi) \ia \cQ^c(\Gamma_K)$ induced by ${\gamma_\chi \mapsto \gamma_K^{w_\chi}}$, where $w_\chi$ is as in proposition \ref{prop:structure_gamma_chi} (compare with the map $j_\chi^F$ therein). These fields fit into the diagram
            \begin{equation}
            \label{eq:iwasawa_algebras_embeddings}
                \begin{tikzcd}
                    Z(\cQ^c(\cG)e_\chi) = \cQ^c(\Gamma_\chi) \arrow[r, "j_\chi", hook]       & \cQ^c(\Gamma_K)                                                           \\
                    Z(\cQ^c(\cG')e_{\chi'}) = \cQ^c(\Gamma_{\chi'}) \arrow[r, "j_{\chi'}", hook] & \cQ^c(\Gamma_{K'}) \arrow[u, "\gamma_{K'} \mapsto \gamma_K^{p^M}"', hook]
                \end{tikzcd}
            \end{equation}
            with $j_{\chi'}$ defined analogously to $j_\chi$ for $\chi'$ as above. Recall that these maps are closely related to our $\psi_\chi \colon K_1(\cQ(\cG)) \to \units{\cQ^c(\Gamma_\chi)}$, as shown in \eqref{eq:psi_and_j}.

            \begin{prop}
            \label{prop:character_res_interpolating}
                Setting \ref{sett:functoriality_2}. Let $\chi \in \Irr_p(\cG)$ and $\chi' \in \Irr_p(\cG')$. Then the restriction map
                \begin{align}
                \label{eq:definition_restriction_series_quotients}
                    \rest_{\chi'}^{\chi} \colon \units{\cQ^c(\Gamma_\chi)} & \to \units{\cQ^c(\Gamma_{\chi'})} \\
                    z & \mapsto j_{\chi'}^{-1} \Bigg(\prod_{\substack{\varsigma \in \Irr_p(\cG)\\\varsigma \sim_W \chi}} j_\varsigma(z)^{\sprod{\indu_{\cG'}^\cG \chi', \varsigma}}\Bigg) \nonumber
                \end{align}
                is a well-defined group homomorphism. Furthermore, $\rest_{\chi'}^{\chi} = \rest_{\widetilde{\chi}'}^{\widetilde{\chi}}$ for any $\widetilde{\chi} \sim_W \chi$ and $\widetilde{\chi}' \sim_W \chi'$, and the diagram
                \begin{center}
                    \begin{tikzcd}[column sep=large]
                        K_1(\cQ(\cG)) \arrow[r, "\prod_\chi \psi_\chi"] \arrow[d, "K_1^\rest(\iota)"] & \prod_\chi \units{\cQ^c(\Gamma_\chi)} \arrow[d, "\prod_{\chi} \prod_{\chi'} \rest_{\chi'}^\chi"] \\
                        K_1(\cQ(\cG')) \arrow[r, "\prod_{\chi'} \psi_{\chi'}"]                          & \prod_{\chi'} \units{\cQ^c(\Gamma_{\chi'})}
                    \end{tikzcd}
                \end{center}
                commutes, where $\chi$ and $\chi'$ run over $\Irr_p(\cG)/{\sim}_W$ and $\Irr_p(\cG')/{\sim}_W$, respectively. For any ${z \in \units{\cQ^c(\Gamma_\chi)}}$, one has
                \[
                    ev_{\gamma_{\chi'}}\big(\rest_{\chi'}^{\chi}(z) \big) = \prod_{\substack{\varsigma \in \Irr_p(\cG)\\\varsigma \sim_W \chi}} ev_{\gamma_\varsigma}(z)^{\sprod{\indu_{\cG'}^\cG \chi', \varsigma}}
                \]
                if the factors on the right-hand side with non-zero exponent are all finite.
            \end{prop}

            Before delving into the proof, we clarify that the last product is a shorthand for
            \[
                \prod_{\substack{\varsigma \in \Irr_p(\cG)\\\varsigma \sim_W \chi \\\varsigma \mid \indu_{\cG'}^\cG \chi'}} ev_{\gamma_\varsigma}(z)^{\sprod{\indu_{\cG'}^\cG \chi', \varsigma}},
            \]
            i.e. factors with exponent 0 are entirely disregarded (in particular, the product is finite). In other words, we adopt the convention that $0^0 = \infty^0 = 1$, which will be in place in other similar products as well.

            \begin{proof}
                As in \cite{rwii} theorem 7, we define $\Map^W(\Irr_p(\cG), \cQ^c(\Gamma_K))$ as the set of all maps $f \colon \Irr_p(\cG) \to \cQ^c(\Gamma_K)$ such that $f(\varsigma \otimes \rho) = \rho^\sharp(f(\varsigma))$ for all $\varsigma \in \Irr_p(\cG)$ and $\rho$ of type $W$. Here $\rho^\sharp$ is as in \eqref{eq:twisted_evaluation_maps_composition}, namely the field automorphism of $\cQ^c(\Gamma_K)$ induced by $\rho^\sharp(\gamma_K) = \rho(\gamma_K) \gamma_K$. The ring structure of $\cQ^c(\Gamma_K)$ induces a ring structure on $\Map^W(\Irr_p(\cG), \cQ^c(\Gamma_K))$, the unit group of which is precisely
                $\units{\Map^W(\Irr_p(\cG), \cQ^c(\Gamma_K))} = \Map^W(\Irr_p(\cG), \units{\cQ^c(\Gamma_K)})$. The proof of the cited theorem shows that
                \begin{align}
                \label{eq:isomorphism_hom_description}
                    \delta \colon \units{Z(\cQ^c(\cG))} &\to \Map^W(\Irr_p(\cG), \units{\cQ^c(\Gamma_K)}) \\
                    z &\mapsto \left[\varsigma \mapsto j_\varsigma(z e_\varsigma) \nonumber
                    \right]
                \end{align}
                is a group isomorphism. In particular, given $f \in \Map^W(\Irr_p(\cG), \units{\cQ^c(\Gamma_K)})$ and $\varsigma \in \Irr_p(\cG)$, one has $f(\varsigma) \in j_\varsigma(\units{\cQ(\Gamma_\varsigma)}) \subseteq \units{\cQ^c(\Gamma_K)}$. For any such $f$, the map
                \begin{align*}
                    \rest^W(f) \colon \Irr_p(\cG') & \to \units{\cQ^c(\Gamma_{K'})} \\
                               \varsigma' & \mapsto \prod_{\substack{\varsigma \in \Irr_p(\cG)\\}} f(\varsigma)^{\sprod{\indu_{\cG'}^\cG \varsigma', \varsigma}}
                \end{align*}
                is a well-defined element of $\Map^W(\Irr_p(\cG'), \units{\cQ^c(\Gamma_{K'})})$ under the vertical arrow in \eqref{eq:iwasawa_algebras_embeddings} by \cite{rwii} lemma 9.

                The central square of the diagram
                \begin{equation}
                \label{diag:large_diagram_restriction}
                    \begin{tikzcd}[row sep=large]
                                                                                                                             &                                                                              &                                                                                  & \units{Z(\cQ^c(\cG))} \arrow[ldd, "\delta"] \arrow[ddddd, dashed] \\
                                                                                                                             &                                                                              &                                                                                  &                                                                   \\
                        K_1(\cQ(\cG)) \arrow[d, "K_1^\rest(\iota)"] \arrow[r] \arrow[rrruu, "\prod_{\varsigma /{\sim}_W}\psi_\varsigma"] & K_1(\cQ^c(\cG)) \arrow[d, "K_1^\rest(\iota)"] \arrow[rruu, "\nr"'] \arrow[r, "\mu"] & {\Map^W(\Irr_p(\cG), \units{\cQ^c(\Gamma_K)})} \arrow[d, "\rest^W"]              &                                                                   \\
                        K_1(\cQ(\cG')) \arrow[r] \arrow[rrrdd, "\prod_{\varsigma'/{\sim}_W}\psi_{\varsigma'}"']                          & K_1(\cQ^c(\cG')) \arrow[r, "\mu'"] \arrow[rrdd, "\nr"]                               & {\Map^W(\Irr_p(\cG'), \units{\cQ^c(\Gamma_{K'})})} \arrow[rdd, "(\delta')^{-1}"] &                                                                   \\
                                                                                                                             &                                                                              &                                                                                  &                                                                   \\
                                                                                                                             &                                                                              &                                                                                  & \units{Z(\cQ^c(\cG'))}
                    \end{tikzcd}
                \end{equation}
                where $\mu$ and $\mu'$ are defined by commutativity of the triangles they are part of (note that $\delta'$ is an isomorphism), commutes by the same result\footnote{The lemma in question concerns $\Hom^\ast(\cR_p(\cG), \units{\cQ^c(\Gamma_K)})$, where $\cR_p(\cG)$ denotes the ring of virtual Artin characters of $\cG$ (cf. section \ref{sec:representations_of_finite_groups}). The essential difference between $\Map^W$ and $\Hom^\ast$ is that the latter adds a an extra Galois-equivariance condition. However, the same proof shows the case of interest to us - one simply needs to replace $P$ by a finitely generated $\cQ^c(\cG)$-module in the claim on p. 560 of \cite{rwii}, and the objects in the commutative triangle on p. 558 by their counterparts with scalars in $\QQ_p^c$.}. The dashed vertical arrow is the homomorphism $(\delta')^{-1} \rest^W \delta$, and its restriction to $\chi$- and $\chi'$-parts (in the domain and codomain, respectively) is precisely the map $\rest_{\chi'}^{\chi}$ from the statement. The fact that it is invariant under $\rho$-twist of $\chi$ and $\chi'$ is clear, since so are all elements of the diagram and the simple components of $\units{Z(\cQ^c(\cG))}$ and $\units{Z(\cQ^c(\cG'))}$. It can also be deduced directly from \eqref{eq:definition_restriction_series_quotients}.

                In order to determine the behaviour under twisted evaluation maps, we resort to the commutative diagram
                \begin{center}
                    \begin{tikzcd}
                                                                                                      & \units{\cQ^c(\Gamma_\varsigma)} \arrow[ld, "j_{\varsigma}"', hook] \arrow[rr, "ev_{\gamma_\varsigma}"]        &  & \QQ_p^c \cup \set{\infty} \arrow[dd, no head, equals] \\
                        \units{\cQ^c(\Gamma_K)} \arrow[rrru, "ev_{\gamma_K}"']                        &                                                                                                &  &                                                           \\
                                                                                                      & \units{\cQ^c(\Gamma_{\chi'})} \arrow[ld, "j_{\chi'}"', hook] \arrow[rr, "ev_{\gamma_{\chi'}}"] &  & \QQ_p^c \cup \set{\infty}                                 \\
                        \units{\cQ^c(\Gamma_{K'})} \arrow[uu, hook] \arrow[rrru, "ev_{\gamma_{K'}}"'] &                                                                                                &  &
                    \end{tikzcd}
                \end{center}
                for $\varsigma \in \Irr_p(\cG)$. We have already shown why the top and bottom faces commute (see for instance the triangles in \eqref{eq:diagram_ev_independent_gamma_K}), and the front face does for essentially the same reason: $\gamma_{\chi'}$ is mapped to 1 either way, and this determines the map uniquely. For any $z \in \units{\cQ^c(\Gamma_\chi)}$, one therefore has
                \begin{align*}
                     ev_{\gamma_{\chi'}}\big(\rest_{\chi'}^{\chi}(z) \big)
                     &= ev_{\gamma_{\chi'}}\Bigg(j_{\chi'}^{-1} \Bigg(\prod_{\substack{\varsigma \in \Irr_p(\cG)\\\varsigma \sim_W \chi}} j_\varsigma(z)^{\sprod{\indu_{\cG'}^\cG \chi', \varsigma}}\Bigg)\Bigg) \\
                     &= ev_{\gamma_{K'}}\Bigg(\prod_{\substack{\varsigma \in \Irr_p(\cG)\\\varsigma \sim_W \chi}} j_\varsigma(z)^{\sprod{\indu_{\cG'}^\cG \chi', \varsigma}}\Bigg) \\
                     &= ev_{\gamma_{K}}\Bigg(\prod_{\substack{\varsigma \in \Irr_p(\cG)\\\varsigma \sim_W \chi}} j_\varsigma(z)^{\sprod{\indu_{\cG'}^\cG \chi', \varsigma}}\Bigg) \\
                     &= \prod_{\substack{\varsigma \in \Irr_p(\cG)\\\varsigma \sim_W \chi}} ev_{\gamma_K}(j_\varsigma(z))^{\sprod{\indu_{\cG'}^\cG \chi', \varsigma}} \\
                     &= \prod_{\substack{\varsigma \in \Irr_p(\cG)\\\varsigma \sim_W \chi}} ev_{\gamma_\varsigma}(z)^{\sprod{\indu_{\cG'}^\cG \chi', \varsigma}}.
                \end{align*}
                The only equalities which do not follow from the last diagram are the first one, which is simply the definition of $\rest_{\chi'}^{\chi}$, and the fourth one, which follows from remark \ref{rem:definition_of_evaluation} ii).
            \end{proof}

            One may be tempted to rewrite the product in \eqref{eq:definition_restriction_series_quotients} as
            \[
                \prod_{\rho} j_\varsigma(z)^{\sprod{\indu_{\cG'}^\cG \chi', \chi \otimes \rho}}
            \]
            with $\rho$ running over the type-$W$ characters of $\cG$. We point out that the two expressions will differ in general, since $\chi \otimes \rho$ and $\chi \otimes \tilde{\rho}$ may coincide even if $\rho$ and $\tilde{\rho}$ do not (cf. remark \ref{rem:properties_of_ev} i)).

            The previous proposition establishes a relation between the Interpolation Conjectures for $\cG$ and $\cG'$, and this in turn between the corresponding equivariant Main Conjectures:

            \begin{prop}
            \label{prop:functoriality_2}
                Setting \ref{sett:functoriality_2}. Let $\alpha$ and $\beta$ be chosen as in setting \ref{sett:formulation} for the extension $L_\infty/K$. In particular, they are valid choices for $L_\infty/K'$ as explained above. Then:
                \begin{enumerate}[i)]
                    \item{
                        Let $\chi' \in \Irr_p(\cG')$. If \hyperref[conje:ic]{IC($L_\infty/K, \chi, L, S, T, \alpha, \beta$)} holds for all $\chi \in \Irr_p(\cG)$ dividing $\indu_{\cG'}^\cG \chi'$, then so does \hyperref[conje:ic]{IC($L_\infty/K', \chi', L, S(K'), T(K'), \alpha, \beta$)}.
                    }
                    \item{
                        \hyperref[conje:emc]{eMC($L_\infty/K, L, S, T, \alpha, \beta$)} implies \hyperref[conje:emc]{eMC($L_\infty/K', L, S(K'), T(K'), \alpha, \beta$)}.
                    }
                \end{enumerate}
            \end{prop}

            \begin{proof}
                For each irreducible $\chi \mid \indu_{\cG'}^\cG \chi'$, let $F_{S, T, \chi}^{\alpha, \beta} \in \units{\cQ^c(\Gamma_\chi)}$ be the unique series quotient satisfying IC($L_\infty/K, \chi, L, S, T, \alpha, \beta$). Choose a set $\Xi_{\chi'}$ of representatives of these $\chi$ modulo $\sim_W$ (we do not claim that $\chi \mid \indu_{\cG'}^\cG \chi'$ implies $\chi \otimes \rho \mid \indu_{\cG'}^\cG \chi'$) and define
                \[
                    F_{\chi'} = \prod_{\chi \in \Xi_{\chi'}} \rest_{\chi'}^\chi (F_{S, T, \chi}^{\alpha, \beta}) \in \units{\cQ^c(\Gamma_{\chi'})}.
                \]

                Given a character $\varsigma \in \Irr_p(\cG)$ which is $W$-equivalent to some (necessarily unique) $\chi \in \Xi_{\chi'}$, set
                \[
                    F_{S, T, \varsigma}^{\alpha, \beta} = F_{S, T, \chi}^{\alpha, \beta} \in \units{\cQ^c(\Gamma_\chi)} = \units{\cQ^c(\Gamma_\varsigma)}.
                \]
                Note that if two $W$-equivalent characters $\chi$ and $\widetilde{\chi}$ both divide $\indu_{\cG'}^\cG \chi'$, one has $F_{S, T, \chi}^{\alpha, \beta} = F_{S, T, \widetilde{\chi}}^{\alpha, \beta}$ by proposition \ref{prop:ic_rho_twist} so that the two definitions of $F_{S, T, \widetilde{\chi}}^{\alpha, \beta}$ (one from the Interpolation Conjecture and one from the last equation) do not conflict. The same proposition shows that $F_{S, T, \varsigma}$ satisfies IC($L_\infty/K, \varsigma, L, S, T, \alpha, \beta$).

                Let $\rho'$ be a character of $\cG'$ of type $W$ and choose a type-$W$ character $\rho$ of $\cG$ such that $\rest_{\cG'}^\cG \rho = \rho'$ as explained earlier. It is easy to see (for instance, using Frobenius reciprocity) that
                \begin{equation}
                \label{eq:induction_and_rho_twists}
                    \indu_{\cG'}^\cG (\chi' \otimes \rho') = (\indu_{\cG'}^\cG \chi') \otimes \rho.
                \end{equation}

                Proposition \ref{prop:character_res_interpolating} then yields
                \begin{align*}
                    ev_{\gamma_{\chi' \otimes \rho'}}(F_{\chi'})
                    &= \prod_{\chi \in \Xi_{\chi'}} ev_{\gamma_{\chi' \otimes \rho'}}(\rest_{\chi'}^\chi(F_{S, T, \chi}^{\alpha, \beta})) \\
                    &= \prod_{\chi \in \Xi_{\chi'}} ev_{\gamma_{\chi' \otimes \rho'}}(\rest_{\chi' \otimes \rho'}^\chi(F_{S, T, \chi}^{\alpha, \beta})) \\
                    & = \prod_{\chi \in \Xi_{\chi'}} \: \prod_{\substack{\varsigma \in \Irr_p(\cG)\\\varsigma \sim_W \chi}} ev_{\gamma_\varsigma}(F_{S, T, \varsigma}^{\alpha, \beta})^{\sprod{\indu_{\cG'}^\cG (\chi' \otimes \rho'), \varsigma}} \\
                    & = \prod_{\substack{\varsigma \in \Irr_p(\cG)\\\varsigma \mid \indu_{\cG'}^\cG (\chi' \otimes \rho')}} ev_{\gamma_\varsigma}(F_{S, T, \varsigma}^{\alpha, \beta})^{\sprod{\indu_{\cG'}^\cG (\chi' \otimes \rho'), \varsigma}}
                \end{align*}
                whenever all factors are finite, where the last equality follows from \eqref{eq:induction_and_rho_twists}. Consider now the set
                \[
                    \Omega = \set{\varsigma \in \Irr_p(\cG) \et{such that} \varsigma \sim_W \chi \et{for some} \chi \in \Xi_{\chi'}}.
                \]
                which coincides with
                \[
                    \set{\varsigma \in \Irr_p(\cG) \et{such that} \varsigma \mid \indu_{\cG'}^\cG (\chi' \otimes \rho') \et{for some} \rho' \et{of type} W}
                \]
                by \eqref{eq:induction_and_rho_twists}. Then
                \[
                    ev_{\gamma_{\varsigma}}(F_{S, T, \varsigma}^{\alpha, \beta}) = \frac{\beta^{-1}(L_{K, S, T}^\ast(\beta \check{\varsigma}, 0))}{R_S^\beta(\alpha, \varsigma)}
                \]
                for almost all $\varsigma \in \Omega$. The finitely many $\varsigma$ where interpolation fails divide only finitely many irreducible characters of $\cG'$ by \eqref{eq:division_induction_finitely_many}, which implies that
                \[
                    ev_{\gamma_{\chi' \otimes \rho'}}(F_{\chi'}) = \prod_{\substack{\varsigma \in \Irr_p(\cG)\\\varsigma \mid \indu_{\cG'}^\cG (\chi' \otimes \rho')}} ev_{\gamma_\varsigma}(F_{S, T, \varsigma}^{\alpha, \beta})^{\sprod{\indu_{\cG'}^\cG (\chi' \otimes \rho'), \varsigma}} = \prod_{\substack{\varsigma \in \Irr_p(\cG)\\\varsigma \mid \indu_{\cG'}^\cG (\chi' \otimes \rho')}} \left(\frac{\beta^{-1}(L_{K, S, T}^\ast(\beta \check{\varsigma}, 0))}{R_S^\beta(\alpha, \varsigma)}\right)^{\sprod{\indu_{\cG'}^\cG (\chi' \otimes \rho'), \varsigma}}
                \]
                for almost all type-$W$ characters $\rho'$ of $\cG'$. Now we simply use the functoriality of the regulated special $L$-values, namely their behaviour with respect to character addition and induction. For the leading coefficient at 0, this is lemma \ref{lem:properties_of_L-functions} iii) and iv). We cannot directly speak of functoriality of the Stark-Tate regulator from section \ref{sec:regulators}, since it was only defined for irreducible characters. However, its relation to that defined in \cite{tate} and the functorial properties of the latter suffice to show
                \begin{equation}
                \label{eq:special_value_functorial}
                    \prod_{\substack{\varsigma \in \Irr_p(\cG)\\\varsigma \mid \indu_{\cG'}^\cG (\chi' \otimes \rho')}} \left(\frac{\beta^{-1}(L_{K, S, T}^\ast(\beta \check{\varsigma}, 0))}{R_S^\beta(\alpha, \varsigma)}\right)^{\sprod{\indu_{\cG'}^\cG (\chi' \otimes \rho'), \varsigma}}
                    = \frac{\beta^{-1}(L_{K', S(K'), T(K')}^\ast(\beta (\check{\chi'} \otimes (\rho')^{-1}), 0))}{R_{S(K')}^\beta(\alpha, \chi' \otimes \rho')}.
                \end{equation}
                More specifically, the equality of denominators is a consequence of equations 6.4 (1) and (2) on p. 29 of \cite{tate} together with \eqref{eq:relation_regulator_tate_mine} above. Recall that Tate only considers $S$-modified Artin $L$-functions (i.e. with $T = \varnothing$), but the $\delta$-factors at places in $T$ obey the same formalism as already argued in the proof of proposition \ref{prop:stark_independent_of_t}. This proves that $F_{\chi'}$ satisfies IC($L_\infty/K', \chi', L, S(K'), T(K'), \alpha, \beta$).

                The implication of equivariant Main Conjectures follows readily. Suppose $\zeta_{S, T}^{\alpha, \beta} \in K_1(\cQ(\cG))$ satisfies eMC($L_\infty/K, L, S, T, \alpha, \beta$). In particular, for each irreducible Artin character $\chi$ of $\Irr_p(\cG)$, the series quotient $F_{S, T, \chi}^{\alpha, \beta} = \psi_{\chi}(\zeta_{S, T}^{\alpha, \beta}) \in \units{\cQ^c(\Gamma_\chi)}$ satisfies IC($L_\infty/K, \chi, L, S, T, \alpha, \beta$). Define
                \[
                    \zeta' = K_1^\rest(\iota)(\zeta_{S, T}^{\alpha, \beta}) \in K_1(\cQ(\cG')).
                \]

                For a fixed $\chi' \in \Irr_p(\cG')$, define $\Xi_{\chi'}$ as above and extend it to a set $\Xi$ of representatives of $\Irr_p(\cG)/{\sim}_W$. Then by proposition \ref{prop:character_res_interpolating}, one has
                \[
                    \psi_{\chi'}(\zeta') = \psi_{\chi'}(K_1^\rest(\iota)(\zeta_{S, T}^{\alpha, \beta}) ) = \prod_{\chi \in \Xi} \rest_{\chi'}^\chi(\psi_\chi(\zeta_{S, T}^{\alpha, \beta})) = \prod_{\chi \in \Xi} \rest_{\chi'}^\chi(F_{S, T, \chi}^{\alpha, \beta}) = \prod_{\chi \in \Xi_{\chi'}} \rest_{\chi'}^\chi(F_{S, T, \chi}^{\alpha, \beta}),
                \]
                where the last equality relies on the observation that, for $\chi \in \Xi \setminus \Xi_{\chi'}$, the map $\rest_{\chi'}^\chi$ is identically 1. Note that by the proof of part i), the right-hand side is precisely the element $F_{S(K'), T(K'), \chi'}^{\alpha, \beta}$ satisfying IC($L_\infty/K', \chi', L, S(K'), T(K'), \alpha, \beta$).

                Moving on to the homological side, let
                \[
                    t^\alpha \colon \cQ(\cG) \otimes_{\Lambda(\cG)} H^1(\cC_{S, T}\q) \isoa \cQ(\cG) \otimes_{\Lambda(\cG)} H^0(\cC_{S, T}\q)
                \]
                and
                \[
                    \tilde{t}^\alpha \colon \cQ(\cG') \otimes_{\Lambda(\cG')} H^1(\cC_{S(K'), T(K')}\q) \isoa \cQ(\cG') \otimes_{\Lambda(\cG')} H^0(\cC_{S(K'), T(K')}\q)
                \]
                denote the trivialisations induced by $\alpha$ as a homomorphism of $\Lambda(\cG)$- and $\Lambda(\cG')$-modules, respectively, as in \eqref{eq:integral_trivialisation}. Then $\tilde{t}^\alpha = \restr{t^\alpha}{\cQ(\cG')}$ and $\cC_{S(K'), T(K')}\q = \restr{\cC_{S, T}\q}{\Lambda(\cG')}$ in the notation of lemma \ref{lem:functoriality_rec} ii), which shows
                \begin{align*}
                    \partial(\zeta') = \partial(K_1^\rest(\iota)(\zeta_{S, T}^{\alpha, \beta})) &= K_0^\rest(\iota, \iota)(\partial(\zeta_{S, T}^{\alpha, \beta})) \\
                    &=  K_0^\rest(\iota, \iota)(-\chi_{\Lambda(\cG), \cQ(\cG)}(\cC_{S, T}\q, t^\alpha)) \\
                    &= - \chi_{\Lambda(\cG'), \cQ(\cG')}(\restr{\cC_{S, T}\q}{\Lambda(\cG')}, \restr{t^\alpha}{\cQ(\cG')}) \\
                    &= - \chi_{\Lambda(\cG'), \cQ(\cG')}(\cC_{S(K'), T(K')}\q, \tilde{t}_\alpha)
                \end{align*}
                (for the second equality, see \eqref{eq:localisation_sequence_restriction}). Therefore, $\zeta'$ satisfies eMC($L_\infty/K', L, S(K'), T(K'), \alpha, \beta$).
            \end{proof}

        \subsection{Reduction steps}
        \label{subsec:reduction_steps}

            We end this section, and by extension this chapter, by exploring the converse to the question treated in the preceding lines, namely: can the conjectures for $L_\infty/K$ be deduced from their counterparts for smaller extensions? The answer is affirmative if the latter are chosen appropriately, and we give two examples of such possible choices. The technique we use mirrors the argument in \cite{johnston_nickel} section 10, the main difference being that the relation between the Interpolation Conjectures must be proved in our setting - as opposed to the totally real case the cited paper is concerned with, where $p$-adic $L$-functions interpolating $L$-values are known to exist.

            Before proceeding with the outlined plan, we make a minor detour to state an immediate consequence of the two preceding functoriality results. Only in this corollary do we dispense with the explicit notation of all parameters of the conjectures, opting instead for the abbreviations \hyperref[conje:ic]{IC($L_\infty/K, \chi, \beta$)} and \hyperref[conje:emc]{eMC($L_\infty/K, \beta$)} introduced at the beginning of the present chapter (recall theorem \ref{thm:independence_of_parameters}).

            \begin{cor}
                Let the choice of an odd prime $p$ and an isomorphism $\beta \colon \CC_p \isoa \CC$ be fixed for all conjectures below.
                \begin{enumerate}[i)]
                    \item{
                        Suppose \hyperref[conje:ic]{IC($\widetilde{L}_\infty/\QQ, \chi, \beta$)} holds for every choice of $\widetilde{L}_\infty/K$ as in setting \ref{sett:construction} with $K = \QQ$ and every $\chi \in \Irr_p(\Gal(\widetilde{L}_\infty/\QQ))$. Then so does \hyperref[conje:ic]{IC($L_\infty'/K', \chi', \beta$)} for any valid choice of $L_\infty'/K'$ and any $\chi' \in \Irr_p(\Gal(L_\infty'/K'))$.
                    }
                    \item{
                        Suppose \hyperref[conje:emc]{eMC($\widetilde{L}_\infty/\QQ, \beta$)} holds for every choice of $\widetilde{L}_\infty/K$ as in setting \ref{sett:construction} with $K = \QQ$. Then so does \hyperref[conje:emc]{eMC($L_\infty'/K', \beta$)} for any valid choice of $L_\infty'/K'$.
                    }
                \end{enumerate}
            \end{cor}

            \begin{proof}
                Let $\QQ_\infty$ denote the cyclotomic $\ZZ_p$-extension of $\QQ$. Given $L_\infty'/K'$ as in the statement, choose a number field $L' \supseteq K'$ whose cyclotomic $\ZZ_p$-extension coincides with $L_\infty'$ and define $\widetilde{L}$ as the Galois closure of $L'/\QQ$. Its cyclotomic $\ZZ_p$-extension $\widetilde{L}_\infty$ is Galois over $\QQ$ (as it coincides with $\widetilde{L}\QQ_\infty$) and has finite degree over $\QQ_\infty$, which means $\widetilde{L}_\infty/\QQ$ is a suitable extension for the formulation of the conjectures. Propositions \ref{prop:functoriality_2} and \ref{prop:functoriality_1} now yield
                \begin{align*}
                    & \ \text{IC($\widetilde{L}_\infty/\QQ, \chi, \beta$) for all} \ \chi \in \Irr_p(\Gal(\widetilde{L}_\infty/\QQ))\\
                    \Rightarrow & \ \text{IC($\widetilde{L}_\infty/K', \widetilde{\chi}, \beta$) for all} \ \widetilde{\chi} \in \Irr_p(\Gal(\widetilde{L}_\infty/K'))\\
                    \Rightarrow & \ \text{IC($L_\infty'/K', \chi', \beta$) for all} \ \chi' \in \Irr_p(\Gal(L_\infty'/K'))
                \end{align*}
                and
                \vspace{-1em}
                \begin{center}
                    eMC($\widetilde{L}_\infty/\QQ, \beta$) $\Rightarrow$ eMC($\widetilde{L}_\infty/K', \beta$) $\Rightarrow$ eMC($L_\infty'/K', \beta$).
                \end{center}
            \end{proof}

            Returning to the main purpose of this subsection, the first step is to reduce the Interpolation Conjecture to the case of linear characters, which conspicuously hints at Brauer's induction theorem \ref{thm:brauer_induction} and thus leads to the consideration of elementary groups. Since our treatment of Artin characters of $\cG$ often relies on factoring them through some finite layer, the following definition from \cite{johnston_nickel} section 10 should not come as a surprise:
            \begin{defn}
            \label{defn:elementary_infinite}
                Let $p$ be a prime and $\cG$ a profinite group which fits into a short exact sequence $H \ia \cG \sa \overline{\Gamma}$ with $H$ finite and $\overline{\Gamma} \iso \ZZ_p$. Given a prime $l$, we say $\cG$ is \textbf{$l$-elementary}\index{p-elementary@$p$-elementary!one-dimensional $p$-adic Lie group} if there exists an open central subgroup $\Gamma \iso \ZZ_p$ of $\cG$ such that $\cG/\Gamma$ is an $l$-elementary finite group (see for instance section \ref{sec:representations_of_finite_groups} for a definition). More generally, $\cG$ is said to be \textbf{elementary} if it is $l$-elementary for at least one $l$.\qedef
            \end{defn}

            In the notation of setting \ref{sett:construction}, let $\chi$ be an Artin character of $\cG$. Choose a finite quotient $\cG_n = \cG/\Gp{n}$ over which $\chi$ factors and denote the corresponding projection by $\overline{\chi}$. By Brauer's theorem, there exists a decomposition
            \[
                \overline{\chi} = \sum_{i = 1}^n z_i \; \indu_{\overline{\cH}_i}^{\cG_n} \overline{\lambda}_i,
            \]
            where $z_i \in \ZZ$, $\overline{\cH}_i$ is an elementary subgroup of $\cG_n$ and $\overline{\lambda}_i$ is a $\QQ_p^c$-valued linear character of $\overline{\cH}_i$. The preimage $\cH_i$ of $\overline{\cH}_i$ under $\cG \sa \cG_n$ is elementary in the sense defined above, and $\overline{\lambda}_i$ inflates to a linear character $\lambda_i = \infl_{\overline{\cH}_i}^{\cH_i} \overline{\lambda}_i$ which fits into
            \begin{equation}
            \label{eq:brauer_infinite}
                \chi = \sum_{i = 1}^n z_i \; \indu_{\cH_i}^\cG \lambda_i.
            \end{equation}
            As explained in the previous subsection, a choice of parameters as in setting \ref{sett:formulation} for $L_\infty/K$ yields a natural set of parameters for $L_\infty/K^{\cH_i}$ for each $i$.

            \begin{lem}
            \label{lem:ic_brauer_induction}
                Setting \ref{sett:formulation}. Given $\chi \in \Irr_p(\cG)$, choose a decomposition $\chi = \sum_{i = 1}^n z_i \; \indu_{\cH_i}^\cG \lambda_i$ as in \eqref{eq:brauer_infinite}. For each $i$, set $K^i = L_\infty^{\cH_i}$ and $L^i = L_\infty^{\cH_i \cap \Gamma}$. Then, if \hyperref[conje:ic]{IC($L_\infty/K^i, \lambda_i, L^i, S(K^i), T(K^i), \alpha, \beta$)} holds for all $i$, so does \hyperref[conje:ic]{IC($L_\infty/K, \chi, L, S, T, \alpha, \beta$)}.
            \end{lem}

            \begin{proof}
                Let $\Gamma_{K^i}$ denote the Galois group over $K^i$ of its cyclotomic $\ZZ_p$-extension $K^i_\infty = K_\infty K^i$. Restriction of Galois automorphisms yields an embedding $\Gamma_{K^i} \ia \Gamma_K$ with image $\Gamma_K^{p^{M_i}}$ for some $M_i \in \NN$. The formulation of the Interpolation Conjecture for $\lambda_i$ contains the implicit choice of a topological generator $\gamma_{K^i}$ of $\Gamma_{K^i}$, which we choose to be the unique preimage of $\gamma_K^{p^{M_i}}$.

                Consider the injective ring homomorphisms
                \[
                    \cQ^c(\Gamma_{\lambda_i}) \xhookrightarrow{j_{\lambda_i}} \cQ^c(\Gamma_{K^i}) \xhookrightarrow{j_{K^i}} \cQ^c(\Gamma_K),
                \]
                where $j_{K^i}$ is induced by $\Gamma_{K^i} \ia \Gamma_K$, and $\cQ^c(\Gamma_{\lambda_i}) = Z(\cQ^c(\cG_i)e_{\lambda_i})$ and $j_{\lambda_i} \colon \gamma_{\lambda_i} \mapsto \gamma_{K^i}^{w_{\lambda_i}} = \gamma_{K^i}$ are as in section \ref{sec:evaluation_maps} (cf. proposition \ref{prop:structure_gamma_chi}). Note that $w_{\lambda_i} = 1$ because $\lambda_i$ is linear.

                Let $F_i = F_{S(K^i), T(K^i), \lambda_i}^{\alpha, \beta} \in \units{\cQ^c(\Gamma_{\lambda_i})}$ be the element predicted by IC($L_\infty/K^i, \lambda_i, L^i, S(K^i), T(K^i), \alpha, \beta$). Recall that, for any type-$W$ character $\rho'$ of $\cH_i$, one has $\cQ^c(\Gamma_{\lambda_i}) = \cQ^c(\Gamma_{\lambda_i \otimes \rho'}) \subseteq Z(\cQ^c(\cH_i)e_{\lambda_i \otimes \rho'})$ and, as shown in proposition \ref{prop:ic_rho_twist}, $F_i$ also satisfies IC($L_\infty/K^i, \lambda_i \otimes \rho', L^i, S(K^i), T(K^i), \alpha, \beta$).

                Fix now a type-$W$ character $\rho$ of $\cG$ and set $\rho_i = \rest_{\cH_i}^\cG \rho$ (which is a type-$W$ character of $\cH_i$) for each $i$. Define
                \[
                    F_\rho = \prod_{i = 1}^n j_{K_i}(j_{\lambda_i \otimes \rho_i}(F_i^{z_i})) \in \units{\cQ^c(\Gamma_K)}.
                \]
                We first prove an interpolation property for this element. Namely, for any type-$W$ character $\tilde{\rho}$ of $\cG$, one has
                \[
                   \restr{F_\rho}{T_K = \tilde{\rho}(\gamma_K) - 1} = \prod_{i = 1}^n \restr{j_{K_i}(j_{\lambda_i \otimes \rho_i}(F_i^{z_i}))}{T_K = \tilde{\rho}(\gamma_K) - 1} = \prod_{i = 1}^n ev_{\gamma_K}(j_{K_i}(j_{\lambda_i \otimes \rho_i \otimes \tilde{\rho}_i}(F_i^{z_i}))) = \prod_{i = 1}^n ev_{\gamma_{\lambda_i \otimes \rho_i \otimes \tilde{\rho}_i}}(F_i)^{z_i}
                \]
                whenever all factors in the last product are finite (this is used in the first and last equalities). Here $\tilde{\rho}_i = \rest_{\cH_i}^\cG \tilde{\rho}$, the last equality is a property we have already encountered multiple times, and the second one follows from the commutativity of
                \begin{equation}
                \label{eq:j_chi_in_subfields}
                    \begin{tikzcd}[column sep=large]
                        \cQ^c(\Gamma_{\lambda_i \otimes \rho_i}) \arrow[rr, "j_{\lambda_i \otimes \rho_i}"] \arrow[d, equals]                             &  & \cQ^c(\Gamma_{K_i}) \arrow[r, "j_{K_i}"] \arrow[d, "\tilde{\rho}_i^\sharp"] & \cQ^c(\Gamma_K) \arrow[d, "\tilde{\rho}^\sharp"] \arrow[rr, "T_K = \tilde{\rho}(\gamma_K) - 1"] &  & \QQ_p^c \cup \set{\infty} \arrow[d, equals] \\
                        \cQ^c(\Gamma_{\lambda_i \otimes \rho_i \otimes \tilde{\rho}_i}) \arrow[rr, "j_{\lambda_i \otimes \rho_i \otimes \tilde{\rho}_i}"] &  & \cQ^c(\Gamma_{K_i}) \arrow[r, "j_{K_i}"]                                    & \cQ^c(\Gamma_K) \arrow[rr, "ev_{\gamma_K}"]                                             &  & \QQ_p^c \cup \set{\infty}
                    \end{tikzcd}
                \end{equation}
                (recall that $\rho^\sharp$ is the automorphism of $\cQ^c(\Gamma_K)$ determined by mapping $\gamma_K$ to $\rho(\gamma_K)\gamma_K$, and analogously for $\rho_i^\sharp \colon \cQ^c(\Gamma_{K_i}) \isoa \cQ^c(\Gamma_{K_i}), \gamma_{K_i} \mapsto \rho_i(\gamma_{K_i})\gamma_{K_i}$).

                In order to determine the above product, let $\cK_i$ be the set of all $\rho' \in \cK_{S(K^i)}^\alpha(\lambda_i)$ such that
                \[
                    ev_{\gamma_{\lambda_i \otimes \rho_i \otimes \rho'}}(F_i) = \frac{\beta^{-1}(L_{K^i, S(K^i), T(K^i)}^\ast(\beta ((\lambda_i \otimes \rho_i)\check{} \otimes (\rho')^{-1}), 0))}{R_{S(K^i)}^\beta(\alpha, \lambda_i \otimes \rho_i \otimes \rho')},
                \]
                which contains almost all type-$W$ characters of $\cH_i$ by IC($L_\infty/K_i, \lambda_i \otimes \rho_i, L^i, S(K^i), T(K^i), \alpha, \beta$). It follows that, for almost all type-$W$ characters $\tilde{\rho}$ of $\cG$, one has: $\rest_{\cH_i}^\cG(\rho) \in \cK_i$ for all $i$. In particular,
                \begin{align}
                \label{eq:equation_f_rho_interpolates}
                    \restr{F_\rho}{T_K = \tilde{\rho}(\gamma_K) - 1}
                    = \prod_{i = 1}^n ev_{\gamma_{\lambda_i \otimes \rho_i \otimes \tilde{\rho}_i}}(F_i)^{z_i}
                    & = \prod_{i = 1}^n \left(\frac{\beta^{-1}(L_{K^i, S(K^i), T(K^i)}^\ast(\beta ((\lambda_i \otimes \rho_i)\check{} \otimes \tilde{\rho}_i^{-1}), 0))}{R_{S(K^i)}^\beta(\alpha, \lambda_i \otimes \rho_i \otimes \tilde{\rho}_i)}\right)^{z_i}\nonumber \\
                    & = \frac{\beta^{-1}(L_{K, S, T}^\ast(\beta (\check{\chi} \otimes \rho^{-1} \otimes \tilde{\rho}^{-1}), 0))}{R_S^\beta(\alpha, \chi \otimes \rho \otimes \tilde{\rho})}
                \end{align}
                for any such $\tilde{\rho}$, where the last equality is a consequence of
                \[
                    \chi \otimes \rho \otimes \tilde{\rho} = \sum_{i = 1}^n z_i (\indu_{\cG_i}^\cG \lambda_i) \otimes \rho \otimes \tilde{\rho} = \sum_{i = 1}^n z_i \; \indu_{\cG_i}^\cG (\lambda_i \otimes \rho_i \otimes \tilde{\rho}_i)
                \]
                and an analogous argument to that for \eqref{eq:special_value_functorial} (which is unaffected by the fact that the $z_i$ may be negative).

                We now prove that $F_\bbone$ lies in the image of $\units{\cQ^c(\Gamma_{\chi})} \xhookrightarrow{j_{\chi}} \units{\cQ^c(\Gamma_K)}$. Suppose first that ${\chi \otimes \rho = \chi \otimes \rho'}$ for two type-$W$-characters $\rho$ and $\rho'$ of $\cG$. Then $F_\rho$ and $F_{\rho'}$ take the same value at infinitely many points of $B_1^{\QQ_p^c}(0)$ by \eqref{eq:equation_f_rho_interpolates}, since the regulated $L$-values for $\chi \otimes \rho \otimes \tilde{\rho}$ and $\chi \otimes \rho' \otimes \tilde{\rho}$ coincide for any $\tilde{\rho}$. Thus, lemma \ref{lem:series_quotient_uniqueness} implies $F_\rho = F_{\rho'}$. In particular, the map ${f \colon \Irr_p(\cG) \to \units{\cQ^c(\Gamma_K)}}$ given by
                \begin{equation}
                \label{eq:definition_brauer_induction_hom}
                    f(\varsigma) =
                    \begin{cases}
                        F_\rho, & \varsigma = \chi \otimes \rho \et{for some} \rho \et{of type} W \\
                        1, & \text{otherwise}
                    \end{cases}
                \end{equation}
                is well defined.

                We claim that $f \in \Map^W(\Irr_p(\cG), \units{\cQ^c(\Gamma_K)})$ in the notation of the previous subsection (see the proof of \ref{prop:character_res_interpolating}), i.e. $f(\varsigma \otimes \rho) = \rho^\sharp(f(\varsigma))$ for all $\varsigma \in \Irr_p(\cG)$ and $\rho$ of type $W$. We distinguish three cases:
                \begin{itemize}
                    \item{
                        If $\varsigma \nsim_W \chi$ then $f(\varsigma \otimes \rho) = f(\varsigma) = 1$.
                    }
                    \item{
                        If $\varsigma = \chi$, one has
                        \[
                            \rho^\sharp(f(\chi)) = \prod_{i = 1}^n \rho^\sharp(j_{K_i}(j_{\lambda_i}(F_i^{z_i}))) = \prod_{i = 1}^n j_{K_i}(\rho_i^\sharp(j_{\lambda_i}(F_i^{z_i}))) = \prod_{i = 1}^n j_{K_i}(j_{\lambda_i \otimes \rho_i}(F_i^{z_i})) = f(\chi \otimes \rho),
                        \]
                        where the second and third equalities follow from an analogous diagram to \eqref{eq:j_chi_in_subfields}.
                    }
                    \item{
                        If $\varsigma = \chi \otimes \tilde{\rho}$ for some $\tilde{\rho}$ of type $W$, the previous case implies
                        \[
                            f(\varsigma \otimes \rho) = f(\chi \otimes (\tilde{\rho} \otimes \rho)) = (\tilde{\rho} \otimes \rho)^\sharp(f(\chi)) = \rho^\sharp(\tilde{\rho}^\sharp(f(\chi))) = \rho^\sharp(f(\chi \otimes \tilde{\rho})) = \rho^\sharp(f(\varsigma)).
                        \]
                    }
                \end{itemize}
                Therefore, \eqref{eq:isomorphism_hom_description} implies that $F_\bbone = f(\chi) = j_{\chi}(\tilde{F}_\chi)$ for some $\tilde{F}_{\chi} \in \units{\cQ^c(\Gamma_\chi)} \subseteq \units{Z(\cQ^c(\cG)e_\chi)}$. Now the commutativity of
                \begin{center}
                    \begin{tikzcd}[column sep=large]
                        \cQ^c(\Gamma_{\chi}) \arrow[d, equals] \arrow[r, "j_\chi"]                         & \cQ^c(\Gamma_K) \arrow[rr, "T_K = \tilde{\rho}(\gamma_K) - 1"] &  & \QQ_p^c \cup \set{\infty} \arrow[d, equals] \\
                        \cQ^c(\Gamma_{\chi \otimes \tilde{\rho}}) \arrow[rrr, "ev_{\gamma_{\chi \otimes \tilde{\rho}}}"] &                                                                &  & \QQ_p^c \cup \set{\infty}
                    \end{tikzcd}
                \end{center}
                for any $\tilde{\rho}$, together with \eqref{eq:equation_f_rho_interpolates}, shows that $\tilde{F}_\chi \in \units{\cQ^c(\Gamma_\chi)}$ satisfies IC($L_\infty/K, \chi, L, S, T, \alpha, \beta$).
            \end{proof}

            The following technical result will reduce the number of elementary subextensions which need to be considered in order to deduce the Interpolation Conjecture for characters of $L_\infty/K$:
            \begin{lem}
            \label{lem:rho_twist_factors}
                Let $p$, $\cG$ and $\Gamma$ be as in setting \ref{sett:construction}. Then, given any $\chi \in \Irr_p(\cG)$, there exists a type-$W$ character $\rho$ of $\cG$ such that $\chi \otimes \rho$ factors through $\cG_0 = \cG/\Gamma$.
            \end{lem}

            \begin{proof}
                Let $\overline{\chi}$ be the projection of $\chi$ to a quotient $\cG_n = \cG/\Gp{n}$ through which it factors. Since the subgroup $\Gamma_n = \Gamma/\Gp{n}$ is central in $\cG_n$, Clifford theory (cf. \cite{cr1} theorem 11.1) implies $\rest_{\Gamma_n}^{\cG_n} \overline{\chi} = e \lambda$ for some $e \in \NN$ and $\lambda \in \Irr_{\QQ_p^c}(\Gamma_n)$. By the abelianity of $\Gamma_n$, $\lambda$ is linear and $e = \overline{\chi}(1) = \chi(1)$. In particular, $\lambda(\overline{\gamma})^{p^n} = \lambda(\overline{\gp{n}}) = 1$ and so $\lambda(\overline{\gamma})$ is a $p^n$-th root of unity $\zeta \in \QQ_p^c$. Let $p^N$ be as in setting \ref{sett:construction}, that is, the index in $\Gamma_K$ of the image of $\Gamma$ under the canonical projection $\cG \sa \Gamma_K$. Choose $\tilde{\zeta} \in \QQ_p^c$ such that $\tilde{\zeta}^{p^N} = \zeta$ and define $\rho$ by (lifting to $\cG$) $\rho(\gamma_K) = \tilde{\zeta}^{-1}$. Then $\chi \otimes \rho$ has the desired property.
            \end{proof}

            The two preceding lemmas, together with some algebraic results from \cite{johnston_nickel}, allow us to characterise the Interpolation and equivariant Main Conjectures for $L_\infty/K$ in terms of certain families of subextensions. In doing so, we follow closely the approach in section 10 of the cited article. The key difference is that no analogue of lemma \ref{lem:ic_brauer_induction} is needed there, since the existence of $p$-adic $L$-functions which play the role of our $F_{S, T, \chi}^{\alpha, \beta}$ is a known fact in the totally real case. Furthermore, it will be necessary for us to prove a certain Galois-invariance property of interpolating series quotients.

            As explained in remark \ref{rem:emcu}, the connecting homomorphism $\partial \colon K_1(\cQ(\cG)) \to K_0(\Lambda(\cG), \cQ(\cG))$ is surjective for any $\cG$ as in setting \ref{sett:construction}. By the localisation sequence \eqref{eq:exact_sequence_k-theory}, its kernel is the image of $K_1(\Lambda(\cG))$ in $K_1(\cQ(\cG))$, which implies the existence of a homomorphism
            \begin{equation}
            \label{eq:nu_g}
                \nu_\cG \colon K_0(\Lambda(\cG), \cQ(\cG)) \to \faktor{\units{Z(\cQ(\cG))}}{\nr(K_1(\Lambda(\cG)))}
            \end{equation}
            via the reduced norm.
            \begin{rem}
            \label{rem:emc_reformulation}
                It is clear (alternatively, see \cite{johnston_nickel} lemma 10.13) that \hyperref[conje:emc]{eMC($L_\infty/K, L, S, T, \alpha, \beta$)} can be reformulated as the combination of the claims:
                \begin{itemize}
                    \item{
                        \hyperref[conje:ic]{IC($L_\infty/K, \chi, L, S, T, \alpha, \beta$)} holds for all $\chi \in \Irr_p(\cG)$.
                    }
                    \item{
                        The product $F_{S, T, \cG}^{\alpha, \beta} = \prod_{\chi/{\sim}_W} F_{S, T, \chi}^{\alpha, \beta} \in \units{Z(\cQ^c(\cG))}$ of the interpolating series quotients lies in $\units{Z(\cQ(\cG))}$ (recall at this point diagram \eqref{diag:centres_galois_invariants}).
                    }
                    \item{
                        One has $\nu_\cG(-\chi_{\Lambda(\cG), \cQ(\cG)}(\cC_{S, T}\q, t^\alpha)) = [F_{S, T, \cG}^{\alpha, \beta}]$, where $[-]$ denotes class modulo $\nr(K_1(\Lambda(\cG)))$.
                    }
                \end{itemize}
                \qedef
            \end{rem}

            We introduce two families of one-dimensional compact $p$-adic Lie groups: on one hand,
            \[
                \cS = \set{\pi_0^{-1}(\cG_0') : \cG_0' \ \text{an elementary subgroup of $\cG_0$}},
            \]
            where $\pi_0$ denotes the canonical projection $\pi_0 \colon \cG \sa \cG_0 = \cG/\Gamma = \Gal(L/K)$; and on the other,
            \[
                \cE_p = \set{\cG'/H' : \cG' \ \text{a $p$-elementary subgroup of $\cG$ containing $\Gamma$}, H' \et{a finite normal subgroup of} \cG'}.
            \]
            Here the term \textit{elementary} refers to the classical notion for finite groups in the case of $\cS$, and to the notion from definition \ref{defn:elementary_infinite} in that of $\cE_p$. Both families are finite, as $\Gamma$ is open in $\cG$ and $\cG$ has only finitely many subgroups of finite order (they are all contained in $H$).

            As described in the two preceding subsections, a choice of parameters for the formulation of the Main Conjecture for $L_\infty/K$ yields parameters for the subextensions defined by elements of $\cS$ and $\cE_p$. The choice of prolongations $v^c$ of places of $K$ to $\QQ^c$ plays no role, since, as already explained, different prolongations give rise to naturally isomorphic modules.

            \begin{prop}
            \label{prop:reduction_step}
                Setting \ref{sett:formulation} with $p$ odd. The following are equivalent:
                \begin{enumerate}[i)]
                    \item{
                        \hyperref[conje:emc]{eMC($L_\infty/K, L, S, T, \alpha, \beta$)} holds.
                    }
                    \item{
                        \hyperref[conje:emc]{eMC($L_\infty/K', L', S(K'), T(K'), \alpha, \beta$)} holds for all $\cG' \in \cS$, where $K' = L_\infty^{\cG'}$ and $L' = LK'$.
                    }
                    \item{
                        \hyperref[conje:emc]{eMC($L_\infty'/K', L', S(K'), T(K'), \alpha', \beta$)} holds for all $\cG'/H' \in \cE_p$, where $K' = L_\infty^{\cG'}$, $L_\infty' = L_\infty^{H'}$ and $L'$ and $\alpha'$ are constructed from the corresponding parameters for $L_\infty/K'$ as in setting \ref{sett:functoriality_1} and lemma \ref{lem:change_of_alpha_quotient}, respectively.
                    }
                \end{enumerate}
            \end{prop}

            \begin{proof}
                The fact that i) implies ii) and iii) is a direct consequence of propositions \ref{prop:functoriality_2} and \ref{prop:functoriality_1}. In order to prove their respective converses, we rely on the reformulation from remark \ref{rem:emc_reformulation}.

                \begin{itemize}
                    \item{
                        \underline{ii) $\Rightarrow$ i)}: We first need to show that IC($L_\infty/K, \chi, L, S, T, \alpha, \beta$) holds for any given ${\chi \in \Irr_p(\cG)}$. By lemmas \ref{prop:ic_rho_twist} and \ref{lem:rho_twist_factors}, we can assume $\chi$ factors through $\cG_0$. In particular, a Brauer decomposition $\chi = \sum_{i = 1}^n z_i \; \indu_{\cH_i}^\cG \lambda_i$ as in \eqref{eq:brauer_infinite} can be chosen such that ${\Gamma \subseteq \ker(\lambda_i) \subseteq \cH_i}$ for all $i$. It follows that all $\cH_i$ belong to $\cS$ by construction, and thus lemma \ref{lem:ic_brauer_induction} immediately yields IC($L_\infty/K, \chi, L, S, T, \alpha, \beta$).

                        The next step is to prove that $F_{S, T, \cG}^{\alpha, \beta} = \prod_{\chi/{\sim}_W} F_{S, T, \chi}^{\alpha, \beta} \in \units{Z(\cQ^c(\cG))}$ lies in fact in $\units{Z(\cQ(\cG))}$, or equivalently by diagram \eqref{diag:centres_galois_invariants}, that it is $G_{\QQ_p}$-invariant. The proof of lemma \ref{lem:ic_brauer_induction} realises each $F_{S, T, \chi}^{\alpha, \beta}$ as the preimage of a certain map $f_\chi$ ($f$ in the referenced proof) under the isomorphism $\delta \colon \units{Z(\cQ^c(\cG))} \isoa \Map^W(\Irr_p(\cG), \units{\cQ^c(\Gamma_K)})$. This isomorphism is in fact $G_{\QQ_p}$-equivariant if one endows its codomain with the natural left $G_{\QQ_p}$-action $(\sigma g)(\varsigma) = \sigma(g(\sigma^{-1}\varsigma))$ for $g \in \Map^W(\Irr_p(\cG), \units{\cQ^c(\Gamma_K)})$, $\varsigma \in \Irr_p(\cG)$ (cf. \cite{rwii} proof of theorem 8). Therefore, it suffices to show that $f_\cG = \prod_{\chi/{\sim}_W} f_\chi$ is itself invariant under that action. But this is a direct consequence of its definition \eqref{eq:definition_brauer_induction_hom} and the equivariant Main Conjecture for the $\cG' \in \cS$: given $\sigma \in G_{\QQ_p}$ and $\chi \in \Irr_p(\cG)$ with Brauer decomposition as before (where we can again assume $\Gamma \subseteq \cH_i \in \cS$), one has $\sigma^{-1} \chi = \sum_{i = 1}^n z_i \; \indu_{\cH_i}^\cG \sigma^{-1} \lambda_i$ and hence
                        \begin{align*}
                            (\sigma f_\cG)(\chi) = \sigma(f_\cG(\sigma^{-1}\chi)) = \sigma(f_{\sigma^{-1}\chi}(\sigma^{-1}\chi))
                            & = \sigma\left(\prod_{i = 1}^n j_{K^i}(j_{\sigma^{-1}\lambda_i}((F_{S(K^i), T(K^i), \sigma^{-1}\lambda_i}^{\alpha, \beta})^{z_i}))\right) \\
                            & = \prod_{i = 1}^n j_{K^i}(\sigma(j_{\sigma^{-1}\lambda_i}((F_{S(K^i), T(K^i), \sigma^{-1}\lambda_i}^{\alpha, \beta})^{z_i}))) \\
                            & = \prod_{i = 1}^n j_{K^i}(j_{\lambda_i}((F_{S(K^i), T(K^i), \lambda_i}^{\alpha, \beta})^{z_i})) \\
                            & = f_\cG(\chi),
                        \end{align*}
                        where $K^i = L_\infty^{\cH_i}$ and the next-to-last equality follows from $F_{S(K^i), T(K^i), \lambda_i}^{\alpha, \beta} = \sigma(F_{S(K^i), T(K^i), \sigma^{-1}\lambda_i}^{\alpha, \beta})$, which is in turn implied by eMC($L_\infty/K^i, LK^i, S(K^i), T(K^i), \alpha, \beta$) (see for instance the proof of corollary \ref{cor:emc_implies_weak_stark}). The only subtlety lies in the third equality, as it is not guaranteed that the decomposition ${\sigma^{-1} \chi = \sum_{i = 1}^n z_i \; \indu_{\cH_i}^\cG \sigma^{-1} \lambda_i}$ was the one used to \textit{construct} $f_{\sigma^{-1}\chi}$. However, the cited lemma shows that it is at least a \textit{valid} one for the construction, and the uniqueness of the interpolating elements (cf. remark \ref{rem:ic} ii)) ensures that all choices of a decomposition lead to the same map $f_{\sigma^{-1}\chi}$. The same goes for the last equality. Since the value of $\sigma f_\cG$ at $\chi$ determines its values at all $\rho$-twists of $\chi$, applying the above argument to a suitable set of representatives of $\Irr_p(\cG)/{\sim_W}$ yields $\sigma f_\cG = f_\cG$, as required.

                        We now move on to the equivariant part of the proof. For each $\cG' \in \cS$, the canonical embeddings $\iota_{\cG'} \colon \Lambda(\cG') \ia \Lambda(\cG)$ and $\iota_{\cG'} \colon \cQ(\cG') \ia \cQ(\cG)$ induce restriction maps on $K$-theory as explained in section \ref{sec:algebraic-k-theory}. Furthermore, the homomorphism $\rest^W$ from the proof of proposition \ref{prop:character_res_interpolating} gives rise to an arrow $\iota_{Z, \cG'}$ making the diagram
                        \begin{center}
                            \begin{tikzcd}
                                K_1(\cG) \arrow[d, "K_1^\rest(\iota_{\cG'})"] \arrow[r, "\nr"] & \units{Z(\cQ(\cG))} \arrow[d, "\iota_{Z, \cG'}"] \\
                                K_1(\cG') \arrow[r, "\nr"]                                     & \units{Z(\cQ(\cG'))}
                            \end{tikzcd}
                        \end{center}
                        commute (cf. \cite{johnston_nickel} (3.9)), which in particular factors as
                        \[
                            [\iota_{Z, \cG'}] \colon \faktor{\units{Z(\cQ(\cG))}}{\nr(K_1(\Lambda(\cG)))} \to \faktor{\units{Z(\cQ(\cG'))}}{\nr(K_1(\Lambda(\cG')))}.
                        \]
                        Now \eqref{eq:localisation_sequence_restriction} implies the commutativity of
                        \begin{center}
                            \begin{tikzcd}[column sep=large]
                                {K_0(\Lambda(\cG), \cQ(\cG))} \arrow[d, "{\prod_{\cG'} K_0^\rest(\iota_{\cG'}, \iota_{\cG'})}"] \arrow[r, "{\nu_\cG}"]    & \faktor{\units{Z(\cQ(\cG))}}{\nr(K_1(\Lambda(\cG)))} \arrow[d, "{\prod_{\cG'} [\iota_{Z, \cG'}]}"]              \\
                                {\prod_{\cG' \in \cS} K_0(\Lambda(\cG'), \cQ(\cG'))} \arrow[r, "{\prod_{\cG'} \nu_{\cG'}}"] & \prod_{\cG' \in \cS} \faktor{\units{Z(\cQ(\cG'))}}{\nr(K_1(\Lambda(\cG')))}
                            \end{tikzcd}
                        \end{center}
                        with $\nu_{\cG'}$ as in \eqref{eq:nu_g}. The proof of proposition \ref{prop:functoriality_2}, together with \eqref{diag:large_diagram_restriction}, shows that
                        \[
                            [\iota_{Z, \cG'}]([F_{S, T, \cG}^{\alpha, \beta}]) = [F_{S, T, \cG'}^{\alpha, \beta}],
                        \]
                        where $F_{S, T, \cG'}^{\alpha, \beta} = \prod_{\chi' \in \Irr_p(\cG')/{\sim}_W} F_{S(K'), T(K'), \chi'}^{\alpha, \beta} \in \units{Z(\cQ(\cG'))} \subseteq \units{Z(\cQ^c(\cG'))}$. The cornerstone of the present argument is the injectivity of the arrow $\prod_{\cG'} [\iota_{Z, \cG'}]$ in the last diagram, which is a direct application of \cite{johnston_nickel} corollary 10.12 (this is the reason we require $p \neq 2$). Since the left vertical arrow sends $-\chi_{\Lambda(\cG), \cQ(\cG)}(\cC_{S, T}\q, t^\alpha)$ to
                        \[
                            \prod_{\cG' \in \cS} K_0^\rest(\iota_{\cG'}, \iota_{\cG'})(-\chi_{\Lambda(\cG), \cQ(\cG)}(\cC_{S, T}\q, t^\alpha)) = \prod_{\cG' \in \cS} - \chi_{\Lambda(\cG'), \cQ(\cG')}(\cC_{S(K'), T(K')}\q, \restr{t^\alpha}{\cQ(\cG')})
                        \]
                        by the proof of \ref{prop:functoriality_2}, $\nu_\cG(-\chi_{\Lambda(\cG), \cQ(\cG)}(\cC_{S, T}\q, t^\alpha))$ is the \textit{unique} preimage of $\prod_{\cG' \in \cS} [F_{S, T, \cG'}^{\alpha, \beta}]$ under $\prod_{\cG' \in \cS} [\iota_{Z, \cG'}]$, i.e. $[F_{S, T, \cG}^{\alpha, \beta}]$.
                    }
                    \item{
                        \underline{iii) $\Rightarrow$ i)}: The proof that IC($L_\infty/K, \chi, L, S, T, \alpha, \beta$) holds for a given $\chi \in \Irr_p(\cG)$ is essentially the same as above, the only difference being that the $\cH_i$ appearing in the Brauer decomposition of $\chi$ may not belong to $\cE_p$. Instead, we reason as follows: each $\lambda_i$ is a linear character of $\cH_i$ (as before, we can assume $\Gamma \subseteq \ker(\lambda_i) \subseteq \cH_i$), and hence establishes an isomorphism between $\cH_i/\ker(\lambda_i)$ and a finite subgroup of $\units{(\QQ_p^c)}$, which is necessarily cyclic. Set $K^i = L_\infty^{\cH_i}$ and $\widetilde{K}^i = L_\infty^{\ker(\lambda_i)}$, so $K \subseteq K^i \subseteq \widetilde{K}^i \subseteq L$.

                        Let $\widetilde{K}_\infty^i$ be the cyclotomic $\ZZ_p$-extension of $\widetilde{K}^i$ and denote $\Gal(L_\infty/\widetilde{K}_\infty^i) = H \cap \ker(\lambda_i)$ by $H_i$. Then $H_i$ is a normal subgroup of $\cH_i$ and $\lambda_i$ factors through the abelian quotient $\cH_i/H_i = \Gal(\widetilde{K}_\infty^i/K^i)$ as a character $\tilde{\lambda}_i$. This quotient has a (central) open subgroup ${\Gal(\widetilde{K}_\infty^i/\widetilde{K}^i) \iso \ZZ_p}$ with cyclic factor group $\Gal(\widetilde{K}^i/K^i) \iso \cH_i/\ker(\lambda_i)$ and is therefore $p$-elementary in the sense of definition \ref{defn:elementary_infinite}. In particular, $\cH_i/H_i \in \cE_p$ and the Interpolation Conjecture holds for $\tilde{\lambda}_i$ and thus also for $\lambda_i$ by proposition \ref{prop:functoriality_1} i). As a result, so does IC($L_\infty/K, \chi, L, S, T, \alpha, \beta$). The fact that $F_{S, T, \cG}^{\alpha, \beta} = \prod_{\chi/{\sim}_W} F_{S, T, \chi}^{\alpha, \beta} \in \units{Z(\cQ^c(\cG))}$ is $G_{\QQ_p}$-invariant follows from the analogous property for each $\cH_i$, which is in turn a consequence of all arrows of
                        \begin{center}
                            \begin{tikzcd}
                                \units{Z(\cQ^c(\cH_i))} \arrow[d, "\rsim"] \arrow[r, "\pi_{\cH_i/H_i}"] & \units{Z(\cQ^c(\cH_i/H_i))} \arrow[d, "\rsim"]           \\
                                \prod_{\chi' \in \Irr_p(\cH_i)/{\sim}_W} \units{\cQ^c(\Gamma_{\chi'})} \arrow[r] & \prod_{\tilde{\chi} \in \Irr_p(\cH_i/H_i)/{\sim}_W} \units{\cQ^c(\Gamma_{\tilde{\chi}})}
                            \end{tikzcd}
                        \end{center}
                        being $G_{\QQ_p}$-equivariant and the equivariant Main Conjecture for $\cH_i/H_i \in \cE_p$. Recall that the bottom arrow is an isomorphism on the components corresponding to characters which factor through $H_i$, and identically 1 elsewhere (cf. lemma \ref{lem:functoriality_1_chi_parts} iii)).  Note that, given $\sigma \in G_{\QQ_p}$ and $\chi' \in \Irr_p(\cH_i)$, $\chi'$ factors through $H_i$ if and only if $\sigma \chi'$ does.

                        The equivariant part of the proof is completely analogous to that for ii) $\Rightarrow$ i), with the caveat that the restriction maps (on $K$-theory, centres and characters) associated to $\cG' \ia \cG$ now have to be composed with the projection maps arising from $\pi_{\cG'/H'} \colon \cG' \sa \cG'/H'$ for $\cG'/H' \in \cE_p$ (these were studied in subsection \ref{subsec:functoriality_1}). The full argument can be found in \cite{johnston_nickel}, specifically in theorem 10.14 and the results leading up to it. We limit ourselves to pointing out that the key commutative diagram is
                        \begin{center}
                            \begin{tikzcd}[column sep=huge]
                            {K_0(\Lambda(\cG), \cQ(\cG))} \arrow[r, "\nu_\cG"] \arrow[d, "{\prod_{\cG'} K_0^\rest(\iota_{\cG'}, \iota_{\cG'})}"]                    & \faktor{\units{Z(\cQ(\cG))}}{\nr(K_1(\Lambda(\cG)))} \arrow[d, "{\prod_{\cG'} [\iota_{Z, \cG'}]}"]                                              \\
                            {\prod_{\cG'} K_0(\Lambda(\cG'), \cQ(\cG'))} \arrow[r, "\prod_{\cG'} \nu_{\cG'}"] \arrow[d, "{\prod_{\cG'/H'} K_0^\rest(\pi_{\cG'/H'}, \pi_{\cG'/H'})}"] & \prod_{\cG'} \faktor{\units{Z(\cQ(\cG'/H'))}}{\nr(K_1(\Lambda(\cG'/H')))} \arrow[d, "{\prod_{\cG'/H'} [\pi_{Z, \cG'/H'}]}"] \\
                            {\prod_{\cG'/H' \in \cE_p} K_0(\Lambda(\cG'/H'), \cQ(\cG'/H'))} \arrow[r, "\prod_{\cG'/H'} \; \nu_{\cG'/H'}"]                            & \prod_{\cG'/H' \in \cE_p} \faktor{\units{Z(\cQ(\cG'/H'))}}{\nr(K_1(\Lambda(\cG'/H')))}
                            \end{tikzcd}
                        \end{center}
                        where the products in the middle row run over $\cG'$ for each $\cG'/H' \in \cE_p$ and the composition of the vertical arrows on the right is now injective by \cite{johnston_nickel} theorem 10.5.
                    }
                \end{itemize}
            \end{proof}

            \begin{rem}
                The proof also shows that, in the notation of the statement of the proposition, the following are equivalent:
                \begin{enumerate}[i)]
                    \item{
                        \hyperref[conje:ic]{IC($L_\infty/K, \chi, L, S, T, \alpha, \beta$)} holds for all $\chi \in \Irr_p(\cG)$.
                    }
                    \item{
                        \hyperref[conje:ic]{IC($L_\infty/K', \lambda', L', S(K'), T(K'), \alpha, \beta$)} holds for all $\cG' \in \cS$ and all linear $\lambda' \in \Irr_p(\cG')$.
                    }
                    \item{
                        \hyperref[conje:ic]{IC($L_\infty'/K', \tilde{\lambda}, L', S(K'), T(K'), \alpha', \beta$)} holds for all $\cG'/H' \in \cE_p$ and all linear $\tilde{\lambda} \in \Irr_p(\cG'/H')$.
                    }
                \end{enumerate}
                \qedef
            \end{rem}

\newpage
\chapter{Compatibility with existing conjectures and new cases}
\label{chap:compatibility_with_existing_conjectures}

    As mentioned in the preceding chapters, setting \ref{sett:construction} is, to a certain extent, a generalisation of that of some existing Main Conjectures. Most notably, $\cG = \Gal(L_\infty/K)$ is not required to be abelian, or $L_\infty$ totally real. In this final chapter, we make that claim precise by discussing the Main Conjectures of Burns, Kurihara and Sano; and of Ritter and Weiss. Although more details will be given at the beginning of each section, the conclusion is the following:
    \begin{itemize}
        \item{
            In the abelian case, our conjecture recovers a slight modification of that in \cite{bks}. This allows us to deduce validity of \hyperref[conje:emc]{eMC($L_\infty/\QQ, \beta$)} whenever $L_\infty/\QQ$ is abelian.
        }
        \item{
            For a suitably defined CM extension $L_\infty/K$, the \textit{minus part} of our conjecture recovers that in \cite{rwii} for the maximal totally real subextension $L_\infty\pl/K$. This implies that the minus part of \hyperref[conje:emc]{eMC($L_\infty/K, \beta$)} holds (for extensions of that form) in the so-called $\mu = 0$ case.
        }
    \end{itemize}

    While not difficult, the proofs of these these equivalences necessitate some recalls on objects such as determinant functors and $\Ext$-modules of Iwasawa modules. We strive to do so in a moderately self-contained manner.

    We point out that the Main Conjecture from chapter \ref{chap:formulation_of_the_main_conjecture} is by no means the most general conjecture in Iwasawa theory of number fields in every possible regard. For instance, it is limited to the rank-one cyclotomic case - which suggests a natural direction for further research.

    \section{The conjecture of Burns, Kurihara and Sano}
    \label{sec:the_conjecture_of_burns_kurihara_and_sano}

        The aim of this section is to prove that the Main Conjecture from chapter \ref{chap:formulation_of_the_main_conjecture} is \textit{essentially equivalent} to that put forward by Burns, Kurihara and Sano in 2017 in the cases where both can be formulated. Specifically, we are alluding to conjecture 3.1 in \cite{bks}, the setup of which was partially introduced in section \ref{sec:description_in_terms_of_rgamma_complexes}. As explained in said section, the main difference is that the cited article is restricted to the case where $L_\infty/K$ is an abelian extension, and in turn does not require $K_\infty$ to be the \textit{cyclotomic} $\ZZ_p$-extension of $K$. In order to compare the two conjectures, therefore, we will be working in setting \ref{sett:formulation} (sometimes simply \ref{sett:construction}) under the additional assumption that $\cG$ is abelian. By ``essentially equivalent'' we mean that a claim about all characters must be replaced by one about almost all characters - in line with our Main Conjecture.

        \begin{rem}
        \label{rem:k1_commutative}
            The commutativity of $\cG$, and therefore of $\Lambda(\cG)$, has profound implications. For instance, it means that $\cQ(\cG)$ is a commutative semisimple Artinian ring and therefore a finite product of fields by Wedderburn's theorem. In particular, the reduced norm becomes an isomorphism $\nr \colon K_1(\cQ(\cG)) \isoa \units{\cQ(\cG)}$, which renders the equivariant Main Conjectures with and without uniqueness equivalent (cf. remark \ref{rem:emcu} i)). Another consequence is that $\cG \iso \Gamma_K \times H$ and all irreducible Artin characters of $\cG$ are linear.\qedef
        \end{rem}

        Recall that abelianity of $\cG$ was already assumed in theorem \ref{thm:iso_bks_complex}, which proves the crucial fact that the complexes employed by Burns, Kurihara and Sano and ourselves are isomorphic in the derived category $\cD(\Lambda(\cG))$. The relevant invariants of complexes (determinant functors and refined Euler characteristics) are unaffected by such isomorphisms, and thus it suffices to consider our usual main complex $\cC_{S, T}\q$ moving forward. We need to introduce its finite-level counterpart: in the notation of setting \ref{sett:construction}, given $n \in \NN$, we define the complex of $\Lambda(\cG_n)$-modules
        \[
            \cC_{L_n, S, T}\q = \Lambda(\cG_n) \otimes_{\Lambda(\cG)}^\LL \cC_{S, T}\q.
        \]
        By lemma \ref{lem:complexes_coinvariants_finite_level}, $\cC_{L_n, S, T}\q$ is isomorphic in $\cD(\Lambda(\cG_n))$ to the complex introduced in \cite{bks} p. 1535, denoted by $\cB_{L_n, S, T}\q$ in section \ref{sec:description_in_terms_of_rgamma_complexes} above. As stated in said article, $\cC_{S, T}\q = \varprojlim_n \cC_{L_n, S, T}\q$ (cf. also the proof of theorem \ref{thm:iso_bks_complex}) and $\cC_{L_n, S, T}\q$ is a perfect complex which is acyclic outside degrees 0 and 1 and satisfies:
        \begin{itemize}
            \item{
                $H^0(\cC_{L_n, S, T}\q) \iso \ZZ_p \otimes \units{\cO_{L_n, S, T}}$.
            }
            \item{
                $H^1(\cC_{L_n, S, T}\q)$ fits into the short exact sequence of $\Lambda(\cG_n)$-modules
                \begin{equation}
                \label{eq:cohomology_finite_level}
                    0 \to Cl_{L_n, S, T}(p) \to H^1(\cC_{L_n, S, T}\q) \to \cX_{L_n, S} \to 0,
                \end{equation}
                where $Cl_{L_n, S, T}(p)$ is the $p$-part of the $(S, T)$-ray class group $Cl_{L_n, S, T}$ of $L_n$ (see the proof of proposition \ref{prop:kernel_cokernel_h0_alpha} for a definition).
            }
        \end{itemize}
        Compare this with the cohomology of $\cC_{S, T}\q$ established in theorem \ref{thm:cohomology_of_complex}. Since $Cl_{L_n, S, T}$ is finite, sequence \eqref{eq:cohomology_finite_level} induces an isomorphism $\CC_p \otimes_{\ZZ_p} H^1(\cC_{L_n, S, T}\q) \iso \CC_p \otimes_{\ZZ_p} \cX_{L_n, S}$.

        The conjecture of Burns, Kurihara and Sano is formulated in terms of determinant functors of these complexes. A brief introduction to these functors and some of their relevant properties can be found in appendix \ref{app:determinant_functors}. We now explain how to express $\Det_{\CC_p[\cG_n]}(\CC_p \otimes_{\ZZ_p}^\LL \cC_{L_n, S, T}\q)$ in terms of $\Det_{\CC_p[\cG_n]}(\CC_p \otimes_{\ZZ_p} H^0(\cC_{L_n, S, T}\q))$ and $\Det_{\CC_p[\cG_n]}(\CC_p \otimes_{\ZZ_p} H^1(\cC_{L_n, S, T}\q))$. For simplicity, assume that $\cC_{S, T}\q$ has a strictly perfect representative of the form
        \begin{equation}
        \label{eq:stricly_perfect_representative_two_degrees}
            \cP\q = [\stackrel{0}{\cP^0} \to \stackrel{1}{\cP^1}]
        \end{equation}
        in $\cD(\Lambda(\cG))$ with $\cP^0$ and $\cP^1$ free. This is always the case if the set $T$ is chosen large enough that $\units{\cO_{L_n, S, T}}$ is $\ZZ$-torsion free for all $n$ (cf. \cite{bks} p. 1538 or \cite{bullach_hofer} equation (12)). Since neither conjecture depends on $T$ (cf. subsection \ref{subsec:the_choice_of_s_and_t} and \cite{bks} p. 1541), this assumption is harmless. In particular, $\cC_{L_n, S, T}\q$ has a strictly perfect representative $\cP_n\q = [\cP_n^0 \to \cP_n^1]$ in $\cD(\Lambda(\cG_n))$ with $\cP_n^0$ and $\cP_n^1$ free.

        Split the four-term exact sequence induced by \eqref{eq:stricly_perfect_representative_two_degrees} into two short ones:
        \begin{center}
            \begin{tikzcd}[column sep=small]
                {H^0(\cC_{L_n, S, T}\q)} \arrow[rr, hook] &  & \cP_n^0 \arrow[rr] \arrow[rd, two heads] &                    & \cP_n^1 \arrow[rr, two heads] &  & {H^1(\cC_{L_n, S, T}\q)} \\
                                                          &  &                                      & W \arrow[ru, hook] &                           &  &
            \end{tikzcd}
        \end{center}

        After applying $\CC_p \otimes_{\ZZ_p} -$, the semisimplicity of $\CC_p[\cG_n]$ ensures the existence of splittings
        \begin{center}
            \begin{tikzcd}
                {\CC_p \otimes_{\ZZ_p} H^0(\cC_{L_n, S, T}\q)} \arrow[r, hook] & \CC_p \otimes_{\ZZ_p} \cP_n^0 \arrow[r, two heads] & \CC_p \otimes_{\ZZ_p} W \arrow[l, "\sigma_0", dashed, bend left]                \\
                \CC_p \otimes_{\ZZ_p} W \arrow[r, hook]                        & \CC_p \otimes_{\ZZ_p} \cP_n^1 \arrow[r, two heads] & {\CC_p \otimes_{\ZZ_p} H^1(\cC_{L_n, S, T}\q)} \arrow[l, "\sigma_1", dashed, bend left]
            \end{tikzcd}
        \end{center}
        which in turn induce isomorphisms
        \begin{equation}
        \label{eq:choice_of_splittings_iso_det}
            \Det_{\CC_p[\cG_n]}(\CC_p \otimes_{\ZZ_p} H^0(\cC_{L_n, S, T}\q)) \otimes_{\CC_p[\cG_n]} \Det_{\CC_p[\cG_n]}(\CC_p \otimes_{\ZZ_p} W) \xrightarrow{\ \sim \ } \Det_{\CC_p[\cG_n]}(\CC_p \otimes_{\ZZ_p} \cP_n^0)
        \end{equation}
        and
        \begin{equation}
        \label{eq:choice_of_splittings_iso_det_two}
            \Det_{\CC_p[\cG_n]}(\CC_p \otimes_{\ZZ_p} W) \otimes_{\CC_p[\cG_n]} \Det_{\CC_p[\cG_n]}(\CC_p \otimes_{\ZZ_p} H^1(\cC_{L_n, S, T}\q)) \xrightarrow{\ \sim \ } \Det_{\CC_p[\cG_n]}(\CC_p \otimes_{\ZZ_p} \cP_n^1)
        \end{equation}
        as in \eqref{eq:additivity_det_functor}. This allows for the following definition from \cite{bks} paragraph 3B1:
        \begin{defn}
        \label{defn:bks_map}
            Setting \ref{sett:construction} with $\cG$ abelian. Let $\beta \colon \CC_p \xrightarrow{\sim} \CC$ be a ring isomorphism. Assume $T$ is large enough that $\cC_{S, T}\q$ has a strictly perfect representative $\cP\q$ of the form \eqref{eq:stricly_perfect_representative_two_degrees}.

            Given $\chi \in \Irr_p(\cG)$, choose $n \in \NN$ such that $\chi$ factorises over $\cG_n$. Set $\cP_n\q = \Lambda(\cG_n) \otimes_{\Lambda(\cG)}^\LL \cP\q$, which is a strictly perfect representative of $\cC_{L_n, S, T}\q$. We define the $\ZZ_p$-homomorphism $\lambda_\chi$ as the composition
            \begin{align*}
                \lambda_\chi \colon \Det_{\Lambda(\cG)}(\cC_{S, T}\q) = & \ \Det_{\Lambda(\cG)}(\cP\q) \\
                \to & \ \Det_{\Lambda(\cG_n)}(\cP_n\q) \\
                \ia & \ \Det_{\CC_p[\cG_n]}(\CC_p \otimes_{\ZZ_p}^\LL \cP_n\q) \\
                = & \ \Det_{\CC_p[\cG_n]}(\CC_p \otimes_{\ZZ_p} \cP_n^0) \otimes_{\CC_p[\cG_n]} \Det_{\CC_p[\cG_n]}(\CC_p \otimes_{\ZZ_p} \cP_n^1)^{-1} \\
                \xrightarrow{\sim} & \ \Det_{\CC_p[\cG_n]}(\CC_p \otimes_{\ZZ_p} H^0(\cC_{L_n, S, T}\q)) \otimes_{\CC_p[\cG_n]} \Det_{\CC_p[\cG_n]}(\CC_p \otimes_{\ZZ_p} H^1(\cC_{L_n, S, T}\q))^{-1} \\
                \xrightarrow{\sim} & \ \Det_{\CC_p[\cG_n]}(\CC_p \otimes \units{\cO_{L_n, S, T}}) \otimes_{\CC_p[\cG_n]} \Det_{\CC_p[\cG_n]}(\CC_p \otimes_{\ZZ_p} \cX_{L_n, S})^{-1} \\
                \xrightarrow{\sim} & \ \Det_{\CC_p[\cG_n]}(\CC_p \otimes_{\ZZ_p} \cX_{L_n, S}) \otimes_{\CC_p[\cG_n]} \Det_{\CC_p[\cG_n]}(\CC_p \otimes_{\ZZ_p} \cX_{L_n, S})^{-1} \\
                \xrightarrow{\sim} & \ \CC_p[\cG_n] \\
                \sa & \ \CC_p.
            \end{align*}
            Here the first two arrows are induced by extension of scalars (cf. appendix \ref{app:determinant_functors}), whereas the fourth one is given by \eqref{eq:choice_of_splittings_iso_det}, \eqref{eq:choice_of_splittings_iso_det_two} and evaluation $\Det_{\CC_p[\cG_n]}(\CC_p \otimes_{\ZZ_p} W) \otimes_{\CC_p[\cG_n]} \Det_{\CC_p[\cG_n]}(\CC_p \otimes_{\ZZ_p} W)^{-1} \xrightarrow{\sim} \CC_p[\cG_n]$. The last three arrows are, in order, the determinant $\Det_{\CC_p[\cG_n]}(\lambda_{n, S}^\beta)$ of the $p$-adic Dirichlet regulator $\lambda_{n, S}^\beta$ (definition \ref{defn:p-adic_dirichlet_regulator_map}), the evaluation map and the homomorphism induced by $\chi$ (i.e. the projection to $e(\chi) \CC_p[\cG_n] \iso \CC_p$).\qedef
        \end{defn}

        \begin{rem}
        \phantomsection
        \label{rem:map_bks}
            \begin{enumerate}[i)]
                \item{
                    The definition in \cite{bks} passes to the finite level $G_\chi = \Gal(L_\infty^{\ker\chi}/K)$ instead of $\cG_n$. However, the resulting map $\lambda_\chi$ is identical by virtue of the surjection $\CC_p[\cG_n] \sa \CC_p[G_\chi]$. In particular, it is independent of the choice of $n$.
                }
                \item{
                    Instead of irreducible (hence linear) $p$-adic Artin characters of $\cG$, the article considers linear $\CC$-valued characters with open kernel - the collection of which is denoted by $\hat{\cG}$. Note that $\beta$ induces a group isomorphism $\Irr_p(\cG) \xrightarrow{\sim} \hat{\cG}, \chi \mapsto \beta \chi$, and therefore everything can be formulated in terms of $\Irr_p(\cG)$ instead.
                }
                \item{
                    The homomorphism $\lambda_\chi$ is known to be independent of the choice of a strictly perfect representative $\cP\q$ of $\cC_{S, T}\q$, as well as splittings $\sigma_0$ and $\sigma_1$ inducing sequences \eqref{eq:choice_of_splittings_iso_det} and \eqref{eq:choice_of_splittings_iso_det_two}, by the properties of determinant functors (see for instance \cite{ckv} p. 71). However, it is not independent of the choice of $\beta \colon \CC_p \isoa \CC_p$ - and hence neither is the conjecture in \cite{bks}.
                }
                \item{
                    The condition on $T$ is not necessary for the definition: the same map can be constructed using an arbitrary representative $\cP\q$ (this is indeed the case in \cite{bks}, although the explicit construction is not given). This assumption, which will be dropped in the statement of the conjecture below, appears above only to simplify the description of the maps. Recall that both conjectures are independent of $T$.
                }
            \end{enumerate}
        \end{rem}

        We are now in a position to state conjecture 3.1 from \cite{bks}\footnote{The cited, published article contains the following erratum: the $\hat{G}_\chi$ in the statement should read $\hat{\cG}_\chi$. This has been confirmed by the authors and is in fact correct in available versions of the article prior to publication.}. However, in order to prove equivalence with our conjecture (in the abelian case), it is necessary to apply a minor tweak to the former. We refer to this as the \textbf{modified BKS conjecture}\index{equivariant Main Conjecture!of Burns, Kurihara and Sano}:
        \begin{conje*}[BKS'($L_\infty/K, L, S, T, \beta$)]
        \phantomsection
        \label{conje:bksp}
            Setting \ref{sett:construction} with $\cG$ abelian.

            There exists a generator\footnote{Despite the notation, which is taken from \cite{bks}, this element is unrelated to the $\cL_{S, T}(\chi, f) \in \CC_p$ from Stark's conjecture \hyperref[conje:stark]{Stark\textsuperscript{T}($L/K, \chi, f, S, T$)}.} $\cL_{S, T}$ of the $\Lambda(\cG)$-module $\Det_{\Lambda(\cG)}(\cC_{S, T}\q)$ such that, for almost all ${\psi \in \Irr_p(\cG)}$, one has $\lambda_\psi(\cL_{S, T}) = \beta^{-1}(L_{K, S, T}^\ast(\beta \psi^{-1}, 0))$.
        \end{conje*}

        \begin{rem}
            Although perhaps not entirely evident, we claim that the above conjecture is indeed equivalent to the outcome of replacing ``every $\psi \in \hat{\cG}_\chi$'' by ``almost every $\psi \in \hat{\cG}_\chi$'' in \cite{bks} conjecture 3.1. We recall the following definitions therein, which will not be used beyond the scope of this remark: for $\psi \in \Irr_p(\cG)$, let
            \[
                r_{\beta \psi, S} =
                \begin{cases}
                    \abs{\set{v \in S \colon \cG_v \subseteq \ker(\psi)}}, & \psi \neq \bbone \\
                    \abs{S} - 1, & \psi = \bbone.
                \end{cases}
            \]
            This is precisely the order of vanishing of the Artin $L$-function $L_{K, S, T}(\beta \psi, s)$ at $s = 0$ (i.e. $r_{\beta \psi}(S)$ in the notation of section \ref{sec:artin_l-series}) - see for instance \cite{bks} subsection 2A. It is immediate from the definition that $r_{\beta \psi, S} = r_{\beta \check{\psi}, S}$.

            Given a character $\eta \in \Irr_p(H) = \Hom(H, \units{\CC_p})$ (in the article: a $\chi \in \hat{\Delta}$), let $\cG_\eta = \cG/\ker(\eta)$ (where $\ker(\eta)$ is really a subgroup of $H$). We regard the canonical embedding $\Irr_p(\cG_\eta) \ia \Irr_p(\cG)$ as an inclusion. Set
            \[
                r_\eta = \abs{\set{v \in S \colon \cG_v \subseteq \ker(\eta)}},
            \]
            which is the number of places in $S$ which  split completely in $L_\infty^{\ker(\eta)}/K$. Since $K_\infty/K$ is the cyclotomic $\ZZ_p$-extension, that can only be the case for archimedean places.

            Assume first \hyperref[conje:bksp]{BKS'($L_\infty/K, L, S, T, \beta$)}. Let $\eta \in \Irr_p(H)$. By assumption, one has\footnote{In the notation of the article, $L_{K, S, T}^{(r)}(\beta \psi^{-1}, 0)$ is the coefficient of $s^r$ in the series expansion around 0, rather than the value of the $r$-th derivative at 0 (see p. 1533). In particular, $L_{K, S, T}^{(r_{\beta\psi, S})}(\beta \psi^{-1}, 0) = L_{K, S, T}^\ast(\beta \psi^{-1}, 0)$.}
            \[
                \lambda_\psi(\cL_{S, T}) = \beta^{-1}(L_{K, S, T}^\ast(\beta \psi^{-1}, 0))
            \]
            for almost all $\psi \in \Irr_p(\cG)$. In particular, the same holds for almost all $\psi \in \Irr_p(\cG_\eta) \subset \Irr_p(\cG)$ satisfying $r_\eta = r_{\beta\psi, S}$ (or any other condition). This shows one implication.

            The converse is subtler because of the way characters are grouped in terms of $\Irr_p(\cG_\eta)$ in \cite{bks} conjecture 3.1. We first point out that it is not necessarily true that, for a given $\eta \in \Irr_p(H)$, one has $r_\eta = r_{\beta\psi, S}$ for almost all $\psi \in \Irr_p(\cG_\eta)$. For instance, if $p \neq 2, L = \QQ(\zeta_p), K = L^+$ (the maximal totally real subfield of $L$), $S$ is arbitrary and $\eta$ is the non-trivial character of $H \iso \zmod{2}$, then $r_\eta = 0$ but any $\psi \in \Irr_p(\cG_\eta) = \Irr_p(\cG)$ which is trivial on $H$ will have $r_{\beta\psi, S} > 0$ (as $H$ contains the decomposition groups of all archimedean places) and there are infinitely many such $\psi$.

            However, the following does hold: given $\eta \in \Irr_p(H)$, the equality $r_\eta = r_{\beta\psi, S}$ is satisfied for almost all $\psi \in \Irr_p(\cG)$ such that $\rest_H^\cG \psi = \eta$ (this latter condition is stricter than $\psi \in \Irr_p(\cG_\eta) \subseteq \Irr_p(\cG)$, which just amounts to $\ker(\eta) \subseteq \ker(\psi)$). To see why, observe that if $\psi \in \Irr_p(\cG)$ restricts to $\eta$, one has
            \[
                \set{v \in S_\infty \colon \cG_v \subseteq \ker(\psi)} = \set{v \in S_\infty \colon \cG_v \subseteq \ker(\eta)},
            \]
            as all archimedean places have finite decomposition group and $H$ contains all elements of $\cG$ of finite order. Therefore, any potential difference between $r_{\beta\psi, S}$ and $r_\eta$ must come from the non-archimedean places, that is,
            \[
                r_{\beta\psi, S} \neq r_\eta \quad \iff \quad \set{v \in S_f \colon \cG_v \subseteq \ker(\psi)} \neq \varnothing.
            \]
            But $\cG_v \subseteq \ker(\psi)$ implies $\cap_{v \in S_f} \cG_v \subseteq \ker(\psi)$, which is only satisfied by finitely many $\psi$ (precisely those which factor over the finite quotient $\cG/\cap_{v \in S_f} \cG_v$).

            Having established this, the remaining implication follows easily: assume \cite{bks} conjecture 3.1 with ``every $\psi \in \hat{\cG}_\eta$'' replaced by ``almost every $\psi \in \hat{\cG}_\eta$''. The set $\Irr_p(\cG)$ is the disjoint union of the equivalence classes of the relation $\psi \sim \psi' \Leftrightarrow \rest_H^\cG \psi = \rest_H^\cG \psi'$ (although unnecessary, we know this is precisely the $W$-twist equivalence relation $\sim_W$). Let $C_\eta$ be the the class represented by $\eta \in \Irr_p(\cG)$, which in particular contains $\Irr_p(\cG_\eta)$. The assumption implies $\lambda_\psi(\cL_{S, T}) = \beta^{-1}(L_{K, S, T}^\ast(\beta \psi^{-1}, 0))$ for all $\psi \in \Irr_p(\cG_\eta)$ satisfying $r_\eta = r_{\beta\psi, S}$, which we have shown is the case for almost all $\psi \in C_\eta$. Now \hyperref[conje:bksp]{BKS'($L_\infty/K, L, S, T, \beta$)} follows from the fact that there are only finitely many classes $C_\eta$.
            \qedef
        \end{rem}

        In preparation for the central result of this section, we prove two technical results. The first one concerns determinant functors and (somewhat in disguise) refined Euler characteristics:

        \begin{prop}
        \label{prop:generator_det_rec}
            Let $\varphi \colon R \ia S$ be an injective homomorphism of commutative rings such such that $S$ is a flat right $R$-module (via $\varphi$) and semisimple Artinian. Suppose given a strictly perfect (cochain) complex $\cP\q$ of $R$-modules and an isomorphism of $S$-modules $t \colon S \otimes_R \cP^\even \xrightarrow{\sim} S \otimes_R \cP^\odd$ (in the notation of the beginning of section \ref{sec:an_integral_trivialisation}). Consider the composition
            \begin{equation}
            \label{diag:det_and_rec}
                \begin{tikzcd}[column sep=small]
                {\varphi_t \colon \Det_R(\cP\q)} \arrow[r, hook] & \Det_S(S \otimes_R^\LL \cP\q) \arrow[r, equals] & \Det_S(S \otimes_R \cP^\even) \otimes_S \Det_S(S \otimes_R \cP^\odd)^{-1} \arrow[d, "\Det_S(t) \otimes_S \Id"] &   \\
                                                                                 &                                     & \Det_S(S \otimes_R \cP^\odd) \otimes_S \Det_S(S \otimes_R \cP^\odd)^{-1} \arrow[r, "ev"]                       & S
                \end{tikzcd}
            \end{equation}
            where arrows are defined as in appendix \ref{app:determinant_functors}. Then, for any generator $g$ of $\Det_R(\cP\q)$ as an $R$-module (if it exists), one has $\varphi_t(g) \in \units{S} \subseteq S$.

            \begin{enumerate}[i)]
                \item{
                    Assume that $[\cP^\even, t, \cP^\odd]$ lies in the image of the connecting map ${\partial \colon K_1(S) \to K_0(R, S)}$ (cf. \eqref{eq:exact_sequence_k-theory}). Then any $g$ as above satisfies
                    \begin{equation}
                    \label{eq:det_and_pseudo_rec}
                        \partial(\varphi_t(g)) = [\cP^\even, t, \cP^\odd],
                    \end{equation}
                    where $\varphi_t(g) \in \units{S}$ is regarded as an element of $K_1(S)$ via the canonical injection $\units{S} \ia K_1(S)$ (cf. \cite{cr2} (40.27)).
                }
                \item{
                    Assume that $[\cP^\even, t, \cP^\odd] = \partial(x)$ for some $x \in K_1(S)$. Then there exists a generator $g$ of $\Det_R(\cP\q)$ as an $R$-module such $\varphi_t(g) = \nr(x) \in \units{S}$.
                }
            \end{enumerate}
        \end{prop}

        \begin{proof}
            The first claim is immediate. If $g$ generates $\Det_R(\cP\q)$ over $R$, then $1 \otimes g$ is a generator of ${S \otimes_R \Det_R(\cP\q) = \Det_S(S \otimes_R^\LL \cP\q)}$ over $S$. The remaining maps defining $\varphi_t$ are isomorphisms of $S$-modules, and thus $\varphi_t(g) \in S$ generates $S$ as an $S$-module - that is, it is a unit.

            Recall that, since $S$ is commutative semisimple, the reduced norm $\nr \colon K_1(S) \to \units{S}$ is an isomorphism (cf. \eqref{eq:reduced_norm_iso}) with obvious inverse $s \mapsto [(s)]$.

            \begin{enumerate}[i)]
                \item{
                    We can assume without loss of generality that $\cP^\even$ and $\cP^\odd$ are free: since $[\cP^\even, t, \cP^\odd]$ lies in $\img(\partial)$, its image $[\cP^\even] - [\cP^\odd]$ in $K_0(R)$ is zero by exactness of \eqref{eq:exact_sequence_k-theory}. This implies there exists a finitely generated projective $R$-module $N$ such that $\cP^\even \oplus N \iso \cP^\odd \oplus N$, and we can assume this sum to be free. Now define
                    \[
                        \widetilde{\cP}\q = \cP\q \oplus [\stackrel{0}{N} \xrightarrow{\Id} \stackrel{1}{N}].
                    \]
                    Then the relations of $K_0(R, S)$ immediately imply $[\cP^\even, t, \cP^\odd] = [\widetilde{\cP}^\even, t \oplus \Id, \widetilde{\cP}^\odd]$, and the isomorphisms
                    \[
                        \Det_R(\widetilde{\cP}^\even) = \Det_R(\cP^\even \oplus N) \xrightarrow{\sim} \Det_R(\cP^\even) \otimes_R \Det_R(N),
                    \]
                    $\Det_R(\widetilde{\cP}^\odd) \xrightarrow{\sim} \Det_R(\cP^\odd) \otimes_R \Det_R(N)$ (analogously) and $ev \colon \Det_R(N) \otimes_R \Det_R(N)^{-1} \xrightarrow{\sim} R$ induce an $R$-isomorphism $f$ making the diagram
                    \begin{center}
                        \begin{tikzcd}
                            \Det_R(\cP\q) \arrow[d, "f"] \arrow[r, "\varphi_t"]                & S \arrow[d, equals] \\
                            \Det_R(\widetilde{\cP}\q) \arrow[r, "\varphi_{t \oplus \Id}"] & S
                        \end{tikzcd}
                    \end{center}
                    commute. Hence it suffices to show that $\partial(\varphi_{t \oplus \Id}(f(g))) = [\widetilde{\cP}^\even, t \oplus \Id, \widetilde{\cP}^\odd]$, which allows us to assume $\cP^\even$ and $\cP^\odd$ to be free from now on. In particular, $\Det_R(\cP^\even)$ and $\Det_R(\cP^\odd)$ are free $R$-modules of rank 1.

                    In order to prove \eqref{eq:det_and_pseudo_rec}, we perform some simplification steps. Firstly, the theory of determinant functors shows that $[\cP^\even, t, \cP^\odd] = [\Det_R(\cP^\even), \Det_S(t), \Det_R(\cP^\odd)]$. Choose and fix an isomorphism of $S$-modules $\iota \colon \Det_S(S \otimes_R \cP^\even) \xrightarrow{\sim} S$. The commutative diagram
                    \begin{center}
                        \begin{tikzcd}
                            \Det_S(S \otimes_R \cP^\even) \arrow[d, "\iota"] \arrow[r, "\varphi_t(g)"] & \Det_S(S \otimes_R \cP^\even) \arrow[d, "\iota"] \\
                            S \arrow[r, "\varphi_t(g)"]                                              & S
                        \end{tikzcd}
                    \end{center}
                    where horizontal arrows denote multiplication by the scalar $\varphi_t(g) \in S$, implies that
                    \[
                        [\Det_R(\cP^\even), \varphi_t(g), \Det_R(\cP^\even)] = [R, \varphi_t(g), R] = \partial(\varphi_t(g))
                    \]
                    in $K_0(R, S)$ (see relation i) in the definition of $K_0(R, S)$ from section \ref{sec:algebraic-k-theory}). This reduces the proof of \eqref{eq:det_and_pseudo_rec} to showing that
                    \[
                        [\Det_R(\cP^\even), \varphi_t(g), \Det_R(\cP^\even)] = [\Det_R(\cP^\even), \Det_S(t), \Det_R(\cP^\odd)],
                    \]
                    or equivalently,
                    \[
                        [\Det_R(\cP^\even), \Det_S(t) \circ \varphi_t(g)^{-1}, \Det_R(\cP^\odd)] = 0.
                    \]
                    Equation \eqref{eq:iso_k0_over_r_trivial} implies it suffices in turn to prove
                    \begin{equation}
                    \label{eq:det_and_scalars_in_k_0}
                        \varphi_t(g)^{-1} \cdot \Det_S(t) (1 \otimes \Det_R(\cP^\even)) = 1 \otimes \Det_R(\cP^\odd) \subseteq S \otimes \Det_R(\cP^\odd),
                    \end{equation}
                    as this would induce an isomorphism of $R$-modules
                    \[
                        \Det_R(\cP^\even) \iso 1 \otimes \Det_R(\cP^\even) \xrightarrow{\varphi_t(g)^{-1} \cdot \Det_S(t)} 1 \otimes \Det_R(\cP^\odd) \iso \Det_R(\cP^\odd)
                    \]
                    by the injectivity of $R \ia S$.

                    Fix an arbitrary generator $g^\odd$ of $\Det_R(\cP^\odd)$ as an $R$-module and consider the generator $g^{\odd, *}$ of $\Det_R(\cP^\odd)^{-1} = \Hom_R(\Det_R(\cP^\odd), R)$ uniquely determined by ${g^{\odd, *}(g^\odd) = 1}$ as described in lemma \ref{lem:invertible_module_dual}. Then ${g = g^\even \otimes g^{\odd, *}}$ for some generator $g^\even$ of $\Det_R(\cP^\even)$. To see this, choose any generator $\tilde{g}^\even$ of $\Det_R(\cP^\even)$ and note that, since both $g$ and $\tilde{g}^\even \otimes g^{\odd, *}$ generate $\Det_R(\cP\q)$, they must differ by a unit in $R$.

                    The image of $g$ in $\Det_S(S \otimes_R^\LL \cP\q) = (S \otimes_R \Det_R(\cP^\even)) \otimes_S (S \otimes_R \Det_R(\cP^\odd)^{-1})$ is simply $(1 \otimes g^\even) \otimes (1 \otimes g^{\odd, *})$, which the vertical arrow in \eqref{diag:det_and_rec} sends to
                    \[
                        \Det_S(t)(1 \otimes g^\even) \otimes (1 \otimes g^{\odd, *}) \in (S \otimes_R \Det_R(\cP^\odd)) \otimes_S (S \otimes_R \Det_R(\cP^\odd)^{-1}).
                    \]
                    But $\Det_S(t)(1 \otimes g^\even)$ is equal to $(s \otimes g^\odd)$ for some $s \in S$, in which case one has
                    \[
                        \varphi_t(g) = ev(\Det_S(t)(1 \otimes g^\even) \otimes (1 \otimes g^{\odd, *})) = s
                    \]
                    by definition of the dual generator $g^{\odd, *}$. Since $g^\even$ and $g^\odd$ generate $\Det_R(\cP^\even)$ and $\Det_R(\cP^\odd)$ over $R$ respectively and $\Det_S(t)$ is $S$-(and hence $R$-)equivariant, \eqref{eq:det_and_scalars_in_k_0} follows.
                }
                \item{
                    By the same argument as above, we can assume without loss of generality that $\cP^\even$ and $\cP^\odd$ are free $R$-modules. In particular, $\Det_R(\cP\q)$ is free of rank 1 and we may choose a generator $g'$. Part i) implies $\partial(\varphi_t(g')) = [\cP^\even, t, \cP^\odd] = \partial(x)$.

                    By the exact sequence of $K$-theory \eqref{eq:exact_sequence_k-theory}, there exists a $y \in K_1(R)$ such that $x = \varphi_t(g') K_1(\varphi)(y)$. Then $g = \det(y) g'$  is a generator of $\Det_R(\cP\q)$ which satisfies
                    \[
                        \varphi_t(g) = \det(y) \varphi_t(g') = \nr(K_1(\varphi)(y) \varphi_t(g')) = \nr(x).
                    \]
                }
            \end{enumerate}

        \end{proof}

        \begin{rem}
            \begin{enumerate}[i)]
                \item{
                    The connecting homomorphism $\partial \colon K_1(\cQ(\cG)) \to K_0(\Lambda(\cG), \cQ(\cG))$ is surjective in our case of interest, as explained in remark \ref{rem:emcu} ii). Therefore, the hypothesis in part i) of the proposition is automatically satisfied.
                }
                \item{
                    The only reason for the flatness assumption is to ensure $S \otimes_R^\LL \cP\q$ coincides with the result of applying $S \otimes_R -$ to $\cP\q$ degree-wise, which allows us to state the proposition in a convenient form for our purposes. However, the proof yields a completely analogous result without the flatness requirement after replacing $\Det_R(\cP\q) \ia \Det_S(S \otimes_R^\LL \cP\q)$ by
                    \[
                        \Det_R(M) \otimes_R \Det_R(N)^{-1} \ia \Det_S(S \otimes_R M) \otimes_S \Det_S(S \otimes_R N)^{-1},
                    \]
                    where $M$ and $N$ are two finitely generated projective $R$-modules playing the role of $\cP^\even$ and $\cP^\odd$, respectively.
                }
            \end{enumerate}
        \end{rem}

        The second preparatory result, which is rather specific to our setting, will be instrumental in relating the maps in our Main Conjecture to the $\lambda_\chi$ from definition \ref{defn:bks_map} used by Burns, Kurihara and Sano:

        \begin{lem}
        \label{lem:trivialisation_det_infinite_finite_level}
            Setting \ref{sett:construction} with $\cG$ abelian. Let $M$ and $N$ be free $\Lambda(\cG)$-modules of the same rank and $\tau \colon M \to N$ a homomorphism between them such that $\cQ(\cG) \otimes_{\Lambda(\cG)} \tau$ is a $\cQ(\cG)$-isomorphism. For a given $\chi \in \Irr_p(\cG)$, choose $n \in \NN$ large enough that $\chi$ factorises through $\cG_n$. Denote by $\lambda_l$ and $\lambda_r$ the composition of the left and right columns of the diagram
            \begin{center}
                \addtolength{\leftskip} {-2cm}
                \addtolength{\rightskip} {-2cm}
                \begin{tikzcd}[column sep=tiny]
                    \Det_{\Lambda(\cG)}(M) \otimes_{\Lambda(\cG)} \Det_{\Lambda(\cG)}(N)^{-1} \arrow[d] \arrow[r, equals]                                                                                                                                                      & \Det_{\Lambda(\cG)}(M) \otimes_{\Lambda(\cG)} \Det_{\Lambda(\cG)}(N)^{-1} \arrow[d, hook]                                                                                                               \\
                    {\Det_{\QQ_p^c[\cG_n]}(\QQ_p^c \otimes_{\ZZ_p} \gcoinv{M}{n}) \otimes_{\QQ_p^c[\cG_n]} \Det_{\QQ_p^c[\cG_n]}(\QQ_p^c \otimes_{\ZZ_p} \gcoinv{N}{n})^{-1}} \arrow[d, "{\Det_{\QQ_p^c[\cG_n]}(\QQ_p^c \otimes_{\ZZ_p} \gcoinv{\tau}{n}) \otimes \Id}"] & \Det_{\cQ(\cG)}(\cQ(\cG) \otimes_{\Lambda(\cG)} M) \otimes_{\cQ(\cG)} \Det_{\cQ(\cG)}(\cQ(\cG) \otimes_{\Lambda(\cG)} N)^{-1} \arrow[d, "\Det_{\cQ(\cG)}(\cQ(\cG) \otimes_{\Lambda(\cG)} \tau) \otimes \Id"] \\
                    {\Det_{\QQ_p^c[\cG_n]}(\QQ_p^c \otimes_{\ZZ_p} \gcoinv{N}{n}) \otimes_{\QQ_p^c[\cG_n]} \Det_{\QQ_p^c[\cG_n]}(\QQ_p^c \otimes_{\ZZ_p} \gcoinv{N}{n})^{-1}} \arrow[d, "ev"]                                                                            & \Det_{\cQ(\cG)}(\cQ(\cG) \otimes_{\Lambda(\cG)} N) \otimes_{\cQ(\cG)} \Det_{\cQ(\cG)}(\cQ(\cG) \otimes_{\Lambda(\cG)} N)^{-1} \arrow[d, "ev"]                                                                \\
                    {\QQ_p^c[\cG_n]} \arrow[d, "\chi"]                                                                                                                                                                                                                 & \cQ(\cG)                                                                                                                                                                                                \\
                    \QQ_p^c
                \end{tikzcd}
            \end{center}
            respectively. Then, for any generator $g$ of $\Det_{\Lambda(\cG)}(M) \otimes_{\Lambda(\cG)} \Det_{\Lambda(\cG)}(N)^{-1}$ over $\Lambda(\cG)$, one has $\lambda_r(g) \in \units{\cQ(\cG)}$ and
            \[
                \lambda_l(g) = ev_{\gamma_\chi}(\psi_\chi(\nr^{-1}(\lambda_r(g)))),
            \]
            where the maps $\units{\cQ(\cG)} \xrightarrow{\nr^{-1}} K_1(\cQ(\cG)) \xrightarrow{\psi_\chi} \units{\cQ^c(\Gamma_\chi)} \xrightarrow{ev_{\gamma_\chi}} \QQ_p^c \cup \set{\infty}$ are given by remark \ref{rem:k1_commutative} and definitions \ref{defn:k1_to_gamma_chi} and \ref{defn:evaluation_at_zero}.
        \end{lem}

        \begin{proof}
            Recall the notation $-^\ast$ for dual generators of determinant modules introduced in lemma \ref{lem:invertible_module_dual}. Since $\Det_{\Lambda(\cG)}(M)$ and $\Det_{\Lambda(\cG)}(N)^{-1}$ are free rank-one $\Lambda(\cG)$-modules, any $g$ as in the statement is necessarily of the form $g = \mu \otimes \nu^\ast$ for some generators $\mu$ and $\nu$ of $\Det_{\Lambda(\cG)}(M)$ and $\Det_{\Lambda(\cG)}(N)$, respectively. The two bottom arrows in the right column of the diagram are $\cQ(\cG)$-isomorphisms, and therefore $g$ is mapped under $\lambda_r$ to a $\cQ(\cG)$-generator of $\cQ(\cG)$, that is, a unit $\lambda_r(g) \in \units{\cQ(\cG)}$.

            Consider the ring homomorphism $\overline{-} \colon \Lambda(\cG) \sa \Lambda(\cG_n) \ia \QQ_p^c[\cG_n]$, which induces a homomorphism of $\Lambda(\cG)$-modules $\overline{-} \colon T \to \QQ_p^c[\cG_n] \otimes_{\Lambda(\cG)} T = \QQ_p^c \otimes_{\ZZ_p} \gcoinv{T}{n}$ for any $\Lambda(\cG)$-module $T$. The elements $\overline{\mu}$ and  $\overline{\nu^\ast} = \overline{\nu}^\ast$ generate $\Det_{\QQ_p^c[\cG_n]}(\QQ_p^c \otimes_{\ZZ_p} \gcoinv{M}{n})$ and $\Det_{\QQ_p^c[\cG_n]}(\QQ_p^c \otimes_{\ZZ_p} \gcoinv{N}{n})^{-1}$ over $\QQ_p^c[\cG_n]$, respectively (see for instance \eqref{eq:wedge_product_basis}). Similarly, $1 \otimes \mu$ is a $\cQ(\cG)$-generator of $\cQ(\cG) \otimes_{\Lambda(\cG)} \Det_{\Lambda(\cG)}(M)$ (which coincides with $\Det_{\cQ(\cG)}(\cQ(\cG) \otimes_{\Lambda(\cG)} M)$), and analogously for $(1 \otimes \nu)^\ast = 1 \otimes \nu^\ast$.

            Since $\nu$ generates $\Det_{\Lambda(\cG)}(N)$, there exists an $x \in \Lambda(\cG)$ such that $\Det_{\Lambda(\cG)}(\tau)(\mu) = x \nu$. It follows that $\Det_{\QQ_p^c[\cG_n]}(\QQ_p^c \otimes_{\ZZ_p} \gcoinv{\tau}{n})(\mu) = \overline{x} \overline{\nu}$ and $\Det_{\cQ(\cG)}(\cQ(\cG) \otimes_{\Lambda(\cG)} \tau)(1 \otimes \mu) = x(1 \otimes \nu) = x \otimes \nu$. Therefore, the image of $g = \mu \otimes \nu^\ast$ in $\QQ_p^c[\cG_n]$ (that is, under $\lambda_l$ minus the last arrow) is $ev(\overline{x} \overline{\nu} \otimes \overline{\nu}^\ast) = \overline{x}$, whereas $\lambda_r(g) = \lambda_r(\mu \otimes \nu^\ast) = ev(x(1 \otimes \overline{\nu}) \otimes (1 \otimes \overline{\nu})^\ast) = x \in \cQ(\cG)$.

            Finally, we apply $\eqref{eq:relation_reduced_norm_and_invariants}$ and \eqref{eq:evaluation_j_and_psi} to $P = \Lambda(\cG)$ and $f = x$ (the endomorphism of $\Lambda(\cG)$ given by left multiplication by $x$) to conclude that
            \[
                \lambda_l(g) = \chi(\overline{x}) = \det(\Hom_{\QQ_p^c[\cG_n]}(V_\chi, \overline{f})) = ev_{\gamma_\chi}(\psi_{\chi}([\cQ(\cG), \cQ(\cG) \otimes_{\Lambda(\cG)} f])) = ev_{\gamma_\chi}(\psi_\chi(\nr^{-1}(\lambda_r(g)))),
            \]
            as desired.
        \end{proof}

        We are now in a position to prove the main result of this section:

        \begin{thm}
        \label{thm:bks}
            Setting \ref{sett:formulation} with $\cG$ abelian. The modified BKS conjecture \hyperref[conje:bksp]{BKS'($L_\infty/K, L, S, T, \beta$)} holds if and only if \hyperref[conje:emc]{eMC($L_\infty/K, L, S, T, \alpha, \beta$)} does.
        \end{thm}

        \begin{proof}
            In this proof, we denote $\Det_R(-)$ for a commutative ring $R$ by $\rD_R(-)$. In the diagrams and displayed equations, we will simplify the notation of $\Lambda(\cG), \Lambda(\cG_n), \cQ(\cG)$ and $\cQ^c(\cG)$ to $\Lambda, \Lambda_n, \cQ$ and $\cQ^c$ respectively. In diagrams, we will furthermore remove the subscript $R$ from $\otimes_R$ when the involved modules are of the form $\Det_R(M)$ for some $R$-module $M$.

            As explained after equation \eqref{eq:stricly_perfect_representative_two_degrees}, we may assume $\cC_{S, T}\q$ has a representative in $\cD(\Lambda(\cG))$ of the form ${\cP\q = [\cP^0 \to \cP^1]}$ (in degrees 0 and 1), with both $\cP^i$ free, by replacing the original set $T$ with a suitable one, which does not affect the validity of either conjecture. We set $\cP_n\q = \Lambda(\cG_n) \otimes_{\Lambda(\cG)}^\LL \cP\q$.

            The first step is to bridge some differences in the formulation of both conjectures. Let $\chi \in \Irr_p(\cG)$ satisfy the kernel condition \ref{manualcond:kc}. Choose an $n \in \NN$ such that $\chi$ factors over $\cG_n = \cG/\Gp{n}$ and let ${}_{\chi, \CC_p} -$ denote $e(\chi) \cdot \CC_p \otimes_{\ZZ_p} -$, where $e(\chi) \in \QQ_p^c[\cG_n] \subseteq \CC_p[\cG_n]$ is the primitive central idempotent corresponding to (the projection of) $\chi$. Then one has the following commutative diagram:
            \begin{center}
            \label{diag:bks_diagram}
                \begin{tikzcd}
                    {\rD_\Lambda(\cC_{S, T}\q)} \arrow[dddddd, "\lambda_\chi"] \arrow[rr, equals] &  & {\rD_\Lambda(\cC_{S, T}\q)} \arrow[d]                                                                                                                                                                              \\
                      &  & \rD_{\Lambda_n}(\cP_n\q) \arrow[d]                                                                                                                                                                       \\
                      &  & {\rD_{\CC_p[\cG_n]}(\CC_p \otimes_{\ZZ_p} \cP_n^\even) \otimes \rD_{\CC_p[\cG_n]}(\CC_p \otimes_{\ZZ_p} \cP_n^\odd)^{-1}} \arrow[d]                                                               \\
                      &  & {\rD_{\CC_p[\cG_n]}({}_{\chi, \CC_p} H^\even(\cC_{L_n, S, T}\q)) \otimes \rD_{\CC_p[\cG_n]}({}_{\chi, \CC_p} H^\odd(\cC_{L_n, S, T}\q))^{-1}} \arrow[d, "{\rD_{\CC_p[\cG_n]}(({}_{\chi, \CC_p} \varphi_n^\alpha)^{-1}) \otimes \Id}"] \\
                      &  & {\rD_{\CC_p[\cG_n]}({}_{\chi, \CC_p} H^\odd(\cC_{L_n, S, T}\q)) \otimes \rD_{\CC_p[\cG_n]}({}_{\chi, \CC_p} H^\odd(\cC_{L_n, S, T}\q))^{-1}} \arrow[d, "ev"]                                                             \\
                      &  & {\CC_p[\cG_n]} \arrow[d, "\chi"]                                                                                                                                                                                   \\
\CC_p                                                                 &  & \CC_p \arrow[ll, "{R_S^\beta(\alpha, \chi)}"']
                \end{tikzcd}
            \end{center}
            Here the bottom arrow is to be understood as multiplication by the regulator $R_S^\beta(\alpha, \chi) \in \units{\CC_p}$. Commutativity follows from the fact that the right column is defined by essentially the same maps as $\lambda_\chi$ (cf. definition \ref{defn:bks_map}, including the implicit choice of splittings $\sigma_0$ and $\sigma_1$), the difference being that this time we take $\chi$-parts \textit{before} the last step (which does not affect the resulting composition), and we choose as the isomorphism
            \[
                \rD_{\CC_p[\cG_n]}({}_{\chi, \CC_p} H^\even(\cC_{L_n, S, T}\q)) \xrightarrow{\sim} \rD_{\CC_p[\cG_n]}({}_{\chi, \CC_p} H^\odd(\cC_{L_n, S, T}\q))
            \]
            the map $\rD_{\CC_p[\cG_n]}(({}_{\chi, \CC_p} \varphi_n^\alpha)^{-1})$ rather than $\rD_{\CC_p[\cG_n]}({}_{\chi, \CC_p} \lambda_S^\beta)$. Since the Stark-Tate regulator $R_S^\beta(\alpha, \chi)$ is precisely the determinant of the $\CC_p$-action of ${}_{\chi, \CC_p} \lambda_S^\beta \circ ({}_{\chi, \CC_p} \varphi_n^\alpha)$ on $\Hom_{\CC_p[\cG_n]}(V_\chi, H^\odd(\cC_{L_n, S, T}\q))$, and determinant functors take endomorphisms of free modules to actual matrix determinants (cf. \eqref{eq:determinant_map_free_modules}), a simple computation shows the diagram commutes indeed. This relies crucially on the fact that $\chi(1) = 1$ by virtue of the abelianity of $\cG$.

            Consider now the right column of said diagram. Since it does not feature the Dirichlet regulator map, all arrows and objects after $\rD_\Lambda(\cC_{S, T}\q)$ can be replaced by their analogues with scalars in $\QQ_p^c$ instead of $\CC_p$. To do so, we also replace, ${}_{\chi, \CC_p} -$ by ${}_\chi - = e(\chi) \cdot \QQ_p^c \otimes_{\ZZ_p} -$. Here we tacitly use the fact that splittings exist over $\QQ_p^c[\cG_n]$ by semisimplicity (the resulting map is independent of the choice of splittings, similarly to remark \ref{rem:map_bks}). Let $\tilde{\lambda}_\chi \colon \rD_{\Lambda(\cG)}(\cC_{S, T}\q) \to \QQ_p^c$ denote the composition of the entire resulting column.

            In order to relate the finite and infinite levels, we would like to shift the complex so that $\varphi_n^\alpha$ does not need to be inverted (the trivialisation $t^\alpha$ is a map from odd to even degree, too). This can be done as follows: since $\cP^\even$ and $\cP^\odd$ are free $\Lambda(\cG)$-modules by assumption, $\rD_{\Lambda(\cG)}(\cP^\even)$ and $\rD_{\Lambda(\cG)}(\cP^\odd)$ are free of rank 1. Let $g^\even$ and $g^\odd$ be arbitrary respective generators. In particular, $g^{\odd, \ast}$ is a generator of $\rD_{\Lambda(\cG)}(\cP^\odd)^{-1}$ in the notation of lemma \ref{lem:invertible_module_dual}, and $g^\even \otimes g^{\odd, \ast}$ generates $\rD_{\Lambda(\cG)}(\cP\q) = \rD_{\Lambda(\cG)}(\cC_{S, T}\q)$. This yields an isomorphism
            \begin{equation}
            \label{eq:isomorphism_shift_generators}
                \iota_{g^\even} \otimes \iota_{g^{\odd, \ast}} \colon \rD_\Lambda(\cC_{S, T}\q) = \rD_\Lambda(\cP^\even) \otimes \rD_\Lambda(\cP^\odd)^{-1} \xrightarrow{\sim} \rD_\Lambda(\cP^\even)^{-1} \otimes \rD_\Lambda(\cP^\odd) = \rD_\Lambda(\cC_{S, T}\q[1]),
            \end{equation}
            with $\iota_{g^\even}$ and $\iota_{g^{\odd, \ast}}$ given by \eqref{eq:isomorphism_inverse_module}. Lemma \ref{lem:trivialising_dual_inverse} then shows that the diagram
            \begin{equation}
            \label{diag:bks_diagram_2}
                \begin{tikzcd}
                    {\rD_\Lambda(\cC_{S, T}\q)} \arrow[dddddd, "\tilde{\lambda}_\chi"] \arrow[rr, "{\iota_{g^\even} \otimes \iota_{g^{\odd, \ast}}}"] &  & {\rD_\Lambda(\cC_{S, T}\q[1])} \arrow[d]                                                                                                                                                                              \\
                      &  & \rD_{\Lambda_n}(\cP_n\q[1]) \arrow[d]                                                                                                                                                                                   \\
                      &  & {\rD_{\QQ_p^c[\cG_n]}(\QQ_p^c \otimes_{\ZZ_p} \cP_n^\even)^{-1} \otimes \rD_{\QQ_p^c[\cG_n]}(\QQ_p^c \otimes_{\ZZ_p} \cP_n^\odd)} \arrow[d]                                                                                                                 \\
                      &  & {\rD_{\QQ_p^c[\cG_n]}({}_\chi H^\even(\cC_{L_n, S, T}\q))^{-1} \otimes \rD_{\QQ_p^c[\cG_n]}({}_\chi H^\odd(\cC_{L_n, S, T}\q))} \arrow[d, "{\Id \otimes \rD_{\QQ_p^c[\cG_n]}({}_\chi \varphi_n^\alpha)}"] \\
                      &  & {\rD_{\QQ_p^c[\cG_n]}({}_\chi H^\even(\cC_{L_n, S, T}\q)^{-1}) \otimes \rD_{\QQ_p^c[\cG_n]}({}_\chi H^\even(\cC_{L_n, S, T}\q))} \arrow[d, "ev"]                                                             \\
                      &  & {\QQ_p^c[\cG_n]} \arrow[d, "\chi"]                                                                                                                                                                                   \\
                    \QQ_p^c \arrow[rr, "{\tilde{\lambda}_\chi(g^\even \otimes g^{\odd, \ast})^{-2}}"]  &  & \QQ_p^c
                    \end{tikzcd}
            \end{equation}
            commutes, where the bottom arrow again denotes multiplication by the square of the inverse of $\tilde{\lambda}_\chi(g^\even \otimes g^{\odd, \ast}) \in \units{(\QQ_p^c)}$. This relies on the fact that the arrow $\chi \colon \QQ_p^c[\cG_n] \to \QQ_p^c$ is a ring homomorphism.

            The last necessary simplification in order to relate the finite- and infinite-level maps is to replace the above morphisms in cohomology by maps on $\cP\q$. A convenient way to do so is to find a map on integral level $\Lambda(\cG)$ which induces both the finite-level splittings used in the construction of $\lambda_\chi$ (definition \ref{defn:bks_map}); and the infinite-level splittings used in the construction of the refined Euler characteristic (cf. \eqref{eq:refined_euler_characteristic_map}). This presents some difficulties, which are however not new: we encountered essentially the same issue in the proof of proposition \ref{prop:independence_of_alpha} (where $\alpha$ itself might not be invertible), and we shall use the same workaround to resolve it. Namely, let $\pi_0$ be the arrow in $\cP\q = [\cP^0 \to \cP^1]$, $W = \img(\pi_0) \subseteq \cP^1$ its image, and $\pi_1$ the surjection $\pi_1 \colon \cP^1 \sa H^1(\cP\q) = H^1(\cC_{S, T}\q)$. Chose splittings $\sigma_0$ and $\sigma_1$ for the complex $\cQ(\cG) \otimes_{\Lambda(\cG)}^\LL \cP\q$ as in \eqref{eq:refined_euler_characteristic_map}:
            \begin{equation}
            \label{diag:splittings_rec}
                \begin{tikzcd}[column sep=2em]
                    \cQ \otimes_\Lambda H^0(\cP\q) \arrow[r, hook] & \cQ \otimes_\Lambda \cP^0 \arrow[rr] \arrow[rd, "\cQ \otimes_\Lambda \pi_0", two heads] &                                                                          & \cQ \otimes_\Lambda \cP^1 \arrow[rr, "\cQ \otimes_\Lambda \pi_1", two heads] &  & \cQ \otimes_\Lambda  H^1(\cP\q) \arrow[ll, "\sigma_1", bend left] \\
                                                                   &                                                                                         & \cQ \otimes_\Lambda W \arrow[ru, hook] \arrow[lu, "\sigma_0", bend left] &                                                                              &  &
                \end{tikzcd}
            \end{equation}

            We can scale $\sigma_0$ and $\sigma_1$ by regular elements elements $d_0, d_1 \in \Lambda(\Gamma) \subseteq \Lambda(\cG)$ to obtain $\Lambda(\cG)$-homomorphisms $d_0 \sigma_0 \colon W \to \cP^0$ and $d_1 \sigma_1 \colon H^1(\cP\q) \to \cP^1$ (recall, for instance, equation \eqref{eq:tensor_hom}). By lemma \ref{lem:torsion_modules_bounded_coinvariants}, there exists an $n_0 \in \NN$ such that multiplication by $d_0 d_1$ (or, equivalently, its projection $\overline{d_0 d_1} \in \Lambda(\cG_m)$) is an automorphism of $e(\chi)\QQ_p^c[\cG_m]$ for all $m \geq n_0$ (simply take $M = \Lambda(\cG)/\ideal{d_0 d_1}$ regarded as a $\Lambda(\Gamma)$-module) \textit{as long as $\chi$ doesn't factor over $\cG_{n_0}$}. This is essentially the same argument as in \eqref{eq:multiplication_by_scalar_isomorphism_on_chi_parts}.

            Fix $d_0, d_1$ and $n_0$ as above, which do not depend on the choice of a character $\chi$. Assume $\chi$ factors through $\cG_n$ but not $\cG_{n_0}$, and it still satisfies \ref{manualcond:kc}. Let $\delta_0$ and $\delta_1$ be the images of $d_0$ and $d_1$ under $\Lambda(\cG) \sa \Lambda(\cG_n) \ia \QQ_p^c[\cG_n] \xrightarrow{\chi} \QQ_p^c$ respectively, both of which are non-zero. Then $\delta_0^{-1} \cdot {}_\chi \gcoinv{(d_0 \sigma_0)}{n}$ and $\delta_1^{-1} \cdot {}_\chi \gcoinv{(d_1 \sigma_1)}{n}$ are splittings for the finite-level sequences
            \begin{equation}
            \label{diag:bks_splittings_finite_level}
                \begin{tikzcd}[column sep=2.2em]
                    {}_\chi \gcoinv{H^0(\cP\q)}{n} \arrow[rr, hook] &  & {}_\chi \gcoinv{\cP^0}{n} \arrow[rr] \arrow[rd, two heads, "{}_\chi \gcoinv{(\pi_0)}{n}"] &                                                                                                                  & {}_\chi \gcoinv{\cP^1}{n} \arrow[rr, two heads, , "{}_\chi \gcoinv{(\pi_1)}{n}"] &  & {}_\chi \gcoinv{H^1(\cP\q)}{n} \arrow[ll, "\delta_1^{-1} \cdot {}_\chi \gcoinv{(d_1 \sigma_1)}{n}", bend left] \\
                                                        &  &                                     & {}_\chi \gcoinv{W}{n} \arrow[ru, hook] \arrow[lu, "\delta_0^{-1} \cdot {}_\chi \gcoinv{(d_0 \sigma_0)}{n}", bend left] &                                     &  &
                \end{tikzcd}
            \end{equation}
            Note that $\gcoinv{\cP^0}{n}$ and $\gcoinv{\cP^1}{n}$ coincide with $\cP_n^0$ and $\cP_n^1$ respectively, and ${}_\chi \gcoinv{H^0(\cP\q)}{n} = {}_\chi \cE_{L_n, S, T}$ is precisely ${}_\chi \ZZ_p \otimes_\ZZ \units{\cO_{L_n, S, T}} = {}_\chi H^0(\cP_n\q)$ by proposition \ref{prop:isomorphism_on_chi_parts} and the preceding results. The injectivity of ${}_\chi \gcoinv{W}{n} \ia {}_\chi \gcoinv{\cP^1}{n}$ then follows from a $\QQ_p^c$-dimension argument.

            We can now define our desired integral map
            \[
                \tau \colon \cP^\odd = \cP^1 \to H^\odd(\cP\q) \oplus W \xrightarrow{t_\iota^\alpha \oplus \Id} H^\even(\cP\q) \oplus W \to \cP^0 = \cP^\even,
            \]
            where $t_\iota^\alpha$ is the integral trivialisation defined in \eqref{eq:integral_trivialisation}, and the first and last arrows are given by $d_1 \pi_1 \oplus (d_1\Id - (d_1\sigma_1)\pi_1)$ and $(d_0 \sigma_0) \oplus d_0 \Id$. These are simply the standard maps arising from a split short exact sequence, only scaled by $d_0$ and $d_1$ to make them integral (which is why we write $(d_1\sigma_1)$ and $(d_0 \sigma_0)$ in brackets).

            Denote by $\dot{\lambda_\chi}$ the composite map given by the second column of \eqref{diag:bks_diagram_2}. Since the choice of splittings implicit in the third vertical arrow is irrelevant, we can assume they coincide with those in \eqref{diag:bks_splittings_finite_level} on $\chi$-parts, after which an easy verification shows the commutativity of the leftmost square in our last (and key) diagram
            \begin{center}
                \addtolength{\leftskip} {-2cm}
                \addtolength{\rightskip} {-2cm}
                \begin{tikzcd}[column sep=tiny]
                    {\rD_\Lambda(\cC_{S, T}\q[1])} \arrow[ddddd, "\dot{\lambda}_\chi"] \arrow[r, equals] & {\rD_\Lambda(\cC_{S, T}\q[1])} \arrow[d] \arrow[r, hook]                                                                                    & {\cQ  \rD_\Lambda(\cC_{S, T}\q[1])} \arrow[dd, equals] \arrow[r, equals]                                                                             & {\cQ  \rD_\Lambda(\cC_{S, T}\q[1])} \arrow[dd, equals]                                                                                                \\
                     & {\rD_{\Lambda_n}(\cP\q[1])} \arrow[d]                                                                                                       &                                                                                                                                      &                                                                                                                                               \\
                     & {\rD_n(\QQ_p^c \otimes_{\ZZ_p} \cP_n^\even)^{-1} \otimes \rD_n(\QQ_p^c \otimes_{\ZZ_p} \cP_n^\odd)} \arrow[d, "\Id \otimes \rD_n(\QQ_p \otimes_{\ZZ_p} \gcoinv{\tau}{n})"'] & \rD_{\cQ}(\cQ  \cP^\even)^{-1} \otimes \rD_{\cQ}(\cQ  \cP^\odd) \arrow[d, "\Id \otimes \rD_{\cQ}(\cQ \otimes_\Lambda \tau)"'] \arrow[r, equals] & \rD_{\cQ}(\cQ  \cP^\even)^{-1} \otimes \rD_{\cQ}(\cQ  \cP^\odd) \arrow[d, "\Id \otimes \rD_{\cQ}((d_0 d_1)^{-1} \cdot \cQ \otimes_\Lambda \tau)"'] \\
                     & {\rD_n(\QQ_p^c \otimes_{\ZZ_p} \cP_n^\even)^{-1} \otimes \rD_n(\QQ_p^c \otimes_{\ZZ_p} \cP_n^\even)} \arrow[d, "ev"]                                                       & \rD_{\cQ}(\cQ  \cP^\even)^{-1} \otimes \rD_\cQ(\cQ  \cP^\even) \arrow[d, "ev"]                                                 & \rD_{\cQ}(\cQ  \cP^\even)^{-1} \otimes \rD_{\cQ}(\cQ  \cP^\even) \arrow[d, "ev"]                                                        \\
                     & {\QQ_p^c[\cG_n]} \arrow[d, "\chi"]                                                                                                          & \cQ \arrow[r, "(d_0 d_1)^{-r}"]                                                                                                      & \cQ                                                                                                                                           \\

                    \QQ_p^c \arrow[r, "(\delta_0 \delta_1)^r"]                                       & \QQ_p^c                                                                                                                                     &
                    \units{\cQ} \arrow[r, "(d_0 d_1)^{-r}"] \arrow[u, hook]                                                                       & \units{\cQ} \arrow[u, hook]
                    \\
                    & & K_1(\cQ) \arrow[d, "ev_{\gamma_\chi} \psi_\chi"] \arrow[r, "(d_0 d_1)^{-r}"] \arrow[u, "\nr"]                                                                       & K_1(\cQ) \arrow[d, "ev_{\gamma_\chi} \psi_\chi"] \arrow[u, "\nr"]                                                                              \\
                                                                                                 &                                                                                                                                             & \QQ_p^c \cup \set{\infty} \arrow[r, "(\delta_0 \delta_1)^{-r}"]                                                                      & \QQ_p^c \cup \set{\infty}
                \end{tikzcd}
            \end{center}
            where $r = \rank_{\Lambda}\cP^0 = \rank_{\Lambda}\cP^1$. In this diagram (and only there), we have made the following further notational simplification: $\rD_n -$ and $\cQ -$ stand for $\rD_{\QQ_p^c[\cG_n]}$ and $\cQ \otimes_{\Lambda} -$, respectively. Commutativity of the bottom-right square follows from \eqref{eq:relation_reduced_norm_and_invariants} and \eqref{eq:evaluation_j_and_psi} (choosing $f$ to be be multiplication by $(d_0 d_1)^r$ on a free rank one $\Lambda(\cG)$-module), and the rest is clear.

            When then together, the above diagrams draw a clear connection between the map $\lambda_\chi$ of Burns, Kurihara and Sano and our evaluation maps $ev_{\gamma_\chi}$. In combination with the two previous lemmas, this suffices to conclude the proof. We only need to consider characters $\chi$ satisfying the following condition:
            \begin{equation}
            \tag{KC'}
            \label{eq:kcp}
                \Gp{\max(n_0, n(S, \alpha))} \nsubseteq \ker(\chi).
            \end{equation}
            This is a slight strengthening of \ref{manualcond:kc} (i.e. $\Gp{n(S, \alpha)} \nsubseteq \ker(\chi)$) which still holds for almost all $\chi \in \Irr_p(\cG)$.

            Assume first BKS'($L_\infty/K, L, S, T, \beta$) and let $\cL_{S, T}$ be the element predicted therein. As a generator of $\rD_{\Lambda(\cG)}(\cC_{S, T}\q) = \rD_{\Lambda(\cG)}(\cP^\even) \otimes_{\Lambda(\cG)} \rD_{\Lambda(\cG)}(\cP^\odd)^{-1}$, it admits an expression $\cL_{S, T} = g^\even \otimes g^{\odd, \ast}$ as above. Set $\widetilde{\cL}_{S, T} = (\iota_{g^\even} \otimes \iota_{g^{\odd, \ast}})(\cL_{S, T}) \in \rD_{\Lambda(\cG)}(\cC_{S, T}\q[1])$ and define $\widetilde{\zeta}$ as the image of $\widetilde{\cL}_{S, T}$ under the composition
            \[
                \rD_\Lambda(\cC_{S, T}\q[1]) \ia
                \rD_{\cQ}(\cQ \otimes_\Lambda \cP^\even)^{-1} \otimes \rD_\cQ(\cQ \otimes_\Lambda \cP^\odd) \to
                \rD_{\cQ}(\cQ \otimes_\Lambda \cP^\even)^{-1} \otimes \rD_{\cQ}(\cQ \otimes_\Lambda \cP^\even) \xrightarrow{ev} \cQ
            \]
            where the middle arrow is $\rD_{\cQ(\cG)}((d_0 d_1)^{-1} \cdot \cQ(\cG) \otimes_{\Lambda(\cG)} \tau)$. This amounts to the top row and first vertical arrows in the fourth column of the last diagram above. Proposition \ref{prop:generator_det_rec} i) now shows that $\widetilde{\zeta}$ lies in lies in $\units{\cQ(\cG)}$, which we identify with $K_1(\cQ(\cG))$ via the reduced norm; and its inverse, which we denote by $\zeta \in K_1(\cQ(\cG))$, satisfies
            \[
                \partial(\zeta) = -\partial(\widetilde{\zeta}) = - [\cP^\odd, (d_0 d_1)^{-1} \cdot \cQ \otimes_\Lambda \tau, \cP^\even]
            \]
            (beware the difference in shift between $\cP\q[1]$ here and $\cP\q$ in the proposition). But $(d_0 d_1)^{-1} \cdot \cQ \otimes_\Lambda \tau$ is precisely the map constructed from the trivialisation $t^\alpha = \cQ(\cG) \otimes_{\Lambda(\cG)} t_\iota^\alpha$ and the splittings $\sigma_0$ and $\sigma_1$ from \eqref{diag:splittings_rec}. In other words, $\partial(\zeta) = - \chi_{\Lambda(\cG), \cQ(\cG)}(\cC_{S, T}\q, t^\alpha)$ is the inverse of the refined Euler characteristic of the trivialised complex $(\cC_{S, T}\q, t^\alpha)$. On the analytic side, we have
            \begin{align*}
                ev_{\gamma_\chi}(\psi_\chi(\zeta))
                & \stackrel{\cdot}{=} ev_{\gamma_\chi}(\psi_\chi(\widetilde{\zeta}))^{-1} \\
                & = \dot{\lambda}_\chi(\widetilde{\cL}_{S, T})^{-1} \\
                & = \tilde{\lambda}_\chi(\cL_{S, T}) \\
                & = R_S^\beta(\chi, \alpha)^{-1} \cdot \lambda_\chi(\cL_{S, T}) \stackrel{\cdot}{=} \frac{\beta^{-1}(L_{K, S, T}^\ast(\beta \check{\chi}, 0))}{R_S^\beta(\alpha, \chi)},
            \end{align*}
            where $\stackrel{\cdot}{=}$ denotes that the equality holds for almost all $\chi \in \Irr_p(\cG)$ satisfying \eqref{eq:kcp}. The equalities follow from, in order: multiplicativity of $ev_{\gamma_\chi}$ (this holds whenever the evaluation is neither $0$ nor $\infty$, which is true a fortiori for almost all $\chi$ by the remaining equalities); the last diagram above together with lemma \ref{lem:trivialisation_det_infinite_finite_level}; diagram \eqref{diag:bks_diagram_2}; diagram $\eqref{diag:bks_diagram}$; and the assumption BKS'($L_\infty/K, L, S, T, \beta$).

            Since $\zeta$ does not depend on $\chi$ and the map $\psi_\chi$ is invariant under $\rho$-twists (see the explanation before \eqref{eq:psi_and_j}), the element $F_{S, T, \chi}^{\alpha, \beta} = \psi_\chi(\zeta) \in \units{\cQ^c(\Gamma_\chi)}$ satisfies IC($L_\infty/K, \chi, L, S, T, \alpha, \beta$). It follows that $\zeta$ itself satisfies eMC($L_\infty/K, L, S, T, \alpha, \beta$).

            The converse is proved very similarly. Namely, suppose the zeta element $\zeta_{S, T}^{\alpha, \beta} \in K_1(\cQ(\cG))$ satisfies eMC($L_\infty/K, L, S, T, \alpha, \beta$). Then proposition \ref{prop:generator_det_rec} ii) yields a generator $\widetilde{\cL}_{S, T} \in \rD_{\Lambda(\cG)}(\cC_{S, T}\q[1])$ which is mapped to $\zeta_{S, T}^{\alpha, \beta}$ under the top row and rightmost column of the last diagram. The preimage $\cL_{S, T}$ of $\widetilde{\cL}_{S, T}$ under the isomorphism $\iota_{g^\even} \otimes \iota_{g^{\odd, \ast}}$ is a generator of $\rD_{\Lambda(\cG)}(\cC_{S, T}\q)$ which can be easily verified to have the interpolation property in BKS'($L_\infty/K, L, S, T, \beta$) using the same argument as above.
        \end{proof}

        \begin{rem}
            It is a by-product of this theorem that our Main Conjecture is independent from the choice $\alpha$, as proved separately in subsection \ref{subsec:the_choice_of_alpha} - indeed, some of the core ideas of both arguments coincide. Note, however, that the above result concerns the abelian case exclusively.\qedef
        \end{rem}

        This settles the case ``abelian over $\QQ$'' as well as, under some conditions, ``abelian over imaginary quadratic'':
        \begin{cor}
            Setting \ref{sett:formulation} with $\cG$ abelian. Conjecture \hyperref[conje:emc]{eMC($L_\infty/\QQ, L, S, T, \alpha, \beta$)} holds in the following cases:
            \begin{enumerate}[i)]
                \item{
                    $K = \QQ$ (so $p \neq 2$).
                }
                \item{
                    $K$ is an imaginary quadratic field where $p$ does not split, and the $\mu$-invariant $\mu(X_S^{cs})$ vanishes (for the extension $L_\infty/L$). Condition $\mu(X_S^{cs}) = 0$ is automatically satisfied if $[L : K]$ is a power of $p$.
                }
            \end{enumerate}
        \end{cor}

        \begin{proof}
            Conjecture 3.1 from \cite{bks} is known in both cases, which implies \hyperref[conje:bksp]{BKS'($L_\infty/\QQ, L, S, T, \beta$)} by the above theorem. For part i), see \cite{bks} remark 5.6, which in turn refers to \cite{burns_greither}. Remark 6.1 ii) in the latter points to the original work of Mazur and Wiles.

            For part ii), see \cite{bullach_hofer} theorem 5.5 and proposition 5.6.
        \end{proof}

        \begin{rem}
            As an application, we outline the proof of a new case of the Main Conjecture which relies on the above corollary and the functoriality results from section \ref{sec:functoriality}. The relevance of this example is that, to the best of our knowledge, it is not covered by existing formulations of a Main Conjecture - let alone proved.

            Consider setting \ref{sett:formulation} with $K = \QQ$ (so $p \neq 2$) and $\Gal(L/\QQ)$ of the form $A \rtimes \zmod{2}$, where $A$ is abelian of odd order and the non-trivial element of $\zmod{2}$ acts on $A$ by inversion. Such a group is said to be \textbf{generalised dihedral}. Assume furthermore that $L^A$ is imaginary (quadratic) and satisfies the conditions of part ii) of the corollary.

            By proposition \ref{prop:reduction_step}, the Main Conjecture for $L_\infty/\QQ$ can be deduced from that for $L_\infty/L^U$ for each elementary subgroup $U$ of $\Gal(L/K)$. It is easy to verify that any such $U$ satisfies at least one of the following:
            \begin{enumerate}[1)]
                \item{
                    $U \leq A$.
                }
                \item{
                    $\abs{U} = 2$.
                }
            \end{enumerate}
            The Main Conjecture for $L_\infty/L^A$ holds by ii) of the corollary, which settles all subgroups of type 1) by proposition \ref{prop:functoriality_2}.

            For $U$ of type 2), it is convenient to recall the notion of so-called \textit{maximal-order} Main Conjectures. These have the same structure as our equivariant Main Conjecture except for the fact that a preimage in $K_1(\cQ(\cG))$ of the refined Euler characteristic of the trivialised complex is claimed to have reduced norm congruent to the analytic object modulo \textit{not} the image of $K_1(\Lambda(\cG))$ (recall \eqref{eq:nu_g}), but rather all of $\units{Z(\fM(\cG))}$, where $\fM(\cG) \subseteq \Lambda(\cG)$ denotes a maximal $\Lambda(\Gamma)$-order in $\cQ(\cG)$ (note that the resulting claim is weaker). These conjectures are known to enjoy similar functoriality properties to the ones in chapter \ref{chap:properties_of_the_main_conjecture} and, more importantly, to decompose into $\chi$-parts (with $\chi$ an irreducible character), since so does $\fM(\cG)$.

            If $U$ is of type 2), then one has $\Gal(L_\infty/L^U) = \Gamma \times U$ and $\Lambda(\Gamma \times U)$ is a maximal $\Lambda(\Gamma)$-order in ${\cQ(\Gamma \times U)}$. Therefore, it suffices to prove the maximal-order Main Conjecture for $L_\infty/\QQ$ by functoriality - and for this, in turn, to prove the $\chi$-part of said conjecture for each $\chi \in \Irr_p(\cG)$.

            We distinguish two further cases. If $\chi$ is linear, it factors as a character of an abelian extension of $\QQ$, in which case the full equivariant Main Conjecture is known by part i) of the corollary - and therefore so is (the $\chi$-part of) the maximal-order conjecture.

            If $\chi$ is not linear, then \cite{jn_brumer} theorem 9.3 shows that $\chi = \indu_A^{\Gal(L/K)} \chi'$ for some non-trivial irreducible character $\chi'$ of $A$. But then functoriality together with case 1) above yields the desired conjecture. This covers the remaining cases and therefore concludes the argument.\qedef
        \end{rem}

    \section{The conjecture of Ritter and Weiss}
    \label{sec:the_conjecture_of_ritter_and_weiss}

        This last section studies the relation between our Main Conjecture and that formulated by Ritter and Weiss in \cite{rwii}, the first example of an equivariant Main Conjecture for arbitrary totally real base field and one of the main motivations for the present work. It was proved in the so-called ``$\mu = 0$ case'' by the same authors in a series of articles leading up to \cite{rw_on_the}. Independently from this, Kakde formulated and proved a Main Conjecture in \cite{kakde_mc} which was shown to be equivalent to that of Ritter and Weiss by, separately, Nickel (cf. \cite{nickel_plms}) and Venjakob (cf. \cite{venjakob_on_rw}).

        A common feature of the aforementioned conjectures is their limitation to the totally real case. Through a duality argument, it will be shown that, whenever $L$ is a CM field containing a primitive $p$-th root of unity, the \textit{minus part} of our Main Conjecture for $L_\infty/K$ is equivalent to the conjecture of Ritter and Weiss for the maximal totally real subextension $L_\infty\pl/K$. The set of places $T$ can be disregarded altogether, as it is not featured in \cite{rwii}. The setting is as follows:

        \begin{sett}
        \label{sett:rw}
        \addcontentsline{toc}{subsection}{Setting E}
            We consider the same objects as in setting \ref{sett:construction} with the exception of $T$. The following additional restrictions are imposed:
            \begin{itemize}
                \item{
                    $K$ is totally real. In particular, $p$ is odd.
                }
                \item{
                    $L$ is a CM field containing a primitive $p$-th root of unity. The former amounts to the existence of a totally real number field $L\pl \subseteq L$ such that $[L : L\pl] = 2$. This $L\pl$ is necessarily a Galois extension of $K$, as is its own cyclotomic $\ZZ_p$-extension $L_\infty\pl$. We set $\cG\pl = \Gal(L_\infty\pl/K)$.
                }
            \end{itemize}
            The Galois group of $L/L\pl$ is generated by complex conjugation, which has order 2 and will be denoted by $\tau$ here. It identifies naturally with the generator of $\Gal(L_\infty/L_\infty\pl)$, which allows us to write $\cG\pl = \cG/\ideal{\tau}$.\qedef
        \end{sett}

        We point out that the $\alpha$ and $\beta$ from setting \ref{sett:formulation} are not necessary here for reasons which will become apparent soon.

        Since $\tau$ is central in $\cG$ and has order $2$, every $\chi \in \Irr_p(\cG)$ satisfies either $\chi(\tau) = \chi(1)$ or ${\chi(\tau) = -\chi(1)}$. We say $\chi$ is \textbf{totally even}\index{character!totally even} in the former case, and \textbf{totally odd}\index{character!totally odd} in the latter. We denote the set of totally even Artin characters of $\cG$ by $\Irr_p(\cG)\pl$, which is in bijection with $\Irr_p(\cG\pl)$; and that of totally odd ones, by $\Irr_p(\cG)\mi$. The subgroup $H = \Gal(L_\infty/K_\infty)$ contains $\tau$, and therefore the $W$-equivalence relation $\sim_W$ introduced in section \ref{sec:morphisms_on_finite_level}
        restricts separately to $\Irr_p(\cG)\pl$ and $\Irr_p(\cG)\mi$.

        The primitive central idempotents $e\pl = (1 + \tau)/2$ and $e\mi = (1 - \tau)/2$ yield a ring decomposition $\Lambda(\cG) = e\pl \Lambda(\cG) \times e\mi \Lambda(\cG)$, and hence also for $\cQ(\cG)$ and $\cQ^c(\cG)$. We denote the \textbf{plus part}\index{plus part} of a $\Lambda(\cG)$-module $M$ by $M\pl = e\pl M$, and analogously for its \textbf{minus part}\index{minus part} $M\mi$. An immediate computation shows that $M\pl$ and $M\mi$ consist precisely of the elements of $M$ upon which $\tau$ acts as the identity and as $-1$, respectively. Taking plus parts is an exact functor (whether one chooses the target category to be that of left modules over $\Lambda(\cG)$ or $\Lambda(\cG)\pl$), as is taking minus parts. Note that these functors identify naturally with $\Lambda(\cG)\pl \otimes_{\Lambda(\cG)} -$ and $\Lambda(\cG)\mi \otimes_{\Lambda(\cG)} -$, respectively. The canonical augmentation map $\Lambda(\cG) \sa \Lambda(\cG\pl)$ restricts to a ring isomorphism on $\Lambda(\cG)\pl$ and to the zero map on $\Lambda(\cG)\mi$.

        The reduced norm and structure of $Z(\cQ^c(\cG))$ also decompose into plus and minus parts:
        \begin{center}
            \begin{tikzcd}[row sep=large, column sep=small]
                K_1(\cQ(\cG)) \arrow[r, equals] \arrow[d]                                          & K_1(\cQ(\cG)\pl) \times K_1(\cQ(\cG)\mi) \arrow[d, shift right=13] \arrow[d, shift left=13]                                                                                            \\
                K_1(\cQ^c(\cG)) \arrow[r, equals] \arrow[d, "\nr"]                                 & K_1(\cQ^c(\cG)\pl) \times K_1(\cQ^c(\cG)\mi) \arrow[d, "\nr", shift right=13] \arrow[d, "\nr", shift left=13]                                                                                 \\
                \units{Z(\cQ^c(\cG))} \arrow[r, equals] \arrow[d, "\rsim"]                         & \units{Z(\cQ^c(\cG)\pl)} \times \units{Z(\cQ^c(\cG)\mi)} \arrow[d, "\rsim", shift right=18] \arrow[d, "\rsim", shift left=18]                                                                   \\
                \prod_{\chi \in \Irr_p(\cG)/{\sim}_W} \units{\cQ^c(\Gamma_\chi)} \arrow[r, equals] & \prod_{\chi \in \Irr_p(\cG)\pl/{\sim}_W} \units{\cQ^c(\Gamma_\chi)} \times \prod_{\chi \in \Irr_p(\cG)\mi/{\sim}_W} \units{\cQ^c(\Gamma_\chi)}
            \end{tikzcd}
        \end{center}
        This corresponds essentially to the discussion in the first part of subsection \ref{subsec:functoriality_1} for the particular case $\widetilde{H} = \ideal{\tau}$. Recall in particular lemma \ref{lem:functoriality_1_chi_parts} and the first few lines following remark \ref{rem:projection_gamma_chi_independent}. Similarly, the relative $K_0$ group decomposes canonically as
        \[
            K_0(\Lambda(\cG), \cQ(\cG)) = K_0(\Lambda(\cG)\pl, \cQ(\cG)\pl) \times K_0(\Lambda(\cG)\mi, \cQ(\cG)\mi).
        \]

        Although it will be stated in precise terms shortly, the shape of the ``Main Conjecture on minus parts'' should be apparent from the above discussion. There is no need, however, for an analogue of the Interpolation Conjecture: it follows from the well-known existence of $p$-adic $L$-functions and a recent result of Ellerbrock and Nickel. We briefly recall these concepts now.

        The \textbf{$p$-adic cyclotomic character}\index{cyclotomic character}\index{character!cyclotomic} of $\cG$ (introduced in more generality in section \ref{sec:iwasawa_algebras_and_modules}) is the homomorphism
        \[
            \chi_\tcyc \colon \cG = \Gal(L_\infty/K) \sa \Gal(K(\mu_{p^\infty})/K) \ia \units{\ZZ_p} = \ZZ_p^{(1)} \times \mu_{p - 1},
        \]
        where $\ZZ_p^{(1)} = 1 + p \ZZ_p$ and $\mu_{p - 1}$ denote the principal units and the $(p - 1)$-th roots of unity in $\ZZ_p$, respectively. Since complex conjugation acts on roots of unity as inversion, one has $\chi_\tcyc(\tau) = -1$. The projection of
        $\chi_\tcyc$ to $\ZZ_p^{(1)}$ factors as a character of $\Gal(K_\infty/K) = \Gamma_K = \overline{\ideal{\gamma_K}}$ and is often denoted by $\kappa$. The projection of $\chi_\tcyc$ to $\mu_{p - 1}$ factors as a character of $\Gal(K(\zeta_p)/K)$ and its usual notation is $\omega$: the \textbf{Teichmüller character}\index{character!Teichmüller}.

        Let $\chi \in \Irr_p(\cG)\pl$ be a totally even character of $\cG$. Then there exists a unique $p$-adic meromorphic function $L_{p, S}(\chi, -) \colon \ZZ_p \to \CC_p$ such that, for all integers $r > 1$, one has
        \begin{equation}
        \label{eq:p-adic_l_function}
            L_{p, K, S}(\chi, 1 - r) = L_{K, S}(\chi \omega^{-r}, 1 - r).
        \end{equation}
        The right-hand side should be interpreted as $\beta^{-1}(L_{K, S}(\beta(\chi \omega^{-r}), 1 - r))$ for an arbitrarily chosen $\beta \colon \CC_p \isoa \CC$. This notation, which is standard, reflects the classical fact - due to Siegel \cite{siegel} - that the outcome is independent of that choice. The function $L_{p, K, S}(\chi, -)$ is known as the \textbf{$S$-truncated $p$-adic $L$-function attached to $\chi$}\index{p-adic L-function@$p$-adic $L$-function}. Its construction in the abelian case (i.e. for linear $\chi$) was given, independently, by Pierrette Cassou-Nogu\`es in \cite{cassou-nogues}; Deligne and Ribet in \cite{deligne_ribet}; and Barsky in \cite{barsky}. Interpolation \eqref{eq:p-adic_l_function} at 0 (that is, for $r = 1$), which is the point of interest to us, was shown to hold later on.

        Greenberg extended the definition of $L_{p, K, S}$ beyond the abelian case by means of Brauer induction (cf. \cite{greenberg}). Note that interpolation at 0 does not follow automatically from the linear case due to the potential vanishing of $L$-values in the denominator. However, Ellerbrock and Nickel have shown it to hold through the study of bounds on denominators of Stickelberger elements along cyclotomic towers. More precisely, \eqref{eq:p-adic_l_function} holds for $r = 1$ and arbitrary $\chi \in \Irr_p(\cG)\pl$ by \cite{ellerbrock_nickel} theorem 3.

        A crucial fact about $p$-adic $L$-functions is their relation to series quotients. Namely, in the above notation, there exists a $\Phi_{p, K, S, \chi} \in \QQ_p^c \otimes_{\QQ_p} \ffrac(\ZZ_p[[T]])$ such that, for all $s \in \ZZ_p$, one has
        \[
            L_{p, K, S}(\chi, 1 - s) = \Phi_{p, K, S, \chi}(\kappa(\gamma_K)^s - 1).
        \]
        These series quotients satisfy the usual functoriality properties and exhibit the following behaviour under $W$-twists: for any $\rho \in \Irr_p(\cG)$ of type $W$, there is an equality
        \begin{equation}
        \label{eq:p-adic_l-function_twist}
             \Phi_{p, K, S, \chi \otimes \rho}(T) =  \Phi_{p, K, S, \chi}(\rho(\gamma_K)(T + 1) - 1)
        \end{equation}
        (see for instance \cite{rwii} p. 563). It is therefore natural to identify $T$ with $T_K = \gamma_K - 1$ and thus regard $\Phi_{p, K, S, \chi}$ as an element of $\QQ_p^c \otimes_{\QQ_p} \ffrac(\ZZ_p[[T_K]]) = \cQ^c(\Gamma_{K})$, since then the above equation becomes
        \[
            \Phi_{p, K, S, \chi \otimes \rho} =  \rho^\sharp(\Phi_{p, K, S, \chi})
        \]
        in the notation of section \ref{sec:evaluation_maps}.

        In order to relate these power series to the Interpolation Conjecture on minus parts, we introduce the continuous $\ZZ_p$-algebra automorphism $t_\tcyc \colon \Lambda(\cG) \to \Lambda(\cG)$ uniquely determined by $t_\tcyc(g) = \chi_\tcyc(g) g$ for all $g \in \cG$. This is a particular case of the $t_\tcyc^s$ from \cite{jn_brumer} section 6.1. It is easy to verify it induces field automorphisms on $\cQ(\cG)$ and $\cQ^c(\cG)$, and therefore group automorphisms on $K_1(\cQ(\cG))$, $K_1(\cQ^c(\cG))$ and $\units{Z(\cQ^c(\cG))}$ - all of which we denote by $t_\tcyc$ as well. Note that $t_\tcyc$ swaps plus and minus parts, as $t_\tcyc(e\pl) = e\mi$ and $t_\tcyc(e\mi) = e\pl$.

        Recall also the involution $\iota$ of $\Lambda(\cG)$ given by $g \mapsto g^{-1}$ introduced immediately before proposition \ref{prop:isomorphism_complexes_derived_extension}, which induces a map $\iota \colon \units{Z(\cQ^c(\cG))} \to \units{Z(\cQ^c(\cG))}$. Since the dual $\check{\chi}$ of a character $\chi \in \Irr_p(\cG)$ is precisely given by $\check{\chi}(g) = \chi(g^{-1})$, a simple argument shows that $\iota(\gamma_{\chi}) = \units{\gamma_{\check{\chi}} \in Z(\cQ^c(\cG))}$.

        The following is a direct application of some results in \cite{jn_brumer}:
        \begin{lem}
        \label{lem:ic_totally_real_case}
            Setting \ref{sett:rw}. There exists a unique element $F_{S, \cG}\pl \in \units{Z(\cQ(\cG)\pl)} \subseteq \units{Z(\cQ(\cG))}$ such that, for any $\chi \in \Irr_p(\cG)\pl$, one has
            \[
                j_\chi(F_{S, \cG}\pl) = \Phi_{p, K, S, \chi} \in \units{\cQ^c(\Gamma_K)},
            \]
            where $j_\chi \colon Z(\cQ^c(\cG)) \to \units{\cQ^c(\Gamma_K)}$ denotes the map $j_\chi^c$ from \eqref{eq:psi_and_j}.

            For each $\psi \in \Irr_p(\cG)\mi$, let $F_{S, \psi} = \iota(t_\tcyc(F_{S, \cG}\pl))e_\psi \in \units{\cQ^c(\Gamma_\psi)} \subseteq \units{Z(\cQ^c(\cG)\mi)} \subseteq \units{Z(\cQ^c(\cG))}$. Then one has
            \begin{equation}
                ev_{\gamma_\psi}(F_{S, \psi}) = L_{K, S}(\check{\psi}, 0)
            \end{equation}
            for all such $\psi$, where the notation of the right-hand is as in \eqref{eq:p-adic_l_function}.
        \end{lem}

        \begin{proof}
             The first claim takes place entirely on plus parts and a brief proof can be found in \cite{nickel_plms} p. 1231. The argument is straightforward: the twisting property \eqref{eq:p-adic_l-function_twist} together with the $\Map^W$-description \eqref{eq:isomorphism_hom_description} of $\units{Z(\cQ^c(\cG\pl))} \iso \units{Z(\cQ^c(\cG)\pl)}$ immediately yield a unique $F_{S, \cG}\pl \in \units{Z(\cQ^c(\cG)\pl)}$ with the asserted property. The fact that $F_{S, \cG}\pl$ actually lies in $\units{Z(\cQ(\cG)\pl)}$, i.e. it is $G_{\QQ_p}$-invariant (cf. \eqref{diag:centres_galois_invariants}), is a direct consequence of behaviour of the $G_{\QQ_p}$-action on $\Phi_{p, K, S, \chi}$ (see for instance \cite{rwii} p. 563 property (5) and p. 558 equation (*)).

             Now the second claim follows from
             \begin{align*}
                ev_{\gamma_\psi}(F_{S, \psi}) & = ev_{\gamma_K}(j_\psi(\iota(t_\tcyc(F_{S, \cG}\pl)))) \\
                & = ev_{\gamma_K}(j_{\check{\psi}}(t_\tcyc(F_{S, \cG}\pl))) \\
                & = ev_{\gamma_K}(j_{\check{\psi}\omega}^1(F_{S, \cG}\pl)) \\
                & = L_{p, S}(\check{\psi}\omega, 1 - 1) \\
                & = L_{K, S}(\check{\psi}\omega\omega^{-1}, 0) \\
                & = L_{K, S}(\check{\psi}, 0).
             \end{align*}
             The map $j_{\check{\psi}\omega}^1$ is defined in \cite{jn_brumer} section 6.1. It will not appear again, so we shall dispense with its definition. The two equalities involving it (third and fourth) are lemma 6.1 and equation (6.4) in the cited article, whereas the remaining equalities are clear from the preceding discussion. Note at this point that $\check{\psi}\omega$ is even, as the Teichmüller character $\omega$ is odd.
        \end{proof}

        \begin{rem}
        \phantomsection
        \label{rem:ic_totally_real_case}
            \begin{enumerate}[i)]
                \item{
                    Let $\rho$ be a type-$W$ character of $\cG$. Then $F_{S, \psi}$ and $F_{S, \psi \otimes \rho}$ coincide, since so do $e_\psi$ and $e_{\psi \otimes \rho}$.
                }
                \item{
                    For almost all $\psi \in \Irr_p(\cG)\mi$, the term $L_{K, S}(\check{\psi}, 0)$ is precisely the regulated leading coefficient in the right-hand side of the interpolation property \eqref{eq:ic} (when $\psi = \chi \otimes \rho$). The reason for the apparent lack of $\beta$ was explained after \eqref{eq:p-adic_l_function}. As $L_{K, S_\infty}(\check{\psi}, s)$ never vanishes at $s = 0$ (for $\psi$ totally odd), the leading coefficient $L_{K, S}^\ast(\check{\psi}, 0)$ coincides with the value $L_{K, S}(\check{\psi}, 0)$ whenever the Euler factors introduced by passing from $L_{K, S_\infty}$ to $L_{K, S}$ are non-zero, which is indeed the case for almost all $\psi$ (this follows from a similar argument to lemma \ref{lem:change_s_vanishing_euler_factors}). By virtue of lemma \ref{lem:properties_of_L-functions} v), $L_{K, S}(\check{\psi}, 0) \neq 0$ implies that the $\CC_p$-vector space $\Hom_{\CC_p[\cG_n]}(V_\psi, \CC_p \otimes_{\ZZ_p} \cX_{L_n, S}))$ featured in the definition of the regulator (definition \ref{defn:regulator}) has dimension 0. In particular, $R_S^\beta(\alpha, \psi) = 1$ for any valid choice of $\alpha$ and $\beta$.
                }
                \item{
                    Although not necessary for our purposes, the $G_{\QQ_p}$-invariance of $\prod_{\psi \in \Irr_p(\cG)\mi/{\sim}_W} F_{S, \psi} \in \units{Z(\cQ^c(\cG)\mi)}$ follows immediately from that of $F_{S, \cG}\pl$. In other words, the former lies in $\units{Z(\cQ(\cG)\mi)}$.
                }
            \end{enumerate}
        \end{rem}

        Together with points i) and ii) above, the lemma immediately shows that \hyperref[conje:ic]{IC($L_\infty/K, \psi, L, S, \varnothing, \alpha, \beta$)} holds for any $\psi \in \Irr_p(\cG)\mi$ and any (irrelevant) choice of $L$, $\alpha$ and $\beta$ (one may also add a set $T$ by \ref{cor:independence_of_s_and_t}). This covers the analytic side of our current pursuits. Before delving into the algebraic side, however, we introduce a piece of notation and present the conjecture of Ritter and Weiss.

        Let $\widetilde{R}$ be an integral domain and $\widetilde{S} = \ffrac(\widetilde{R})$ its field of fractions. Suppose given an $\widetilde{R}$-order $R$ inside an $\widetilde{S}$-algebra\footnote{There should be no ambiguity between the ring $S$ and the set of places $S$. In this section, the former (whose notation is carried over from section \ref{sec:an_integral_trivialisation}) appears only in the current paragraph, and the latter only outside of it.} $S = \widetilde{S} \otimes_{\widetilde{R}} R$. Assume furthermore that $S$ is semisimple. Then the canonical embedding $R \ia S$ satisfies the conditions of the $\varphi$ in section \ref{sec:an_integral_trivialisation} - that is, $S$ is flat as an $R$-module. A perfect complex $\cC\q$ of $R$-modules with $\widetilde{R}$-torsion cohomology admits only one trivialisation over $S$, namely the zero map. Accordingly, we set
        \begin{equation}
        \label{eq:class_zero_trivialisation}
            [\cC\q] = \chi_{R, S}(\cC\q, 0) \in K_0({R, S}),
        \end{equation}
        where the right-hand side denotes the refined Euler characteristic defined in \eqref{eq:definition_rec_general}. If $M$ is an $\widetilde{R}$-torsion $R$-module of finite projective dimension over $R$, then $M[0]$ is a perfect complex (isomorphic in $\cD(R)$ to any finite projective resolution of $M$) and one may define
        \[
            [M] = [M[0]] = \chi_{R, S}(M[0], 0) \in K_0({R, S}).
        \]
        It can be shown that, for any short exact sequence of complexes $\cC_1\q \ia \cC_2\q \sa \cC_3\q$ with each $\cC_i\q$ as $\cC\q$ above, one has $[\cC_2\q] = [\cC_1\q] + [\cC_3\q]$ in $K_0({R, S})$. Note that this is a particular instance of the additivity of refined Euler characteristics which appeared in subsection \ref{subsec:the_choice_of_s_and_t}. An analogous relation holds for any short exact sequence of $R$-modules $M_1 \ia M_2 \sa M_3$ with each $M_i$ as $M$ above.

        Let us specialise to our case of interest $\widetilde{R} = \Lambda(\Gamma)$ in the notation of setting \ref{sett:rw}. The relevant $\Lambda(\Gamma)$-orders are $\Lambda(\cG)$ and its plus and minus parts. In a few lines, the localisations of these objects at height-one prime ideals of $\Lambda(\Gamma)$ will be considered as well. Let $\cT_{S}\q(\cG\pl)$ denote the global complex constructed in section \ref{sec:the_global_complex} for the specific extension $L_\infty\pl/K$, which was shown to be perfect in proposition \ref{prop:complexes_are_perfect}. Its cohomology (determined in proposition \ref{prop:global_complex}) is indeed $\Lambda(\Gamma)$-torsion: this is clear for $H^1(\cT_{S}\q(\cG\pl)) = \ZZ_p$, and it follows from theorem \ref{thm:weak_leopoldt} (the weak Leopoldt conjecture) and the fact that $L\pl$ is totally real for $H^0(\cT_{S}\q(\cG\pl)) = X_S(\cG\pl)$ by \cite{nsw} 10.3.22. Therefore, $\cT_{S}\q(\cG\pl)$ defines a class $[\cT_{S}\q(\cG\pl)] \in K_0(\Lambda(\cG\pl), \cQ(\cG\pl))$ as explained above. This accounts for the lack of a trivialisation, or an analogous map to $\alpha$ (cf. setting \ref{sett:formulation}), in \cite{rwii}. In the context of setting \ref{sett:rw}, the \textbf{Main Conjecture of Ritter and Weiss}\index{equivariant Main Conjecture!of Ritter and Weiss} is the following assertion:

        \begin{conje*}[RW($L_\infty\pl/K, S)$]
        \phantomsection
        \label{conje:rw}
            Setting \ref{sett:rw}.

            There exists an element $\zeta_S\pl \in K_1(\cQ(\cG\pl))$ such that $\partial(\zeta_S\pl) = [\cT_{S}\q(\cG\pl)] \in K_0(\Lambda(\cG\pl), \cQ(\cG\pl))$ and $\nr(\zeta_S\pl) = F_{S, \cG}\pl \in \units{Z(\cQ(\cG)\pl)}$, where $F_{S, \cG}\pl$ is as in lemma \ref{lem:ic_totally_real_case}.
        \end{conje*}

        \begin{rem}
        \phantomsection
        \label{rmk:conje_rw}
            \begin{enumerate}[i)]
                \item{
                    The original conjecture in \cite{rwii} p. 564 is formulated in a slightly more general setting than the above. Namely, it is only required that $L_\infty\pl$ be the cyclotomic $\ZZ_p$-extension of a totally real number field $K'$ containing $K$ - under no assumption that $K'$ is of the form $L\pl$. However, the possibility of a direct comparison with our Main Conjecture seems unclear in that generality, as it will shortly become apparent that the argument relating the two relies on the notion of plus and minus parts, as well as the cyclotomic character and Tate twists. We will briefly touch on this topic again after theorem \ref{thm:rw}.
                }
                \item{
                    The assertion in \hyperref[conje:rw]{RW($L_\infty\pl/K, S)$} is not identical to, but rather an exact reformulation of, that in \cite{rwii}. This follows from the argument in \cite{nickel_plms} p. 1231. Essentially, our $F_{S, \cG}\pl \in \units{Z(\cQ(\cG)\pl)}$ corresponds to the $L_{k, S} \in \Hom^\ast(R_p(\cG\pl), \units{\cQ^c(\Gamma_K)})$ of Ritter and Weiss by theorem 8 in their article; and the $\widetilde{\mho}_S$ therein coincides with $[\cT_{S}\q(\cG\pl)]$ by equation \eqref{eq:iso_global_rhom} and \cite{nickel_plms} theorem 2.4.
                }
                \item{
                    The conjecture in \cite{rwii} includes a uniqueness quantifier on $\zeta_S\pl$ (such that $\nr(\zeta_S\pl) = F_{S, \cG}\pl$). The relation between the resulting claim and \hyperref[conje:rw]{RW($L_\infty\pl/K, S)$} mirrors that between the equivariant Main Conjectures with and without uniqueness from section \ref{sec:the_main_conjecture}. Since this has already been discussed (see for instance remark \ref{rem:emcu} i)), we limit ourselves to the version formulated above.
                }
            \end{enumerate}
        \end{rem}

        Our aim is to prove the equivalence of \hyperref[conje:emc]{eMC($L_\infty/K, L, S, T, \alpha, \beta$)} on \textit{minus parts} and \hyperref[conje:rw]{RW($L_\infty\pl/K, S)$}. Consider first the minus part of the main complex $\cC_{S, \varnothing}\q$ constructed in section \ref{sec:the_main_complex} for the extension $L_\infty/K$ and $T = \varnothing$. More formally, this refers to the perfect complex of $\Lambda(\cG)\mi$-modules
        \[
            \cC_S^{\bullet, -} = \Lambda(\cG)\mi \otimes_{\Lambda(\cG)}^\LL \cC_{S, \varnothing}\q.
        \]
        By the exactness of the minus-parts functor, $\Lambda(\cG)\mi \otimes_{\Lambda(\cG)}^\LL -$ amounts to degree-wise $\Lambda(\cG)\mi \otimes_{\Lambda(\cG)} -$ and there are equalities $H^i(\cC_S^{\bullet, -}) = H^i(\cC_{S, \varnothing}\q)\mi$ at all degrees $i$. Theorem \ref{thm:cohomology_of_complex} therefore implies $H^0(\cC_S^{\bullet, -}) \iso E_{S, \varnothing}\mi = E_S\mi$, as well as the existence of a short exact sequence
        \begin{equation}
        \label{eq:ses_h1_minus_parts}
            0 \to X_S^{cs, -} \to H^1(\cC_S^{\bullet, -}) \to \cX_S\mi \to 0.
        \end{equation}
        The term $E_S\mi$ can be determined using \cite{nsw} theorem 11.3.11 ii). Set $\widetilde{\cG} = \Gal(L_\infty/L\pl)$. Since $e\mi \in \Lambda(\widetilde{\cG})$, it is enough to consider the module structure of $E_S$ over this Iwasawa algebra (this is necessary due to the setup of section XI\S3 in the reference). One then has an isomorphism $E_S \iso \ZZ_p(1) \oplus (\Ind_{\ideal{\tau}}^{\widetilde{\cG}} \ZZ_p)^{\abs{S_\infty(L\pl)}}$. The summand $\ZZ_p(1)$ is the inverse limit of the $p$-power roots of unity along $L_\infty/L$, and lies in $E_S\mi$ because complex conjugation $\tau$ acts on it as $-1$. The second summand vanishes on minus parts, as
        \[
            e\mi \Ind_{\ideal{\tau}}^{\widetilde{\cG}} \ZZ_p = e\mi \Lambda(\widetilde{\cG}) \ctp_{\Lambda(\ideal{\tau})} \ZZ_p = \Lambda(\widetilde{\cG}) \ctp_{\Lambda(\ideal{\tau})} e\mi \ZZ_p = 0.
        \]
        Hence, $H^0(\cC_S^{\bullet, -}) \iso E_S\mi \iso \ZZ_p(1)$ (also as a $\Lambda(\cG)$-module) is $\Lambda(\Gamma)$-torsion. As for sequence \eqref{eq:ses_h1_minus_parts}, one has
        \[
            \cX_S\mi = \cY_S\mi = \cY_{S_f}\mi
        \]
        by the same argument as above. Both $X_S^{cs, -}$ and $\cX_S\mi = \cY_{S_f}\mi$ are $\Lambda(\Gamma)$-torsion as well: the former is so even before taking minus parts, as explained shortly after \eqref{eq:integral_trivialisation}, whereas the latter has finite $\ZZ_p$-rank.

        The upshot is that the main complex $\cC_{S, \varnothing}\q$ defines a class $[\cC_S^{\bullet, -}] \in K_0(\Lambda(\cG)\mi, \cQ(\cG)\mi) \subseteq K_0(\Lambda(\cG), \cQ(\cG))$ on minus parts without the need for a map $\alpha$. In light of this, as well as remark \ref{rem:ic_totally_real_case}, the minus part of \hyperref[conje:emc]{eMC($L_\infty/K, L, S, \varnothing, \alpha, \beta$)} is precisely the following assertion, which we unoriginally refer to as the \textbf{equivariant Main Conjecture on minus parts}\index{equivariant Main Conjecture!on minus parts}:

        \begin{conje*}[eMC\textsuperscript{$-$}($L_\infty/K, S)$]
        \phantomsection
        \label{conje:emcmi}
            Setting \ref{sett:rw}.

            There exists an element $\zeta_S\mi \in K_1(\cQ(\cG)\mi)$ such that $\partial(\zeta_S\mi) = - [\cC_S^{\bullet, -}] \in K_0(\Lambda(\cG)\mi, \cQ(\cG)\mi)$ and, for any $\chi \in \Irr_p(\cG)\mi$, one has $\psi_\chi(\zeta_S\mi) = F_{S, \chi} \in \units{\cQ^c(\Gamma_\chi)}$, where $F_{S, \chi}$ is as in lemma \ref{lem:ic_totally_real_case} (with $\chi$ replaced by $\psi$).
        \end{conje*}

        The fundamental algebraic result which will allow us to relate \hyperref[conje:rw]{RW($L_\infty\pl/K, S)$} to \hyperref[conje:emcmi]{eMC\textsuperscript{$-$}($L_\infty/K, S)$} is (essentially) \cite{johnston_nickel} theorem A.8. Intuitively, it states that a relation of the form $\partial(x) = y, \nr(x) = z$, where $x \in K_1(\cQ(\cG))$, is preserved (meaning it still holds for a possibly different $x$) if one replaces $y$ by a \textit{similar enough} $y' \in K_0(\Lambda(\cG), \cQ(\cG))$. The $y$ of interest to us is $[\cT_{S}\q(\cG\pl)]$, but some brief recalls are in order before $y'$ can be introduced. Our reference is \cite{nsw} section V\S4.

        Given a finitely generated left $\Lambda(\cG)$-module $M$ and $i \in \NN$, we set $E^i(M) = \Ext_{\Lambda(\cG)}^i(M, \Lambda(\cG))$, that is, the $i$-th right derived functor of $\Hom_{\Lambda(\cG)}(-, \Lambda(\cG))$ applied to $M$. The left $\Lambda(\cG)$-action on $\Hom_{\Lambda(\cG)}(N, \Lambda(\cG))$ (for arbitrary $N$) given by $(\lambda f)(n) = f(n) \iota(\lambda)$, where $\iota$ is the involution of $\Lambda(\cG)$ introduced before \eqref{eq:involution_action}, endows $E^i(M)$ with a left $\Lambda(\cG)$-module structure. Some basic properties of these $\Ext$-groups include the following:
        \begin{enumerate}[i)]
            \item{
                There are $\Lambda(\Gamma)$-isomorphisms $\Ext_{\Lambda(\cG)}^i(M, \Lambda(\cG)) \iso \Ext_{\Lambda(\Gamma)}^i(M, \Lambda(\Gamma))$ for each $i \in \NN$ (see \cite{nsw} proposition 5.4.17 and recall that $\Lambda(\cG)$ is Noetherian).
            }
            \item{
                $E^0(M) = \Hom_{\Lambda(\cG)}(M, \Lambda(\cG))$, which is trivial if $M$ is $\Lambda(\Gamma)$-torsion by i).
            }
            \item{
                If $M$ is $\Lambda(\Gamma)$-torsion, then so is $E^1(M)$. If, in addition, $\mu(M) = 0$ (as defined after theorem \ref{thm:structure_theorem_iwasawa}), then $E^1(M) \iso \Hom_{\ZZ_p}(M, \ZZ_p)$ (cf. \cite{nsw} corollary 5.5.7).
            }
            \item{
                $E^i(M) = 0$ for all $i \geq 3$, since the global dimension of $\Lambda(\Gamma)$ is 2.
            }
        \end{enumerate}

        Returning to our complex of interest, it is possible to show that $\cC_{S, \varnothing}\q$ has a representative in $\cD(\Lambda(\cG))$ of the form $[\stackrel{0}{A} \to \stackrel{1}{B}]$ with $B$ projective and $\pd_{\Lambda(\cG)} A \leq 1$. Passing to minus parts, this implies the existence of an exact sequence
        \[
            0 \to \ZZ_p(1) \to A\mi \to B\mi \to H^1(\cC_S^{\bullet, -}) \to 0,
        \]
        where the middle arrow represents $\cC_S^{\bullet, -}$ in $\cD(\Lambda(\cG)\mi)$ (and also in $\cD(\Lambda(\cG))$). Now a straightforward homological computation shows that
        \[
            H^i(R\Hom_{\Lambda(\cG)}(\cC_S^{\bullet, -}, \Lambda(\cG))) =
            \begin{cases}
                E^1(H^1(\cC_S^{\bullet, -})), & i = 0\\
                E^1(\ZZ_p(1)) = \ZZ_p(-1), & i = 1\\
                0, & \text{otherwise}.
            \end{cases}
        \]

        The correct complex $y'$ to compare to $\cT_{S}\q(\cG\pl)$ is the Tate twist $R\Hom_{\Lambda(\cG)}(\cC_S^{\bullet, -}, \Lambda(\cG))(1)$. Here and below, the notation $-(r)$ on a complex of modules over $\Lambda(\cG)$ (or a related ring) formally refers to $\ZZ_p(r) \otimes_{\ZZ_p}^\LL -$. This amounts to applying $\ZZ_p(r) \otimes_{\ZZ_p} -$ (i.e. taking the usual $r$-th Tate twist) on each degree, and it commutes with cohomology.

        The last preparatory step is to show that the two aforementioned complexes are indeed similar enough - in a sense which is made precise by localisation. Given a height-one prime ideal $\fp$ of $\Lambda(\Gamma)$, let $\Lambda_\fp(\Gamma)$ denote the corresponding localisation and set $\Lambda_\fp(\cG) = \Lambda_\fp(\Gamma) \otimes_{\Lambda(\Gamma)} \Lambda(\cG)$. Localisation at $\fp$ constitutes a exact functor from the category of left $\Lambda(\cG)$-modules to that of $\Lambda_\fp(\cG)$-modules and induces a functor on the corresponding categories of cochain complexes in the usual way.

        \begin{prop}
        \label{prop:plus_and_minus_locally}
            Setting \ref{sett:rw}. For each height-one prime ideal $\fp$ of $\Lambda(\Gamma)$, one has
            \[
                [R\Hom_{\Lambda(\cG)}(\cC_S^{\bullet, -}, \Lambda(\cG))(1)_\fp] = [\cT_{S}\q(\cG\pl)_\fp]
            \]
            in $K_0(\Lambda_\fp(\cG), \cQ(\cG))$.
        \end{prop}

        \begin{proof}
            We first point out that the classes of the complexes in the statement are well defined in the sense of \eqref{eq:class_zero_trivialisation}. Indeed, $\Lambda_\fp(\cG)$ is a $\Lambda_\fp(\Gamma)$-order in the $\ffrac(\Lambda_\fp(\Gamma)) = \cQ(\Gamma)$-algebra $\cQ(\cG)$, and both complexes have $\Lambda_\fp(\Gamma)$-torsion cohomology - they do even before localising by the preceding discussion. Here we are tacitly using the fact that $E^1$ preserves torsionness.

            Consider the short exact sequence \eqref{eq:ses_h1_minus_parts}, whose last non-zero term was shown to coincide with $\cY_{S_f}\mi$. In the associated long exact $\Ext$-sequence
            \[
                \cdots \to E^0(X_S^{cs, -}) \to E^1(\cY_{S_f}\mi) \to E^1(H^1(\cC_S^{\bullet, -})) \to E^1(X_S^{cs, -}) \to E^2(\cY_{S_f}\mi) \to \cdots,
            \]
            the terms $E^0(X_S^{cs, -})$ and $E^2(\cY_{S_f}\mi)$ vanish. The former does by property ii) of $E^i$ above, whereas in the case of the latter, vanishing follows immediately from the length-one projective resolution of $\Ind_{\cG_v}^\cG \ZZ_p$ constructed in the proof of corollary \ref{cor:difference_euler_characteristics_change_s}. Thus, we obtain a short exact sequence on the $E^1$-terms.

            In order to apply Kummer duality, $E^1(X_S^{cs, -})$ must be replaced by $E^1(X_{nr}\mi)$, which can be achieved as follows: start at the five-term exact sequence of $\Lambda(\cG)$-modules
            \[
                0 \to E_{S_\infty} \to E_S \to \varprojlim_n \bigoplus_{w_n \in S_f(L_n)} \ZZ_p \cdot w_n \to X_{S_\infty}^{cs} \to X_S^{cs} \to 0,
            \]
            which is constructed in a completely analogous manner to \eqref{eq:five_term_sequence_enlarge_s} (note that $X_{S_\infty}^{cs}$ is precisely $X_{nr}$). As explained in remark \ref{rem:difference_galois_modules_s}, the non-$p$-adic places are irrelevant in the middle term, which therefore becomes $\cY_{S_p}$. Since $E_{S_\infty}\mi = E_S\mi = \ZZ_p(1)$ (apply \cite{nsw} theorem 11.3.11 as before), the sequence $\cY_{S_p}\mi \ia X_{nr}\mi \sa X_S^{cs, -}$ is exact. Using the same $\Ext$-sequence argument as above results in a new short exact sequence on $E^1$-terms which, when spliced with the previous one, yields
            \begin{equation}
            \label{eq:four_term_exact_sequence_minus_e}
                0 \to E^1(\cY_{S_f}\mi) \to E^1(H^1(\cC_S^{\bullet, -})) \to E^1(X_{nr}\mi) \sa E^1(\cY_{S_p}\mi) \to 0.
            \end{equation}

            Less effort is needed on plus parts. Let $(S_f \setminus S_p)^{p+}$ denote the set of $v \in S_f \setminus S_p$ such that $L\pl$ contains a primitive $p$-th root of unity locally at any (equivalently, all) prolongations of $v$. Then there exists a short exact sequence of $\Lambda(\cG)$-modules
            \begin{equation}
            \label{eq:exact_sequence_x_plus_parts}
                0 \to \bigoplus_{v \in (S_f \setminus S_p)^{p+}} \Ind_{\cG_v\pl}^{\cG\pl} (\ZZ_p(1)) \to X_S(\cG\pl) \to X_{S_\infty \cup S_p}(\cG\pl) \to 0
            \end{equation}
            by \cite{nsw} theorem 11.3.5. A simple computation shows that
            \[
                \bigoplus_{v \in (S_f \setminus S_p)^{p+}} \Ind_{\cG_v\pl}^{\cG\pl} (\ZZ_p(1)) = \cY_{(S_f \setminus S_p)^{p+}}\mi (1).
            \]
            Furthermore, $\cY_{\set{v}}\mi$ vanishes for any $v \in S_f \setminus S_p$ which does not lie in $(S_f \setminus S_p)^{p+}$: for any such $v$ and any arbitrarily chosen prolongation of $w$ of $v$ to $L$, the extension $L/L\pl$ is not trivial locally at $w$ ($L$ does contain $\zeta_p$), and therefore $\tau \in \Gal(L/L\pl)_w$. As a consequence, one may replace the first non-zero term of \eqref{eq:exact_sequence_x_plus_parts} by $\cY_{S_f \setminus S_p}\mi (1)$.

            The relation between \eqref{eq:four_term_exact_sequence_minus_e} and \eqref{eq:exact_sequence_x_plus_parts} comes in the form of \textit{Kummer duality}, which provides a pseudo-isomorphism
            \[
                X_{S_\infty \cup S_p}(\cG\pl)(-1) \xrightarrow{\approx} E^1(X_{nr}\mi)
            \]
            (see \cite{nsw} corollary 11.4.4 or, more in line with our notation, \cite{sharifi_notes} corollary 3.4.8). In order to prove the claim in the proposition, we distinguish two cases:
            \begin{itemize}
                \item{
                    \underline{If $\fp$ is the singular prime $p\Lambda(\Gamma)$}: Any $\Lambda(\cG)$-module which is finitely generated over $\ZZ_p$ vanishes after localisation at $\fp$ by the structure theorem \ref{thm:structure_theorem_iwasawa}. This is the case for the degree-1 cohomology modules of both $R\Hom_{\Lambda(\cG)}(\cC_S^{\bullet, -}, \Lambda(\cG))(1)$ and $\cT_{S}\q(\cG\pl)$. In particular, one has
                    \begin{align*}
                        [R\Hom_{\Lambda(\cG)}(\cC_S^{\bullet, -}, \Lambda(\cG))(1)_\fp] & = [H^0(R\Hom_{\Lambda(\cG)}(\cC_S^{\bullet, -}, \Lambda(\cG))(1)_\fp)] \\
                        & = [H^0(R\Hom_{\Lambda(\cG)}(\cC_S^{\bullet, -}, \Lambda(\cG)))(1)_\fp] \\
                        & = [E^1(H^1(\cC_S^{\bullet, -}))(1)_\fp]
                    \end{align*}
                    and $[\cT_{S}\q(\cG\pl)_\fp] = [H^0(\cT_{S}\q(\cG\pl))_\fp] = [X_S(\cG\pl)_\fp]$ in $K_0(\Lambda_\fp(\cG), \cQ(\cG))$. Note that the degree-0 cohomology modules have a fortiori finite projective dimension when localised at $\fp$, as localisation preserves perfection.

                    Localise now \eqref{eq:four_term_exact_sequence_minus_e} and \eqref{eq:exact_sequence_x_plus_parts} at $\fp$, which causes the terms of the form $\cY_-$ and $E^1(\cY_-)$ to vanish by their finite-generatedness over $\ZZ_p$. This results in
                    \[
                        [E^1(H^1(\cC_S^{\bullet, -}))(1)_\fp] = [E^1(X_{nr}\mi)(1)_\fp] = [X_{S_\infty \cup S_p}(\cG\pl)_\fp] = [X_S(\cG\pl)_\fp],
                    \]
                    the middle equality being Kummer duality, which concludes the proof of the singular case.
                }
                \item{
                    \underline{If $\fp$ is a regular (that is, not the singular) prime}: The localisation $\Lambda_\fp(\cG)$ is known to have the following property: any finitely generated $\Lambda_\fp(\cG)$-module has finite projective dimension (cf. \cite{swan}). In particular, the class in $K_0(\Lambda_\fp(\cG), \cQ(\cG))$ of a perfect complex with $\Lambda(\Gamma)_\fp$-torsion cohomology can be computed as an alternating sum of the classes of its cohomology modules. Using this fact, together with \eqref{eq:four_term_exact_sequence_minus_e} and \eqref{eq:exact_sequence_x_plus_parts}, Kummer duality and the obvious short exact sequence $\cY_{S_f \setminus S_p}\mi \ia \cY_{S_f}\mi \sa \cY_{S_p}\mi$ , yields the desired result.
                }
            \end{itemize}
        \end{proof}

        Now the main result of this section follows easily:

        \begin{thm}
        \label{thm:rw}
            Setting \ref{sett:rw}. Conjecture \hyperref[conje:emcmi]{eMC\textsuperscript{$-$}($L_\infty/K, S)$} holds if and only if \hyperref[conje:rw]{RW($L_\infty\pl/K, S)$} does.
        \end{thm}

        \begin{proof}
            Suppose first that $\zeta_S\pl \in K_1(\cQ(\cG\pl)) \subseteq K_1(\cQ(\cG))$ satisfies \hyperref[conje:rw]{RW($L_\infty\pl/K, S)$}, so ${\partial(\zeta_S\pl) = [\cT_{S}\q(\cG\pl)]}$ and $\nr(\zeta_S\pl) = F_{S, \cG}\pl$. Apply proposition \ref{prop:plus_and_minus_locally} above and \cite{johnston_nickel} theorem A.8 to conclude that there exists a $\tilde{\zeta}_S$ such that $\partial(\tilde{\zeta}_S) = [R\Hom_{\Lambda(\cG)}(\cC_S^{\bullet, -}, \Lambda(\cG))(1)]$ and $\nr(\tilde{\zeta}_S) = F_{S, \cG}\pl$.  We stress that hypothesis (ii) of the cited theorem may not be satisfied, but its proof shows that hypotheses (i) and (ii) can be replaced by the property in proposition \ref{prop:plus_and_minus_locally}.

            Set $\zeta_S\mi = \iota(t_\tcyc(\tilde{\zeta}_S)^t)$, where $t$ denotes the transpose. For all $\chi \in \Irr_p(\cG)\mi$, one has
            \[
                \psi_\chi(\zeta_S\mi) = \nr(\zeta_S\mi)e_\chi = \iota(t_\tcyc(F_{S, \cG}\pl)) e_\chi = F_{S, \psi}
            \]
            (see lemma \ref{lem:ic_totally_real_case} for the last equality).

            On the homological side, it is well known that $\partial(t_\tcyc(x)) = \partial(x)(-1)$ for any $x \in K_1(\cQ(\cG))$; and if, furthermore, $\partial(x) = [\cC\q]$ for a perfect $\Lambda(\cG)$-complex with torsion cohomology, then one has $\partial(\iota(x^t)) = - [R\Hom_{\Lambda(\cG)}(\cC\q, \Lambda(\cG))]$. It follows that $\partial(\zeta_S\mi) = - [\cC_S^{\bullet, -}]$, and hence $\zeta_S\mi$ satisfies \hyperref[conje:emcmi]{eMC\textsuperscript{$-$}($L_\infty/K, S)$}.

            A completely symmetric argument proves the converse.
        \end{proof}

        As explained in remark \ref{rmk:conje_rw} i), setting \ref{sett:rw} does not cover all cases for which the original conjecture of Ritter and Weiss can be formulated. However, if \hyperref[conje:rw]{RW($L_\infty\pl/K, S)$} holds in general, then so does their original conjecture by the functoriality properties of the latter. The above theorem shows that this is the case if \hyperref[conje:emcmi]{eMC\textsuperscript{$-$}($L_\infty/K, S)$} holds in full generality - which would in turn be implied by the equivariant Main Conjecture from section \ref{sec:the_main_conjecture}.

        We conclude with some immediate consequences of the last theorem:
        \begin{cor}
            Setting \ref{sett:rw}. Conjecture \hyperref[conje:emcmi]{eMC\textsuperscript{$-$}($L_\infty/K, S)$} holds in the following cases:
            \begin{enumerate}[i)]
                \item{
                    The $\mu$-invariant of $X_S(\cG\pl)$ vanishes. For instance, this is always the case if $L\pl$ is an abelian extension of $\QQ$.
                }
                \item{
                    $\cG$ has an abelian Sylow $p$-subgroup.
                }
            \end{enumerate}
        \end{cor}

        \begin{proof}
            In the case $\mu(X_S(\cG\pl)) = 0$, \hyperref[conje:rw]{RW($L_\infty\pl/K, S)$} is known by \cite{rw_on_the} or \cite{kakde_mc}. The second statement in part i) is the celebrated Ferrero-Washington theorem from \cite{ferrero_washington}. Part ii) follows from corollary 1.2 in \cite{johnston_nickel}.
        \end{proof}

        Recall at this point that the vanishing of the $\mu$-invariant for the cyclotomic $\ZZ_p$-extension of any number field was conjectured by Iwasawa.

\newpage

\appendix
\newpage

\chapter{Functoriality of refined Euler characteristics}
\label{app:functoriality_of_refined_euler_characteristics}

    We show that refined Euler characteristics (as defined in section \ref{sec:an_integral_trivialisation}) behave well with respect to derived extension and restriction of scalars. This is probably known to experts, as the proof amounts to a formal verification, but the author is not aware of any reference for the precise formulation required above. Given a ring homomorphism $f \colon R \to S$, we denote the corresponding extension- and restriction-of-scalars functors by $S \otimes_R - \colon \fM_R^l \to \fM_S^l$ and $\restr{-}{R} \colon \fM_S^l \to \fM_R^l$, respectively, where $\fM_R^l$ denotes the category of left $R$-modules and analogously for $\fM_S^l$. Recall the notation $K_0(\rho, \sigma) \colon K_0(R, S) \to K_0(R', S')$ and $K_0^\rest(\rho, \sigma) \colon K_0(R', S') \to K_0(R, S)$ from diagrams         \eqref{eq:functoriality_relative_k0_payoff} and \eqref{eq:localisation_sequence_restriction}.

    \begin{lem}
    \label{lem:functoriality_rec}
        Let
        \begin{center}
            \begin{tikzcd}
                R \arrow[d, "\rho"] \arrow[r, "\varphi"] & S \arrow[d, "\sigma"] \\
                R' \arrow[r, "\varphi'"]                 & S'
            \end{tikzcd}
        \end{center}
        be a commutative diagram of ring homomorphisms. Suppose that $S$ and $S'$ are semisimple, $S$ is flat as a right $R$-module via $\varphi$ and $S'$ is flat as a right $R'$-module via $\varphi'$. Then:

        \begin{enumerate}[i)]
            \item{
                Let $\cC\q$ be a perfect complex of left $R$-modules and $t$ a trivialisation, that is, an isomorphism $t \colon S \otimes_R H^\odd(\cC\q) \isoa S \otimes_R H^\even(\cC\q)$. Then $t$ induces a trivialisation
                \[
                    t' \colon S' \otimes_{R'} H^\odd(R' \otimes_R^\LL \cC\q) \to S' \otimes_{R'} H^\even(R' \otimes_R^\LL \cC\q)
                \]
                such that
                \[
                    K_0(\rho, \sigma)(\chi_{R, S}(\cC\q, t)) = \chi_{R', S'}(R' \otimes_R^\LL \cC\q, t').
                \]
            }
            \item{
                Assume that $R'$ is a finitely generated projective $R$-module via $\rho$. Let $\cC\q$ be a perfect complex of left $R'$-modules and $t \colon S' \otimes_{R'} H^\odd(\cC\q) \to S' \otimes_{R'} H^\even(\cC\q)$ a trivialisation. Then $\restr{\cC\q}{R}$ is a perfect complex of $R$-modules, $\restr{t}{S}$ is a trivialisation for it, and one has
                \[
                    K_0^\rest(\rho, \sigma)(\chi_{R', S'}(\cC\q, t)) = \chi_{R, S}(\restr{\cC\q}{R}, \restr{t}{S}).
                \]
            }
        \end{enumerate}
    \end{lem}

    \begin{proof}
        \begin{enumerate}[i)]
            \item{
                We assume first that $\cC\q$ is strictly perfect and denote it by $\cP\q$ to avoid confusion with the general case. The following facts will be relevant:
                \begin{enumerate}[1)]
                    \item{
                        $(R' \otimes_R^\LL \cP\q)^i = R' \otimes_R \cP^i$, where the left-hand side denotes the $i$-th cochain module of $R' \otimes_R^\LL \cP\q$. In other words, the complex $R' \otimes_R^\LL \cP\q$ is the result of applying $R' \otimes_R -$ to $\cP\q$ degree-wise (and also to the differentials). This is an immediate consequence of $\cP^i$ being projective for all $i$. Since $R' \otimes_R -$ sends finitely generated projective $R$-modules to finitely generated projective $R'$-modules, $R' \otimes_R^\LL \cP\q$ is strictly perfect as well. In particular, $(S' \otimes_{R'}^\LL (R' \otimes_R^\LL \cP\q))^i = S' \otimes_{R'} (R' \otimes_R \cP^i)$.
                    }
                    \item{
                        For the same reason as above, $(S \otimes_R^\LL \cP\q)^i = S \otimes_R \cP^i$. The semisimplicity of $S$ therefore implies
                        \[
                            (S' \otimes_S^\LL (S \otimes_R^\LL \cP\q))^i = S' \otimes_S^\LL (S \otimes_R^\LL \cP\q)^i = S' \otimes_S (S \otimes_R \cP^i).
                        \]
                    }
                    \item{
                        The functors $S' \otimes_S (S \otimes_R -)$ and $S' \otimes_{R'} (R' \otimes_R -)$ are canonically naturally isomorphic by the commutative diagram above, so we may identify them. Using 1) and 2) yields
                        \[
                            S' \otimes_S^\LL (S \otimes_R^\LL \cP\q) = S' \otimes_{R'}^\LL (R' \otimes_R^\LL \cP\q).
                        \]
                    }
                    \item{
                        $S \otimes_R -$, $S' \otimes_{R'} -$ and $S' \otimes_S -$ are additive and exact and hence commute with taking cocycles, coboundaries and cohomology when applied to a complex degree-wise.
                    }
                \end{enumerate}
                Note that we do not require $R'$ to be flat over $R$, and hence $H^i(R' \otimes_R^\LL \cP\q)$ might not coincide with $R' \otimes_R H^i(\cP\q)$. Some care is therefore needed in the order the subsequent operations.

                We first define the new trivialisation $t'$. By virtue of 1) - 4), there are canonical isomorphisms
                \begin{equation}
                \label{eq:derived_tensor_product_isomorphisms}
                    S' \otimes_S H^i(S \otimes_R^\LL \cP\q) \iso H^i(S' \otimes_S^\LL (S \otimes_R^\LL \cP\q)) = H^i(S' \otimes_{R'}^\LL (R' \otimes_R^\LL \cP\q)) \iso S' \otimes_{R'} H^i(R' \otimes_R^\LL \cP\q)
                \end{equation}
                for each $i \in \ZZ$. We let $t'$ be unique arrow making the diagram
                \begin{equation}
                \label{eq:trivialisation_extension_of_scalars}
                    \begin{tikzcd}
                    S \otimes_R H^\odd(\cP\q) \arrow[r, "t"] \arrow[d, equals]                        & S \otimes_R H^\even(\cP\q) \arrow[d, equals]            \\
                    H^\odd(S \otimes_R^\LL \cP\q) \arrow[r, "t"] \arrow[d, "S' \otimes_S \Id"]                & H^\even(S \otimes_R^\LL \cP\q) \arrow[d, "S' \otimes_S \Id"]   \\
                    S' \otimes_S H^\odd(S \otimes_R^\LL \cP\q) \arrow[d, "\rsim"] \arrow[r, "S' \otimes_S t"] & S' \otimes_S H^\even(S \otimes_R^\LL \cP\q) \arrow[d, "\rsim"] \\
                    S' \otimes_{R'} H^\odd(R' \otimes_R^\LL \cP\q) \arrow[r, "t'", dashed]                    & S' \otimes_{R'} H^\even(R' \otimes_R^\LL \cP\q)
                    \end{tikzcd}
                \end{equation}
                commute, where the vertical isomorphisms come from \eqref{eq:derived_tensor_product_isomorphisms}. Since $t$ is an isomorphism, so are $S' \otimes_S t$ and $t'$ in the obvious categories.

                The refined Euler characteristic $\chi_{R, S}(\cC\q, t)$ relies on a choice of splittings of certain short exact sequences (cf. \eqref{eq:refined_euler_characteristic_map} and preceding lines). Those choices induce analogous splittings for the complex $S' \otimes_{R'} (R' \otimes_R^\LL \cP\q)$ in a natural way: for each $i \in \ZZ$, let $\delta_i$ be a splitting for the sequence $B^i(S \otimes_R^\LL \cP\q) \ia Z^i(S \otimes_R^\LL \cP\q) \sa H^i(S \otimes_R^\LL \cP\q)$. Then the $S'$-homomorphism $\delta_i'$ induced by $S' \otimes_S \delta_i$ and the last two vertical arrows of
                \begin{center}
                    \begin{tikzcd}[row sep=large]
                        B^i(S \otimes_R^\LL \cP\q) \arrow[r, hook] \arrow[d, "S' \otimes_S \Id"] & Z^i(S \otimes_R^\LL \cP\q) \arrow[r, "\pi_i", two heads] \arrow[d, "S' \otimes_S \Id"]   & H^i(S \otimes_R^\LL \cP\q) \arrow[d, "S' \otimes_S \Id"] \arrow[l, "\delta_i", in=340, out=200]   \\
                        S' \otimes_S B^i(S \otimes_R^\LL \cP\q) \arrow[r, hook] \arrow[d, "\rsim"]            & S' \otimes_S Z^i(S \otimes_R^\LL \cP\q) \arrow[r, "S' \otimes_S \pi_i", two heads] \arrow[d, "\rsim"] & S' \otimes_S H^i(S \otimes_R^\LL \cP\q) \arrow[d, "\rsim"] \arrow[l, "S' \otimes_S \delta_i", in=340, out=200] \\
                        B^i(S' \otimes_{R'}^\LL (R' \otimes_R^\LL \cP\q)) \arrow[r, hook]                      & Z^i(S' \otimes_{R'}^\LL (R' \otimes_R^\LL \cP\q)) \arrow[r, "\pi_i'", two heads]                       & H^i(S' \otimes_{R'}^\LL (R' \otimes_R^\LL \cP\q)) \arrow[l, "\delta_i'", dashed, in=340, out=200]
                    \end{tikzcd}
                \end{center}
                (where the isomorphisms are analogous to \eqref{eq:derived_tensor_product_isomorphisms}) is a splitting of the bottom row. Likewise, we use the splitting $\tau_i$ of
                \begin{center}
                    \begin{tikzcd}
                        Z^i(S \otimes_R^\LL \cP\q) \arrow[r, hook] & (S \otimes_R^\LL \cP\q)^i \arrow[r, two heads] & B^ {i + 1}(S \otimes_R^\LL \cP\q) \arrow[l, "\tau_i", in=340, out=200]
                    \end{tikzcd}
                \end{center}
                to construct one for
                \begin{center}
                    \begin{tikzcd}
                        Z^i(S' \otimes_{R'}^\LL (R' \otimes_R^\LL \cP\q)) \arrow[r, hook] & (S' \otimes_{R'}^\LL (R' \otimes_R^\LL \cP\q))^i \arrow[r, two heads] & B^ {i + 1}(S' \otimes_{R'}^\LL (R' \otimes_R^\LL \cP\q)) \arrow[l, "\tau_i'", in=340, out=200, dashed].
                    \end{tikzcd}
                \end{center}

                If $\varphi_t \colon S \otimes_R \cP^\odd \xrightarrow{\sim} S \otimes_R \cP^\even$ and $\varphi_{t'} \colon S \otimes_{R'} (R' \otimes_R^\LL \cP\q)^\odd \xrightarrow{\sim} S' \otimes_{R'} (R' \otimes_R^\LL \cP\q)^\even$ are as in \eqref{eq:refined_euler_characteristic_map}, the diagram
                \begin{center}
                    \begin{tikzcd}
                        S \otimes_R \cP^\odd \arrow[r, "\varphi_t"] \arrow[d, "S' \otimes_S \Id"]                  & S \otimes_R \cP^\even \arrow[d, "S' \otimes_S \Id"]    \\
                        S' \otimes_S (S \otimes_R \cP^\odd) \arrow[d, "\rsim"] \arrow[r, "S' \otimes_S \varphi_t"] & S' \otimes_S (S \otimes_R \cP^\even) \arrow[d, "\rsim"] \\
                        S' \otimes_{R'} (R' \otimes_R^\LL \cP\q)^\odd \arrow[r, "\varphi_{t'}"]                         & S' \otimes_{R'} (R' \otimes_R^\LL \cP\q)^\even
                    \end{tikzcd}
                \end{center}
                can be easily checked to commute, and hence
                \begin{align*}
                    K_0(\rho, \sigma)(\chi_{R, S}(\cP\q, t))
                    = & K_0(\rho, \sigma)([\cP^\odd, \varphi_t, \cP^\even]) \\
                    = & [R' \otimes_R \cP^\odd, S' \otimes_S \varphi_t, R' \otimes_R \cP^\even] \\
                    = & [(R' \otimes_R^\LL \cP\q)^\odd, \varphi_{t'}, (R' \otimes_R^\LL \cP\q)^\even] = \chi_{R', S'}(R' \otimes_R^\LL \cP\q, t').
                \end{align*}

                We now treat the general case of an arbitrary perfect complex $\cC\q$ trivialised by some $t$ as in the statement. Let $\cP\q$ be a strictly perfect representative of $\cC\q$ and $q \colon \cC\q \isoa \cP\q$ an isomorphism in the derived category $\cD(R)$. This induces isomorphisms on cohomology $q_H^i \colon H^i(\cC\q) \xrightarrow{\sim} H^i(\cP\q)$ for all $i \in \ZZ$, as well as an isomorphism $R' \otimes_R^\LL q \colon R' \otimes_R^\LL \cC\q \isoa R' \otimes_R^\LL \cP\q$ in $\cD(R')$ (so $R' \otimes_R^\LL \cC\q$ is perfect). In particular, there exist $R'$-isomorphisms
                \[
                    (R' \otimes_R^\LL q)_H^i \colon H^i(R' \otimes_R^\LL \cC\q) \xrightarrow{\sim} H^i(R' \otimes_R^\LL \cP\q)
                \]
                in each degree $i$ (as mentioned above, one should not expect $(R' \otimes_R^\LL q)_H^i = R' \otimes_R q_H^i$). Consider the diagram
                \begin{center}
                    \begin{tikzcd}
                        S \otimes_R H^\odd(\cC\q) \arrow[r, "t"] \arrow[d, "S \otimes_R q_H^\odd"]                                                            & S \otimes_R H^\even(\cC\q) \arrow[d, "S \otimes_R q_H^\even"]                                           \\
                        S \otimes_R H^\odd(\cP\q) \arrow[r, "\tilde{t}", dashed] \arrow[d]                                                                      & S \otimes_R H^\even(\cP\q) \arrow[d]                                                                      \\
                        S' \otimes_{R'} H^\odd(R' \otimes_R^\LL \cP\q) \arrow[r, "\tilde{t}'", dashed] \arrow[d, "S' \otimes_{R'} (R' \otimes_R^\LL q)_H^\odd"] & S' \otimes_{R'} H^\even(R' \otimes_R^\LL \cP\q) \arrow[d, "S' \otimes_{R'} (R' \otimes_R^\LL q)_H^\even"] \\
                        S' \otimes_{R'} H^\odd(R' \otimes_R^\LL \cC\q) \arrow[r, "t'", dashed]                                                                  & S' \otimes_{R'} H^\even(R' \otimes_R^\LL \cC\q)
                    \end{tikzcd}
                \end{center}
                In the top square, the three solid arrows are isomorphisms and we define the dashed one by commutativity. The middle square is \eqref{eq:trivialisation_extension_of_scalars} (with its central square omitted). As for the bottom square, the two vertical arrows are isomorphisms and the bottom one is again determined by commutativity. This concludes the proof, since it implies
                \[
                    \chi_{R, S}(\cC\q, t) = \chi_{R, S}(\cP\q, \tilde{t}) = \chi_{R', S'}(\cP\q, \tilde{t}') = \chi_{R', S'}(\cC\q, t'),
                \]
                the middle equality being case treated above.
            }
            \item{
                The argument for restriction of scalars is simpler than that for extension. The functor $\restr{-}{R}$ induced by $\varphi$
                \begin{itemize}
                    \item{
                        is additive and exact.
                    }
                    \item{
                        sends finitely generated (resp. projective) $R'$-modules to finitely generated (resp. projective) $R$-modules.
                    }
                    \item{
                        sends quasi-isomorphisms to quasi-isomorphisms\footnote{It is not necessary to consider a left derived functor as we did with $R' \otimes_R^{\LL} -$: one can directly restrict scalars in the derived category. It is of note, however, that $\restr{\cC\q}{R}$ might have a larger isomorphism class in $\cD(R)$ than $\cC\q$ did in $\cD(R')$, since there may exist quasi-isomorphisms which are $R$-equivariant, but not $R'$-equivariant.}.
                    }
                    \item{
                        commutes with cohomology, coboundaries and cocycles.
                    }
                \end{itemize}
                The analogous statements for $\restr{-}{S}$ (via $\sigma$) hold too. By the commutativity of the diagram in the statement, we can identify the functors $S \otimes_R (\restr{-}{R}) = \restr{(S' \otimes_{R'} -)}{S} \colon \fM_{R'}^l \to \fM_S^l$ (more precisely, they are canonically naturally isomorphic). This explains why $\restr{t}{S}$ is a trivialisation for $\restr{\cC\q}{R}$.

                By the same formal argument as in part i), we can assume $\cC\q$ is strictly perfect, say $\cP\q$. The splittings used to define $\varphi_t \colon S' \otimes_{R'} \cP^\odd \to S' \otimes_{R'} \cP^\even$ induce splittings for the for their counterparts associated to $\restr{\cP\q}{R}$. This results in an $S$-isomorphism
                \[
                    \varphi_{(\restr{t}{S})} \colon S \otimes_R (\restr{\cP\q}{R})^\odd \isoa S \otimes_R (\restr{\cP\q}{R})^\even
                \]
                which coincides with $\restr{(\varphi_t)}{S}$. Therefore,
                \begin{align*}
                    K_0^\rest(\rho, \sigma)(\chi_{R', S'}(\cP\q, t)) & = K_0^\rest(\rho, \sigma)([\cP^\odd, \varphi_t, \cP^\even]) \\
                    & = [\restr{\cP^\odd}{R}, \restr{(\varphi_t)}{S}, \restr{\cP^\even}{R}] \\
                    & = [(\restr{\cP\q}{R})^\odd, \varphi_{(\restr{t}{S})}, (\restr{\cP\q}{R})^\even] = \chi_{R, S}(\restr{\cP\q}{R}, \restr{t}{S}).
                \end{align*}
            }
        \end{enumerate}
    \end{proof}


\chapter{Determinant functors}
\label{app:determinant_functors}

    We recall the notion of determinant functors in our setting of interest and prove two simple facts about their duals. The original, more general definition of these functors is due to Knudsen and Mumford (cf. \cite{km}), but our needs are limited to the case treated in \cite{weibel} p. 21 ff. Another good reference is \cite{ckv} p. 70 ff.

    Let a commutative ring $R$ be given which decomposes as a direct product of rings $R = \prod_{i = 1}^n R_i$. If we denote the unit element of $R_i$ by $e_i$, which we identify with its natural preimage in $R$, one has $1 = \sum_{i = 1}^n e_i$ and each $e_i$ is an idempotent of $R$. Every $R$-module $M$ decomposes as $M = \bigoplus_{i = 1}^n e_i M$, with $e_i M$ inheriting a natural $R_i$-module structure.

    A prime ideal $\fp$ of $R$ is necessarily of the form $\fp = R_1 \times \cdots \times R_{i - 1} \times \fp_i \times R_{i + 1} \cdots \times R_n$ for some $i$ and some prime ideal $\fp_i$ of $R_i$, which yields an identification of $\Spec(R)$ with the disjoint union $\bigsqcup_{i = 1}^n \Spec(R_i)$. For any $R$-module $M$, the localisation $M_\fp$ is canonically isomorphic to $(e_i M)_{\fp_i}$ as an $R_\fp \iso (R_i)_{\fp_i}$-module.

    Let $R$ be an arbitrary commutative ring and $P$ a finitely generated projective $R$-module. Then the rank function
    \begin{align*}
        \rank_P \colon \Spec(R) & \to \NN \\
                    \fp & \mapsto \rank_{R_{\fp}} P_\fp
    \end{align*}
    is locally constant (cf. \cite{weibel} p. 21) when the domain is endowed with the usual Zariski topology. By the compactness of $\Spec(R)$, there exists a decomposition $R = \prod_{i = 1}^n R_i$ such that $\restr{\rank_P}{\Spec(R_i)}$ is constant. We define the \textbf{determinant} \index{determinant functor!on modules} of $P$ as the $R$-module
    \[
        \Det_R(P) = \bigoplus_{i = 1}^n \bigwedge_{R_i}^{\rank_{P, i}} e_i P,
    \]
    where $e_i$ is as before, $\bigwedge$ is the usual exterior power and $\rank_{P, i}$ denotes the constant value of $\rank_P$ at $\Spec(R_i)$. Recall that, for any $R_i$-module $M$, the zero-th exterior power $\bigwedge_{R_i}^0 M$ is $R_i$ by convention. Different valid decompositions of $R$ result in isomorphic determinant modules, and hence $\Det_R(P)$ is well defined. A concise way to rephrase this is: $\Det_R(P)$ is locally defined as the highest exterior power. In order to avoid sign issues, one should regard $\Det_R(P)$ as a \textit{graded} invertible module (see the comment at the beginning of \cite{bks_on_zeta} section 3.2) with the grading being given by the rank function. We omit this grading from the notation at no risk of ambiguity, as it is univocally determined by the module $P$ itself.

    Since the $r$-th power of a rank-$r$ free module is free of rank 1 for any $r \in \NN$, it is easy to verify that $\Det_R(P)$ is projective and locally free of rank 1 - in other words, an invertible $R$-module. This implies that, letting $\Det_R(P)^{-1} = \Hom_R(\Det_R(P), R)$, the $R$-linear evaluation map
    \begin{align*}
        ev \colon \Det_R(P) \otimes_R \Det_R(P)^{-1} & \to R \\
        m \otimes f & \mapsto f(m)
    \end{align*}
    is an isomorphism.

    Suppose now that $P$ and $Q$ are two finitely generated projective $R$-modules which are locally of the same rank (an important example is the case $P \iso Q$), and choose a decomposition $R = \prod_{i = 1}^n R_i$ such that $\rank_P$ and $\rank_Q$ coincide on $\Spec(R_i)$ for all $i$. Let $P_i$ denote $e_i P$, and analogously for $Q_i$. Then, given an $R$-homomorphism $f \colon P \to Q$, which we identify with the sum $f = \bigoplus_{i = 1}^n f_i$ of the $R$-(or $R_i$-)homomorphisms $f_i \colon P_i \to Q_i$, we define
    \begin{align*}
        \Det(f_i) \colon & \bigwedge_{R_i}^{\rank_{P, i}} P_i  \to  \bigwedge_{Q_i}^{\rank_{Q, i}} Q_i \\
                        & \bigwedge_{j = 1}^{\rank_{P, i}} m_j \mapsto \bigwedge_{j = 1}^{\rank_{Q, i}} f_i(m_j)
    \end{align*}
    (extended by $R$-linearity) and\index{determinant functor!on homomorphisms}
    \begin{equation}
    \label{eq:definition_of_det}
        \Det_R(f) = \bigoplus_{i = 1}^n \Det(f_i) \colon \Det(Q) \to \Det(P).
    \end{equation}
    This construction satisfies the following:
    \begin{enumerate}[i)]
        \item{
            $\Det_R(f)$ is a homomorphism of $R$-modules and it is independent of the chosen decomposition of $R$.
        }
        \item{
            It is functorial: $\Det_R(\Id) = \Id$ and $\Det_R(g \circ f) = \Det_R(g) \circ \Det_R(f)$. In particular, it sends isomorphisms to isomorphisms and inverses to inverses.
        }
    \end{enumerate}

    The motivation behind the name of determinant functors is the following: suppose $P$ is a free $R$-module of finite rank $r$ and fix a basis $m^1, \ldots, m^r$. In particular, the map $\rank_P$ is identically $r$ on all of $\Spec(R)$ and
    \begin{equation}
    \label{eq:wedge_product_basis}
         \mu = \bigwedge_{j = 1}^{r} m^j
    \end{equation}
    is a generator of $\Det_R(P) = \bigwedge_R^r P$. It can then be shown that, for any endomorphism $f \in \End_R(P)$, the diagram
    \begin{equation}
    \label{eq:determinant_map_free_modules}
        \begin{tikzcd}[column sep=large]
            R \arrow[r, "\det(M_f)"] \arrow[d, "\mu"] & R \arrow[d, "\mu"] \\
            \Det_R(P) \arrow[r, "\Det_R(f)"]            & \Det_R(P)
        \end{tikzcd}
    \end{equation}
    commutes, where $\mu$ denotes the isomorphism $s \mapsto s\mu$ and $\det(M_f)$ is multiplication by the determinant of the matrix of $f$ in any basis (cf. \cite{weibel} p. 21). Note that the vertical arrows can a fortiori be replaced by any arbitrary $R$-homomorphism.

    Determinant functors behave well with respect to extension of scalars: a ring homomorphism $R \to S$ (with $R$ and $S$ commutative) induces a natural isomorphism $S \otimes_R \Det_R(P) \iso \Det_S(S \otimes_R P)$. They are furthermore additive in short exact sequences in the sense that $P' \xhookrightarrow{\iota} P \sa P''$ yields an isomorphism
    \begin{equation}
    \label{eq:additivity_det_functor}
        \Det_R(P') \otimes_R \Det_R(P'') \iisoo \Det_R(P)
    \end{equation}
    which is locally defined by mapping $(\bigwedge_{i = 1}^n x_i) \otimes (\bigwedge_{j = 1}^m y_j)$ to $(\bigwedge_{i = 1}^n \iota(x_i)) \wedge (\bigwedge_{j = 1}^m \sigma(y_j))$, where $\sigma$ is an arbitrarily chosen splitting of $P \sa P''$.

    The above notion can be extended to perfect complexes as follows: given a strictly perfect complex $\cP\q$ of $R$-modules, the \textbf{determinant of $\cP\q$}\index{determinant functor!on perfect complexes} is the module
    \[
        \Det_R(\cP\q) = \bigotimes_{i \in \ZZ} \Det_R(\cP^i)^{(-1)^i} = \Det_R(\cP^\even) \otimes_R \Det_R(\cP^\odd)^{-1},
    \]
    where the tensor product in the middle term is over $R$. If $\cC\q$ is now a perfect complex of $R$-modules, $\Det_R(\cC\q)$ is defined as $\Det_R(\cP\q)$ for a strictly perfect representative $\cP\q$ of $\cC\q$. This is in fact well defined (up to canonical isomorphism) regardless of the choice of $\cP\q$ (cf. \cite{ckv} p.71).

    This behaves well with respect to derived extension of scalars: consider $\cC\q$ and $\cP\q$ as above and a homomorphism of commutative rings $R \to S$. Then $S \otimes_R^\LL \cC\q$ is represented by $S \otimes_R^\LL \cP\q$, which is simply the result of applying $S \otimes_R -$ to $\cP\q$ degree-wise. Therefore, extension of scalars on determinant modules induces a natural map $\Det_R(\cC\q) \to \Det_S(S \otimes_R^\LL \cC\q)$.

    We conclude this appendix with two small lemmas used in section \ref{sec:the_conjecture_of_burns_kurihara_and_sano}:

    \begin{lem}
    \label{lem:invertible_module_dual}
        Let $R$ be a commutative ring and $M = \ideal{m}$ a cyclic invertible $R$ module. Then there exists a unique $m^\ast \in M^{-1} = \Hom_R(M, R)$ such that $m^\ast(m) = 1$. Furthermore, any $f \in M^{-1}$ satisfies $f = f(m) m^\ast$, and therefore $m^\ast$ generates $M^{-1}$ as an $R$-module.
    \end{lem}

    \begin{proof}
        By the invertibility of $M$, the $R$-module homomorphism $ev \colon M \otimes_R M^{-1} \to R$ which sends $x \otimes f$ to $f(x)$ is an isomorphism. Let $\sum_{i = 1}^n x_i \otimes f_i$ be the (unique) preimage of $1 \in R$ via $ev$. Since $m$ is a generator of $M$, for each $i$ there exists a scalar $r_i \in R$ such that $r_i m = x_i$. Hence $\sum_{i = 1}^n x_i \otimes f_i = m \otimes \sum_{i = 1}^n r_i f_i$. In particular,
        \[
            1 = ev\bigg(\sum_{i = 1}^n x_i \otimes f_i\bigg) = ev\bigg(m \otimes \sum_{i = 1}^n r_i f_i\bigg) = \bigg(\sum_{i = 1}^n r_i f_i\bigg)(m).
        \]
        This shows existence of $m^\ast$. As for uniqueness, assume $f, g \in \Hom_R(M, R)$ satisfy $f(m) = g(m) = 1$. Given any $x \in M$, there exists a scalar $r \in R$ with $rm = x$, and therefore $f(x) = r f(m) = r = g(x)$. Hence, $f = g$.

        For the last claim, we note that $f(m) = (f(m)m^\ast)(m)$ and, by the same argument as above, this implies $f = f(m) m^\ast$.
    \end{proof}

    In view of this lemma, the choice of a generator $m$ of an invertible module $M$ (if it exists) yields an $R$-isomorphism
    \begin{align}
    \label{eq:isomorphism_inverse_module}
        \iota_m \colon M & \xrightarrow{\sim} M^{-1} \\
                       m & \mapsto m^\ast.\nonumber
    \end{align}
    Note that $\iota_m$ does depend on the choice of $m$: if $m'$ is a different generator, say $m' = rm$ for some $r \in \units{R}$, then an immediate computation shows that $\iota_m = r^2 \iota_{m'}$.

    Any invertible $R$-module $M$ is automatically reflexive. That is, the canonical $R$-homomorphism
    \begin{align*}
        M & \to (M^{-1})^{-1} = \Hom_R(\Hom_R(M, R), R) \\
                       x & \mapsto [f \mapsto f(x)]
    \end{align*}
    is an isomorphism. This follows from the fact that it holds for free modules, and hence in particular for the localisation of $M$ at any prime ideal; and a map which becomes an isomorphism after localisation at every maximal ideal is an isomorphism globally too (cf. \cite{eisenbud} corollary 2.8). We often identify $M$ and $(M^{-1})^{-1}$, in light of which one has $(m^\ast)^\ast = m$ and $\iota_{m^\ast} = \iota_m^{-1}$ for any generator $m$ of $M$ (if it exists).

    The following application of this duality is relevant to our endeavours:

    \begin{lem}
    \label{lem:trivialising_dual_inverse}
        Let $R$ be a commutative ring and $f \colon M \xrightarrow{\sim} N$ an isomorphism of finitely generated projective $R$-modules (which are therefore locally of the same rank). Suppose $\Det_R(M)$ and $\Det_R(N)$ are generated over $R$ by $m$ and $n$, respectively. Consider the homomorphisms $\alpha$ and $\beta$ defined by the rows of
        \begin{equation}
        \label{diag:dualising_inverted_trivialisation}
            \begin{tikzcd}
                \alpha \colon \Det_R(M) \otimes_R \Det_R(N)^{-1} \arrow[d, "i_m \otimes i_{n^\ast}"] \arrow[rr, " \Det_R(f) \otimes \Id"] &  & \Det_R(N) \otimes_R \Det_R(N)^{-1} \arrow[r, "ev"] & R \\
                \beta \colon \Det_R(M)^{-1} \otimes_R \Det_R(N) \arrow[rr, "\Id \otimes \Det_R(f^{-1})"]                            &  & \Det_R(M)^{-1} \otimes_R \Det_R(M) \arrow[r, "ev"] & R
            \end{tikzcd}
        \end{equation}
        Then $\beta((i_m \otimes i_{n^\ast})(m \otimes n^\ast)) = \alpha(m \otimes n^\ast)^{-1} \in \units{R}$.
    \end{lem}

    \begin{proof}
        This is a straightforward computation. Since $\alpha$ is an $R$-isomorphism, it maps the generator $m \otimes n^\ast$ of $\Det_R(M) \otimes_R \Det_R(N)^{-1}$ to a unit $\alpha(m \otimes n^\ast) \in \units{R}$. By definition, one has
        \[
            (i_m \otimes i_{n^\ast})(m \otimes n^\ast) = m^\ast \otimes n \in \Det_R(M)^{-1} \otimes_R \Det_R(N)
        \]
        (note that we are implicitly using the duality $\iota_{n^\ast} = \iota_n^{-1}$ explained above).

        The element $n$ generates $\Det_R(N)$, and hence there exists an $r \in R$ such that $\Det_R(f)(m) = rn$ In particular, $\Det_R(f^{-1})(n) = \Det_R(f)^{-1}(n) = r^{-1} m$. Thus, on one hand,
        \[
            \alpha(m \otimes n^\ast) = ev(\Det_R(f)(m) \otimes n^\ast) = r \cdot ev(n \otimes n^\ast) = r;
        \]
        and on the other,
        \[
            \beta((i_m \otimes i_{n^\ast})(m \otimes n^\ast)) = \beta(m^\ast \otimes n) = ev(m^\ast \otimes \Det_R(f^{-1})(n)) = r^{-1} \cdot ev(m^\ast \otimes m) = r^{-1},
        \]
        as desired.
    \end{proof}

    All arrows in \eqref{diag:dualising_inverted_trivialisation} are isomorphisms, and therefore the diagram can be completed uniquely into a commutative square through a vertical arrow $R \to R$. The lemma shows this arrow is multiplication by $r^{-2}$, where $r = \alpha(m \otimes n^\ast) \in \units{R}$.

%
%
%
%
%


    \newpage
    \cleardoublepage
    \printbibliography[heading=bibintoc]


    \newpage
    \cleardoublepage
    \phantomsection
    \printindex

\end{document}